\documentclass[11pt]{article}
\usepackage{amsmath, amssymb, amsthm, esint, hyperref, cite}
\usepackage{geometry, mathabx}
\newcommand{\bd}{\boldsymbol}

\usepackage{xcolor}
\usepackage{graphicx} 

\def\R{{\mathbb R}}

\newcommand{\aalpha}{{\alpha}}
\newcommand{\ttt}{\nu}

\def\bu{{\bf u}}
\def\bq{{\bf q}}
\def\bv{{\bf v}}
\def\dt{{\Delta t}}

\theoremstyle{definition}

\newtheorem{theorem}{Theorem}[section]

\newtheorem{definition}{Definition}[section]
\newtheorem{proposition}{Proposition}[section]
\newtheorem{remark}{Remark}[section]
\newtheorem{lemma}{Lemma}[section]
\begin{document}

\title{\sc{Fluid-poroviscoelastic structure interaction problem with nonlinear {geometric} coupling}}
%\title{\sc{Well-posedness of a nonlinearly coupled fluid-poroviscoelastic structure interaction problem}}
\author{Jeffrey Kuan, Sun\v{c}ica \v{C}ani\'{c}, Boris Muha}
%\date{}
\maketitle

\begin{abstract}
{We investigate weak solutions to a fluid-structure interaction (FSI) problem between the flow of an incompressible, viscous fluid modeled by the Navier-Stokes equations, and a poroviscoelastic medium modeled by the Biot equations. These systems are coupled nonlinearly across an interface with mass and elastic energy, modeled by a reticular plate equation, which is transparent to fluid flow. We provide a constructive proof of the existence of a weak solution to 
a regularized problem. Next, a weak-{{classical}} consistency result is obtained, showing that the weak solution to the regularized problem converges, as the 
regularization parameter approaches zero, to a {{classical}} solution to the original problem, when {{such a classical}} solution exists. While the assumptions in the first step only require the Biot medium to be poroelastic, the second step requires additional regularity, namely, that the Biot medium is poroviscoelastic. This is the first weak solution existence result for {{an}} FSI problem with nonlinear coupling involving a Biot model for poro(visco)elastic media.}
\iffalse
We prove the existence of a weak solution to a fluid-structure interaction (FSI) problem between the flow of an incompressible, viscous fluid modeled by the Navier-Stokes equations, and a poroviscoelastic medium modeled by the Biot equations. The two are nonlinearly coupled over an interface with mass and elastic energy, modeled
by a reticular plate equation, which is transparent to fluid flow. The existence proof is constructive, consisting of two steps. First, the existence of a weak solution to 
a regularized problem is shown. Next, a weak-{{classical}} consistency result is obtained, showing that the weak solution to the regularized problem converges, as the 
regularization parameter approaches zero, to a {{classical}} solution to the original problem, when {{such a classical}} solution exists. While the assumptions in the first step only 
require the Biot medium to be poroelastic, the second step requires additional regularity, namely, that the Biot medium is poroviscoelastic. This is the first weak solution existence result for {{an}} FSI problem with nonlinear coupling involving 
a Biot model for poro(visco)elastic media. 
\fi
\end{abstract}

\section{Introduction and motivation}\label{intro}

In this paper we study a time-dependent nonlinearly coupled fluid-structure interaction problem between the flow of an incompressible, viscous fluid, modeled by the Navier-Stokes equations, and bulk poroviscoelasticity modeled by the Biot equations.
Bulk poroviscoelasticity means that the dimensions of the ``free fluid flow'' domain and the poroviscoelastic medium domain are the same. In particular, in this manuscript we consider {{a 2D fluid-poroelastic structure interaction (FPSI) problem, which captures the main mathematical difficulties of such coupling, see Fig.~\ref{domain}}}. 
The free fluid flow and the Biot poro(visco)elastic medium are coupled across the current location of the interface, {{which is modeled by a reticular plate that has inertia and elastic energy}}. A reticular plate is a lattice-type structure characterized by
two properties: periodicity and small thickness, where periodicity refers to periodic cells (holes) distributed in all directions \cite{CioranescuBook}.
The reticular plate interface is transparent to fluid flow. 
We are interested in the existence of finite energy weak solutions (of the Leray-Hopf type).

\begin{figure}[htp!]
\center
                    \includegraphics[width = 0.4 \textwidth]{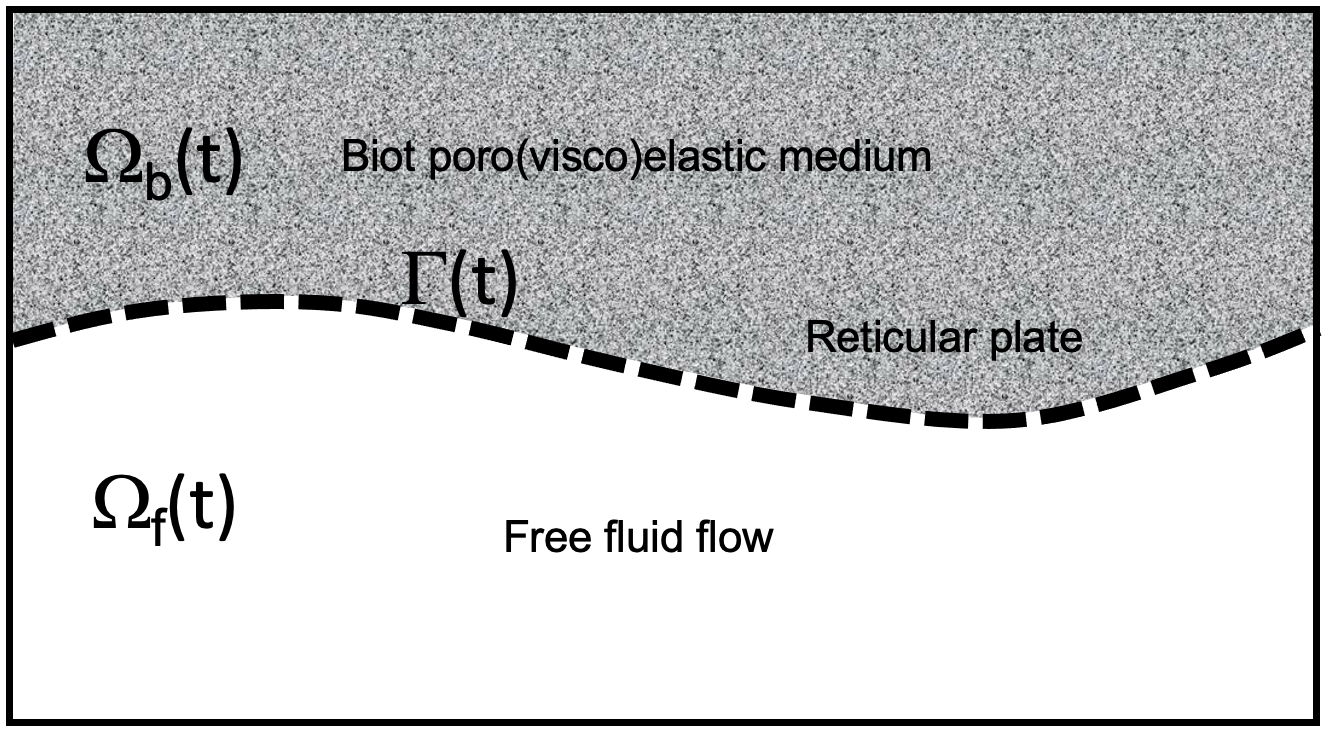}
  \caption{{\small\emph{A sketch of the fluid-poroelastic structure interaction domain. }}}
\label{domain}
\end{figure}

The problem we study here arises in many applications. In particular, we mention encapsulation of bioartificial organs \cite{FluidsCanic} and blood flow in arteries which are modeled as poro(visco)elastic media to study drug transport through the vascular walls \cite{YifanDES,BukacZuninoYotovPoroelastic,QuainiQuarteroniPoroelastic}. The reticular plate can be used to capture the elastodynamics behavior of the intima/elastic laminae layer of arterial walls which is in direct contact with the blood flow on one side, and a poroelastic medium 
consisting of the arterial media/adventitia complex on the other side. 

{From the mathematical point of view the primary difficulties in studying Navier-Stokes equations nonlinearly coupled to bulk poro(visco)elasticity} arise from the fact that 
the finite energy solutions do not posses sufficient regularity to (1) define the moving domain and the corresponding traces, and (2) guarantee that all the integrals in the weak formulation of the problem are well-defined. The first issue is related to the difficulties associated with {fluid-structure coupling, where the fluid and structure domains are of the same dimension}. 
The second issue is a consequence of the geometric nonlinearities associated with moving domain problems.
These are the main reasons why to this day there have been no works on the existence of weak solutions for the Biot-Navier-Stokes coupled problems in which the coupling is assumed over a moving interface. 

To get around these difficulties, we take the following approaches. First, the reticular plate at the interface associates mass and elastic energy to the interface, and regularizes the boundary of the fluid domain. {{In classical fluid-structure interaction problems involving {\emph{elastic}} structures, this usually takes care of the issues related to the regularity of traces in moving boundary problems \cite{Multilayered}. In the case when the structure is {\emph{poroelastic}}, and it satisfies the Biot equations on a {\emph{moving}} domain}}, this is, however, not sufficient {{since the energy estimates do not provide sufficient regularity of the poroelastic matrix displacement for certain integrals over the moving Biot domain in the weak formulation}} to be well defined. 
{{
This is why we take the following two-step approach:
\begin{enumerate}
\item We introduce a ``consistent {\bf{regularized}} weak formulation'' of the coupled problem by defining a suitable convolution in spatial variables and regularizing only the problematic terms in the weak formulation of the coupled problem. We prove the existence of a weak solution to this regularized problem.
\item We show that as the regularization parameter tends to zero, the  solution to this regularized problem converges to the solution of 
the original nonregularized problem in the case when the original problem has a {{classical}} solution and the Biot poroelastic matrix is viscoelastic. {{Here,  a classical solution is a solution that is smooth both temporally and spatially and it hence satisfies the system of PDEs for the {\bf{original fluid-poroelastic structure problem}} pointwise}}. 
\end{enumerate} 
}}

{{The existence of a weak solution to the regularized problem was announced by the authors in \cite{NonlinearFPSI_CRM23}, where only the main steps of the proof were outlined. In particular, the proofs of the existence of weak solutions to the fluid and structure subproblems used in constructing the coupled solution were omitted in \cite{NonlinearFPSI_CRM23}, and only the main steps of the uniform estimates were presented. Most importantly, the proof of the main compactness result needed to address the main difficulty, the geometric nonlinearity in the regularized problem, is only outlined in \cite{NonlinearFPSI_CRM23}.
Furthermore, details of the construction of appropriate test functions that are defined on moving domains are also omitted in \cite{NonlinearFPSI_CRM23}.
Here we present details of all proofs,}} and show the {\bf{weak-classical consistency}} result outlined in step 2 above. 

The {\bf{weak-classical consistency}} result outlined in step 2 {{
is obtained by using a Gronwall-type estimate, which shows that the energy of the difference
between the weak solution of the regularized problem and the {{classical (temporally and spatially smooth)}} solution to the original, nonregularized problem with viscoelastic Biot poroelastic matrix, converges to zero as the regularization parameter tends to zero. While the main idea is simple, the estimates  are quite nontrivial due to the fact that we need to 
work with the integrals over regularized Biot domains and compare them with the integrals over the nonregularized moving domains. 
Details are presented in Section~\ref{weakstrong}.}}

%Namely, we prove that if a smooth solution exists for the fluid-poroviscoelastic structure interaction (FPSI) problem without regularization, then there exists a time $T>0$ such that a sequence of weak solutions to the regularized FPSI problem constructed here, converges to the strong solution on $[0,T]$ as the regularization parameter converges to $0$.

%%%%%%%%%%%START OF ERASED PARAGRAPH%%%%%%%%%%%
\if 1= 0
Therefore, in this manuscript 
{\emph{we prove
the existence of a weak solution to a nonlinearly coupled fluid-structure interaction problem between the flow of an incompressible, viscous fluid modeled by the
Navier-Stokes equations, and a structure consisting of two solids -- a thick poroviscoelastic medium modeled by the Biot equations, and a thin interface with mass modeled by a reticular plate equation.}}
 We mention that no viscoelasticity is needed for the proof of the existence of a weak solution to the regularized problem. The existence of a weak solution
 to the regularized problem holds in the purely poroelastic case (and in the {{viscoelastic}} case). Viscoelasticity of the Biot poroelastic matrix is needed only in the proof of {{weak-classical}} consistency. 
 \fi
 %%%%%%%%%%%%END OF ERASED PARAGRAPH%%%%%%%%%%%%

{{We conclude this section by noting that the main steps of the {\bf{constructive}} proof presented in this manuscript can be used to design a numerical scheme to capture the solutions to the original (non-regularized) FPSI problem, {{see  \cite{AndrewPartitionedSchemeFPSI}.}} 
{{The main constructive steps of the proof can be summarized as follows. 
We semidiscretize the regularized FPSI problem in time by subdividing the time interval into $N$ subintervals of width $\Delta t$. }}
At each time step we split the reticular plate subproblem from the regularized fluid-Biot subproblem using a Lie operator splitting strategy \cite{glowinski2003finite}.
To deal with the moving domains we use
the Lagrangian map  for the Biot domain, and an {{Arbitrary Lagrangian-Eulerian mapping for the fluid domain, which maps}} a fixed, reference domain onto the current, 
physical domain. We switch between the reference domain formulation and moving domain formulation in the proof as needed. {{For each $\Delta t$,}} approximate solutions are constructed by ``solving'' the sequence of semidiscretized (linearized) problems defined on the current (approximate) moving domain for each $t_n = n \Delta t, n = 1,\dots,N$. 

We then show uniform boundedness of the approximate solutions by deriving energy estimates {{that are}} uniform in the time discretization parameter $\Delta t$. This will allow us to deduce the existence of 
weakly and weakly* convergent subsequences. Since the problem is highly nonlinear, just having weakly and weakly* subsequences is not sufficient to pass to the limit in the weak formulations
of {{the approximate problems}}. Hence, we must obtain strong convergence of approximate sequences by using
several compactness results: the classical Aubin-Lions compactness lemma \cite{AubinLions} for the Biot displacement, Arzela-Ascoli for {{the plate displacement}}, 
Dreher and J\"{u}ngel's compactness result \cite{DreherJungel} for the Biot and plate velocity and pore pressure,
and a recent generalized Aubin-Lions-Simon compactness result by Muha and \v{C}ani\'{c} \cite{MuhaCanic13}, to deal 
with the most {{involved}} part, which is the free fluid velocity defined on different time-dependent fluid domains. 

Once strongly convergent subsequences are obtained from the 
compactness results, one would like to pass to the limit in the weak formulation to show that the limits of the subsequences are weak solutions to the regularized 
fluid-poroelastic structure interaction problem. However, this cannot be done yet, since the velocity test functions 
are also defined on moving domains and we need to construct ``appropriate'' test functions which can be compared for different domains,
and for which we can show convergence to a test function of the limiting, continuous problem. Luckily, in contrast with the classical fluid-elastic structure interaction problems, in our case the 
fluid test functions decouple from the structure problem, and so it is a bit easier to construct appropriate test functions for which one can show 
uniform pointwise convergence to a test function for the continuous problem.
With this final step, we can pass to the limit in the weak formulations of approximate problems and show that the limits of approximate subsequences
satisfy the continuous weak formulation of the regularized problem. 

This existence result is local in time because we can guarantee  the nondegeneracy of the fluid domains both for the free fluid flow and the 
filtrating flow through the poroelastic medium only locally in time. However, using the {{approach}} presented in \cite[Section 5]{CDEM} the time of existence can be extended to
the maximal time until {{one of the following three events occurs:}} (1) the {{moving fluid domain or Biot domain}} degenerates (e.g., the interface touches the bottom of the fluid domain {{or the top of the Biot domain}}), (2) the pores in the poroelastic matrix
denegerate in the sense that the Lagrangian mapping stops being injective, or (3) $T = \infty$.
}}

%This {{weak-classical consistency}} proof effectively shows the weak-strong uniqueness, namely, it shows that
%the solution we constructed is unique and equal to the strong solution in energy norm 
%when the strong solution exists. 

\section{Literature review}

There is extensive past work on fluid-structure interaction (FSI) studying fully coupled systems involving incompressible, Newtonian fluids interacting with {{deformable}} structures. 

{{Most of the FSI literature considers models involving purely {\bf{elastic}} structures.
The models first considered were {\bf{linearly coupled}} FSI models \cite{barGruLasTuff2, BarGruLasTuff, KukavicaTuffahaZiane}, which pose the fluid equations on a fixed reference fluid domain, as a linearization that approximates real-life dynamics well when structure displacements 
and deformations are small. 

In cases when displacements and deformations of the structure are large, they can significantly affect the fluid dynamics 
in which case time-dependent moving fluid domains that depend on the displacement itself must be taken into account. 
Such {\bf{nonlinearly coupled}} FSI models have been extensively studied in
\cite{BdV1,CDEM,ChengShkollerCoutand,ChenShkoller,CSS1,CSS2,CG,Grandmont16,FSIforBIO_Lukacova,IgnatovaKukavica,ignatova2014well,Kuk,LengererRuzicka,Lequeurre,MuhaCanic13, BorSun3d,BorSunMultiLayered,BorSunNonLinearKoiter,BorSunSlip,Raymond}.
In such models the time-dependent and a priori unknown fluid domain evolves according to the displacement of the structure, giving rise to a fully coupled problem with two-way coupling between the fluid and structure that has significant geometric nonlinearities arising from the moving boundary. 
There are two broad classes of nonlinearly coupled FSI models:
(1) models in which the elastic structure has a lower spatial dimension than the fluid so that the structure is for example an elastic plate or shell, and (2) models in which the fluid and structure domains have the same spatial dimension. In the first case (involving elastic structures of lower spatial dimension), the works showing existence of {\emph{strong solutions}} include \cite{BdV1, Lequeurre, Grandmont16, FSIforBIO_Lukacova}, and the  works showing existence of {\emph{weak solutions}} include \cite{CDEM, CG, MuhaCanic13, LengererRuzicka, FSIforBIO_Lukacova}. 
In the second case (involving coupled elastic structures and fluids of the same spatial dimension), well-posedness results have been studied in \cite{CSS1, CSS2, ChengShkollerCoutand, ChenShkoller, Kuk, IgnatovaKukavica, ignatova2014well, Raymond}. 

Closest to the work presented in this manuscript is the work of \cite{MuhaCanic13} showing existence of weak solutions to a nonlinearly coupled problem
 between an elastic Koiter shell and an incompressible viscous fluid modeled by the Navier-Stokes equations. In \cite{MuhaCanic13}   a splitting scheme was introduced to prove the existence of a weak solution to the nonlinearly coupled problem by semidiscretizing the fully coupled problem in time and splitting the coupled problem into fluid and structure subproblems. This scheme has proven to be a robust way for analyzing a variety of complex nonlinearly coupled (moving boundary) FSI problems involving elastic or viscoelastic structures, see \cite{MuhaCanic13, BorSun3d, BorSunMultiLayered, BorSunNonLinearKoiter, BorSunSlip}.
In the present manuscript we adapt the splitting scheme approach to the nonlinearly coupled {\bf{fluid-poroelastic structure}} interaction problem.

%However, many elastic materials, such as biological tissues and sediments that interact with fluids are not impermeable and can admit fluid flow through their pores, 
%in which case poroelasticity of the material needs to be taken into account. 

In terms of literature related to {\bf{poroelastic media}} modeled by the Biot equations, we mention the studies by Biot, modeling soil consolidation \cite{Biot1, Biot2}, the studies of fractures in porous and poroelastic materials \cite{GWG15, LDQ11} and more recently, applications to biomedical science, including the study of the ocular tissue related to  the onset of glaucoma \cite{CGH14}, and the modeling of intestinal walls as poroelastic media \cite{YRC14}. The mathematical well-posedness of the Biot equations discussed in these models has been the focus of a number of works, including 
\cite{Biotwell1, Biotwell2, Biotwell3, Biotwell4, Biotwell5, Biotwell6, Biotwell7, Biotwell8,BiotWell22,BiotWell23}. 
}}

%More recently, there has been a need in applications to understand not just poroelastic materials by themselves, but the interaction between poroelastic materials and fluids. 
%Mathematically, such systems are described by coupling fluid equations (e.g. the Navier-Stokes or Stokes system) with poroelasticity. 
{{In terms of fluid-poroelastic structure interaction problems, the analysis of well-posedness for linearly coupled problems were discussed in \cite{AEN19, Ces17, Sho05}.
Recent progress in the design of bioartificial organs, see e.g., \cite{FluidsCanic}, sparked the need to study FPSI problems in which the fluid-structure interface itself has mass and elastic or poroelastic energy. 
%Namely, in the recent work on the design of 
%a bioartificial pancreas \cite{FluidsCanic}, the bioartifical pancreas consists of  an {\emph{encapsulated}} poroelastic agarose gel containing transplanted pancreatic cells, where the capsule containing the poroelastic medium is itself poroelastic, and it is designed to protect the transplanted cells within the poroelastic agarose gel
%from the patient's own immune cells, while allowing the passage of oxygen and nutrients to the cells for long time viability. 
%This capsule is a thin poroelastic membrane/plate which sits at the interface between the poroelastic gel containing the transplanted cells, and the flow of blood 
%carrying oxygen and nutrients to the bioartificial organ.  
The well-posedness for a linearly coupled FPSI problem 
 in which the structure consists of two layers: a thin poroelastic plate located at the interface between
the free fluid flow and a thick poroelastic medium modeled by the Biot equations,  was obtained in \cite{BCMW21} for both the linear and nonlinear Biot equations, where the nonlinearity refers to 
the dependence of the permeability tensor in the Biot equations on the fluid content. 
In \cite{BCMW21}  the fluid-structure interface with mass serves as a regularizing mechanism and 
provides sufficient information about the regularity of the interface and the free fluid domain to allow, for the first time, the proof of the existence of 
 a finite energy weak solution. 
 
 None of the works that address weak solutions to  fluid-structure interaction problems between the flow of an incompressible, viscous fluid and a poroelastic solid have taken into account
 nonlinear coupling over the moving interface. 
 %Such problems, however, are of importance in many applications, including the flow of blood in coronary arteries 
 %sitting on the surface of the heart, and contracting and relaxing with each heart beat \cite{YifanDES,YifanCMAME}. To understand large displacements that occur due to the 
% contractions of the heart muscle, and capture the flow of drugs through the vascular wall, nonlinear coupling between the blood flow and vascular walls,
% modeled as poro(visco)elastic media, needs to be taken into account. 
 %
%
%We emphasize that the existing work on fluid-poroelastic structure interaction is solely for linearly coupled models, in which the fluid equations are posed on a fixed reference fluid domain and the Biot equations describing poroelasticity are posed on a fixed reference poroelastic structure domain. More recently, there has been work, motivated by the design of bioartificial organs, that has considered the well-posedness of a linearly coupled poroelastic model involving a multilayered structure, consisting of a poroelastic structure and a poroelastic plate separating the fluid and poroelastic structure \cite{BCMW21}. 
%
%However, there has been no prior work that has considered the case of nonlinearly coupled fluid-poroelastic structure interaction, in which the poroelastic structure and fluid are both posed on moving domains that depend on the poroelastic structure displacement, as all past work has been for linearly coupled models. 
%
The goal of the current manuscript is to develop a well-posedness theory for a nonlinearly coupled (moving boundary) fluid-poroelastic structure interaction problem by constructing new tools for dealing with the equations of poroelasticity defined on a priori unknown and time-dependent domains. 
}}

%As in \cite{BCMW21}, we consider a multilayered problem, in which there is a plate which separates a poroelastic medium and a fluid, though in the current manuscript, the separating plate will be purely elastic rather than poroelastic. 

\section{Description of the main problem}\label{model}
We study {{fluid-poroelastic structure interaction}} between the flow of an incompressible, viscous fluid and a multilayered poro(visco)elastic structure
consisting of two layers: a thick poro(visco)elastic layer modeled by the Biot equations, and a thin elastic layer modeled by the reticular plate equation. 
The problem is set on a two dimensional domain, which embodies all the main mathematical difficulties associated with the analysis of this problem. 
The entire two dimensional domain $\hat{\Omega}$ is a union of the reference domain for the fluid subproblem $\hat{\Omega}_{f}$, {{the}} reference domain for the Biot poroviscoelastic material $\hat{\Omega}_{b}$, and the reference domain $\hat{\Gamma}$ of the elastic reticular plate
which serves {{as the interface}} separating the free fluid flow and the Biot medium:
\begin{equation*}
	\hat{\Omega} = \hat{\Omega}_{b} \cup \hat{\Omega}_{f} \cup \hat{\Gamma}, \ {\rm where} \ 
	\hat{\Omega}_{b} = (0, L) \times (0, R), \  \hat{\Gamma} = (0, L) \times \{0\}, \  \hat{\Omega}_{f} = (0, L) \times (-R, 0).
\end{equation*}
These domains will evolve in time, 
giving rise to the time-dependent $\Omega(t) = \Omega_{b}(t) \cup \Omega_{f}(t) \cup \Gamma(t)$. 
We will be using the hat notation to denote objects associated with the reference domain.
On each subdomain we will consider the following mathematical models.

\subsection{The Biot equations on a moving domain}

The Biot system consists of the elastodynamics equation, which in this work will be defined on the Lagrangian domain  $\hat{\Omega}_{b}$,
and the fluid equation, which in this work will be defined on the Eulerian domain $ {\Omega}_{b}(t)$.
Let 
$\hat{\bd{\eta}}: [0, T] \times \hat{\Omega}_{b} \to \mathbb{R}^{2}$  denote the displacement of the Biot poroviscoelastic matrix from its reference configuration,
and let $\hat{p}: \hat{\Omega}_{b} \to \mathbb{R}$ denote the fluid pore pressure. 
To specify the fluid equation given in terms of the fluid pore pressure in Eulerian formulation, we introduce  the {\bf{Lagrangian map}} by
\begin{equation}\label{phi}
	\hat{\boldsymbol{\Phi}}_{b}^{\eta}(t, \cdot) = \text{Id} + \hat{\bd{\eta}}(t, \cdot): \hat{\Omega}_{b} \to \Omega_{b}(t),
\end{equation}
with $(\boldsymbol{\Phi}_{b}^{\eta})^{-1}(t, \cdot): \Omega_{b}(t) \to \hat{\Omega}_{b}$ denoting its inverse. 
The Biot equations are then given by:
\begin{align}\label{Biot1}
	\rho_{b} \partial_{tt} \hat{\boldsymbol{\eta}} &= \hat{\nabla} \cdot \hat{S}_{b}(\hat{\nabla} \hat{\boldsymbol{\eta}}, \hat{p}) & \text{ in } \hat{\Omega}_{b},
\\
\label{Biot2}
	\frac{c_{0}}{[\det(\hat{\nabla} \hat{\boldsymbol{\Phi}}^{\eta}_{b})] \circ (\bd{\Phi}^{\eta}_{b})^{-1}} \frac{D}{Dt} p + \alpha \nabla \cdot \frac{D}{Dt} \boldsymbol{\eta} 
	&= \nabla \cdot (\kappa \nabla p) & \text{ in } \Omega_{b}(t),
\end{align}
where {{$\frac{D}{Dt} = \frac{d}{dt} + \left (\left (\partial_t\boldsymbol{\eta}(t, \cdot)\circ(\boldsymbol{\Phi}_{b}^{\eta})^{-1}(t, \cdot)\right ) \cdot \nabla\right)$}} is the material derivative. The first equation describes the elastodynamics of the poroelastic solid matrix, while the second equation models the conservation of mass 
principle of the filtrating fluid, see, e.g. \cite{NonlinearFPSI1, NonlinearFPSI2} { for more details about Biot equations defined on moving domains.}
To recover the filtration fluid velocity $\boldsymbol{q}$, Darcy's law is used:
\begin{equation}\label{darcy}
	\boldsymbol{q} = -\kappa \nabla p \qquad \text{ on } \Omega_{b}(t), 
\end{equation}
where $\kappa$ is a positive permeability constant. 

In this work, we will consider both the viscoelastic and the purely elastic consitutive models for the Biot poroelastic matrix 
with the Piola-Kirchhoff stress tensor for the viscoelastic case given by
\begin{equation}\label{Biotstress}
	\hat{S}_{b}(\nabla \boldsymbol{\eta}, p) = 2\mu_{e} \hat{\bd{D}}(\hat{\bd{\eta}}) + \lambda_{e} (\hat{\nabla} \cdot \hat{\boldsymbol{\eta}}) \boldsymbol{I} + 2\mu_{v} \hat{\bd{D}}(\hat{\bd{\eta}}_{t}) + \lambda_{v} (\hat{\nabla} \cdot \hat{\bd{\eta}}_{t}) \bd{I} - \alpha \det(\hat{\nabla} \hat{\boldsymbol{\Phi}}^{\eta}_{b}) \hat{p} (\hat{\nabla} \hat{\boldsymbol{\Phi}}^{\eta}_{b})^{-t},
\end{equation}
where superscript $t$ denotes matrix transposition and $A^{-t}=(A^{-1})^t$.
The {purely elastic case} has the coefficients $\lambda_{v}$ and $\mu_{v}$ equal to zero.
Here, $\bd{D}$ denotes the symmetrized gradient, $\mu_{e}$ and $\lambda_{e}$ are the Lam\'{e} parameters related to the elastic stress, $\mu_{v}$ and $\lambda_{v}$ are the corresponding parameters related to the viscoelastic stress, and $\hat{\bd{\Phi}}^{\eta}_{b}$ is the Lagrangian map defined above. {{From the definition of the stress tensor \eqref{Biotstress}, one can see that the elastodynamics of the Biot medium in \eqref{Biot1} is described by linear elasticity with an additional term involving pore pressure. This pressure term embodies additional geometric nonlinearities arising from transforming the pressure between the Eulerian and Lagrangian frameworks.

In equation \eqref{Biot2} the Biot material displacement $\bd{\eta}$ and the pore pressure $p$ are defined on the physical domain $\Omega_{b}(t)$ as 
\begin{equation*}
	\bd{\eta}(t, \cdot) = \hat{\bd{\eta}}(t, (\boldsymbol{\Phi}^{\eta}_{b})^{-1}(t, \cdot)), \qquad p(t, \cdot) = \hat{p}(t, (\boldsymbol{\Phi}_{b}^{\eta})^{-1}(t, \cdot)),
	 \ {\rm where} \ \Omega_{b}(t) = \hat{\boldsymbol{\Phi}}_{b}^{\eta}(t, \hat{\Omega}_{b}).
\end{equation*}}}
We remark that in the last term of the Piola-Kirchhoff stress tensor \eqref{Biotstress}, we have used the Piola transform {(e.g. \cite[Section 1.7.]{Ciarlet})}, which is a transformation that maps tensors in Lagrangian coordinates to corresponding tensors in Eulerian coordinates in such a way that divergence-free tensors in Lagrangian coordinates remain divergence free in Eulerian coordinates \cite{Ciarlet}. 

We note that \emph{a priori} the notion of $\Omega_{b}(t)$ is not entirely clear, unless $\hat{\bd{\eta}}$ is sufficiently regular, and furthermore, the formulation of this problem makes sense only if the map $\hat{\bd{\Phi}}^{\eta}_{b} = \text{Id} + \hat{\bd{\eta}}$ is an injective map from $\hat{\Omega}_{b}$ to $\Omega_{b}(t)$. 
We address these important issues later.

\subsection{The reticular plate equation}

{{A reticular plate is a lattice-type structure characterized by two properties: periodicity and small thickness, where periodicity refers to periodic cells (holes) distributed in all directions \cite{CioranescuBook}.
Reticular plates, shells or membranes are models for reticular tissue, which is a connective tissue made up of a network of supportive fibers that provide a framework for soft organs.}}
The elastodynamics of reticular plates, studied in  \cite{CioranescuBook} using homogenization,
is governed by a  plate-type equation, defined on the equilibrium middle surface $\hat\Gamma$ of the homogenized plate or shell. 
The homogenized equation is given in terms of transverse displacement $\hat{\boldsymbol{\omega}} = \hat{\omega} \boldsymbol{e}_{y}$
from the reference configuration:
\begin{equation}\label{plate}
	{{\rho_{p}}} \partial_{tt} \hat{\omega} + \hat{\Delta}^{2} \hat{\omega} = \hat{F}_{p}, \qquad \text{ on } \hat{\Gamma},
\end{equation}
where {{$\rho_{p}$}} is the plate density coefficient and $\hat{F}_{p}$ is the external forcing on the plate in $y$ direction, to be specified later in the coupling conditions.
The constant {{$\rho_p$}} is the ``average'' plate density,
which depends on the periodic structure. The in-plane bi-Laplacian $\hat\Delta^2$ (Laplace-Beltrami operator for curved $\hat\Gamma$'s) is associated with 
the elastic energy of the plate. Typically, there is a coefficient $\tilde{D}$ in front of the bi-Laplacian, which contains 
information about the periodicity of the structure and its stiffness properties \cite{CioranescuBook}. 
In the present work, {{without loss of generality,}} we will assume that it is equal to $1$. 
The source term $\hat{F}_p$ corresponds to the loading of the poroelastic plate, which will come
from the jump in the normal stress (traction) between the free fluid on one side and the thick Biot poroelastic structure on the other, see (7) below.

In our problem, the reticular plate separates the regions of free fluid flow and the Biot poroviscoelastic medium, and is transparent to the flow between the two. 
{{This means, in particular, that there is no resistance to the fluid flow passing through the reticular place. 
%As a consequence, the fluid flow quantities, e.g., pressure and normal stress, satisfy the coupling condition \eqref{pressurebalance} below, which is the same as the coupling condition that would appear between the thick Biot medium and free fluid flow in an FPSI problem without the reticular plate. 
However, due to the inertia and elastic energy of the plate, the analysis of the problem will be simplified due to the regularizing effects of the plate inertia and elastic energy, as we shall see below (see e.g., Remark~\ref{RemarkAboutInterface}). }}

 The time-dependent configuration of the plate 
\begin{equation*}
	\Gamma(t) = \{(x, y) : 0 < x < L,\; y = \hat{\omega}(t, x)\},
\end{equation*}
forms the bottom boundary of the moving Biot domain $\Omega_{b}(t)$, and the remaining left, top, and right boundaries of the moving Biot domain $\Omega_{b}(t)$ are fixed in time. 
Hence, we impose $\bd{\eta} = 0$ on the left, top, and right boundaries of $\Omega_{b}(t)$.
See Fig.~\ref{domain}. Hence, we can describe the moving domain $\Omega_{b}(t)$ as
\begin{equation*}
	\Omega_{b}(t) = \{(x, y) : 0 < x < L, \; \hat{\omega}(t, x) < y < R\}.
\end{equation*}

\subsection{The Navier-Stokes equations on a moving domain}

The free flow of an incompressible, viscous fluid will be modeled by  the Navier-Stokes equations
\begin{equation}\label{NS1}
\left.
\begin{array}{rcl}
	\displaystyle{\partial_{t} \boldsymbol{u} + (\boldsymbol{u} \cdot \nabla) \boldsymbol{u}} &=& \nabla \cdot \boldsymbol{\sigma}_{f}(\nabla \boldsymbol{u}, \pi)
	\\
	\displaystyle{\nabla \cdot \boldsymbol{u}} &=& 0
\end{array}
\right\}
\quad
\text{ in } \Omega_{f}(t),
\end{equation}
where $\boldsymbol{u}$ is the fluid velocity and $\pi$ is the fluid pressure. The Cauchy stress tensor is given by
\begin{equation*}
	\boldsymbol{\sigma}_{f}(\nabla \boldsymbol{u}, \pi) = 2\nu \boldsymbol{D}(\boldsymbol{u}) - \pi \boldsymbol{I},
\end{equation*}
where $\pi$ is the fluid pressure and $\nu$ is kinematic viscosity coefficient. 
Notice that the fluid problem is defined on a moving domain, which is not {{known}} {\emph{a priori}}. 
The moving fluid domain  $\Omega_{f}(t)$ is a function of time and it is determined by the plate displacement $\hat{\omega}$, as follows:
\begin{equation*}
	\Omega_{f}(t) = \{(x, y) : 0 < x < L, -R < y < \hat{\omega}(t, x)\}.
\end{equation*}
The fact that the free fluid domain depends on one of the unknowns in the problem presents a geometric nonlinearity that is difficult to deal with.
We will be using the following {\bf{Arbitrary Lagrangian Eulerian (ALE) mapping}} 
$\hat{\boldsymbol{\Phi}}^{\omega}_{f}: \hat{\Omega}_{f} \to \Omega_{f}(t)$ to map the fixed reference domain $\hat{\Omega}_{f}$ onto the current, physical 
domain $\Omega_f(t)$:
	\begin{equation}\label{phif}
		\hat{\boldsymbol{\Phi}}^{\omega}_{f}(\hat{x}, \hat{y}) = \left(\hat{x}, \hat{y} + \left(1 + \frac{\hat{y}}{R}\right)\hat{\omega} \right), \qquad (\hat{x}, \hat{y}) \in \hat{\Omega}_{f}.
	\end{equation}
In our analysis, we will use this ALE mapping to will switch between the fixed and moving boundary formulations of the coupled problem as needed.
{ 
\begin{remark}\label{RemarkALE}
In numerical computations, it is typical to employ harmonic extension to construct the Arbitrary Lagrangian-Eulerian (ALE) mapping. However, given the simplicity of our geometry, we chose to utilize the explicit formula for extension to simplify the calculations related to the change of variables. Since our methodology is not contingent on the particular selection of the ALE map, in scenarios involving more complex geometries where an explicit formula is not viable, alternatives such as harmonic extension can also be utilized.
\end{remark}
}

\subsection{The coupling conditions}
The Navier-Stokes equations \eqref{NS1}, the Biot equations \eqref{Biot1}, \eqref{Biot2}, and the reticular plate equation \eqref{plate} are coupled 
 across the moving reticular plate interface $\Gamma(t)$ via two sets of coupling conditions: the kinematic and dynamic coupling conditions. To state these conditions, we introduce the following notation:
\begin{itemize}
\item {{The Biot Cauchy stress tensor defined on the physical domain ${S}_{b}(\nabla \boldsymbol{\eta}, p)$ is obtained by applying the Piola transform to the Biot Cauchy stress tensor 
$\hat{S}_{b}(\nabla \boldsymbol{\eta}, p)$ defined on the reference domain, to obtain:}}
\begin{align}
	&S_{b}(\nabla \boldsymbol{\eta}, p) = [\det(\hat{\nabla} \hat{\boldsymbol{\Phi}}^{\eta}_{b})^{-1} \hat{S}_{b}(\hat{\nabla} \hat{\boldsymbol{\eta}}, \hat{p}) (\hat{\nabla} \hat{\boldsymbol{\Phi}}^{\eta}_{b})^{t}] \circ (\bd{\Phi}^{\eta}_{b})^{-1} 
	\nonumber 
	\\
	&= \left(\frac{1}{\det(\hat{\nabla} \hat{\boldsymbol{\Phi}}^{\eta}_{b})} \left[2\mu_{e} \hat{\boldsymbol{D}}(\hat{\boldsymbol{\eta}}) + \lambda_{e} (\hat{\nabla} \cdot \hat{\boldsymbol{\eta}}) + 2\mu_{v} \hat{\boldsymbol{D}}(\hat{\bd{\eta}}_{t}) + \lambda_{v} (\hat{\nabla} \cdot \hat{\bd{\eta}}_{t})\right] (\hat{\nabla} \hat{\boldsymbol{\Phi}}^{\eta}_{b})^{t}\right) \circ (\bd{\Phi}^{\eta}_{b})^{-1} - \alpha p \boldsymbol{I}.
	\label{physicalstress}
\end{align}
\item
The Eulerian structure velocity of the Biot poroviscoelastic matrix is given at each point of the physical domain $\Omega_{b}(t)$ by
\begin{equation}\label{xi}
	\boldsymbol{\xi}(t, \cdot) = \partial_{t}\hat{\boldsymbol{\eta}}\left(t, (\boldsymbol{\Phi}^{\eta}_{b})^{-1}(t, \cdot)\right).
\end{equation}
\item
The normal unit vector  to the moving interface $\Gamma(t)$ will be denoted by $\boldsymbol{n}(t)$, and the normal unit vector to the reference configuration of the interface $\hat\Gamma$ will be denoted by $\hat{\bd{n}}$.
Note that $\hat{\bd{n}} = \bd{e}_{y}$. The vectors $\bd{n}(t)$ and $\hat{\bd{n}}$ point outward from $\Omega_{f}(t)$ and $\Omega_{f}$, and inward towards $\Omega_{b}(t)$ and $\Omega_{b}$. 
\end{itemize}
\noindent
The following two sets of coupling conditions give rise to a well-defined bounded energy of the coupled problem:
\vskip 0.1in
\noindent
{\bf{(I) Kinematic coupling conditions:}}
\begin{itemize}
	\item Continuity of normal components of velocity (conservation of mass of the fluid):
	\begin{equation}\label{mass}
		\boldsymbol{u} \cdot \boldsymbol{n}(t) = (\boldsymbol{q} + \boldsymbol{\xi}) \cdot \boldsymbol{n}(t), \qquad \text{ on } (0, T) \times \Gamma(t).
	\end{equation}
	\item Slip in the tangential component of free fluid velocity, known as the Beavers-Joseph-Saffman condition \cite{JM96,JM00}:
	\begin{equation}\label{bjs}
		\beta(\boldsymbol{\xi} - \boldsymbol{u}) \cdot \boldsymbol{\tau}(t) = \boldsymbol{\sigma}_{f} \boldsymbol{n}(t) \cdot \boldsymbol{\tau}(t), \qquad \text{ on } (0, T) \times \Gamma(t),
	\end{equation}
	where $\beta \ge 0$ is a constant and $\bd{\tau}(t)$ is the rightward pointing unit tangent vector to $\Gamma(t)$. 	
	\item Continuity of displacements:
	\begin{equation}\label{DisplCont}
		\hat{\boldsymbol{\eta}} = \hat{\omega} \boldsymbol{e}_{y}, \qquad \text{ on } (0, T) \times \hat{\Gamma}.
	\end{equation} 
\end{itemize}
{\bf{(II) Dynamic coupling conditions:}}
\begin{itemize}
	\item Balance of forces describing the body forcing on the plate as the difference between the normal components of normal stress coming from the Biot medium
	on one side, and free fluid flow on the other:
	\begin{equation}\label{dynamic}
		\hat{F}_{p} = -\det(\nabla \hat{\bd{\Phi}}^{\omega}_{f}) [\bd{\sigma}_{f}(\nabla \bd{u}, \pi) \circ \hat{\bd{\Phi}}^{\omega}_{f}] (\nabla \hat{\bd{\Phi}}^{\omega}_{f})^{-t} \hat{\bd{n}} \cdot \hat{\bd{n}} + \hat{S}_{b}(\hat{\nabla} \hat{\bd{\eta}}, \hat{p}) \hat{\bd{n}} \cdot \hat{\bd{n}}|_{\hat{\Gamma}}, \qquad \text{ on } {{(0, T) \times \hat{\Gamma}}},
	\end{equation}
	
	\if 1 = 0
	\begin{equation}\label{dynamic}
		\hat{F}_{p} = - \hat{\mathcal{J}}^{\omega}_{\Gamma} \cdot (\boldsymbol{\sigma}_{f}(\nabla \boldsymbol{u}, \pi) \boldsymbol{n} \cdot \boldsymbol{e}_{y})|_{\Gamma(t)} \circ (\bd{\Phi}^{\omega}_{\Gamma})^{-1} + \hat{S}_{b}(\hat{\nabla} \hat{\boldsymbol{\eta}}, \hat{p}) \boldsymbol{e}_{y} \cdot \boldsymbol{e}_{y}|_{\hat{\Gamma}}, \qquad \text{ on } \hat{\Gamma},
	\end{equation}
	\fi 
	
	where $\hat{\bd{\Phi}}^{\omega}_{f}$ is the Arbitrary Lagrangian-Eulerian (ALE) mapping defined in \eqref{phif}.
	\item Balance of pressure at the interface:
	\begin{equation}\label{pressurebalance}
		-\boldsymbol{\sigma}_{f}(\nabla \boldsymbol{u}, \pi) \boldsymbol{n}(t) \cdot \boldsymbol{n}(t) + \frac{1}{2} |\boldsymbol{u}|^{2} = p, \qquad \text{ on } (0, T) \times \Gamma(t).
	\end{equation}
	{{See \cite{MuhaCanic13, CanicLectureNotes, NSDarcy, DynamicPressure} for the use of the dynamic (total) pressure $\pi + \frac{1}{2}|\bd{u}|^{2}$ on the left-hand side of \eqref{pressurebalance}.}}
\end{itemize}

\subsection{The {initial and} boundary conditions}\label{boundary}

For the fluid, we will assume rigid walls on $\partial \Omega_{f}(t) \setminus \Gamma(t)$  and impose a no-slip condition 
\begin{equation*}
	\boldsymbol{u} = 0, \qquad \text{ on } \partial \Omega_{f}(t) \setminus \Gamma(t).
\end{equation*}
Similarly, we will assume that the boundaries of the Biot poroviscoelastic medium, excluding the interface $\Gamma(t)$, are rigid
and impose 
\begin{equation*}
	\hat{\boldsymbol{\eta}} = 0 \ \ \text{ and } \ \ \hat{p} = 0, \qquad \text{ on } \partial \hat{\Omega}_{b} \setminus \hat{\Gamma}.
\end{equation*}
{Finally, we prescribe the following initial conditions:
\begin{align*}
	 \boldsymbol{u}(0)=\boldsymbol{u}_0\quad {\rm in}\; \Omega_f(0),
	 \\
	 \hat{\bd{\eta}}(0)=\hat{\bd{\eta}}_0,\;
	 \partial_t\hat{\bd{\eta}}(0)=\hat{\bd{\xi}}_0\quad {\rm in}\; \hat{\Omega}_{b},
	 \\
	 \hat{\omega}(0)=\hat{\omega}_0,\;
	 \partial_t\hat{\omega}(0)=\hat{\zeta}_0\quad{\rm in}\; \hat{\Gamma},
	 \\
	 \hat{p}(0)=\hat{p}_0\quad {\rm in}\; \hat{\Omega}_{b}.
\end{align*}
}

\subsection{Preview of the main results}
{{Our first main result is the existence of a weak solution to a regularized FPSI problem, where there is a regularization parameter $\delta > 0$. The regularization will involve spatially regularizing the Biot displacement by extending the displacement $\hat{\bd{\eta}}$ on $\hat{\Omega}_{b}$ to a larger domain and using spatial convolution by a smooth compactly supported kernel, scaled by $\delta$. This regularized FPSI problem will be introduced in Sec.~\ref{regularized}.}}
 The existence result {{for the regularized FPSI problem}} holds for both elastic and {{viscoelastic}} Biot material. Here we state the theorem informally and refer the reader to Theorem \ref{MainThm1} for the precise statement.
\begin{theorem}\label{InformalThm1}[Existence of a weak solution to the regularized problem]
	Let {{$\rho_b,\mu_e,\lambda_e,\alpha,\rho_p,\nu>0$ and $\mu_v,\lambda_v\geq 0$}}. Moreover, assume that initial data are in {{the finite energy class}} and that initially, the interface does not touch the {{bottom boundary of the fluid domain and the top boundary of the Biot domain}}, and assume that certain compatibility conditions are satisfied. Then for every regularization parameter $\delta>0$, there exists $T>0$ (potentially depending on $\delta > 0$) such that there is a weak solution on $[0,T]$ to the regularized problem with regularization parameter $\delta$. {{Furthermore, the weak solution to the regularized problem exists on a maximal time interval $[0, T]$, where either (1) $T = \infty$ or (2) $T$ is finite and is the time at which either:
	\begin{itemize}
	\item the fluid or Biot domain degenerates so that the moving interface collides with the bottom boundary of {{$\hat{\Omega}_{f}$ or the top boundary of $\hat{\Omega}_{b}$}} or
	\item the (regularized) Lagrangian mapping $\hat{\bd{\Phi}}^{\eta^{\delta}}_{b}$ for the Biot domain is no longer injective.
	\end{itemize}}}
\end{theorem}

Our second main result is a weak-classical consistency result. Namely, in order to justify our regularization procedure and the corresponding definition of weak solutions to the regularized problem, we prove that weak solutions to the regularized problem indeed converge to the solution of the original {{(non-regularized)}} FPSI problem. More precisely, we prove the following result, made precise in Theorem \ref{weakstrongunique}.
\begin{theorem}\label{InformalThm2}[Weak-{{classical}} consistency]
	Assume that a classical (smooth) solution to the FPSI problem with a Biot poroviscoelastic medium exists {on time-interval [0,T]} for the case for which the viscoelasticity parameters $\mu_{v}, \lambda_{v} > 0$. Then every sequence of weak solutions to the regularized problem with regularization parameter $\delta > 0$ converges to the classical solution on $[0,T]$ as the regularization parameter $\delta$ converges to $0$. In particular, the time interval of existence for the weak solutions to the regularized problem is uniform in regularization parameter { and solutions to the regularized problem {{exist}} on the same time interval where the classical solution exists.}
\end{theorem}

{
\begin{remark}
An alternative formulation for Theorems \ref{InformalThm1} and \ref{InformalThm2} is that there exists a weak solution to an approximate problem of the original FPSI problem. Specifically, Theorem \ref{InformalThm2} asserts that the strong solution can be approximated by solutions to the regularized problem, the existence of which is guaranteed by Theorem \ref{InformalThm1}.
\end{remark}
}

The heart of the proof of this theorem is a bootstrap argument presented in Section \ref{bootstrap}. Namely, the main issue is that geometric quantities, such as the determinant of the displacement, cannot be estimated by the energy, and thus are not uniformly bounded in the regularization parameter $\delta$. We derive appropriate bounds by using a bootstrap argument in combination with optimal convergence rate estimates for the convolution regularization. The main technical issue in comparing the classical solution with weak solutions to the regularized problem is the fact that they are defined on different domains. Therefore, we use a change of variables that transfers fluid velocities as vector fields and preserves the divergence-free condition. This transformation was introduced by \cite{InoueWakimoto} and was used in proving weak-strong type of results in the context of FSI in \cite{WeakStrongRigid1,WeakStrongRigid2,WeakStrongFSI}. The corresponding estimates are carried out in Section \ref{weakstrongestimate}.

\section{Definition of a weak solution}\label{sec:weak}

Because the problem under consideration is nonlinearly coupled, the fluid domain $\Omega_{f}(t)$ and the Biot poroviscoelastic domain $\Omega_{b}(t)$ in physical space are time-dependent and  not known apriori.
To handle the moving domains, it is useful to introduce the mappings that map the reference domains $\hat{\Omega}_{b}$, $\hat{\Gamma}$, and $\hat{\Omega}_{f}$ onto 
the moving domains that depend on time and on the solution itself.
\subsection{Mappings between reference and physical domains}
Let
\begin{equation*}
	\hat{\boldsymbol{\Phi}}^{\eta}_{b}(t, \cdot): \hat{\Omega}_{b} \to \Omega_{b}(t), \qquad \hat{\boldsymbol{\Phi}}^{\omega}_{\Gamma}(t, \cdot): \hat{\Gamma} \to \Gamma(t), \qquad \hat{\boldsymbol{\Phi}}^{\omega}_{f}(t, \cdot): \hat{\Omega}_{f} \to \Omega_{f}(t),
\end{equation*}
be such that
\begin{equation}\label{phif}	
\begin{array}{lll}
	&\hat{\boldsymbol{\Phi}}_{b}^{\eta} = \text{Id} + \hat{\boldsymbol{\eta}}(\hat{x}, \hat{y}),\quad & (\hat{x}, \hat{y}) \in \hat{\Omega}_b
	\\
&\hat{\boldsymbol{\Phi}}^{\omega}_{\Gamma}(\hat{x}, 0) = (\hat{x}, \hat{\omega}(\hat{x})), &{{\hat{x} \in \hat\Gamma}}
\\
&\hat{\boldsymbol{\Phi}}^{\omega}_{f}(\hat{x}, \hat{y}) = \left(\hat{x}, \hat{y} + \left(1 + \frac{\hat{y}}{R}\right)\hat{\omega}(\hat{x}) \right),  &(\hat{x}, \hat{y}) \in \hat{\Omega}_{f},
\end{array}
\end{equation}
with the inverse
\begin{equation}\label{Jf}
	(\boldsymbol{\Phi}^{\omega}_{f})^{-1}(x, y) = \left(x, -R + \frac{R}{R + \hat{\omega}}(R + y)\right).
\end{equation}
We are using $(\hat{x}, \hat{y})$ to denote the coordinates on the reference domain and $(x, y)$ the coordinates on the physical domain. 
Note that these mapings are \textit{time-dependent}, even though in the rest of this manuscript we will not explicitly notate this time dependence for ease of notation. 

The {\bf{Jacobians of the transformations}} are given by:
\begin{equation}\label{Jf}
\begin{array}{lll}
	\hat{\mathcal{J}}^{\omega}_{f} = 1 + \frac{\hat{\omega}}{R},
	\qquad
	\hat{\mathcal{J}}^{\eta}_{b} = \det(\boldsymbol{I} + \hat{\nabla} \hat{\boldsymbol{\eta}}),
	\qquad
        \hat{\mathcal{J}}^{\omega}_{\Gamma} = \sqrt{1 + |\partial_{\hat{x}} \hat{\omega}|^{2}},
 \end{array}
\end{equation}
where $\hat{\mathcal{J}}^{\omega}_{\Gamma}$ measures the arc length difference of between the reference and deformed configuration of the plate.
Notice that in the Jacobian $\hat{\mathcal{J}}^{\omega}_{f}$ we dropped the absolute value sign since 
 our results will hold up until the time of domain degeneracy when $|\hat{\omega}| \ge R$.
 
Under these mappings the {{functions are transformed as follows}}.

{\bf{Tranformations under $\boldsymbol{\Phi}^{\omega}_{f}$}}.
The fluid velocity $\boldsymbol{u}$ defined on $\Omega_{f}(t)$ is transferred to the fixed reference domain $\hat{\Omega}_{f}$ by 
\begin{equation*}
	\hat{\bd{u}}(t, \hat{x}, \hat{y}) = \boldsymbol{u} \circ \hat{\boldsymbol{\Phi}}_{f}, \qquad \text{ for } (\hat{x}, \hat{y}) \in \hat{\Omega}_{f}.
\end{equation*}
Recall that on the moving domain $\Omega_{f}(t)$, the fluid velocity $\boldsymbol{u}$ is divergence free, i.e., $\nabla \cdot \boldsymbol{u} = 0$. However, when we pull the fluid velocity back to the reference domain, $\hat{\bd{u}}$ is not necessarily divergence free on $\hat{\Omega}_{f}$. 
Hence, we want to reformulate the divergence free condition on the fixed reference domain. 

{\bf{The divergence free condition.}} Let $g$ be a function defined on $\Omega_{f}(t)$, then
\begin{equation*}
	\nabla g = \nabla \left(\hat{g} \circ (\boldsymbol{\Phi}^{\omega}_{f})^{-1}\right) = (\hat{\nabla}^{\omega}_{f} \hat{g}) \circ (\boldsymbol{\Phi}^{\omega}_{f})^{-1}, 
\end{equation*}
where $\hat{\nabla}^{\omega}_{f}$ is  the {\emph{transformed gradient operator}}: 
\begin{equation}\label{transformedgradient}
	\hat{\nabla}^{\omega}_{f} = 
	\begin{pmatrix} \partial_{\hat{x}} - (R + y) \partial_{\hat{x}} \hat{\omega} \frac{R}{(R + \hat{\omega})^{2}} \partial_{\hat{y}} \\ \frac{R}{R + \hat{\omega}} \partial_{\hat{y}} 
	\end{pmatrix} 
	\quad {\rm{where}} \quad 
	y = \hat{y} + \left(1 + \frac{\hat{y}}{R}\right)\hat{\omega}.
\end{equation}
Therefore, the divergence free condition and the symmetrized gradient on the fixed reference domain 
$\hat{\Omega}_{f}$ are:
$$
\hat{\nabla}^{\omega}_{f} \cdot \hat{\bd{u}} = 0, \qquad \displaystyle \hat{\bd{D}}^{\omega}_{f}(\hat{\bd{u}}) = \frac{1}{2}\left(\hat{\nabla}^{\omega}_{f} \hat{\bd{u}} + (\hat{\nabla}^{\omega}_{f} \hat{\bd{u}})^{t}\right).
$$

{\bf{Time derivatives.}} The {{time derivative transforms}} under the map $\hat{\boldsymbol{\Phi}}^{\omega}_{f}$ as follows:
%\begin{equation*}
%	\partial_{t} \boldsymbol{u}(t, x, y) = \partial_{t} (\hat{\bd{u}}(t, (\bd{\Phi}^{\omega}_{f})^{-1}(\hat{x}, \hat{y})) = \partial_{t} \hat{\bd{u}} - \partial_{\hat{y}} \hat{\bd{u}} \cdot (R + y) \cdot \frac{R}{(R + \hat{\omega})^{2}} \partial_{t}\hat{\omega} = \partial_{t} \hat{\boldsymbol{u}} - \partial_{\hat{y}} \hat{\bd{u}} \frac{(R + \hat{y}) \partial_{t} \hat{\omega}}{R + \hat{\omega}}.
%\end{equation*}
%Thus,
\begin{equation}\label{womega}
\partial_{t} \boldsymbol{u} = \partial_{t} \hat{\bd{u}} - (\hat{\boldsymbol{w}} \cdot \hat{\nabla}^{\omega}_{f}) \hat{\bd{u}}  \quad  {\rm{where}} \quad 
	\hat{\bd{w}} = \frac{R + \hat{y}}{R}\partial_{t}\hat{\omega} \boldsymbol{e}_{y}.
\end{equation}
%we have that
%\begin{equation}\label{timetransform}
%	\partial_{t} \boldsymbol{u} = \partial_{t} \hat{\bd{u}} - (\hat{\boldsymbol{w}} \cdot \hat{\nabla}^{\omega}_{f}) \hat{\bd{u}}.
%\end{equation}

{\bf{Tranformations under $\boldsymbol{\Phi}^{\omega}_{b}$}}.
Given a scalar function $g$ defined on $\Omega_{b}(t)$ the pull back of $g$ to the reference domain $\hat{\Omega}_{b}$ is given by
\begin{equation*}
	\hat{g} = g \circ \hat{\boldsymbol{\Phi}}^{\eta}_{b}.
\end{equation*}
We claim that for some differential operator $\hat{\nabla}^{\eta}_{b}$, which we will determine below,
\begin{equation*}
	\nabla g = \nabla\left(\hat{g} \circ (\boldsymbol{\Phi}^{\eta}_{b})^{-1}\right) = (\hat{\nabla}_{b}^{\eta} \hat{g}) \circ (\bd{\Phi}^{\eta}_{b})^{-1}, 
\end{equation*}
where $\nabla$ is a gradient on the physical domain, $\hat{\nabla}$ is a gradient on the reference domain, and $\hat{\nabla}^{\eta}_{b}$ is a differential operator (different from $\hat{\nabla}$) on the reference domain. For any function $g$ defined on the physical domain, we have that
\begin{equation*}
	\hat{\nabla} \left(g \circ \hat{\bd{\Phi}}^{\eta}_{b}\right) = [(\nabla g) \circ \hat{\bd{\Phi}}^{\eta}_{b}] \cdot (\boldsymbol{I} + \hat{\nabla} \hat{\boldsymbol{\eta}}).
\end{equation*}
Hence, for
\begin{equation*}
	\hat{\nabla}^{\eta}_{b} \hat{g} = (\nabla g) \circ \hat{\boldsymbol{\Phi}}^{\eta}_{b},
\end{equation*}
we get the following explicit formula for the {\emph{transformed gradient operator}} $\hat{\nabla}^{\eta}_{b}$ on $\hat{\Omega}_{b}$:
\begin{equation}\label{nablaeta}
	\hat{\nabla}^{\eta}_{b} \hat{g} = \left(\frac{\partial \hat{g}}{\partial \hat{x}}, \frac{\partial \hat{g}}{\partial \hat{y}}\right) \cdot (\boldsymbol{I} + \hat{\nabla} \hat{\boldsymbol{\eta}})^{-1}.
\end{equation}
Notice that the invertibility of the matrix $\boldsymbol{I} + \hat{\nabla} \hat{\boldsymbol{\eta}}$ will be related to whether the map $(\hat{x}, \hat{y}) \to (\hat{x}, \hat{y}) + \hat{\bd{\eta}}(\hat{x}, \hat{y})$ is a bijection between $\hat{\Omega}_{b}$ and $\Omega_{b}(t)$. 

\subsection{Weak solution}
We now derive the definition of a weak solution to the given FPSI problem, by means of the following formal calculation. We start with the fluid equations and multiply by a test function $\boldsymbol{v}$. Recall the definition of the Eulerian structure velocity $\bd{\xi}$ from \eqref{xi}. For the inertia term of the Navier-Stokes equations, using the Reynold's transport theorem and integration by parts, we obtain:
\begin{multline*}
\int_{\Omega_{f}(t)} (\partial_{t}\boldsymbol{u} + (\boldsymbol{u} \cdot \nabla) \boldsymbol{u})) \cdot \boldsymbol{v} = \frac{d}{dt} \int_{\Omega_{f}(t)} \bd{u} \cdot \bd{v} - \int_{\Omega_{f}(t)} \bd{u} \cdot \partial_{t}\bd{v} - \int_{\Gamma(t)} (\bd{\xi} \cdot \bd{n}) \bd{u} \cdot \bd{v} \\
+ \frac{1}{2} \int_{\Omega_{f}(t)} [((\bd{u} \cdot \nabla) \bd{u}) \cdot \bd{v} - (\bd{u} \cdot \nabla) \bd{v}) \cdot \bd{u}] + \frac{1}{2} \int_{\Gamma(t)} (\bd{u} \cdot \bd{n}) \bd{u} \cdot \bd{v} \\
= \frac{d}{dt} \int_{\Omega_{f}(t)} \bd{u} \cdot \bd{v} - \int_{\Omega_{f}(t)} \bd{u} \cdot \partial_{t}\bd{v} + \frac{1}{2} \int_{\Omega_{f}(t)} [((\bd{u} \cdot \nabla) \bd{u}) \cdot \bd{v} - ((\bd{u} \cdot \nabla) \bd{v}) \cdot \bd{u}] + \frac{1}{2} \int_{\Gamma(t)} (\bd{u} \cdot \bd{n} - 2\bd{\xi} \cdot \bd{n}) \bd{u} \cdot \bd{v}.
\end{multline*}
For the diffusive term of the Navier Stokes equations, we integrate by parts to obtain
\begin{equation*}
-\int_{\Omega_{f}(t)} (\nabla \cdot \bd{\sigma}_{f}(\nabla \bd{u}, \pi)) \cdot \bd{v} = 2\nu \int_{\Omega_{f}(t)} \bd{D}(\bd{u}) : \bd{D}(\bd{v}) - \int_{\Gamma(t)} \bd{\sigma}_{f}(\nabla \bd{u}, \pi) \bd{n} \cdot \bd{v},
\end{equation*}
where we used the fact that the test function $\boldsymbol{v}$ is divergence free to eliminate the pressure, and we use that the test function satisfies $\bd{v} = 0$ on $\partial \Omega_{f}(t) \setminus \Gamma(t)$ due to the boundary conditions for $\bd{u}$.

Next, we multiply {{the structure equation \eqref{Biot1}}} by a test function $\hat{\bd{\psi}}$ to obtain
\begin{align*}
&\int_{\hat{\Omega}_{b}} (\rho_{b} \partial_{tt} \hat{\bd{\eta}} - \hat{\nabla} \cdot \hat{S}_{b}(\hat{\nabla} \hat{\bd{\eta}}, \hat{p})) \cdot \hat{\bd{\psi}} = \rho_{b} \left(\frac{d}{dt} \int_{\hat{\Omega}_{b}} \partial_{t} \hat{\boldsymbol{\eta}} \cdot \hat{\boldsymbol{\psi}} - \int_{\Omega_{b}} \partial_{t} \hat{\bd{\eta}} \cdot \partial_{t} \hat{\bd{\psi}}\right) 
\\
&+ \int_{\hat{\Omega}_{b}} \hat{S}_{b}(\hat{\nabla} \hat{\bd{\eta}}, \hat{p}) : \hat{\nabla} \hat{\bd{\psi}} + \int_{\hat{\Gamma}} \hat{S}_{b}(\hat{\nabla} \hat{\bd{\eta}}, \hat{p}) \bd{e}_{y} \cdot \hat{\bd{\psi}} = \rho_{b}\left(\frac{d}{dt} \int_{\hat{\Omega}_{b}} \partial_{t} \hat{\boldsymbol{\eta}} \cdot \hat{\boldsymbol{\psi}} - \int_{\hat{\Omega}_{b}} \partial_{t} \hat{\bd{\eta}} \cdot \partial_{t} \hat{\bd{\psi}} \right) \\
&+ \int_{\hat{\Omega}_{b}} (2\mu_{e} \hat{\bd{D}}(\hat{\bd{\eta}}) : \hat{\bd{D}}(\hat{\bd{\psi}}) + \lambda_{e} (\hat{\nabla} \cdot \hat{\bd{\eta}})(\hat{\nabla} \cdot \hat{\bd{\psi}}) + 2\mu_{v} \hat{\bd{D}}(\partial_{t}\hat{\bd{\eta}}) : \hat{\bd{D}}(\hat{\bd{\psi}}) + \lambda_{v} (\hat{\nabla} \cdot \partial_{t} \hat{\bd{\eta}}) (\hat{\nabla} \cdot \hat{\bd{\psi}})) \\
&- \alpha \int_{\Omega_{b}(t)} p (\nabla \cdot \bd{\psi}) + \int_{\hat{\Gamma}} \hat{S}_{b}(\nabla \hat{\bd{\eta}}, \hat{p}) \bd{e}_{r} \cdot \hat{\bd{\psi}}.
\end{align*}
Except on $\hat{\Gamma}$, there are no boundary terms, because $\hat{\boldsymbol{\eta}} = 0$ on the left, top, and right boundaries of $\hat{\Omega}_{b}$, and hence the same condition holds for the corresponding test function $\hat{\boldsymbol{\psi}}$. Note that in the integral over $\Omega_{b}(t)$, $\bd{\psi} := \hat{\bd{\psi}} \circ (\bd{\Phi}^{\eta}_{b})^{-1}$.

Finally, we test the {{second equation \eqref{Biot2}}} corresponding to the evolution of the pore pressure for the Biot poroviscoelastic medium with a test function $r$, and  recall the definition of the Darcy velocity $\bd{q}$ from \eqref{darcy},  keeping in mind that $\bd{n}$ is the \textit{inward} normal vector to $\Omega_{b}(t)$:
\begin{align*}
&\int_{\Omega_{b}(t)} \left(\frac{c_{0}}{[\det(\hat{\nabla} \hat{\bd{\Phi}}^{\eta}_{b})] \circ (\bd{\Phi}_{b}^{\eta})^{-1}} \frac{D}{Dt} p + \alpha  \nabla \cdot \frac{D}{Dt} \bd{\eta} - \nabla \cdot (\kappa \nabla p) \right) r \\
&= \int_{\hat{\Omega}_{b}} c_{0} \partial_{t}\hat{p} \cdot \hat{r} + \int_{\Omega_{b}(t)} \alpha \left(\nabla \cdot \frac{D}{Dt} \bd{\eta}\right) r + \int_{\Omega_{b}(t)} \kappa \nabla p \cdot \nabla r - \int_{\Gamma(t)} (\bd{q} \cdot \bd{n})r \\
&= \frac{d}{dt} \int_{\hat{\Omega}_{b}} c_{0}\hat{p}\cdot \hat{r} - \int_{\hat{\Omega}_{b}} c_{0}\hat{p} \cdot \partial_{t}\hat{r} - \int_{\Omega_{b}(t)} \alpha \frac{D}{Dt} \bd{\eta} \cdot \nabla r - \alpha \int_{\Gamma(t)} (\bd{\xi} \cdot \bd{n}) r + \int_{\Omega_{b}(t)} \kappa \nabla p \cdot \nabla r - \int_{\Gamma(t)} (\bd{q} \cdot \bd{n}) r.
\end{align*}
There are no boundary terms except on $\Gamma(t)$ from the integration by parts in the integral involving $\alpha$ and in the integral involving $\kappa$ because of the Dirichlet boundary condition $r = 0$ (since $p = 0$) on the left, top, and right boundaries of $\hat{\Omega}_{b}$.

After adding the two stress terms, and recalling 
the definition of $\hat{\bd{\Phi}}^{\omega}_{\Gamma}$ in \eqref{phif} and  $\hat{\mathcal{J}}^{\omega}_{\Gamma}$ in \eqref{Jf} we obtain:
\begin{multline*}
-\int_{\Gamma(t)} \bd{\sigma}_{f}(\nabla \bd{u}, \pi) \bd{n} \cdot \bd{v} + \int_{\hat{\Gamma}} \hat{S}_{b}(\hat{\nabla} \hat{\bd{\eta}}, \hat{p}) \bd{e}_{y} \cdot \hat{\bd{\psi}} \\
= \int_{\Gamma(t)} \bd{\sigma}_{f}(\nabla \bd{u}, \pi) \bd{n} \cdot (\bd{\psi} - \bd{v}) + \int_{\hat{\Gamma}} (\hat{S}_{b}(\hat{\nabla} \hat{\bd{\eta}}, \hat{p}) \boldsymbol{e}_{y} - \hat{\mathcal{J}}^{\omega}_{\Gamma} \cdot (\boldsymbol{\sigma}_{f}(\nabla \boldsymbol{u}, \pi) \boldsymbol{n}|_{\Gamma(t)} \circ \bd{\Phi}^{\omega}_{\Gamma}) \cdot \hat{\bd{\psi}}.
\end{multline*}
Since the displacement of the plate is only in the $y$ direction so that $\hat{\bd{\eta}} = \hat{\omega} \bd{e}_{y}$ on $\hat{\Gamma}$, the test function $\hat{\bd{\psi}}$ points in the $y$ direction on $\hat{\Gamma}$ as well. We will denote by $\hat{\varphi}$ the magnitude of $\hat{\bd{\psi}}|_{\hat{\Gamma}}$ so that $\hat{\bd{\psi}} = \hat{\varphi} \bd{e}_{y}$ on $\hat{\Gamma}$. By the dynamic coupling condition \eqref{dynamic}, we have that the previous expression is equal to
\begin{align*}
&= \int_{\Gamma(t)} \bd{\sigma}_{f}(\nabla \bd{u}, \pi) \bd{n} \cdot (\bd{\psi} - \bd{v}) + 
\int_{\hat{\Gamma}} \hat{F}_{p} \cdot \hat{\varphi} = \int_{\Gamma(t)} \bd{\sigma}_{f}(\nabla \bd{u}, \pi) \bd{n} \cdot (\bd{\psi} - \bd{v}) + \int_{\hat{\Gamma}} (\rho_{p} \partial_{tt} \hat{\omega} + \hat{\Delta}^{2} \hat{\omega}) \hat{\varphi} \\
&= \int_{\Gamma(t)} \bd{\sigma}_{f}(\nabla \bd{u}, \pi) \bd{n} \cdot \bd{n} (\psi_{n} - v_{n}) + \int_{\Gamma(t)} \bd{\sigma}_{f}(\nabla \bd{u}, \pi) \bd{n} \cdot \bd{\tau} (\psi_{\tau} - v_{\tau}) + \int_{\hat{\Gamma}} (\rho_{p} \partial_{tt} \hat{\omega} + \hat{\Delta}^{2} \hat{\omega}) \hat{\varphi} \\
&= \int_{\Gamma(t)} \bd{\sigma}_{f}(\nabla \bd{u}, \pi) \bd{n} \cdot \bd{n} (\psi_{n} - v_{n}) + \int_{\Gamma(t)} \beta (\bd{\xi} - \bd{u}) \cdot \bd{\tau} (\psi_{\tau} - v_{\tau}) + \int_{\hat{\Gamma}} (\rho_{p} \partial_{tt} \hat{\omega} + \hat{\Delta}^{2} \hat{\omega}) \hat{\varphi} \\
&= \int_{\Gamma(t)} \left(\frac{1}{2}|\bd{u}|^{2} - p\right) (\psi_{n} - v_{n}) + \int_{\Gamma(t)} \beta (\bd{\xi} - \bd{u}) \cdot \bd{\tau} (\psi_{\tau} - v_{\tau}) \\
&+ \frac{d}{dt}\left(\int_{\hat{\Gamma}} \rho_{p} \partial_{t}\hat{\omega} \cdot \hat{\varphi}\right) - \int_{\hat{\Gamma}} \rho_{p} \partial_{t}\hat{\omega} \cdot \partial_{t}\hat{\varphi} + \int_{\hat{\Gamma}} \hat{\Delta} \hat{\omega} \cdot \hat{\Delta} \hat{\varphi},
\end{align*}
where we used the coupling conditions \eqref{bjs} and \eqref{pressurebalance} in the last step. {{For clarity, we note that in the preceding calculation and in the remainder of the manuscript, $\bd{n}$ denotes the unit normal vector along $\Gamma(t)$ that points upward towards $\Omega_{b}(t)$, and $\bd{\tau}$ denotes the unit tangential vector along $\Gamma(t)$ that points to the right.}}

The weak formulation then follows by summing everything together. {{Before we state the definition of a weak solution to our FPSI problem, we introduce the following notation.
{{Let $\hat{\zeta}$ denote  the {\bf{transverse velocity}} of the plate so that 
\begin{equation}\label{zeta}
\partial_{t}\hat{\omega} = \hat{\zeta},
\end{equation}
and let $\zeta = \hat{\zeta} \circ (\bd{\Phi}^{\omega}_{\Gamma})^{-1}$.}}}}

\begin{definition}\label{weak_solution}
The ordered four-tuple $(\boldsymbol{u}, \hat{\omega}, \hat{\boldsymbol{\eta}}, p)$ {{satisfies the weak formulation to the nonlinearly coupled FPSI problem}} if for every test function $(\boldsymbol{v}, \hat{\varphi}, \hat{\boldsymbol{\psi}}, r)$ that is $C^{1}_{c}$ in time on {{$[0, T)$}} taking values in the test space, satisfying $\hat{\boldsymbol{\psi}} = \hat{\varphi} \bd{e}_{y}$ on $\hat{\Gamma}$, we have that
\begin{multline}\label{weakphysical}
-\int_{0}^{T} \int_{\Omega_{f}(t)} \boldsymbol{u} \cdot \partial_{t}\boldsymbol{v} + \frac{1}{2} \int_{0}^{T} \int_{\Omega_{f}(t)} [((\boldsymbol{u} \cdot \nabla) \boldsymbol{u}) \cdot \boldsymbol{v} - ((\boldsymbol{u} \cdot \nabla)\boldsymbol{v}) \cdot \boldsymbol{u}] + \frac{1}{2} \int_{0}^{T} \int_{\Gamma(t)} (\boldsymbol{u} \cdot \boldsymbol{n} - 2\zeta \bd{e}_{y} \cdot \boldsymbol{n}) \boldsymbol{u} \cdot \boldsymbol{v} \\
+ 2\nu \int_{0}^{T} \int_{\Omega_{f}(t)} \boldsymbol{D}(\boldsymbol{u}) : \boldsymbol{D}(\boldsymbol{v}) + \int_{0}^{T} \int_{\Gamma(t)} \left(\frac{1}{2}|\boldsymbol{u}|^{2} - p\right)(\psi_{n} - v_{n}) + \beta \int_{0}^{T}  \int_{\Gamma(t)} (\zeta \bd{e}_{y} - \bd{u}) \cdot \bd{\tau} (\bd{\psi} - \bd{v}) \cdot \bd{\tau} \\
- \rho_{p} \int_{0}^{T} \int_{\hat{\Gamma}} \partial_{t} \hat{\omega} \cdot \partial_{t} \hat{\varphi} + \int_{0}^{T} \int_{\hat{\Gamma}} \hat{\Delta} \hat{\omega} \cdot \hat{\Delta} \hat{\varphi} - \rho_{b} \int_{0}^{T} \int_{\hat{\Omega}_{b}} \partial_{t} \hat{\boldsymbol{\eta}} \cdot \partial_{t} \hat{\boldsymbol{\psi}} + 2\mu_{e} \int_{0}^{T} \int_{\hat{\Omega}_{b}} \hat{\boldsymbol{D}}(\hat{\boldsymbol{\eta}}) : \hat{\boldsymbol{D}}(\hat{\boldsymbol{\psi}}) \\
+ \lambda_{e} \int_{0}^{T} \int_{\hat{\Omega}_{b}} (\hat{\nabla} \cdot \hat{\boldsymbol{\eta}})(\hat{\nabla} \cdot \hat{\boldsymbol{\psi}}) + 2\mu_{v} \int_{0}^{T} \int_{\hat{\Omega}_{b}} \hat{\bd{D}}(\partial_{t}\hat{\bd{\eta}}) : \hat{\bd{D}}(\hat{\bd{\psi}}) + \lambda_{v} \int_{0}^{T} \int_{\hat{\Omega}_{b}} (\hat{\nabla} \cdot \partial_{t} \hat{\bd{\eta}})(\hat{\nabla} \cdot \hat{\bd{\psi}}) \\
- \alpha \int_{0}^{T} \int_{\Omega_{b}(t)} p \nabla \cdot \boldsymbol{\psi} - c_{0} \int_{0}^{T} \int_{\hat{\Omega}_{b}} \hat{p} \cdot \partial_{t}\hat{r} - \alpha \int_{0}^{T} \int_{\Omega_{b}(t)} \frac{D}{Dt} \boldsymbol{\eta} \cdot \nabla r - \alpha \int_{0}^{T} \int_{\Gamma(t)} (\zeta \bd{e}_{y} \cdot \boldsymbol{n}) r \\
+ \kappa \int_{0}^{T} \int_{\Omega_{b}(t)} \nabla p \cdot \nabla r - \int_{0}^{T} \int_{\Gamma(t)} ((\boldsymbol{u} - \zeta \bd{e}_{y}) \cdot \boldsymbol{n}) r \\
= \int_{\Omega_{f}(0)} \boldsymbol{u}(0) \cdot \boldsymbol{v}(0) + \rho_{p} \int_{\hat{\Gamma}} \partial_{t}\hat{\omega}(0) \cdot \hat{\varphi}(0) + \rho_{b} \int_{\hat{\Omega}_{b}} \partial_{t}\hat{\boldsymbol{\eta}}(0) \cdot \hat{\boldsymbol{\psi}}(0) + c_{0} \int_{\hat{\Omega}_{b}} \hat{p}(0) \cdot \hat{r}(0).
\end{multline}
\end{definition}
\begin{remark}\label{RemarkNonregularized}
{{It is immediate to see that a classical (temporally and spatially smooth) solution to the FPSI problem satisfies the weak formulation stated above. However, when considering less regular solutions (in particular, weak solutions in the class of finite-energy solutions), the above weak formulation \textit{{{is inadequate}} for the regularity of finite-energy solutions} for the following reason.}} By the energy estimates (see Section \ref{energy}), the regularity of the structure displacement $\hat{\boldsymbol{\eta}}$ on $\hat{\Omega}_{b}$  is $L^{\infty}(0, T, H^{1}(\hat{\Omega}_{b}))$, which is not enough regularity to interpret the term
\begin{equation*}
\alpha \int_{\Omega_{b}(t)} p \nabla \cdot \boldsymbol{\psi},
\end{equation*}
since the test function has regularity $\hat{\boldsymbol{\psi}} \in H^{1}(\hat{\Omega}_{b})$ on the \textit{fixed reference domain}, due to the corresponding finite energy regularity of $\hat{\bd{\eta}}$. Hence, after changing variables, which adds an extra factor of $\det(\boldsymbol{I} + \hat{\nabla} \hat{\boldsymbol{\eta}})$ arising from the Jacobian, which is only in $L^{\infty}(0, T; L^{1}(\hat{\Omega}_{b}))$ in two dimensions, there is not enough regularity to guarantee that this integral is finite.
Therefore, we cannot interpret the above notion of weak solution properly in the space of finite energy solutions, as the finite energy space does not have enough regularity to make sense of certain integrals in the weak formulation, involving the deformed domain $\Omega_{b}(t)$. 
\end{remark}

This is why we introduce a regularized problem, which is consistent with the original problem in the sense that weak solutions to the regularized problem converge, as the regularization parameter tends to zero, to a smooth solution of the original, nonregularized problem, when a smooth solution exists. This {{weak-classical consistency}} result will be shown in Sec.~\ref{weakstrong}. 

\section{Regularized weak solution and statement of existence result}\label{regularized}

Since all the mathematical challenges related to the inability to properly interpret all of the terms in the weak solution arise fundamentally from the \textit{lack of regularity} of $\hat{\boldsymbol{\eta}}$ on $\hat{\Omega}_{b}$,  we will regularize  $\hat{\boldsymbol{\eta}}$ via a convolution with a smooth, compactly supported kernel, and introduce an appropriate \textit{regularized weak formulation} of the original FPSI problem. 
Because we are working on a bounded domain $\hat{\Omega}_{b}$, we must be careful to introduce the convolution in a way that preserves the Dirichlet condition on the left, top, and right boundaries of $\hat{\Omega}_{b}=(0, L) \times (0, R)$.

This is why we define an {\emph{extended domain}} $\tilde{\Omega}_{b}$: 
\begin{equation*}
\tilde{\Omega}_{b} = [-L, 2L] \times [-R, 2R],
\end{equation*}
so that for $\delta < \min(L, R)$ the convolution of a function on $\tilde{\Omega}_{b}$ with a smooth function of compact support in the closed ball of radius $\delta$ gives a function defined on $\hat{\Omega}_{b}$.
We then introduce an odd extension along the lines $\hat{x} = 0$, $\hat{x} = L$, $\hat{y} = 0$ and $\hat{y} = R$ as follows.
\begin{definition}\label{extension}
Given $\hat{\bd{\eta}}$ defined on $\hat{\Omega}_{b}$ satisfying $\hat{\bd{\eta}} = 0$ on $\hat{x} = 0$, $\hat{x} = L$, and $\hat{y} = R$ and $\hat{\bd{\eta}} = \hat{\omega} \bd{e}_{y}$ on $\hat{y} = 0$, define the \textbf{odd extension of $\hat{\bd{\eta}}$ to $\tilde{\Omega}_{b}$} by keeping $\hat{\bd{\eta}}$ the same on 
$\hat{\Omega}_{b}=[0, L] \times [0, R]$ and defining $\hat{\bd{\eta}}$ outside of the closure of $\hat{\Omega}_{b}$ as follows:
\begin{enumerate}
\item On $[0, L] \times [-R, 0]$, set $\hat{\bd{\eta}}(\hat{x}, \hat{y}) = \hat{\omega}(\hat{x}) \bd{e}_{y} + (\hat{\omega}(\hat{x}) \bd{e}_{y} - \hat{\bd{\eta}}(\hat{x}, -\hat{y}))$.
\item On $[0, L] \times [R, 2R]$, set $\hat{\bd{\eta}}(\hat{x}, \hat{y}) = -\hat{\bd{\eta}}(\hat{x}, 2R - \hat{y})$. 
\item On $[-L, 0] \times [-R, 2R]$, set $\hat{\bd{\eta}}(\hat{x}, \hat{y}) = -\hat{\bd{\eta}}(-\hat{x}, \hat{y})$.
\item On $[L, 2L] \times [-R, 2R]$, set $\hat{\bd{\eta}}(\hat{x}, \hat{y}) = -\hat{\bd{\eta}}(2L - \hat{x}, \hat{y})$.
\end{enumerate}
\end{definition}

Let $\sigma$ be a radially symmetric function on $\mathbb{R}^{2}$ with compact support in the closed ball of radius one such that $\displaystyle \int_{\mathbb{R}^{2}} \sigma = 1$, and define
\begin{equation*}
\sigma_{\delta} = \delta^{-2} \sigma(\delta^{-1} \bd{x}), \qquad \text{ on } \mathbb{R}^{2}.
\end{equation*}

\begin{definition}
We define the following {\bf{regularized functions}} which are spatially smooth on $\hat{\Omega}_{b}$:
\begin{itemize}
\item The regularized Biot displacement{{, which is}} obtained by extending $\hat{\boldsymbol{\eta}}$ to $\tilde{\Omega}_{b}$ by odd extension and defining:
{{\begin{equation}\label{etadelta}
\hat{\boldsymbol{\eta}}^{\delta}(t, \hat{x}, \hat{y}) = (\hat{\boldsymbol{\eta}} * \sigma_{\delta})(t, \hat{x}, \hat{y}) := \int_{\R^{2d}} \hat{\bd{\eta}}(t, \hat{x} - z, \hat{y} - w) \sigma_{\delta}(z, w) dz dw, \qquad \text{ on } \hat{\Omega}_{b},
\end{equation}}}
\item The regularized Lagrangian mapping:
{{\begin{equation}\label{phidelta}
\hat{\boldsymbol{\Phi}}_{b}^{\eta^{\delta}}(t, \cdot) = \text{Id} + \hat{\boldsymbol{\eta}}^{\delta}(t, \cdot),
\end{equation}}}
\item The regularized moving Biot domain:
{{\begin{equation}\label{omegadelta}
\Omega^{\delta}_{b}(t) = \hat{\boldsymbol{\Phi}}_{b}^{\eta^{\delta}}(t, \hat{\Omega}_{b}).
\end{equation}}}
Note that even though the kinematic coupling condition holds for $\hat{\bd{\eta}}$ in the sense that $\hat{\bd{\eta}}|_{\hat{\Gamma}} = \hat{\omega} \bd{e}_{y}$, it is \textit{not necessarily true} that $\hat{\bd{\eta}}^{\delta}|_{\hat{\Gamma}} = \hat{\omega} \bd{e}_{y}$. Therefore, we will also define:
\item The regularized moving interface:
{{\begin{equation*}
\Gamma^{\delta}(t) = \hat{\bd{\Phi}}^{\eta^{\delta}}_{b}(t, \hat{\Gamma}). 
\end{equation*}}}
Alternatively, $\hat{\Gamma}^{\delta}$ is the plate interface if it were displaced from the reference configuration $\hat{\Gamma}$ in the direction $\hat{\bd{\eta}}^{\delta}|_{\hat{\Gamma}}$, which is a purely transverse $y$ displacement, as one can verify.  
\end{itemize}
\end{definition}

Note that by the way we extended $\hat{\bd{\eta}}$ to the larger domain $\tilde{\Omega}_{b}$ we have that
\begin{equation*}
\hat{\boldsymbol{\eta}}^{\delta} = 0 \quad {\rm on} \quad \partial \hat{\Omega}_{b} \setminus \hat{\Gamma}.
\end{equation*}

With these regularized versions of the Biot structure displacement and velocity, we can now define the notion of a \textit{weak solution to
the regularized weak FPSI problem with the regularization parameter $\delta$}. We start by defining the solution and test space, which are motivated by the energy estimates in Section \ref{energy}, and then we state the regularized weak formulation in the moving domain framework and in the fixed reference domain framework.

{{
\begin{remark}\label{RemarkAboutInterface}
We have regularized the physical Biot domain using the regularized Biot displacement, which results in the regularized moving Biot domain $\Omega_{b}^{\delta}(t)$ as stated in \eqref{omegadelta}. We emphasize that the main reason for this regularization is because the structure displacement $\bd{\eta}$, which is  in the finite energy space $H^{1}(\Omega_{b})$ (without regularization), does not posses sufficient regularity to make sense of the definition of $\Omega_{b}^{\delta}(t)$. However, while the moving Biot domain requires regularization, we emphasize that there is no need to regularize the fluid domain $\Omega_{f}(t)$ because the fluid domain $\Omega_{f}(t)$ can be defined without explicit reference to the Biot displacement $\bd{\eta}$, and hence it is not affected by the regularity issues associated with $\bd{\eta}$. In particular, $\Omega_{f}(t)$ can already be well-defined by using just the plate displacement $\omega$. Even though $\omega$ is the trace of $\bd{\eta}$ along $\Gamma$, the plate displacement $\omega$ has additional regularity in the finite energy space, i.e.,  $\omega \in H_{0}^{2}(\Gamma)$ due to the fact that  $\omega$ itself satisfies the plate equation. This makes $\omega$ a continuous function on $\Gamma$ which allows us to define $\Omega_{f}(t)$ unambiguously. 
\end{remark}}}

\subsection{Functional spaces and definition of weak solutions}

\begin{definition}(Solution and test spaces for the regularized problem)
\begin{itemize}
\item \textit{Fluid function space (moving domain/Eulerian formulation).}
\begin{equation}\label{Vft}
V_{f}(t) = \{\bd{u} = (u_{x}, u_{y}) \in H^{1}(\Omega_{f}(t)) : \nabla \cdot \bd{u} = 0, \text{ and } \bd{u} = 0 \text{ when } x = 0, x = L, y = -R\},
\end{equation}
\begin{equation}\label{fmovingspace}
\mathcal{V}_{f} = L^{\infty}(0, T; L^{2}(\Omega_{f}(t))) \cap L^{2}(0, T; V_{f}(t)).
\end{equation}
\item \textit{Fluid function space (fixed domain/Lagrangian formulation).}
\begin{equation}\label{Vf}
V^{\omega}_{f} = \{\hat{\boldsymbol{u}} = (\hat{u}_{x}, \hat{u}_{y}) \in H^{1}(\hat{\Omega}_{f}) : \hat{\nabla}^{\omega}_{f} \cdot \hat{\boldsymbol{u}} = 0, \text{ and } \hat{\boldsymbol{u}} = 0 \text{ when } \hat{x} = 0, \hat{x} = L, \hat{y} = -R\},
\end{equation}
\begin{equation}\label{frefspace}
\mathcal{V}^{\omega}_{f} = L^{\infty}(0, T; L^{2}(\hat{\Omega}_{f})) \cap L^{2}(0, T; V^{\omega}_{f}).
\end{equation}
\item \textit{Plate function space.}
\begin{equation}\label{omegarefspace}
\mathcal{V}_{\omega} = W^{1, \infty}(0, T; L^{2}(\hat{\Gamma})) \cap L^{\infty}(0, T; H_{0}^{2}(\hat{\Gamma})).
\end{equation}
\item \textit{Biot displacement function space.}
\begin{equation}\label{Vd}
V_{d} = \{\hat{\boldsymbol{\eta}} = (\hat{\eta}_{x}, \hat{\eta}_{y}) \in H^{1}(\hat{\Omega}_{b}): \hat{\boldsymbol{\eta}} = 0  \text{ for } \hat{x} = 0, \hat{x} = L, \hat{y} = R, \text{ and } \hat{\eta}_{x} = 0 \text{ on } \hat{\Gamma}\},
\end{equation}
\begin{equation}\label{brefspace}
\mathcal{V}_{b} = W^{1, \infty}(0, T; L^{2}(\hat{\Omega}_{b})) \cap L^{\infty}(0, T; V_{d}) \cap H^{1}(0, T; V_{d}).
\end{equation}
\item \textit{Biot pore pressure function space.} 
\begin{equation}\label{Vp}
V_{p} = \{\hat{p} \in H^{1}(\hat{\Omega}_{b}): \hat{p} = 0 \text{ for } \hat{x} = 0, \hat{x} = L, \hat{y} = R\},
\end{equation}
\begin{equation}\label{prefspace}
\mathcal{Q}_{b} = L^{\infty}(0, T; L^{2}(\hat{\Omega}_{b})) \cap L^{2}(0, T; V_{p}).
\end{equation}
\item \textit{Weak solution space (moving domain).}
\begin{equation}\label{solnspacemoving}
\mathcal{V}_{\text{sol}} = \{(\boldsymbol{u}, \hat{\omega}, \hat{\boldsymbol{\eta}}, \hat{p}) \in \mathcal{V}_{f} \times \mathcal{V}_{\omega} \times \mathcal{V}_{b} \times \mathcal{Q}_{b} : \hat{\boldsymbol{\eta}} = \hat{\omega} \boldsymbol{e}_{y} \text{ on } \hat{\Gamma}\}.
\end{equation}
\item \textit{Weak solution space (fixed domain).} 
\begin{equation}\label{solnspace}
\mathcal{V}^{\omega}_{\text{sol}} = \{(\hat{\boldsymbol{u}}, \hat{\omega}, \hat{\boldsymbol{\eta}}, \hat{p}) \in \mathcal{V}^{\omega}_{f} \times \mathcal{V}_{\omega} \times \mathcal{V}_{b} \times \mathcal{Q}_{b} : \hat{\boldsymbol{\eta}} = \hat{\omega} \boldsymbol{e}_{y} \text{ on } \hat{\Gamma}\}.
\end{equation}
\item \textit{Test space (moving domain).}
\begin{equation}\label{testspacemoving}
\mathcal{V}_{\text{test}} = \{(\boldsymbol{v}, \hat{\varphi}, \hat{\boldsymbol{\psi}}, \hat{r}) \in C_{c}^{1}([0, T); V_{f}(t) \times H_{0}^{2}(\hat{\Gamma}) \times V_{d} \times V_{p}) : \hat{\boldsymbol{\psi}} = \hat{\varphi} \boldsymbol{e}_{y} \text{ on } \hat{\Gamma}\}.
\end{equation}
\item \textit{Test space (fixed domain).} 
\begin{equation}\label{testspace}
\mathcal{V}^{\omega}_{\text{test}} = \{(\hat{\boldsymbol{v}}, \hat{\varphi}, \hat{\boldsymbol{\psi}}, \hat{r}) \in C_{c}^{1}([0, T); V^{\omega}_{f} \times H_{0}^{2}(\hat{\Gamma}) \times V_{d} \times V_{p}) : \hat{\boldsymbol{\psi}} = \hat{\varphi} \boldsymbol{e}_{y} \text{ on } \hat{\Gamma}\}.
\end{equation}
\end{itemize}
\end{definition}

\begin{remark}
Because $\hat{\Gamma}$ is one dimensional, for plate displacements $\hat{\omega} \in \mathcal{V}_{\omega}$, we have that $\hat{\omega} \in C(0, T; C^{1}(\hat{\Gamma}))$ and hence, there is a one-to-one correspondence between functions in $\mathcal{V}_{\text{sol}}$ and $\mathcal{V}^{\omega}_{\text{sol}}$ and functions in $\mathcal{V}_{\text{test}}$ and $\mathcal{V}^{\omega}_{\text{test}}$, given by composition with the ALE mapping \eqref{phif}.
% for the fluid domain, $\hat{\bd{\Phi}}^{\omega}_{f}: \hat{\Omega}_{f} \to \Omega_{f}(t)$, for the component involving fluid velocities. 
\end{remark}

{{Next, we state the weak formulation to the regularized problem as follows.}}

\begin{definition}\label{regularizedmoving}(Weak solution to the regularized problem, {\bf{moving fluid domain formulation}})
An ordered four-tuple $(\boldsymbol{u}, \hat{\omega}, \hat{\boldsymbol{\eta}}, p) \in \mathcal{V}_{\text{sol}}$ is a \textit{weak solution to the regularized nonlinearly coupled FPSI problem with regularization parameter $\delta$} if for every test function $(\boldsymbol{v}, \hat{\varphi}, \hat{\boldsymbol{\psi}}, \hat{r}) \in \mathcal{V}_{\text{test}}$,
\begin{multline}\label{deltaweakphysical}
-\int_{0}^{T} \int_{\Omega_{f}(t)} \boldsymbol{u} \cdot \partial_{t}\boldsymbol{v} + \frac{1}{2} \int_{0}^{T} \int_{\Omega_{f}(t)} [((\boldsymbol{u} \cdot \nabla) \boldsymbol{u}) \cdot \boldsymbol{v} - ((\boldsymbol{u} \cdot \nabla)\boldsymbol{v}) \cdot \boldsymbol{u}] + \frac{1}{2} \int_{0}^{T} \int_{\Gamma(t)} (\boldsymbol{u} \cdot \boldsymbol{n} - 2\zeta \bd{e}_{y} \cdot \boldsymbol{n}) \boldsymbol{u} \cdot \boldsymbol{v} \\
+ 2\nu \int_{0}^{T} \int_{\Omega_{f}(t)} \boldsymbol{D}(\boldsymbol{u}) : \boldsymbol{D}(\boldsymbol{v}) + \int_{0}^{T} \int_{\Gamma(t)} \left(\frac{1}{2}|\boldsymbol{u}|^{2} - p\right)(\psi_{n} - v_{n}) + \beta \int_{0}^{T} \int_{\Gamma(t)} (\zeta \bd{e}_{y} - \bd{u}) \cdot \bd{\tau} (\bd{\psi} - \bd{v}) \cdot \bd{\tau} \\
- \rho_{p} \int_{0}^{T} \int_{\hat{\Gamma}} \partial_{t} \hat{\omega} \cdot \partial_{t} \hat{\varphi} + \int_{0}^{T} \int_{\hat{\Gamma}} \hat{\Delta} \hat{\omega} \cdot \hat{\Delta} \hat{\varphi} - \rho_{b} \int_{0}^{T} \int_{\hat{\Omega}_{b}} \partial_{t} \hat{\boldsymbol{\eta}} \cdot \partial_{t}\hat{\boldsymbol{\psi}} + 2\mu_{e} \int_{0}^{T} \int_{\hat{\Omega}_{b}} \hat{\boldsymbol{D}}(\hat{\boldsymbol{\eta}}) : \hat{\boldsymbol{D}}(\hat{\boldsymbol{\psi}}) \\
+ \lambda_{e} \int_{0}^{T} \int_{\hat{\Omega}_{b}} (\hat{\nabla} \cdot \hat{\boldsymbol{\eta}})(\hat{\nabla} \cdot \hat{\boldsymbol{\psi}}) + 2\mu_{v} \int_{0}^{T} \int_{\hat{\Omega}_{b}} \hat{\bd{D}}(\partial_{t} \hat{\bd{\eta}}) : \hat{\bd{D}}(\hat{\bd{\psi}}) + \lambda_{v} \int_{0}^{T} \int_{\hat{\Omega}_{b}} (\hat{\nabla} \cdot \partial_{t}\hat{\bd{\eta}}) (\hat{\nabla} \cdot \hat{\bd{\psi}}) \\
- \alpha \int_{0}^{T} \int_{\Omega_{b}^{\delta}(t)} p \nabla \cdot \boldsymbol{\psi} - c_{0} \int_{0}^{T} \int_{\hat{\Omega}_{b}} \hat{p} \cdot \partial_{t}\hat{r} - \alpha \int_{0}^{T} \int_{\Omega_{b}^{\delta}(t)} \frac{D^{\delta}}{Dt} \boldsymbol{\eta} \cdot \nabla r - \alpha \int_{0}^{T} \int_{\Gamma^{\delta}(t)} (\zeta \bd{e}_{y} \cdot \boldsymbol{n}^{\delta}) r \\
+ \kappa \int_{0}^{T} \int_{\Omega^{\delta}_{b}(t)} \nabla p \cdot \nabla r - \int_{0}^{T} \int_{\Gamma(t)} ((\boldsymbol{u} - \zeta \bd{e}_{y}) \cdot \boldsymbol{n}) r \\
= \int_{\Omega_{f}(0)} \boldsymbol{u}(0) \cdot \boldsymbol{v}(0) + \rho_{p} \int_{\hat{\Gamma}} \partial_{t}\hat{\omega}(0) \cdot \hat{\varphi}(0) + \rho_{b} \int_{\hat{\Omega}_{b}} \partial_{t}\hat{\boldsymbol{\eta}}(0) \cdot \hat{\boldsymbol{\psi}}(0) + c_{0} \int_{\hat{\Omega}_{b}} \hat{p}(0) \cdot \hat{r}(0),
\end{multline}
where $\frac{D^{\delta}}{Dt} = \frac{d}{dt} + (\boldsymbol{\xi}^{\delta} \cdot \nabla)$ with {{$\bd{\xi}^{\delta}(t, \cdot) = \partial_{t} \hat{\bd{\eta}}^{\delta}(t, (\bd{\Phi}^{\eta^{\delta}}_{b})^{-1}(t, \cdot))$}} is the material derivative with respect to the regularized displacement, $\boldsymbol{n}$ denotes the upward pointing normal vector to $\Gamma(t)$, and $\bd{n}^{\delta}$ denotes the upward pointing normal vector to $\Gamma^{\delta}(t)$. 
\end{definition}

Notice that only four terms contain regularization via convolution with parameter $\delta$. While there are many different ways to write the regularized 
weak formulation, the regularization presented above is a regularization that deviates from the original, nonregularized problem, in the smallest possible
number of terms, and is still consistent  with the original, nonregularized problem, as we show later.

\begin{remark}
While the solution to the regularized problem above depends on the regularization parameter $\delta$ implicitly, to simplify notation we will drop the
$\delta$ notation whenever it is clear from the context that we are working with the solution to the regularized problem. 
\end{remark}

\begin{remark}
We simplify notation by omitting the explicit compositions with the maps $\hat{\bd{\Phi}}^{\omega}_{f}$, $\hat{\bd{\Phi}}^{\omega}_{\Gamma}$, $\hat{\bd{\Phi}}^{\eta}_{b}$, and $\hat{\bd{\Phi}}^{\eta^{\delta}}_{b}$,  and their inverses. The necessary compositions with such mappings will be clear from the context. For example,  
\begin{equation*}
-\alpha \int_{0}^{T} \int_{\Omega_{b}^{\delta}(t)} p \nabla \cdot \bd{\psi}
\quad  {\rm means} \quad 
-\alpha \int_{0}^{T} \int_{\Omega_{b}^{\delta}(t)} \left(\hat{p} \circ (\bd{\Phi}^{\eta^{\delta}}_{b})^{-1}\right) \nabla \cdot \left(\hat{\bd{\psi}} \circ (\bd{\Phi}^{\eta^{\delta}}_{b})^{-1}\right),
\end{equation*}
and 
\begin{equation*}
-\int_{0}^{T} \int_{\Gamma(t)} ((\bd{u} - \zeta \bd{e}_{y}) \cdot \bd{n}) r 
\quad {\rm means} \quad 
-\int_{0}^{T} \int_{\Gamma(t)} \left(\left(\bd{u} - (\zeta \circ (\bd{\Phi}^{\omega}_{\Gamma})^{-1})\bd{e}_{y}\right) \cdot \bd{n}\right) \left(\hat{r} \circ (\bd{\Phi}^{\eta}_{b})^{-1}\right).
\end{equation*}
\end{remark}
{Notice that here we tacitly assume that $\bd{\Phi}^{\eta^{\delta}}_{b}$ is invertible. We will later justify this assumption by proving that it holds on some time interval $[0,T_{\delta}]$, where the time $T_{\delta}$ may depend on the regularization parameter $\delta$.}
Next, we reformulate the definition of a regularized weak solution on the {\bf{fixed reference domain}}. Recall that the Jacobians
$\hat{\mathcal{J}}^{\omega}_{f}$, $\hat{\mathcal{J}}^{\eta}_{b}$, and $\hat{\mathcal{J}}^{\omega}_{\Gamma}$ in \eqref{Jf}
 will appear upon using a change of variables to map the problem onto the reference domain.
To transform the first term in the weak formulation  \eqref{deltaweakphysical} above, we use  \eqref{womega} to transform the time derivatives  and  assume that $|\hat{\omega}| < R$ so that there is no domain degeneracy. {{After using \eqref{transformedgradient} and \eqref{womega},}} we get
\begin{align}\label{transferref1}
&\int_{\Omega_{f}(t)} \boldsymbol{u} \cdot \partial_{t}\boldsymbol{v} = \int_{\hat{\Omega}_{f}} \left(1 + \frac{\hat{\omega}}{R}\right) \hat{\boldsymbol{u}} \cdot \partial_{t} \hat{\boldsymbol{v}} - \int_{\hat{\Omega}_{f}} \left(1 + \frac{\hat{\omega}}{R}\right) \hat{\boldsymbol{u}} \cdot [(\hat{\boldsymbol{w}} \cdot \hat{\nabla}^{\omega}_{f}) \hat{\boldsymbol{v}}] 
\nonumber\\
&= \int_{\hat{\Omega}_{f}} \left(1 + \frac{\hat{\omega}}{R}\right) \hat{\bd{u}} \cdot \partial_{t} \hat{\bd{v}} - \frac{1}{R} \int_{\hat{\Omega}_{f}} \hat{\bd{u}} \cdot [(R + \hat{y}) \partial_{t}\hat{\omega} \partial_{\hat{y}} \hat{\bd{v}}] 
\nonumber\\
&= \int_{\hat{\Omega}_{f}} \left(1 + \frac{\hat{\omega}}{R}\right) \hat{\boldsymbol{u}} \cdot \partial_{t} \hat{\boldsymbol{v}} - \frac{1}{2R} \int_{\hat{\Omega}_{f}} \hat{\boldsymbol{u}} \cdot [(R + \hat{y}) \partial_{t}\hat{\omega} \partial_{\hat{y}} \hat{\boldsymbol{v}}] + \frac{1}{2R} \int_{\hat{\Omega}_{f}} (\partial_{t} \hat{\omega}) \hat{\bd{u}} \cdot \hat{\bd{v}} 
\nonumber\\
&+ \frac{1}{2R} \int_{\hat{\Omega}_{f}} [(R + \hat{y}) \partial_{t}\hat{\omega} \partial_{\hat{y}} \hat{\boldsymbol{u}}] \cdot \hat{\boldsymbol{v}} - \frac{1}{2} \int_{\hat{\Gamma}} (\hat{\boldsymbol{u}} \cdot \hat{\boldsymbol{v}}) \partial_{t}\hat{\omega} 
\nonumber\\
&= \int_{\hat{\Omega}_{f}} \left(1 + \frac{\hat{\omega}}{R}\right) \hat{\boldsymbol{u}} \cdot \partial_{t} \hat{\boldsymbol{v}} - \frac{1}{2} \int_{\hat{\Omega}_{f}} \left(1 + \frac{\hat{\omega}}{R}\right) [((\hat{\boldsymbol{w}} \cdot \hat{\nabla}^{\omega}_{f}) \hat{\boldsymbol{v}}) \cdot \hat{\boldsymbol{u}} - ((\hat{\boldsymbol{w}} \cdot \hat{\nabla}^{\omega}_{f}) \hat{\boldsymbol{u}}) \cdot \hat{\boldsymbol{v}}] 
\nonumber\\
&+ \frac{1}{2R}\int_{\hat{\Omega}_{f}} (\partial_{t}\hat{\omega}) \hat{\boldsymbol{u}} \cdot \hat{\boldsymbol{v}} - \frac{1}{2} \int_{\hat{\Gamma}} (\hat{\boldsymbol{u}} \cdot \hat{\boldsymbol{v}}) \partial_{t}\hat{\omega},
\end{align}
where we integrated by parts in the $\hat{y}$ direction. %Recall that $\hat{\boldsymbol{w}}$ is defined by \eqref{womega}. 
Note that the final term in \eqref{transferref1} will combine with the following term in \eqref{deltaweakphysical}:
\begin{equation}\label{equality}
\int_{0}^{T} \int_{\Gamma(t)} (\zeta \bd{e}_{y} \cdot \boldsymbol{n}) \boldsymbol{u} \cdot \boldsymbol{v} = \int_{0}^{T} \int_{\hat{\Gamma}} (\hat{\boldsymbol{u}} \cdot \hat{\boldsymbol{v}}) \partial_{t}\hat{\omega},
\end{equation}
where we used $\displaystyle \boldsymbol{n} = (-\partial_{\hat{x}}\hat{\omega}, 1)/{\hat{\mathcal{J}}^{\omega}_{\Gamma}}$ for the normal vector to the interface and $\zeta \bd{e}_{y}|_{\Gamma(t)} = \partial_{t}\hat{\omega} \bd{e}_{y}$. Because the transformation from $\Gamma(t)$ to $\hat{\Gamma}$ cancels out the factor of $\hat{\mathcal{J}}^{\omega}_{\Gamma}$ in the unit normal vector, it is useful to define the following renormalized normal and tangent vectors:
\begin{equation}\label{ntomega}
\hat{\boldsymbol{n}}^{\omega} = (-\partial_{\hat{x}}\hat{\omega}, 1), \qquad \hat{\boldsymbol{\tau}}^{\omega} = (1, \partial_{\hat{x}}\hat{\omega}).
\end{equation}
We similarly define
\begin{equation}\label{nomegadelta}
\hat{\bd{n}}^{\omega^{\delta}} = (-\partial_{\hat{x}} (\hat{\bd{\eta}}^{\delta}|_{\hat{\Gamma}}), 1).
\end{equation}
We are now ready to state the definition of a weak solution to the regularized problem on the fixed reference domain. 

\begin{definition}\label{regularizedfixed}(Weak solution to the regularized problem, {\bf{fixed fluid domain formulation}})
An ordered four-tuple $(\hat{\boldsymbol{u}}, \hat{\omega}, \hat{\boldsymbol{\eta}}, \hat{p}) \in \mathcal{V}^{\omega}_{\text{sol}}$ is a \textit{weak solution to the regularized nonlinearly coupled FPSI problem with regularization parameter $\delta$} if for all test functions $(\hat{\boldsymbol{v}}, \hat{\varphi}, \hat{\boldsymbol{\psi}}, \hat{r}) \in \mathcal{V}^{\omega}_{\text{test}}$, the following equality holds:
\begin{multline}\label{regularweakref}
-\int_{0}^{T} \int_{\hat{\Omega}_{f}} \left(1 + \frac{\hat{\omega}}{R}\right) \hat{\boldsymbol{u}} \cdot \partial_{t}\hat{\boldsymbol{v}} + \frac{1}{2} \int_{0}^{T} \int_{\hat{\Omega}_{f}} \left(1 + \frac{\hat{\omega}}{R}\right) [((\hat{\boldsymbol{u}} - \hat{\boldsymbol{w}}) \cdot \hat{\nabla}^{\omega}_{f} \hat{\boldsymbol{u}}) \cdot \hat{\boldsymbol{v}} - ((\hat{\boldsymbol{u}} - \hat{\boldsymbol{w}}) \cdot \hat{\nabla}^{\omega}_{f} \hat{\boldsymbol{v}}) \cdot \hat{\boldsymbol{u}}] \\
- \frac{1}{2R} \int_{0}^{T} \int_{\hat{\Omega}_{f}} (\partial_{t}\hat{\omega}) \hat{\boldsymbol{u}} \cdot \hat{\boldsymbol{v}} + \frac{1}{2} \int_{0}^{T} \int_{\hat{\Gamma}} (\hat{\boldsymbol{u}} \cdot \hat{\boldsymbol{n}}^{\omega} - \hat{\zeta} \bd{e}_{y} \cdot \hat{\boldsymbol{n}}^{\omega}) \hat{\boldsymbol{u}} \cdot \hat{\boldsymbol{v}} + 2\nu \int_{0}^{T} \int_{\hat{\Omega}_{f}} \left(1 + \frac{\hat{\omega}}{R}\right) \hat{\boldsymbol{D}}(\hat{\boldsymbol{u}}) : \hat{\boldsymbol{D}}(\hat{\boldsymbol{v}}) \\
+ \int_{0}^{T} \int_{\hat{\Gamma}} \left(\frac{1}{2}|\hat{\boldsymbol{u}}|^{2} - \hat{p}\right)(\hat{\boldsymbol{\psi}} - \hat{\boldsymbol{v}})\cdot \hat{\boldsymbol{n}}^{\omega} + \frac{\beta}{\hat{\mathcal{J}}^{\omega}_{\Gamma}} \int_{0}^{T} \int_{\hat{\Gamma}} (\hat{\zeta}\bd{e}_{y} - \hat{\boldsymbol{u}}) \cdot \hat{\boldsymbol{\tau}}^{\omega} (\hat{\boldsymbol{\psi}} - \hat{\boldsymbol{v}}) \cdot \hat{\boldsymbol{\tau}}^{\omega} \\
- \rho_{p} \int_{0}^{T} \int_{\hat{\Gamma}} \partial_{t} \hat{\omega} \cdot \partial_{t} \hat{\varphi} + \int_{0}^{T} \int_{\hat{\Gamma}} \hat{\Delta} \hat{\omega} \cdot \hat{\Delta} \hat{\varphi} - \rho_{b} \int_{0}^{T} \int_{\hat{\Omega}_{b}} \partial_{t} \hat{\boldsymbol{\eta}} \cdot \partial_{t}\hat{\boldsymbol{\psi}} + 2\mu_{e} \int_{0}^{T} \int_{\hat{\Omega}_{b}} \hat{\boldsymbol{D}}(\hat{\boldsymbol{\eta}}) : \hat{\boldsymbol{D}}(\hat{\boldsymbol{\psi}}) \\
+ \lambda_{e} \int_{0}^{T} \int_{\hat{\Omega}_{b}} (\hat{\nabla} \cdot \hat{\boldsymbol{\eta}})(\hat{\nabla} \cdot \hat{\boldsymbol{\psi}}) + 2\mu_{v} \int_{0}^{T} \int_{\hat{\Omega}_{b}} \hat{\bd{D}}(\partial_{t}\hat{\bd{\eta}}) : \hat{\bd{D}}(\hat{\bd{\psi}}) + \lambda_{v} \int_{0}^{T} \int_{\hat{\Omega}_{b}} (\hat{\nabla} \cdot \partial_{t}\hat{\bd{\eta}})(\hat{\nabla} \cdot \hat{\bd{\psi}}) \\
- \alpha \int_{0}^{T} \int_{\hat{\Omega}_{b}} \hat{\mathcal{J}}^{\eta^{\delta}}_{b} \hat{p} \hat{\nabla}^{\eta^{\delta}}_{b} \cdot \hat{\boldsymbol{\psi}} - c_{0} \int_{0}^{T} \int_{\hat{\Omega}_{b}} \hat{p} \cdot \partial_{t}\hat{r}
- \alpha \int_{0}^{T} \int_{\hat{\Omega}_{b}} \hat{\mathcal{J}}^{\eta^{\delta}}_{b} \partial_{t}\hat{\bd{\eta}} \cdot \hat{\nabla}^{\eta^{\delta}}_{b} \hat{r} \\
- \alpha \int_{0}^{T} \int_{\hat{\Gamma}} (\hat{\zeta} \bd{e}_{y} \cdot \hat{\boldsymbol{n}}^{\omega^{\delta}}) \hat{r} + \kappa \int_{0}^{T} \int_{\hat{\Omega}_{b}} \hat{\mathcal{J}}^{\eta^{\delta}}_{b} \hat{\nabla}^{\eta^{\delta}}_{b} \hat{p} \cdot \hat{\nabla}^{\eta^{\delta}}_{b} \hat{r} - \int_{0}^{T} \int_{\hat{\Gamma}} ((\hat{\boldsymbol{u}} - \hat{\zeta}\bd{e}_{y}) \cdot \hat{\boldsymbol{n}}^{\omega}) \hat{r} \\
= \int_{\Omega_{f}(0)} \boldsymbol{u}(0) \cdot \boldsymbol{v}(0) + \rho_{p} \int_{\hat{\Gamma}} \partial_{t}\hat{\omega}(0) \cdot \hat{\varphi}(0) + \rho_{b} \int_{\hat{\Omega}_{b}} \partial_{t}\hat{\boldsymbol{\eta}}(0) \cdot \hat{\boldsymbol{\psi}}(0) + c_{0} \int_{\hat{\Omega}_{b}} \hat{p}(0) \cdot \hat{r}(0),
\end{multline}
\end{definition}
\noindent where $\hat{\mathcal{J}}^{\eta^{\delta}}_{b}$ and $\hat{\mathcal{J}}^{\omega}_{\Gamma}$ are defined in \eqref{Jf},
$\hat{\boldsymbol{w}}$ is defined in \eqref{womega}, 
$\hat{\nabla}^{\omega}_{f}$ in \eqref{transformedgradient},
$\hat{\nabla}^{\eta^{\delta}}_{b}\hat{g}$ in \eqref{nablaeta}, and
$\hat{\zeta}$ in \eqref{zeta}.

\if 1 = 0 %%%%%%%%%%%%%
We list the definitions of all of the relevant expressions below.
\begin{equation*}
\hat{\boldsymbol{w}} = \frac{R + \hat{y}}{R} \partial_{t}\hat{\omega}\boldsymbol{e}_{y}, \qquad \hat{\nabla}^{\omega}_{f} = \left(\partial_{\hat{x}} - \frac{(R + \hat{y})\partial_{\hat{x}}\hat{\omega}}{R + \hat{\omega}} \partial_{\hat{y}}, \frac{R}{R + \hat{\omega}} \partial_{\hat{y}}\right), \qquad \hat{\nabla}^{\eta^{\delta}}_{b}\hat{g} = \left(\frac{\partial \hat{g}}{\partial \hat{x}}, \frac{\partial \hat{g}}{\partial \hat{y}}\right) \cdot (\boldsymbol{I} + \hat{\nabla} \hat{\boldsymbol{\eta}}^{\delta})^{-1}
\end{equation*}
\begin{equation*}
\hat{\boldsymbol{n}}^{\omega} = (-\partial_{\hat{x}}\hat{\omega}, 1), \qquad \hat{\boldsymbol{\tau}}^{\omega} = (1, \partial_{\hat{x}}\hat{\omega}), \qquad \hat{\bd{n}}^{\omega^{\delta}} = (-\partial_{\hat{x}}(\hat{\bd{\eta}}^{\delta}|_{\hat{\Gamma}}), 1),
\end{equation*}
\begin{equation*}
\hat{\mathcal{J}}^{\eta^{\delta}}_{b} = \det(\boldsymbol{I} + \hat{\nabla} \hat{\boldsymbol{\eta}}^{\delta}), \qquad \hat{\mathcal{J}}^{\omega}_{\Gamma} = \sqrt{1 + |\partial_{\hat{x}} \hat{\omega}|^{2}}.
\end{equation*}
\fi

\subsection{Formal energy inequality}\label{energy}

Here we show that the regularized problem is defined in a way that {{still gives rise to an energy equality, which in fact is the same energy inequality that one would formally obtain for the original problem, except with integrals over the moving Biot domain $\Omega_{b}(t)$ becoming integrals over the regularized moving Biot domain $\Omega^{\delta}_{b}(t)$}}. More precisely, we formally prove that a weak solution to the regularized problem satisfies the following energy equality.
\begin{lemma}
Assuming that a weak solution exists, the following energy equality holds:
\begin{equation}
E^K(T) + {{E}}^E(T) +
\int_0^T \left( {{D}}^V_f(t)  + {{D}}^V_b(t) + D^V_{f_b}(t) + D^V_\beta(t) \right)dt
 = E^K(0)+ {{E}}^E(0)
\end{equation}
where
\begin{align*}
{{E}}^K(t) = \frac{1}{2} \int_{\Omega_{f}(t)} |\boldsymbol{u}(t)|^{2} + 
\frac{1}{2} \rho_{b} \int_{\hat{\Omega}_{b}} |\partial_{t}\hat{\boldsymbol{\eta}}(t)|^{2} +
\frac{1}{2} c_{0} \int_{\hat{\Omega}_{b}} |\hat{p}(t)|^{2} +
\frac{1}{2}\rho_{p} \int_{\hat{\Gamma}} |\partial_{t} \hat{\omega}(t)|^{2}
\end{align*}
is the sum of the kinetic energy of the fluid, the kinetic energy of the Biot poroviscoelastic matrix motion, the kinetic energy of the filtrating fluid flow in the Biot medium,
and the kinetic energy of the plate motion, ${{E}}^E(t)$ is defined by
\begin{align*}
{{E}}^E(t)  = 2\mu_{e} \int_{\hat{\Omega}_{b}} |\hat{\bd{D}}(\hat{\bd{\eta}})(t)|^{2} + 2\lambda_{e} \int_{\hat{\Omega}_{b}} |\hat{\nabla} \cdot \hat{\bd{\eta}}(t)|^{2}
+ \int_{\hat{\Gamma}} |\hat{\Delta} \hat{\omega}(t)|^{2},
\end{align*}
which corresponds to the elastic energy of the Biot poroviscoelastic matrix and the elastic energy of the plate,
and 
\begin{align*}
{{D}}^V_f(t) = 2\nu  \int_{\Omega_{f}(t)} |\boldsymbol{D}(\boldsymbol{u})|^{2},\quad
{{D}}^V_b(t) = 2\mu_{v}  \int_{\hat{\Omega}_{b}} |\hat{\boldsymbol{D}}(\partial_{t} \hat{\bd{\eta}})|^{2} 
	+ \lambda_{v}  \int_{\hat{\Omega}_{b}} |\hat{\nabla} \cdot \partial_{t}\hat{\boldsymbol{\eta}}|^{2},
\end{align*}
\begin{align*}
D^V_{f_b}(t) = 	\kappa \int_{\Omega_{b}^{\delta}(t)} |\nabla p|^{2},\quad
D^V_\beta(t) = 	\beta  \int_{\Gamma(t)} |(\boldsymbol{\xi} - \boldsymbol{u}) \cdot \boldsymbol{\tau}|^{2}
\end{align*}
correspond to dissipation due to fluid viscosity, viscosity of the Biot poroviscoelastic matrix, dissipation due to permeability effects,
and dissipation due to friction in the Beavers-Joseph-Saffman slip condition.
\end{lemma}

\begin{proof}
To derive this energy equality we start by 
substituting 
$
	(\hat{\boldsymbol{v}}, \hat{\varphi}, \hat{\boldsymbol{\psi}}, \hat{r}) = (\hat{\boldsymbol{u}}, \hat{\zeta}, \partial_{t}\hat{\bd{\eta}}, \hat{p})
$
into the regularized weak formulation \eqref{regularweakref} defined on the fixed reference domain 
and calculate 
\begin{equation*}
	\frac{1}{2}\int_{\hat{\Gamma}} (\hat{\boldsymbol{u}} - \hat{\zeta}\bd{e}_{y}) \cdot \hat{\boldsymbol{n}}^{\omega} |\hat{\boldsymbol{u}}|^{2} + \int_{\hat{\Gamma}} \left(\frac{1}{2}|\hat{\boldsymbol{u}}|^{2} - \hat{p}\right) (\hat{\zeta}\bd{e}_{y} - \hat{\boldsymbol{u}}) \cdot \hat{\boldsymbol{n}}^{\omega} - \int_{\hat{\Gamma}} ((\hat{\boldsymbol{u}} - \hat{\zeta}\bd{e}_{y}) \cdot \hat{\boldsymbol{n}}^{\omega}) \hat{p} = 0.
\end{equation*}
Furthermore, using integration by parts one obtains
\begin{align*}
	&\alpha\left(\int_{\hat{\Omega}_{b}} \hat{\mathcal{J}}^{\eta^{\delta}}_{b} \hat{p} \hat{\nabla}^{\eta^{\delta}}_{b} \cdot \partial_{t} \hat{\boldsymbol{\eta}} + \int_{\hat{\Omega}_{b}} \hat{\mathcal{J}}^{\eta^{\delta}}_{b} \partial_{t}\hat{\boldsymbol{\eta}} \cdot \hat{\nabla}^{\eta^{\delta}}_{b} \hat{p} + \int_{\hat{\Gamma}} (\hat{\zeta}\bd{e}_{y} \cdot \hat{\boldsymbol{n}}^{\omega^{\delta}}) \hat{p} \right) \\
	&= \alpha\left(\int_{\Omega_{b}^{\delta}(t)} p \nabla \cdot \bd{\xi} + \int_{\Omega_{b}^{\delta}(t)} \bd{\xi} \cdot \nabla p + \int_{\Gamma^{\delta}(t)} (\zeta \bd{e}_{y} \cdot \boldsymbol{n}^{\delta}) p\right) = 0,
\end{align*}
where $\boldsymbol{n}^{\delta}$ is the upward pointing unit normal vector to $\Gamma^{\delta}(t)$. Finally, by the Reynold's transport theorem
\begin{equation*}
	\int_{0}^{T} \int_{\Omega_{f}(t)} \bd{u} \cdot \partial_{t} \bd{u} + \frac{1}{2} \int_{0}^{T} \int_{\Gamma(t)} (\zeta \bd{e}_{y} \cdot \bd{n}) |\bd{u}|^{2} = \frac{1}{2} \int_{\Omega_{f}(T)} |\bd{u}|^{2} - \frac{1}{2} \int_{\Omega_{f}(0)} |\bd{u}|^{2}.
\end{equation*}
By combining these calculations one {{obtains}} the final energy estimate:
\begin{multline*}
	\frac{1}{2} \int_{\Omega_{f}(T)} |\boldsymbol{u}(T)|^{2} + 2\nu \int_{0}^{T} \int_{\Omega_{f}(t)} |\boldsymbol{D}(\boldsymbol{u})|^{2} + \beta \int_{0}^{T} \int_{\Gamma(t)} |(\boldsymbol{\xi} - \boldsymbol{u}) \cdot \boldsymbol{\tau}|^{2} + \frac{1}{2}\rho_{p} \int_{\hat{\Gamma}} |\partial_{t} \hat{\omega}(T)|^{2} + \int_{\hat{\Gamma}} |\hat{\Delta} \hat{\omega}(T)|^{2} \\
	+ \frac{1}{2} \rho_{b} \int_{\hat{\Omega}_{b}} |\partial_{t}\hat{\boldsymbol{\eta}}(T)|^{2} + 2\mu_{e} \int_{\hat{\Omega}_{b}} |\hat{\bd{D}}(\hat{\bd{\eta}})(T)|^{2} + 2\lambda_{e} \int_{\hat{\Omega}_{b}} |\hat{\nabla} \cdot \hat{\bd{\eta}}(T)|^{2} + 2\mu_{v} \int_{0}^{T} \int_{\hat{\Omega}_{b}} |\hat{\boldsymbol{D}}(\partial_{t} \hat{\bd{\eta}})|^{2} \\
	+ \lambda_{v} \int_{0}^{T} \int_{\hat{\Omega}_{b}} |\hat{\nabla} \cdot \partial_{t}\hat{\boldsymbol{\eta}}|^{2}
	+ \frac{1}{2} c_{0} \int_{\hat{\Omega}_{b}} |\hat{p}(T)|^{2} + \kappa \int_{0}^{T} \int_{\Omega_{b}^{\delta}(t)} |\nabla p|^{2} = \frac{1}{2} \int_{\Omega_{f}(0)} |\boldsymbol{u}(0)|^{2} + \frac{1}{2} \rho_{p} \int_{\hat{\Gamma}} |\partial_{t}\hat{\omega}(0)|^{2} \\
	+ \int_{\hat{\Gamma}} |\hat{\Delta} \hat{\omega}(0)|^{2} + \frac{1}{2} \rho_{b} \int_{\hat{\Omega}_{b}} |\partial_{t}\hat{\bd{\eta}}(0)|^{2} + 2\mu_{e} \int_{\hat{\Omega}_{b}} |\hat{\bd{D}}(\hat{\bd{\eta}})(0)|^{2} + 2\lambda_{e} \int_{\hat{\Omega}_{b}} |\hat{\nabla} \cdot \hat{\bd{\eta}}(0)|^{2} + \frac{1}{2} c_{0} \int_{\hat{\Omega}_{b}} |\hat{p}(0)|^{2}.
\end{multline*}
\end{proof}

\subsection{Statement of the main existence result for the regularized problem}

We now state the main result on the existence of a weak solution to the regularized problem.

\begin{theorem}\label{MainThm1}
{{Let $\rho_b,\mu_e,\lambda_e,\alpha,\rho_p,\nu>0$ and $\mu_v,\lambda_v\geq 0$. Consider initial data for the plate displacement $\hat{\omega}_{0} \in H_{0}^{2}(\hat{\Gamma})$, plate velocity $\hat{\zeta}_{0} \in L^{2}(\hat{\Gamma})$, Biot displacement $\hat{\bd{\eta}}_{0} \in H^{1}(\hat{\Omega}_{b})$, Biot velocity $\hat{\bd{\xi}}_{0} \in H^{1}(\hat{\Omega}_{b})$ in the case of a viscoelastic Biot medium $\mu_{v}, \lambda_{v} > 0$ and $\hat{\bd{\xi}}_{0} \in L^{2}(\hat{\Omega}_{b})$ otherwise for the case of a purely elastic Biot medium,}} Biot pore pressure $\hat{p}_{0} \in L^{2}(\hat{\Omega}_{b})$, and fluid velocity $\bd{u}_{0} \in H^{1}(\Omega_{f}(0))$ which is divergence-free. Suppose further that $|\hat{\omega}_{0}| \le R_{0} < R$ for some $R_{0}$, $\hat{\bd{\eta}}_{0}|_{\Gamma} = \hat{\omega}_{0}\bd{e}_{y}$, and $\hat{\bd{\xi}}_{0}|_{\Gamma} = \hat{\zeta}_{0} \bd{e}_{y}$, and for some arbitrary but fixed regularization parameter $\delta > 0$, suppose that $\text{Id} + \hat{\bd{\eta}}_{0}^{\delta}$ is an invertible map with $\det(\bd{I} + \nabla\hat{\bd{\eta}}_{0}^{\delta}) > 0$. Then, there exists a weak solution $(\bd{u}, \hat{\omega}, \hat{\bd{\eta}}, \hat{p})$ to the regularized FPSI problem with regularization parameter $\delta$ on some time interval $[0, T]$, for some $T > 0$. 
\end{theorem}

While $T$ in general depends on $\delta$, we will show  that if there exists a smooth solution to the nonregularized FPSI problem, then this time $T$ for the regularized problem is independent of $\delta$. This will allow us to pass to the limit as $\delta \to 0$ and show that weak solutions to the regularized FPSI problems constructed in this manuscript,
converge to a smooth solution 
of the original, nonregularized problem, when a smooth solution to the nonregularized problem exists.
{On the other hand, without the additional assumption of the existence of a strong solution, one cannot draw any conclusions about the limit as $\delta\to 0$. This assumption plays a crucial role in demonstrating that $T$ remains independent of $\delta$. Furthermore, as elucidated in Remark \ref{RemarkNonregularized}, energy estimates alone are insufficient to take the limit in certain terms in the weak formulation. Therefore, in order to pass to the limit as $\delta \to 0$ one would need to prove additional regularity estimates (beyond energy estimates) for weak solutions, which appears to be beyond the current state-of-the-art techniques.}

\begin{remark}
The result above is a local result, since it holds up to some time $T > 0$, which needs to be sufficiently small. However, it is easy to show that this $T > 0$ can be made maximal, in the sense that it holds until the time for which $\text{Id} + \hat{\bd{\eta}}^{\delta}$ fails to be invertible or $|\hat{\omega}(\cdot, x)| = R$ for some $x \in \hat{\Gamma}$ when the reticular plate collides with the boundary. This
can be shown using a standard method, see e.g., pg.~397-398 of \cite{CDEM}, or the proof of Theorem 7.1 in \cite{MuhaCanic13}.
\end{remark}

\textbf{An important notational convention.} For notational simplicity, we will no longer use the ``hat" notation to distinguish between functions and domains in the physical or reference configuration: for example, we will denote both the pore pressure $p$ on $\Omega_{b}(t)$ and $\hat{p}$ on $\hat{\Omega}_{b}$ by $p$, as the distinction between these two will be clear from context. In addition, we will remove the ``hat" convention from the reference domains, and for example, we will denote the reference domain $\hat{\Omega}_{b}$ for the Biot medium by $\Omega_{b}$. We will follow this notational convention for the rest of the manuscript.

The proof of Theorem~\ref{MainThm1} is constructive, and based on an operator splitting scheme. 
This is an approach that has been used in constructive existence proofs of weak solutions for a variety of FSI problems, see for example \cite{MuhaCanic13}.

\section{The splitting scheme}\label{splitting}

The splitting scheme is defined as follows.
First, semidiscretize the problem in time by introducing  the time step $\Delta t = T/N$, and subdivide the time interval $[0,T]$ into $N$ subintervals, 
each of width $\Delta t$. {{On each subinterval, we will run two subproblems: (1) a plate subproblem which takes into account the elastodynamics of the reticular plate and updates the plate displacement and the plate velocity, and (2) a fluid/Biot subproblem which updates the fluid velocity, the Biot displacement, the Biot pore pressure, and the plate velocity. Hence, each subinterval $[n\Delta t, (n + 1)\Delta t]$ involves an iteration of the plate subproblem and then an iteration of the fluid/Biot subproblem, and the solution from the previous subproblem is used as data for the subsequent subproblem. The approximations of 
 the fluid velocity, plate displacement and velocity, and Biot poroviscoelastic material displacement and pressure will be denoted by
\begin{equation*}
(\boldsymbol{u}^{n + \frac{i}{2}}_{N}, \omega^{n + \frac{i}{2}}_{N}, \zeta^{n + \frac{i}{2}}_{N}, \boldsymbol{\eta}^{n + \frac{i}{2}}_{N}, p^{n + \frac{i}{2}}_{N}), \qquad \text{ for } n = 0, 1, ...., N \text{ and } i = 0, 1,
\end{equation*}
where they are all defined on the given time subinterval $[n\Delta t, (n + 1)\Delta t]$. Here, the quantities with the superscript $n + 1/2$ denote the resulting approximate solutions obtained after the plate subproblem is solved, and the quantities with the superscript $n + 1$ denote the resulting approximate solutions obtained after the fluid/Biot subproblem is solved.}}
For the splitting scheme we will work on the fixed reference domain and hence, we will semi-discretize the regularized weak formulation \eqref{regularweakref} on the fixed reference domain. Backward Euler discretization will be used to approximate time derivatives, with the following shorthand notation:
{{
\begin{equation}\label{dotdiscretization}
\dot{f}^{n + \frac{i}{2}}_{N} = \frac{f^{n + \frac{i}{2}}_{N} - f^{n + \frac{i}{2} - 1}_{N}}{\Delta t}.
\end{equation}
As a technical comment, in the description of the subproblems below, the backward Euler discretization \eqref{dotdiscretization} can potentially give rise to negative subscripts, so when relevant, we will explicitly define $f_{-1}$ and $f_{-1/2}$ depending on the context.}}

\subsection{The plate subproblem} Only the plate displacement and velocity $\omega^{n + \frac{1}{2}}_{N}$ and $\zeta^{n + \frac{1}{2}}_{N}$ are updated in this subproblem, leaving the remaining variables unchanged:
\begin{equation*}
\boldsymbol{u}^{n + \frac{1}{2}}_{N} = \boldsymbol{u}^{n}_{N}, \qquad \boldsymbol{\eta}^{n + \frac{1}{2}}_{N} = \boldsymbol{\eta}^{n}_{N}, \qquad p^{n + \frac{1}{2}}_{N} = p^{n}_{N}.
\end{equation*}
The new plate displacement and velocity are calculated from the following weak formulation of the plate subproblems:
 find $\omega^{n + \frac{1}{2}}_{N} \in H_{0}^{2}(\Gamma)$ and $\zeta^{n + \frac{1}{2}}_{N} \in H_{0}^{2}(\Gamma)$, such that
\begin{equation}\label{plateweak1}
\int_{\Gamma} \left(\frac{\omega^{n + \frac{1}{2}}_{N} - \omega^{n - \frac{1}{2}}_{N}}{\Delta t}\right) \cdot \phi = \int_{\Gamma} \zeta^{n + \frac{1}{2}}_{N} \cdot \phi, \qquad \text{ for all } \phi \in L^{2}(\Gamma),
\end{equation}
\begin{equation}\label{plateweak2}
\rho_{p} \int_{\Gamma} \left(\frac{\zeta^{n + \frac{1}{2}}_{N} - \zeta^{n}_{N}}{\Delta t}\right) \cdot \varphi + \int_{\Gamma} \Delta \omega^{n + \frac{1}{2}}_{N} \cdot \Delta \varphi = 0, \qquad \text{ for all } \varphi \in H_{0}^{2}(\Gamma).
\end{equation}
When $n = 0$, we set $\omega^{-\frac{1}{2}}_{N} = \omega(0)$ and $\zeta^{0}_{N} = \zeta(0)$. In particular, $\omega(0) \boldsymbol{e}_{y} = \boldsymbol{\eta}(0)|_{\Gamma}$ and $\zeta(0) \boldsymbol{e}_{y} = \bd{\xi}(0)$.

\begin{lemma}
Problem \eqref{plateweak1}, \eqref{plateweak2} has a unique {{solution}} which satisfies the following energy equality:
\begin{align}\label{plateenergy}
&\frac{1}{2}\rho_{p} \int_{\Gamma} |\zeta^{n + \frac{1}{2}}_{N}|^{2} + \frac{1}{2}\rho_{p} \int_{\Gamma} |\zeta^{n + \frac{1}{2}}_{N} - \zeta^{n}_{N}|^{2} + \frac{1}{2} \int_{\Gamma} |\Delta \omega^{n + \frac{1}{2}}_{N}|^{2} + \frac{1}{2} \int_{\Gamma} |\Delta (\omega^{n + \frac{1}{2}}_{N} - \omega^{n - \frac{1}{2}}_{N})|^{2} 
\nonumber\\
&= \frac{1}{2}\rho_{p} \int_{\Gamma} |\zeta^{n - \frac{1}{2}}_{N}|^{2} + \frac{1}{2} \int_{\Gamma} |\Delta \omega^{n - \frac{1}{2}}_{N}|^{2}.
\end{align}
\end{lemma}

\begin{proof}
To prove this, we first notice that {{\eqref{plateweak1} immediately implies that}}
\begin{equation}\label{zeta}
\zeta^{n + \frac{1}{2}}_{N} = \frac{\omega^{n + \frac{1}{2}}_{N} - \omega^{n - \frac{1}{2}}_{N}}{\Delta t}
\end{equation}
{{so that, by substituting into \eqref{plateweak2}, it suffices to find $\omega^{n + \frac{1}{2}}_{N} \in H_{0}^{2}(\Gamma)$ which satisfies:}}
\begin{equation*}
\rho_{p} \int_{\Gamma} \omega^{n + \frac{1}{2}}_{N} \cdot \varphi + (\Delta t)^{2} \int_{\Gamma} \Delta \omega^{n + \frac{1}{2}}_{N} \cdot \Delta \varphi = \rho_{p} \int_{\Gamma} (\omega^{n - \frac{1}{2}}_{N} + (\Delta t) \zeta^{n}_{N}) \cdot \varphi, \qquad \text{ for all } \varphi \in H_{0}^{2}(\Gamma).
\end{equation*}
The bilinear form
\begin{equation*}
B[\omega, \varphi] = \rho_{p} \int_{\Gamma} \omega \cdot \varphi + (\Delta t)^{2} \int_{\Gamma} \Delta \omega \cdot \Delta \varphi
\end{equation*}
is coercive on $H^{2}_{0}(\Gamma)$, and 
\begin{equation*}
\varphi \to \rho_{p} \int_{\Gamma} \left(\omega^{n - \frac{1}{2}}_{N} + (\Delta t) \zeta^{n}_{N}\right) \cdot \varphi
\end{equation*}
is a continuous linear functional on $H^{2}_{0}(\Gamma)$, since we will have $\omega^{n - \frac{1}{2}}_{N} \in H_{0}^{2}(\Gamma)$ and $\zeta^{n}_{N} \in L^{2}(\Gamma)$ by the way our splitting scheme is defined. Thus, by the Lax-Milgram lemma, there exists a unique solution $\omega^{n + \frac{1}{2}}_{N} \in H_{0}^{2}(\Gamma)$, from which we also recover {{$\zeta^{n + \frac{1}{2}}_{N} \in H^{2}_{0}(\Gamma)$}} using \eqref{zeta} above. 

The energy equality above follows by substituting $\varphi = \zeta^{n + \frac{1}{2}}_{N} = \frac{\omega^{n + \frac{1}{2}}_{N} - \omega^{n - \frac{1}{2}}_{N}}{\Delta t} \in H_{0}^{2}(\Gamma)$ into the weak formulation and using  the identity
\begin{equation*}
(a - b) \cdot a = \frac{1}{2}(|a|^{2} + |a - b|^{2} - |b|^{2}).
\end{equation*}
\end{proof}

\subsection{The fluid and Biot subproblem} For the fluid and Biot subproblem, we update the quantities related to the fluid and the Biot medium. Due to the kinematic coupling between the Biot medium displacement and the plate displacement, we must also update the plate velocity, as the dynamics of the Biot medium affect the kinematics of the plate. In this step, only the plate displacement remains unchanged:
\begin{equation*}
\omega^{n + 1}_{N} = \omega^{n + \frac{1}{2}}_{N}.
\end{equation*}
%We will use $\omega^{n}_{N}$ calculated from the previous step for the ALE mapping for the fluid.

To state the weak formulation of the fluid and Biot subproblem, we define the solution and test spaces, respectively:
\begin{eqnarray}\label{semisolnspace}
\mathcal{V}^{n + 1}_{N} &=& \{(\bd{u}, \zeta, \bd{\eta}, p) \in \mathcal{V}^{\omega^{n}_{N}}_{f} \times H_{0}^{2}(\Gamma) \times V_{d} \times V_{p}\},
\\
%\end{equation}
%\begin{equation}
\label{semitestspace}
\mathcal{Q}^{n + 1}_{N} &=& \{(\boldsymbol{v}, \varphi, \boldsymbol{\psi}, r) \in V^{\omega^{n}_{N}}_{f} \times H_{0}^{2}(\Gamma) \times V_{d} \times V_{p} : \boldsymbol{\psi} = \varphi \boldsymbol{e}_{y} \text{ on } \Gamma\},
\end{eqnarray}
where $V^{\omega}_{f}$, $V_{d}$, and $V_{p}$ are defined in \eqref{Vf}, \eqref{Vd}, and \eqref{Vp}. 

The weak formulation now reads: find  $(\boldsymbol{u}^{n + 1}_{N}, \zeta^{n + 1}_{N}, \boldsymbol{\eta}^{n + 1}_{N}, p^{n + 1}_{N}) \in \mathcal{V}^{n + 1}_{N}$ defined on the reference domain, such that for all test functions $(\boldsymbol{v}, \varphi, \boldsymbol{\psi}, r) \in \mathcal{Q}^{n + 1}_{N}$ defined on the reference domain, the following holds:

\begin{equation}\label{weakBiot1}
\begin{aligned}
&\int_{\Omega_{f}} \left(1 + \frac{\omega^{n}_{N}}{R}\right) \boldsymbol{\dot{u}}^{n + 1}_{N} \cdot \boldsymbol{v} 
+ 2\nu \int_{\Omega_{f}} \left(1 + \frac{\omega^{n}_{N}}{R}\right) \boldsymbol{D}^{\omega^{n}_{N}}_{f}(\boldsymbol{u}^{n + 1}_{N}) : \boldsymbol{D}^{\omega^{n}_{N}}_{f}(\boldsymbol{v}) + \int_{\Gamma} \left(\frac{1}{2}\boldsymbol{u}^{n + 1}_{N} \cdot \boldsymbol{u}^{n}_{N} - p^{n + 1}_{N}\right)(\boldsymbol{\psi} - \boldsymbol{v})\cdot \boldsymbol{n}^{\omega^{n}_{N}} 
\\
&+ \frac{1}{2} \int_{\Omega_{f}} \left(1 + \frac{\omega^{n}_{N}}{R}\right) \left[\left(\left(\boldsymbol{u}^{n}_{N} - \zeta^{n + \frac{1}{2}}_{N} \frac{R + y}{R} \boldsymbol{e}_{y}\right) \cdot \nabla^{\omega^{n}_{N}}_{f} \boldsymbol{u}^{n + 1}_{N}\right) \cdot \boldsymbol{v} - \left(\left(\boldsymbol{u}^{n}_{N} - \zeta^{n + \frac{1}{2}}_{N} \frac{R + y}{R} \boldsymbol{e}_{y}\right) \cdot \nabla^{\omega^{n}_{N}}_{f} \boldsymbol{v}\right) \cdot \boldsymbol{u}^{n + 1}_{N}\right] 
\\
&+ \frac{1}{2R} \int_{\Omega_{f}} \zeta^{n + \frac{1}{2}}_{N} \boldsymbol{u}^{n + 1}_{N} \cdot \boldsymbol{v} + \frac{1}{2} \int_{\Gamma} (\boldsymbol{u}^{n + 1}_{N} - \boldsymbol{\dot{\eta}}^{n + 1}_{N}) \cdot \boldsymbol{n}^{\omega^{n}_{N}} (\boldsymbol{u}^{n}_{N} \cdot \boldsymbol{v}) 
\\
&+ \frac{\beta}{\mathcal{J}^{\omega^{n}_{N}}_{\Gamma}} \int_{\Gamma} (\boldsymbol{\dot{\eta}}^{n + 1}_{N} - \boldsymbol{u}^{n + 1}_{N}) \cdot \boldsymbol{\tau}^{\omega^{n}_{N}} (\boldsymbol{\psi} - \boldsymbol{v}) \cdot \boldsymbol{\tau}^{\omega^{n}_{N}} + \rho_{b} \int_{\Omega_{b}} \left(\frac{\boldsymbol{\dot{\eta}}^{n + 1}_{N} - \boldsymbol{\dot{\eta}}^{n}_{N}}{\Delta t}\right) \cdot \boldsymbol{\psi} 
\\
&+ \rho_{p} \int_{\Gamma} \left(\frac{\zeta^{n + 1}_{N} - \zeta^{n + \frac{1}{2}}_{N}}{\Delta t}\right) \varphi + 2\mu_{e} \int_{\Omega_{b}} \boldsymbol{D}(\boldsymbol{\eta}^{n + 1}_{N}) : \boldsymbol{D}(\boldsymbol{\psi}) + \lambda_{e} \int_{\Omega_{b}} (\nabla \cdot \boldsymbol{\eta}^{n + 1}_{N})(\nabla \cdot \boldsymbol{\psi}) 
\\
&+ 2\mu_{v} \int_{\Omega_{b}} \bd{D}(\dot{\bd{\eta}}^{n + 1}_{N}) : \bd{D}(\bd{\psi}) + \lambda_{v} \int_{\Omega_{b}} (\nabla \cdot \dot{\bd{\eta}}^{n + 1}_{N}) (\nabla \cdot \bd{\psi}) - \alpha \int_{\Omega_{b}} \mathcal{J}^{(\eta^{n}_{N})^{\delta}}_{b} p^{n + 1}_{N} \nabla^{(\eta^{n}_{N})^{\delta}}_{b} \cdot \boldsymbol{\psi} 
\\
&+ c_{0}  \int_{\Omega_{b}} \frac{p^{n + 1}_{N} - p^{n}_{N}}{\Delta t} r - \alpha \int_{\Omega_{b}} \mathcal{J}^{(\eta^{n}_{N})^{\delta}}_{b} \boldsymbol{\dot{\eta}}^{n + 1}_{N} \cdot \nabla^{(\eta^{n}_{N})^{\delta}}_{b} r - \alpha \int_{\Gamma} (\boldsymbol{\dot{\eta}}^{n + 1}_{N} \cdot \boldsymbol{n}^{(\omega^{n}_{N})^{\delta}}) r 
\\
&+ \kappa \int_{\Omega_{b}} \mathcal{J}^{(\eta^{n}_{N})^{\delta}}_{b} \nabla^{(\eta^{n}_{N})^{\delta}}_{b} p^{n + 1}_{N} \cdot \nabla^{(\eta^{n}_{N})^{\delta}}_{b} r - \int_{\Gamma} [(\boldsymbol{u}^{n + 1}_{N} - \boldsymbol{\dot{\eta}}^{n + 1}_{N}) \cdot \boldsymbol{n}^{\omega^{n}_{N}}]r = 0,
\end{aligned}
\end{equation}
and
\begin{equation}\label{weakBiot2}
\int_{\Gamma} \left(\frac{\boldsymbol{\eta}^{n + 1}_{N} - \boldsymbol{\eta}^{n}_{N}}{\Delta t} \right) \cdot \boldsymbol{\phi} = \int_{\Gamma} \zeta^{n + 1}_{N} \bd{e}_{y} \cdot \boldsymbol{\phi}, \qquad \text{ for all } \boldsymbol{\phi} \in L^{2}(\Gamma).
\end{equation}
{{We remark that when $n = 0$, the backwards Euler discretization \eqref{dotdiscretization} involves a negative subscript in the definition of $\dot{\bd{\eta}}^{0}_{N}$, so in this case, we will instead set $\dot{\bd{\eta}}^{0}_{N}$ to be the initial plate velocity $\zeta_{0}$.}}

\begin{lemma}\label{Subproblem2}
Problem \eqref{weakBiot1}, \eqref{weakBiot2}  has a unique solution provided that the following assumptions hold:
\begin{enumerate}
\item {\sc{Assumption 1A}}: \textit{Boundedness of the plate displacement away from $R$.} There exists a positive constant $R_{max}$ such that 
\begin{equation}\label{1A}
|\omega^{k + \frac{i}{2}}_{N}| \le R_{max} < R, \qquad \text{ for all } k = 0, 1, ..., n \text{ and } i = 0, 1.
\end{equation}
\item {\sc{Assumption 2A:}} \textit{Invertibility of the map from fixed to moving Biot domain.} The map
\begin{equation}\label{2A}
\text{Id} + (\bd{\eta}^{n}_{N})^{\delta}: \Omega_{b} \to (\Omega_{b})^{n, \delta}_{N} \qquad \text{ is invertible},
\end{equation}
where we define $(\Omega_{b})^{n, \delta}_{N}$ to be the image of $\Omega_{b}$ under the map $\text{Id} + (\bd{\eta}^{n}_{N})^{\delta}$. 
\end{enumerate}
Additionally, the weak solution satisfies the following energy equality:
\begin{align*}
&\frac{1}{2} \int_{\Omega_{f}} \left(1 + \frac{\omega^{n + 1}_{N}}{R}\right) |\boldsymbol{u}^{n + 1}_{N}|^{2} + \frac{1}{2} \rho_{b} \int_{\Omega_{b}} |\dot{\bd{\eta}}^{n + 1}_{N}|^{2} + \frac{1}{2} c_{0} \int_{\Omega_{b}} |p^{n + 1}_{N}|^{2} + \mu_{e} \int_{\Omega_{b}} |\bd{D}(\bd{\eta}^{n + 1}_{N})|^{2} + \frac{1}{2} \lambda_{e} \int_{\Omega_{b}} |\nabla \cdot \bd{\eta}^{n + 1}_{N}|^{2} \\
&+ \frac{1}{2} \rho_{p} \int_{\Gamma} |\zeta^{n + 1}_{N}|^{2} + 2\mu_{v}(\Delta t) \int_{\Omega_{b}} |\bd{D}(\dot{\bd{\eta}}^{n + 1}_{N})|^{2} + \lambda_{v}(\Delta t) \int_{\Omega_{b}} |\nabla \cdot \dot{\bd{\eta}}^{n + 1}_{N}|^{2} + \kappa (\Delta t) \int_{\Omega_{b}} \mathcal{J}^{(\eta^{n}_{N})^{\delta}}_{b} |\nabla^{(\eta^{n}_{N})^{\delta}}_{b} p^{n + 1}_{N}|^{2} \\
&+ \frac{\beta (\Delta t)}{\mathcal{J}^{\omega^{n}_{N}}_{\Gamma}} \int_{\Gamma} |(\boldsymbol{\dot{\eta}}^{n + 1}_{N} - \boldsymbol{u}^{n + 1}_{N}) \cdot \boldsymbol{\tau}^{\omega^{n}_{N}}|^{2} + \frac{1}{2} \rho_{b} \int_{\Omega_{b}} |\bd{\dot{\eta}}^{n + 1}_{N} - \bd{\dot{\eta}}^{n}_{N}|^{2} + \frac{1}{2} c_{0} \int_{\Omega_{b}} |p^{n + 1}_{N} - p^{n}_{N}|^{2} + \mu_{e} \int_{\Omega_{b}} |\bd{D}(\bd{\eta}^{n + 1}_{N} - \bd{\eta}^{n}_{N})|^{2} \\
&+ \frac{1}{2} \lambda_{e} \int_{\Omega_{b}} |\nabla \cdot (\bd{\eta}^{n + 1}_{N} - \bd{\eta}^{n}_{N})|^{2} = \frac{1}{2} \int_{\Omega_{f}} \left(1 + \frac{\omega^{n}_{N}}{R}\right) |\bd{u}^{n}_{N}|^{2} + \frac{1}{2} \rho_{b} \int_{\Omega_{b}} |\dot{\bd{\eta}}^{n}_{N}|^{2} + \frac{1}{2} c_{0} \int_{\Omega_{b}} |p^{n}_{N}|^{2} + \mu_{e} \int_{\Omega_{b}} |\bd{D}(\bd{\eta}^{n}_{N})|^{2} \\
&+ \frac{1}{2} \lambda_{e} \int_{\Omega_{b}} |\nabla \cdot \bd{\eta}^{n}_{N}|^{2} + \frac{1}{2} \rho_{p} \int_{\Gamma} |\zeta^{n + \frac{1}{2}}_{N}|^{2}.
\end{align*}
\end{lemma}

The proof is based on using the Lax-Milgram Lemma. However, in this case the proof is more involved for two reasons. First, the bilinear form
associated with problem \eqref{weakBiot1} and \eqref{weakBiot2} is not coercive  on the Hilbert space 
$
\mathcal{V}^{\omega^{n}_{N}}_{f} \times V_{d} \times V_{p},
$ because of a mismatch between the hyperbolic and parabolic 
scaling in the problem.  The second reason is that it is not {\emph{a priori}} clear that Korn's inequality,
which is needed in the proof of the existence, holds for the Biot domain.
To deal with the first difficulty and recover the coercive structure of the problem, the test functions can be rescaled {{by the factor $\Delta t$}} so that
\begin{equation}\label{rescalling}
\boldsymbol{v} \to (\Delta t)\boldsymbol{v}, \qquad r \to (\Delta t)r.
\end{equation}
This scaling of the test functions is valid because if $(\boldsymbol{v}, \varphi, \boldsymbol{\psi}, v) \in \mathcal{Q}^{n + 1}_{N}$, then the rescaled test function satisfies $((\Delta t)^{-1}\boldsymbol{v}, \varphi, \boldsymbol{\psi}, (\Delta t)^{-1}r) \in \mathcal{Q}^{n + 1}_{N}$ also. 
To deal with the second difficulty, one can show by explicit calculation that the following Korn's inequality holds for this {{problem. We refer the reader to Section 11.2 and Corollary 11.2.22 in \cite{Brenner} for a more general proof of the Korn inequality.}}
\begin{proposition}\label{Korn}{\sc{Korn's inequality for the Biot poroviscoelastic domain.}}
For all $\bd{\eta} \in V_{d}$,
\begin{equation*}
\int_{\Omega_{b}} |\bd{D}(\bd{\eta})|^{2} \ge \frac{1}{2} \int_{\Omega_{b}} |\nabla \bd{\eta}|^{2}.
\end{equation*}
\end{proposition}
\begin{proof}
By a standard approximation argument, it suffices to assume that $\bd{\eta}$ is smooth. Because $\eta_{x} = 0$ on $\Gamma$ and because $\bd{\eta} = 0$ on the left, top, and right boundaries of $\Omega_{b}$, we have from integration by parts, that
\begin{equation*}
\int_{\Omega_{b}} \frac{\partial \eta_{x}}{\partial y} \frac{\partial \eta_{y}}{\partial x} = -\int_{\Omega_{b}} \eta_{x} \frac{\partial^{2} \eta_{y}}{\partial x \partial y} = \int_{\Omega_{b}} \frac{\partial \eta_{x}}{\partial x} \frac{\partial \eta_{y}}{\partial y}.
\end{equation*}
Therefore, by using the inequality $a^{2} + 2ab + b^{2} \ge 0$, we obtain
{{\begin{align*}
\int_{\Omega_{b}} |\bd{D}(\bd{\eta})|^{2} &= \int_{\Omega_{b}} \left(\frac{\partial \eta_{x}}{\partial x}\right)^{2} + \left(\frac{\partial \eta_{y}}{\partial y}\right)^{2} + 2\left[\frac{1}{2}\left(\frac{\partial \eta_{x}}{\partial y} + \frac{\partial \eta_{y}}{\partial x}\right)\right]^{2} \\
&= \int_{\Omega_{b}} \left(\frac{\partial \eta_{x}}{\partial x}\right)^{2} + \left(\frac{\partial \eta_{y}}{\partial y}\right)^{2} + \frac{1}{2}\left(\frac{\partial \eta_{x}}{\partial y} + \frac{\partial \eta_{y}}{\partial x}\right)^{2} \\
&= \int_{\Omega_{b}} \left(\frac{\partial \eta_{x}}{\partial x}\right)^{2} + \left(\frac{\partial \eta_{y}}{\partial y}\right)^{2} + \frac{1}{2}\left[\left(\frac{\partial \eta_{x}}{\partial y}\right)^{2} + \left(\frac{\partial \eta_{y}}{\partial x}\right)^{2}\right] + \frac{\partial\eta_{x}}{\partial y} \frac{\partial \eta_{y}}{\partial x} \\
&= \int_{\Omega_{b}} \left(\frac{\partial \eta_{x}}{\partial x}\right)^{2} + \frac{\partial\eta_{x}}{\partial x} \frac{\partial \eta_{y}}{\partial y} +\left(\frac{\partial \eta_{y}}{\partial y}\right)^{2} + \frac{1}{2}\left[\left(\frac{\partial \eta_{x}}{\partial y}\right)^{2} + \left(\frac{\partial \eta_{y}}{\partial x}\right)^{2}\right] \ge \frac{1}{2} \int_{\Omega_{b}} |\nabla \bd{\eta}|^{2}.
\end{align*}}}
\end{proof}

\begin{proof}{\bf{Proof of Lemma~\ref{Subproblem2}.}}
Rewrite the weak formulation \eqref{weakBiot1} and \eqref{weakBiot2} so that all of the functions at the {{$(n + 1)$st}} time step are on the left hand side while all other quantities are on the right hand side. In addition, we rewrite $\zeta^{n + 1}_{N}$ in terms of $\bd{\eta}^{n}_{N}$ and $\bd{\eta}^{n + 1}_{N}$ by using \eqref{weakBiot2}:
\begin{equation*}
\zeta^{n + 1}_{N} \bd{e}_{y} = \frac{\boldsymbol{\eta}^{n + 1}_{N} - \boldsymbol{\eta}^{n}_{N}}{\Delta t}\Big\vert_{\Gamma}.
\end{equation*}
After using the rescaling \eqref{rescalling} of the test functions, 
the weak formulation involves the following coercive and continuous bilinear form {{$B: H \times H \to \mathbb{R}$, where $H$ is the Hilbert space $\mathcal{V}^{\omega^{n}_{N}}_{f} \times V_{d} \times V_{p}$:}}
\begin{align*}
&{{B([\bd{u}, \bd{\eta}, p), (\bd{v}, \bd{\psi}, r)]}} := (\Delta t)^{2} \int_{\Omega_{f}} \left(1 + \frac{\omega^{n}_{N}}{R}\right) \boldsymbol{u} \cdot \boldsymbol{v} \\
&+ \frac{1}{2} (\Delta t)^{3} \int_{\Omega_{f}} \left(1 + \frac{\omega^{n}_{N}}{R}\right) \left[\left(\left(\boldsymbol{u}^{n}_{N} - \zeta^{n + \frac{1}{2}}_{N} \frac{R + y}{R} \boldsymbol{e}_{y}\right) \cdot \nabla^{\omega^{n}_{N}} \boldsymbol{u}\right) \cdot \boldsymbol{v} - \left(\left(\boldsymbol{u}^{n}_{N} - \zeta^{n + \frac{1}{2}}_{N} \frac{R + y}{R} \boldsymbol{e}_{y}\right) \cdot \nabla^{\omega^{n}_{N}} \boldsymbol{v}\right) \cdot \boldsymbol{u}\right] \\
&+ (\Delta t)^{3} \cdot \frac{1}{2R} \int_{\Omega_{f}} \zeta^{n + \frac{1}{2}}_{N} \boldsymbol{u} \cdot \boldsymbol{v} + \frac{1}{2} (\Delta t)^{3} \int_{\Gamma} (\boldsymbol{u} - (\Delta t)^{-1} \boldsymbol{\eta}) \cdot \boldsymbol{n}^{\omega^{n}_{N}} (\boldsymbol{u}^{n}_{N} \cdot \boldsymbol{v}) \\
&+ 2\nu (\Delta t)^{3} \int_{\Omega_{f}} \left(1 + \frac{\omega^{n}_{N}}{R}\right) \boldsymbol{D}^{\omega^{n}_{N}}_{f}(\boldsymbol{u}) : \boldsymbol{D}^{\omega^{n}_{N}}_{f}(\boldsymbol{v}) + (\Delta t)^{2} \int_{\Gamma} \left(\frac{1}{2}\boldsymbol{u} \cdot \boldsymbol{u}^{n}_{N} - p\right)(\boldsymbol{\psi} - (\Delta t) \boldsymbol{v})\cdot \boldsymbol{n}^{\omega^{n}_{N}} \\
&+ \frac{\beta}{\mathcal{J}^{\omega^{n}_{N}}_{\Gamma}} (\Delta t)^{2} \int_{\Gamma} [(\Delta t)^{-1} \boldsymbol{\eta} - \boldsymbol{u}] \cdot \boldsymbol{\tau}^{\omega^{n}_{N}} (\boldsymbol{\psi} - (\Delta t) \boldsymbol{v}) \cdot \boldsymbol{\tau}^{\omega^{n}_{N}} + \rho_{b} \int_{\Omega_{b}} \boldsymbol{\eta} \cdot \boldsymbol{\psi} + \rho_{p} \int_{\Gamma} \boldsymbol{\eta} \cdot \boldsymbol{\psi} \\
&+ (2\mu_{e} (\Delta t)^{2} + 2\mu_{v}(\Delta t)) \int_{\Omega_{b}} \boldsymbol{D}(\boldsymbol{\eta}) : \boldsymbol{D}(\boldsymbol{\psi}) + (\lambda_{e} (\Delta t)^{2} + \lambda_{v}(\Delta t)) \int_{\Omega_{b}} (\nabla \cdot \boldsymbol{\eta})(\nabla \cdot \boldsymbol{\psi}) \\
&- \alpha (\Delta t)^{2} \int_{\Omega_{b}} \mathcal{J}^{(\eta^{n}_{N})^{\delta}}_{b} p \nabla^{(\eta^{n}_{N})^{\delta}}_{b} \cdot \boldsymbol{\psi} + c_{0} (\Delta t)^{2} \int_{\Omega_{b}} p r - \alpha (\Delta t)^{2} \int_{\Omega_{b}} \mathcal{J}^{(\eta^{n}_{N})^{\delta}}_{b} \boldsymbol{\eta} \cdot \nabla^{(\eta^{n}_{N})^{\delta}}_{b} r \\
&- \alpha (\Delta t)^{2} \int_{\Gamma} (\boldsymbol{\eta} \cdot \boldsymbol{n}^{(\omega^{n}_{N})^{\delta}}) r + \kappa (\Delta t)^{3} \int_{\Omega_{b}} \mathcal{J}^{(\eta^{n}_{N})^{\delta}}_{b} \nabla^{(\eta^{n}_{N})^{\delta}}_{b} p \cdot \nabla^{(\eta^{n}_{N})^{\delta}}_{b} r \\
&- (\Delta t)^{3} \int_{\Gamma} [(\boldsymbol{u} - (\Delta t)^{-1} \boldsymbol{\eta})\cdot \boldsymbol{n}^{\omega^{n}_{N}}]r.
\end{align*}
With this notation, the weak formulation reads: find  $(\boldsymbol{u}^{n + 1}_{N},  \boldsymbol{\eta}^{n + 1}_{N}, p^{n + 1}_{N}) \in \mathcal{V}^{\omega^{n}_{N}}_{f} \times V_{d} \times V_{p}$ such that for all test functions $(\boldsymbol{v}, \boldsymbol{\psi}, r) \in \mathcal{V}^{\omega^{n}_{N}}_{f} \times V_{d} \times V_{p}$, 

\begin{multline}\label{BiotFluidLaxMilgram}
{{B[(\bd{u}^{n + 1}_{N}, \bd{\eta}^{n + 1}_{N}, p^{n + 1}_{N}), (\bd{v}, \bd{\psi}, r)]}} = (\Delta t)^{2} \int_{\Omega_{f}} \left(1 + \frac{\omega^{n}_{N}}{R}\right) \boldsymbol{u}^{n}_{N} \cdot \boldsymbol{v} - \frac{1}{2} (\Delta t)^{2} \int_{\Gamma} \boldsymbol{\eta}^{n}_{N} \cdot \boldsymbol{n}^{\omega^{n}_{N}} (\boldsymbol{u}^{n}_{N} \cdot \boldsymbol{v}) \\
+ \frac{\beta}{\mathcal{J}^{\omega^{n}_{N}}_{\Gamma}} (\Delta t) \int_{\Gamma} \boldsymbol{\eta}^{n}_{N} \cdot \boldsymbol{\tau}^{\omega^{n}_{N}} (\boldsymbol{\psi} - (\Delta t) \boldsymbol{v}) \cdot \boldsymbol{\tau}^{\omega^{n}_{N}} + \rho_{b} \int_{\Omega_{b}} (2\boldsymbol{\eta}^{n}_{N} - \boldsymbol{\eta}^{n - 1}_{N}) \cdot \boldsymbol{\psi} + \rho_{p} \int_{\Gamma} (\boldsymbol{\eta}^{n}_{N} + (\Delta t) \zeta^{n + \frac{1}{2}}_{N} \bd{e}_{y}) \cdot \boldsymbol{\psi} \\
+ 2\mu_{v}(\Delta t) \int_{\Omega_{b}} \bd{D}(\bd{\eta}^{n}_{N}) : \bd{D}(\bd{\psi}) + \lambda_{v} (\Delta t) \int_{\Omega_{b}} (\nabla \cdot \bd{\eta}^{n}_{N}) (\nabla \cdot \bd{\psi}) \\
+ c_{0} (\Delta t)^{2} \int_{\Omega_{b}} p^{n}_{N} r - \alpha (\Delta t)^{2} \int_{\Omega_{b}} \mathcal{J}^{(\eta^{n}_{N})^{\delta}}_{b} \boldsymbol{\eta}^{n}_{N} \cdot \nabla^{(\eta^{n}_{N})^{\delta}}_{b} r  - \alpha (\Delta t)^{2} \int_{\Gamma} (\boldsymbol{\eta}^{n}_{N} \cdot \boldsymbol{n}^{(\omega^{n}_{N})^{\delta}}) r + (\Delta t)^{2} \int_{\Gamma} (\bd{\eta}^{n}_{N} \cdot \bd{n}^{\omega^{n}_{N}})r.
\end{multline}

We now show that the bilinear form {{$B[(\bd{u}, \bd{\eta}, p), (\bd{v}, \bd{\psi}, r)]$}} is coercive and continuous as a bilinear form on the Hilbert space 
$
\mathcal{V}^{\omega^{n}_{N}}_{f} \times V_{d} \times V_{p},
$
with the inner product given by
\begin{equation*}
\langle (\bd{u}, \bd{\eta}, p), (\bd{v}, \bd{\psi}, r)\rangle = \int_{\Omega_{f}} (\bd{u} \cdot \bd{v} + \nabla \bd{u} : \nabla \bd{v}) + \int_{\Omega_{b}} (\bd{\eta} \cdot \bd{\psi} + \nabla \bd{\eta} : \nabla \bd{\psi}) + \int_{\Omega_{b}} (p \cdot r + \nabla p \cdot \nabla r).
\end{equation*}
We focus on establishing coercivity, since continuity follows by standard arguments. 
To show coercivity we calculate {{$B[(\boldsymbol{u}, \boldsymbol{v}, p), (\bd{u}, \boldsymbol{\eta}, p)]$}}.
In this calculation we note that after integration by parts, the sum of the following terms becomes zero:
\begin{equation*}
-\alpha(\Delta t)^{2} \int_{\Omega_{b}} \mathcal{J}^{(\eta^{n}_{N})^{\delta}}_{b} p \nabla^{(\eta^{n}_{N})^{\delta}}_{b} \cdot \boldsymbol{\eta} - \alpha(\Delta t)^{2} \int_{\Omega_{b}} \mathcal{J}^{(\eta^{n}_{N})^{\delta}}_{b} \boldsymbol{\eta} \cdot \nabla^{(\eta^{n}_{N})^{\delta}}_{b} p - \alpha (\Delta t)^{2} \int_{\Gamma} \left(\boldsymbol{\eta} \cdot \boldsymbol{n}^{(\omega^{n}_{N})^{\delta}}\right) p = 0.
\end{equation*}
Indeed, to see this, 
we bring the integrals back to the time-dependent physical domain, \textbf{which we can do as long as $(\bd{\eta}^{n}_{N})^{\delta}$ is a bijection from $\Omega_{b}$ to $(\Omega_{b})^{n, \delta}_{N}$}, which is provided by Assumption 2A \eqref{2A}, and perform the following computation:
\begin{align*}
&-\alpha(\Delta t)^{2} \int_{\Omega_{b}} \mathcal{J}^{(\eta^{n}_{N})^{\delta}}_{b} p \nabla^{(\eta^{n}_{N})^{\delta}}_{b} \cdot \boldsymbol{\eta} - \alpha(\Delta t)^{2} \int_{\Omega_{b}} \mathcal{J}^{(\eta^{n}_{N})^{\delta}}_{b} \boldsymbol{\eta} \cdot \nabla^{(\eta^{n}_{N})^{\delta}}_{b} p - \alpha (\Delta t)^{2} \int_{\Gamma} \left(\boldsymbol{\eta} \cdot \boldsymbol{n}^{(\omega^{n}_{N})^{\delta}}\right) p \\
&= -\alpha (\Delta t)^{2} \left(\int_{(\Omega_{b})^{n, \delta}_{N}} p \nabla \cdot \boldsymbol{\eta} + \int_{(\Omega_{b})^{n, \delta}_{N}} \boldsymbol{\eta} \cdot \nabla p + \int_{\Gamma^{n, \delta}_{N}} (\boldsymbol{\eta} \cdot \boldsymbol{n}) p\right) = 0,
\end{align*}
where we used integration by parts, the fact that $\boldsymbol{n}$ points outwards from $\Omega_{f}$ and hence inwards towards $\Omega_{b}$, and also use  
that $\bd{\eta} = 0$ on the left, right, and top boundaries of $\Omega_{b}$. Combining this with the fact that $(\Delta t) \zeta^{n + \frac{1}{2}}_{N} = \omega^{n + \frac{1}{2}}_{N} - \omega^{n - \frac{1}{2}}_{N} = \omega^{n + 1}_{N} - \omega^{n}_{N}$, we obtain
\begin{multline*}
{{B[(\boldsymbol{u}, \boldsymbol{\eta}, p), (\boldsymbol{u}, \bd{\eta}, p)]}} := (\Delta t)^{2} \int_{\Omega_{f}} \left(1 + \frac{\omega^{n}_{N} + \omega^{n + 1}_{N}}{2R}\right) |\boldsymbol{u}|^{2} + 2\nu (\Delta t)^{3} \int_{\Omega_{f}} \left(1 + \frac{\omega^{n}_{N}}{R}\right) \left|\boldsymbol{D}^{\omega^{n}_{N}}_{f} (\boldsymbol{u})\right|^{2} \\
+ \frac{\beta}{\mathcal{J}^{\omega^{n}_{N}}_{\Gamma}}(\Delta t) \int_{\Gamma} \left|(\boldsymbol{\eta} - (\Delta t)\boldsymbol{u}) \cdot \boldsymbol{\tau}^{\omega^{n}_{N}}\right|^{2} + \rho_{b} \int_{\Omega_{b}} |\boldsymbol{\eta}|^{2} + \rho_{p} \int_{\Gamma} |\bd{\eta}|^{2} + (2\mu_{e} (\Delta t)^{2} + 2\mu_{v}(\Delta t)) \int_{\Omega_{b}} |\boldsymbol{D}(\boldsymbol{\eta})|^{2} \\
+ (\lambda_{e} (\Delta t)^{2} + \lambda_{v} (\Delta t)) \int_{\Omega_{b}} \left|\nabla \cdot \boldsymbol{\eta}\right|^{2} + c_{0} (\Delta t)^{2} \int_{\Omega_{b}} |p|^{2} + \kappa (\Delta t)^{3} \int_{\Omega_{b}} \mathcal{J}^{(\eta^{n}_{N})^{\delta}}_{b} |\nabla^{(\eta^{n}_{N})^{\delta}}_{b} p|^{2}.
\end{multline*}

\if 1 = 0 %%%%%%%%%%%%%%%%
We must show that the symmetrized gradient can be expressed in terms of the usual gradient. It is well-known that one can do this for the fluid, and one even has a Korn \textit{equality} in this case. However, this is not immediate for the Biot material. We thus prove the following explicit Korn inequality for the Biot domain. 
\begin{proposition}[Korn inequality for the Biot poroviscoelastic domain]\label{Korn}
For all $\bd{\eta} \in V_{d}$,
\begin{equation*}
\int_{\Omega_{b}} |\bd{D}(\bd{\eta})|^{2} \ge \frac{1}{2} \int_{\Omega_{b}} |\nabla \bd{\eta}|^{2}.
\end{equation*}
\end{proposition}
\begin{proof}
By a standard approximation argument, it suffices to assume that $\bd{\eta}$ is smooth. Because $\eta_{x} = 0$ on $\Gamma$ and because $\bd{\eta} = 0$ on the left, top, and right boundaries of $\Omega_{b}$, we have from integration by parts, that
\begin{equation*}
\int_{\Omega_{b}} \frac{\partial \eta_{x}}{\partial y} \frac{\partial \eta_{y}}{\partial x} = -\int_{\Omega_{b}} \eta_{x} \frac{\partial^{2} \eta_{y}}{\partial x \partial y} = \int_{\Omega_{b}} \frac{\partial \eta_{x}}{\partial x} \frac{\partial \eta_{y}}{\partial y}.
\end{equation*}
Therefore,
\begin{multline*}
\int_{\Omega_{b}} |\bd{D}(\bd{\eta})|^{2} = \int_{\Omega_{b}} \left(\frac{\partial \eta_{x}}{\partial x}\right)^{2} + \left(\frac{\partial \eta_{y}}{\partial y}\right)^{2} + \frac{1}{2}\left(\frac{\partial \eta_{x}}{\partial y} + \frac{\partial \eta_{y}}{\partial x}\right)^{2} \\
= \int_{\Omega_{b}} \left(\frac{\partial \eta_{x}}{\partial x}\right)^{2} + \left(\frac{\partial \eta_{y}}{\partial y}\right)^{2} + \frac{1}{2}\left[\left(\frac{\partial \eta_{x}}{\partial y}\right)^{2} + \left(\frac{\partial \eta_{y}}{\partial x}\right)^{2}\right] + \frac{\partial\eta_{x}}{\partial y} \frac{\partial \eta_{y}}{\partial x} \\
= \int_{\Omega_{b}} \left(\frac{\partial \eta_{x}}{\partial x}\right)^{2} + \frac{\partial\eta_{x}}{\partial x} \frac{\partial \eta_{y}}{\partial y} +\left(\frac{\partial \eta_{y}}{\partial y}\right)^{2} + \frac{1}{2}\left[\left(\frac{\partial \eta_{x}}{\partial y}\right)^{2} + \left(\frac{\partial \eta_{y}}{\partial x}\right)^{2}\right] \ge \frac{1}{2} \int_{\Omega_{b}} |\nabla \bd{\eta}|^{2},
\end{multline*}
by using the inequality $a^{2} + 2ab + b^{2} \ge 0$.
\end{proof}
\fi %%%%%%%%%%%%%%%%

Coercivity of this form follows from the fact that $|\omega^{k + \frac{i}{2}}_{N}| < R$, see Assumption 1A in \eqref{1A}, 
and Korn inequality, see Proposition~\ref{Korn}, once we handle the last term and show that 
\begin{equation*}
\kappa (\Delta t)^{3} \int_{\Omega_{b}} \mathcal{J}^{(\eta^{n}_{N})^{\delta}}_{b} |\nabla^{(\eta^{n}_{N})^{\delta}}_{b} p|^{2} \ge c\int_{\Omega_{b}} |\nabla p|^{2}, 
\end{equation*}
for some positive constant $c > 0$. To show this, we first recall the definitions
\begin{equation*}
\mathcal{J}^{(\eta^{n}_{N})^{\delta}}_{b} = \det(\bd{I} + \nabla (\bd{\eta}^{n}_{N})^{\delta}), \qquad \nabla^{(\eta^{n}_{N})^{\delta}}_{b} p = \nabla p \cdot (\bd{I} + \nabla (\bd{\eta}^{n}_{N})^{\delta})^{-1}.
\end{equation*}
Then, letting $| \cdot |$ denote the matrix norm, we have 
\begin{equation}\label{coercivepressure}
\kappa (\Delta t)^{3} \int_{\Omega_{b}} \mathcal{J}^{(\eta^{n}_{N})^{\delta}}_{b} |\nabla^{(\eta^{n}_{N})^{\delta}}_{b} p|^{2} \ge \kappa(\Delta t)^{3} \int_{\Omega_{b}} \mathcal{J}^{(\eta^{n}_{N})^{\delta}}_{b} |\bd{I} + \nabla (\bd{\eta}^{n}_{N})^{\delta}|^{-2} |\nabla p|^{2}.
\end{equation}
Assumption 2A \eqref{2A} implies that $\bd{I} + (\bd{\eta}^{n}_{N})^{\delta}$ is an invertible map from $\Omega_{b}$ to $(\Omega_{b})^{n, \delta}_{N}$, and 
we further note that $|\bd{I} + \nabla (\bd{\eta}^{n}_{N})^{\delta}|$ is continuous on $\overline{\Omega_{b}}$ and hence is bounded from above. Thus, $|\bd{I} + \nabla (\bd{\eta}^{n}_{N})^{\delta}|^{-2} \ge c_{0} > 0$ for some positive constant $c_0$. The assumption that $\bd{I} + (\bd{\eta}^{n}_{N})^{\delta}$ is invertible implies that $\det(\bd{I} + \nabla (\bd{\eta}^{n}_{N})^{\delta}) > 0$. However, since this determinant is a continuous function on the compact set $\overline{\Omega_{b}}$, we conclude that there exists a positive constant $c_{1} > 0$ such that $\det(\bd{I} + \nabla (\bd{\eta}^{n}_{N})^{\delta}) \ge c_{1} > 0$. This establishes coercivity. 

Existence of a unique weak solution $(\bd{u}^{n + 1}_{N}, \bd{\eta}^{n + 1}_{N}, p^{n + 1}_{N}) \in \mathcal{V}^{\omega^{n}_{N}}_{f} \times V_{d} \times V_{p}$  now follows from the Lax-Milgram lemma. 
From here, we recover $\zeta^{n + 1}_{N}$, by using $\displaystyle \zeta^{n + 1}_{N} \bd{e}_{y} = \frac{\bd{\eta}^{n + 1}_{N} - \bd{\eta}^{n}_{N}}{\Delta t} \Big\vert_{\Gamma}$. Note that $\displaystyle \frac{\bd{\eta}^{n + 1}_{N} - \bd{\eta}^{n}_{N}}{\Delta t} \Big\vert_{\Gamma}$ points in the $y$ direction because the trace of any function $\bd{\eta} \in V_{d}$ on $\Gamma$ points in the $y$ direction by definition, see \eqref{Vd}.

\textbf{Energy equality:} We substitute $\boldsymbol{v} = \boldsymbol{u}^{n + 1}_{N}$, $\varphi = \zeta^{n + 1}_{N}$, $\boldsymbol{\psi} = \boldsymbol{\dot{\eta}}^{n + 1}_{N}$, and $r = p^{n + 1}_{N}$ into \eqref{weakBiot1}, and use the identity
\begin{equation*}
(a - b) \cdot a = \frac{1}{2}(|a|^{2} + |a - b|^{2} - |b|^{2}).
\end{equation*}
Since $\omega^{n + 1}_{N} = \omega^{n + \frac{1}{2}}_{N}$ and $(\Delta t) \zeta^{n + \frac{1}{2}}_{N} = \omega^{n + \frac{1}{2}}_{N} - \omega^{n}_{N}$, we obtain the following energy equality:
\begin{align*}
&\frac{1}{2} \int_{\Omega_{f}} \left(1 + \frac{\omega^{n + 1}_{N}}{R}\right) |\boldsymbol{u}^{n + 1}_{N}|^{2} + \frac{1}{2} \rho_{b} \int_{\Omega_{b}} |\dot{\bd{\eta}}^{n + 1}_{N}|^{2} + \frac{1}{2} c_{0} \int_{\Omega_{b}} |p^{n + 1}_{N}|^{2} + \mu_{e} \int_{\Omega_{b}} |\bd{D}(\bd{\eta}^{n + 1}_{N})|^{2} + \frac{1}{2} \lambda_{e} \int_{\Omega_{b}} |\nabla \cdot \bd{\eta}^{n + 1}_{N}|^{2} \\
&+ \frac{1}{2} \rho_{p} \int_{\Gamma} |\zeta^{n + 1}_{N}|^{2} + 2\mu_{v}(\Delta t) \int_{\Omega_{b}} |\bd{D}(\dot{\bd{\eta}}^{n + 1}_{N})|^{2} + \lambda_{v}(\Delta t) \int_{\Omega_{b}} |\nabla \cdot \dot{\bd{\eta}}^{n + 1}_{N}|^{2} + \kappa (\Delta t) \int_{\Omega_{b}} \mathcal{J}^{(\eta^{n}_{N})^{\delta}}_{b} |\nabla^{(\eta^{n}_{N})^{\delta}}_{b} p^{n + 1}_{N}|^{2} \\
&+ \frac{\beta (\Delta t)}{\mathcal{J}^{\omega^{n}_{N}}_{\Gamma}} \int_{\Gamma} |(\boldsymbol{\dot{\eta}}^{n + 1}_{N} - \boldsymbol{u}^{n + 1}_{N}) \cdot \boldsymbol{\tau}^{\omega^{n}_{N}}|^{2} + \frac{1}{2} \rho_{b} \int_{\Omega_{b}} |\bd{\dot{\eta}}^{n + 1}_{N} - \bd{\dot{\eta}}^{n}_{N}|^{2} + \frac{1}{2} c_{0} \int_{\Omega_{b}} |p^{n + 1}_{N} - p^{n}_{N}|^{2} + \mu_{e} \int_{\Omega_{b}} |\bd{D}(\bd{\eta}^{n + 1}_{N} - \bd{\eta}^{n}_{N})|^{2} \\
&+ \frac{1}{2} \lambda_{e} \int_{\Omega_{b}} |\nabla \cdot (\bd{\eta}^{n + 1}_{N} - \bd{\eta}^{n}_{N})|^{2} = \frac{1}{2} \int_{\Omega_{f}} \left(1 + \frac{\omega^{n}_{N}}{R}\right) |\bd{u}^{n}_{N}|^{2} + \frac{1}{2} \rho_{b} \int_{\Omega_{b}} |\dot{\bd{\eta}}^{n}_{N}|^{2} + \frac{1}{2} c_{0} \int_{\Omega_{b}} |p^{n}_{N}|^{2} + \mu_{e} \int_{\Omega_{b}} |\bd{D}(\bd{\eta}^{n}_{N})|^{2} \\
&+ \frac{1}{2} \lambda_{e} \int_{\Omega_{b}} |\nabla \cdot \bd{\eta}^{n}_{N}|^{2} + \frac{1}{2} \rho_{p} \int_{\Gamma} |\zeta^{n + \frac{1}{2}}_{N}|^{2},
\end{align*}
where the terms containing parameter $\alpha$ cancel out after 
bringing the integrals back to the time-dependent domain, integrating by parts, and recalling that the normal vector points inward towards the Biot domain:
\begin{multline*}
-\alpha \int_{\Omega_{b}} \mathcal{J}^{(\eta^{n}_{N})^{\delta}}_{b} p^{n + 1}_{N} \nabla^{(\eta^{n}_{N})^{\delta}}_{b} \cdot \boldsymbol{\dot{\eta}}^{n + 1}_{N} - \alpha \int_{\Omega_{b}} \mathcal{J}^{(\eta^{n}_{N})^{\delta}}_{b} \boldsymbol{\dot{\eta}}^{n + 1}_{N} \cdot \nabla^{(\eta^{n}_{N})^{\delta}}_{b} p^{n + 1}_{N} - \alpha \int_{\Gamma} \left(\boldsymbol{\dot{\eta}}^{n + 1}_{N} \cdot \boldsymbol{n}^{(\omega^{n}_{N})^{\delta}}\right) p^{n + 1}_{N} \\ 
= -\alpha\int_{(\Omega_{b})^{n, \delta}_{N}} p^{n + 1}_{N} (\nabla \cdot \boldsymbol{\dot{\eta}}^{n + 1}_{N}) - \alpha \int_{(\Omega_{b})^{n, \delta}_{N}} \boldsymbol{\dot{\eta}}^{n + 1}_{N} \cdot \nabla p^{n + 1}_{N} - \alpha \int_{\Gamma^{n, \delta}_{N}} (\boldsymbol{\dot{\eta}}^{n + 1}_{N} \cdot \boldsymbol{n}) p^{n + 1}_{N} = 0.
\end{multline*}
This completes the proof of the Lemma.
\end{proof}

\subsection{The coupled semi-discrete problem: weak formulation and energy}
To obtain uniform energy estimates for approximate solutions of our semidiscretized scheme it is useful to present the scheme in monolithic form:
\begin{equation}\label{semi1}
\begin{aligned}
&\int_{\Omega_{f}} \left(1 + \frac{\omega^{n}_{N}}{R}\right) \boldsymbol{\dot{u}}^{n + 1}_{N} \cdot \boldsymbol{v} 
+ 2\nu \int_{\Omega_{f}} \left(1 + \frac{\omega^{n}_{N}}{R}\right) \boldsymbol{D}^{\omega^{n}_{N}}_{f}(\boldsymbol{u}^{n + 1}_{N}) : \boldsymbol{D}^{\omega^{n}_{N}}_{f}(\boldsymbol{v}) + \int_{\Gamma} \left(\frac{1}{2}\boldsymbol{u}^{n + 1}_{N} \cdot \boldsymbol{u}^{n}_{N} - p^{n + 1}_{N}\right)(\boldsymbol{\psi} - \boldsymbol{v})\cdot \boldsymbol{n}^{\omega^{n}_{N}} 
\\
&+ \frac{1}{2} \int_{\Omega_{f}} \left(1 + \frac{\omega^{n}_{N}}{R}\right) \left[\left(\left(\boldsymbol{u}^{n}_{N} - \zeta^{n + \frac{1}{2}}_{N} \frac{R + y}{R} \boldsymbol{e}_{y}\right) \cdot \nabla^{\omega^{n}_{N}}_{f} \boldsymbol{u}^{n + 1}_{N}\right) \cdot \boldsymbol{v} - \left(\left(\boldsymbol{u}^{n}_{N} - \zeta^{n + \frac{1}{2}}_{N} \frac{R + y}{R} \boldsymbol{e}_{y}\right) \cdot \nabla^{\omega^{n}_{N}}_{f} \boldsymbol{v}\right) \cdot \boldsymbol{u}^{n + 1}_{N}\right] 
\\
&+ \frac{1}{2R} \int_{\Omega_{f}} \zeta^{n + \frac{1}{2}}_{N} \boldsymbol{u}^{n + 1}_{N} \cdot \boldsymbol{v} + \frac{1}{2} \int_{\Gamma} (\boldsymbol{u}^{n + 1}_{N} - \boldsymbol{\dot{\eta}}^{n + 1}_{N}) \cdot \boldsymbol{n}^{\omega^{n}_{N}} (\boldsymbol{u}^{n}_{N} \cdot \boldsymbol{v})
+ \frac{\beta}{\mathcal{J}^{\omega^{n}_{N}}_{\Gamma}} \int_{\Gamma} (\boldsymbol{\dot{\eta}}^{n + 1}_{N} - \boldsymbol{u}^{n + 1}_{N}) \cdot \boldsymbol{\tau}^{\omega^{n}_{N}} (\boldsymbol{\psi} - \boldsymbol{v}) \cdot \boldsymbol{\tau}^{\omega^{n}_{N}}
 \\
& + \rho_{b} \int_{\Omega_{b}} \left(\frac{\boldsymbol{\dot{\eta}}^{n + 1}_{N} - \boldsymbol{\dot{\eta}}^{n}_{N}}{\Delta t}\right) \cdot \boldsymbol{\psi}
 + \rho_{p} \int_{\Gamma} \left(\frac{\zeta^{n + 1}_{N} - \zeta^{n}_{N}}{\Delta t}\right) \varphi + 2\mu_{e} \int_{\Omega_{b}} \boldsymbol{D}(\boldsymbol{\eta}^{n + 1}_{N}) : \boldsymbol{D}(\boldsymbol{\psi}) + \lambda_{e} \int_{\Omega_{b}} (\nabla \cdot \boldsymbol{\eta}^{n + 1}_{N})(\nabla \cdot \boldsymbol{\psi}) 
\\
&+ 2\mu_{v} \int_{\Omega_{b}} \bd{D}(\dot{\bd{\eta}}^{n + 1}_{N}) : \bd{D}(\bd{\psi}) + \lambda_{v} \int_{\Omega_{b}} (\nabla \cdot \dot{\bd{\eta}}^{n + 1}_{N}) (\nabla \cdot \bd{\psi}) - \alpha \int_{\Omega_{b}} \mathcal{J}^{(\eta^{n}_{N})^{\delta}}_{b} p^{n + 1}_{N} \nabla^{(\eta^{n}_{N})^{\delta}}_{b} \cdot \boldsymbol{\psi} 
+ c_{0}  \int_{\Omega_{b}} \frac{p^{n + 1}_{N} - p^{n}_{N}}{\Delta t} r 
\\
&- \alpha \int_{\Omega_{b}} \mathcal{J}^{(\eta^{n}_{N})^{\delta}}_{b} \boldsymbol{\dot{\eta}}^{n + 1}_{N} \cdot \nabla^{(\eta^{n}_{N})^{\delta}}_{b} r - \alpha \int_{\Gamma} (\boldsymbol{\dot{\eta}}^{n + 1}_{N} \cdot \boldsymbol{n}^{(\omega^{n}_{N})^{\delta}}) r 
+ \kappa \int_{\Omega_{b}} \mathcal{J}^{(\eta^{n}_{N})^{\delta}}_{b} \nabla^{(\eta^{n}_{N})^{\delta}}_{b} p^{n + 1}_{N} \cdot \nabla^{(\eta^{n}_{N})^{\delta}}_{b} r
\\
& - \int_{\Gamma} [(\boldsymbol{u}^{n + 1}_{N} - \boldsymbol{\dot{\eta}}^{n + 1}_{N}) \cdot \boldsymbol{n}^{\omega^{n}_{N}}]r + \int_{\Gamma} \Delta \omega^{n + \frac{1}{2}}_{N} \cdot \Delta \varphi = 0,\ 
\forall (\bd{v}, \varphi, \bd{\psi}, r) \in \mathcal{Q}^{n + 1}_{N},
\end{aligned}
\end{equation}
\begin{equation}\label{semi2}
\int_{\Gamma} \left(\frac{\omega^{n + \frac{1}{2}}_{N} - \omega^{n - \frac{1}{2}}_{N}}{\Delta t}\right) \phi = \int_{\Gamma} \zeta^{n + \frac{1}{2}}_{N} \phi, \qquad \int_{\Gamma} \left(\frac{\boldsymbol{\eta}^{n + 1}_{N} - \boldsymbol{\eta}^{n}_{N}}{\Delta t}\right) \cdot \boldsymbol{\phi} = \int_{\Gamma} \zeta^{n + 1}_{N} \bd{e}_{y} \cdot \boldsymbol{\phi}, \ \forall \phi, \boldsymbol{\phi} \in L^{2}(\Gamma).
\end{equation}

{{Next, we will obtain uniform energy estimates for the approximate solutions generated from the splitting scheme. To do this, we define the discrete energy $E^{n + \frac{i}{2}}_{N}$ and discrete dissipation $D^{n + 1}_{N}$ as follows:}}
\begin{equation}\label{discenergy}
\begin{aligned}
E^{n + \frac{i}{2}}_{N} &= \frac{1}{2}\int_{\Omega_{f}} \left(1 + \frac{\omega^{n}_{N}}{R}\right) |\boldsymbol{u}^{n + \frac{i}{2}}_{N}|^{2} + \frac{1}{2}\rho_{b} \int_{\Omega_{b}} 
|\boldsymbol{\dot{\eta}}^{n + \frac{i}{2}}_{N}|^{2} + \frac{1}{2}c_{0}\int_{\Omega_{b}} |p^{n + \frac{i}{2}}_{N}|^{2} + \mu_{e} \int_{\Omega_{b}} |\boldsymbol{D}(\boldsymbol{\eta}^{n + \frac{i}{2}}_{N})|^{2}, . \\
&+ \frac{1}{2} \lambda_{e} \int_{\Omega_{b}} |\nabla \cdot \boldsymbol{\eta}^{n + \frac{i}{2}}_{N}|^{2} + \frac{1}{2} \rho_{p} \int_{\Gamma} |\zeta^{n + \frac{i}{2}}_{N} |^{2} + \frac{1}{2} \int_{\Gamma} |\Delta\omega^{n + \frac{i}{2}}_{N}|^{2}, \ i = 0,1.
\\
%\label{discdissipation}
D^{n + 1}_{N} &= 2\nu(\Delta t) \int_{\Omega_{f}} \left(1 + \frac{\omega^{n}_{N}}{R}\right) \left|\boldsymbol{D}^{\omega^{n}_{N}}_{f}(\boldsymbol{u}^{n + 1}_{N})\right|^{2} + 2\mu_{v} (\Delta t) \int_{\Omega_{b}} |\bd{D}(\dot{\bd{\eta}}^{n + 1}_{N})|^{2} + \lambda_{v} (\Delta t) \int_{\Omega_{b}} |\nabla \cdot \dot{\bd{\eta}}^{n + 1}_{N}|^{2}  \\
&+ \kappa(\Delta t) \int_{\Omega_{b}} \mathcal{J}^{(\eta^{n}_{N})^{\delta}}_{b} |\nabla^{(\eta^{n}_{N})^{\delta}}_{b} p^{n + 1}_{N}|^{2} + \frac{\beta(\Delta t)}{\mathcal{J}^{\omega^{n}_{N}}_{\Gamma}} \int_{\Gamma} \left|(\boldsymbol{\dot{\eta}}^{n + 1}_{N} - \boldsymbol{u}^{n + 1}_{N}) \cdot \boldsymbol{\tau}^{\omega^{n}_{N}}\right|^{2}.
\end{aligned}
\end{equation}

{{Then, the semidiscrete weak formulation \eqref{semi1} and \eqref{semi2} implies the following uniform estimates on the discretized energy and dissipation.}}
\begin{lemma}\label{discreteenergyestimates}
The following {\bf{discrete energy equalities}} hold for the semi-discretized formulation \eqref{semi1}, \eqref{semi2}:
\begin{equation}\label{energyeq1}
E^{n + \frac{1}{2}}_{N} + \frac{1}{2} \rho_{p} \int_{\Gamma} \left|\zeta^{n + \frac{1}{2}}_{N} - \zeta^{n}_{N}\right|^{2} + \frac{1}{2} \int_{\Gamma} \left|\Delta(\omega^{n + \frac{1}{2}}_{N} - \omega^{n - \frac{1}{2}}_{N})\right|^{2} = E^{n}_{N}
\end{equation}
\begin{multline}\label{energyeq2}
E^{n + 1}_{N} + D^{n + 1}_{N} + \frac{1}{2} \int_{\Omega_{f}} \left(1 + \frac{\omega^{n}_{N}}{R}\right) \left|\boldsymbol{u}^{n + 1}_{N} - \boldsymbol{u}^{n}_{N}\right|^{2} + \frac{1}{2} \rho_{b} \int_{\Omega_{b}} \left|\boldsymbol{\dot{\eta}}^{n + 1}_{N} - \boldsymbol{\dot{\eta}}^{n}_{N}\right|^{2} + \frac{1}{2} c_{0} \int_{\Omega_{b}} \left|p^{n + 1}_{N} - p^{n}_{N} \right|^{2} \\
+ \mu_{e} \int_{\Omega_{b}} \left|\boldsymbol{D}(\boldsymbol{\eta}^{n + 1}_{N} - \boldsymbol{\eta}^{n}_{N})\right|^{2} + \frac{1}{2}\lambda_{e} \int_{\Omega_{b}} \left|\nabla \cdot (\boldsymbol{\eta}^{n + 1}_{N} - \boldsymbol{\eta}^{n}_{N})\right|^{2} + \frac{1}{2}\rho_{p} \int_{\Gamma} |\zeta^{n + 1}_{N} - \zeta^{n + \frac{1}{2}}_{N}|^{2} = E^{n + \frac{1}{2}}_{N},
\end{multline}
{{where the discrete energy $E^{n + \frac{i}{2}}_{N}$ and the discrete dissipation $D^{n + 1}_{N}$ are defined in \eqref{discenergy}.}}
\end{lemma}
We remark that the terms not included in the definition of $E^{n + \frac{i}{2}}_{N}$ and $D^{n + 1}_{N}$, appearing in \eqref{energyeq1} and \eqref{energyeq2}, are {\bf{numerical dissipation terms}}. 

These energy identities immediately imply that $E^{n + \frac{i}{2}}_{N}$ and $\sum_{n = 1}^{N} D^{n}_{N}$ are {\bf{uniformly bounded}} by a constant $C$ independent of $n$ and $N$.

The semidiscretized splitting scheme defines semidiscretized approximations of the solution to the regularized problem at discrete time points. To work with approximate functions
and show that they converge to the solution of the continuous problem, we need to extend the semidiscrete approximations to the entire time interval
and investigate uniform boundedness of those approximate solution functions. This is done next.

\section{Approximate solutions}\label{approximate}

Now that we have defined the numerical solutions at each time step, we collect the solutions into approximate solutions defined on the whole time interval $[0, T]$, for which we will obtain uniform estimates from our previous energy estimates. 

We define the following two extensions of the approximate functions to the entire interval $[0,T]$:

\begin{itemize}
\item Piecewise constant approximate solutions, for $(n - 1) \Delta t < t \le n\Delta t$:
\begin{multline}\label{approxconstant}
\boldsymbol{u}_{N}(t) = \boldsymbol{u}^{n}_{N}, \quad \boldsymbol{\eta}_{N}(t) = \boldsymbol{\eta}^{n}_{N}, \quad p_{N}(t) = p^{n}_{N}, \quad \omega_{N}(t) = \omega^{n - \frac{1}{2}}_{N}, \quad \zeta_{N}(t) = \zeta^{n - \frac{1}{2}}_{N}, \quad \zeta_{N}^{*}(t) = \zeta^{n}_{N};
\end{multline}
\item Linear interpolations:
{{
\begin{equation}\label{linearinterpolation}
\boldsymbol{\overline{\eta}}_{N}(n\Delta t) = \boldsymbol{\eta}^{n}_{N}, \qquad \overline{p}_{N}(n\Delta t) = p^{n}_{N}, \qquad \overline{\omega}_{N}(n\Delta t) = \omega^{n - \frac{1}{2}}_{N}, \qquad \text{ for } n = 0, 1, ..., N,
\end{equation}}}
where we formally set $\omega^{-\frac{1}{2}}_{N} = \omega_{0}$. 
\end{itemize}
Note that by construction, we have that
{{
\begin{equation}\label{approxdt}
\partial_{t} \overline{\omega}_{N} = \zeta_{N}, \qquad \partial_{t} \boldsymbol{\overline{\eta}}_{N}|_{\Gamma} = \zeta^{*}_{N} \boldsymbol{e}_{y}.
\end{equation}}}

From the preceding energy estimates, we have the following lemma on uniform boundedness.

\begin{lemma}\label{uniform}{\sc{Uniform boundedness of approximate solutions.}}
Assume:
\begin{enumerate}
\item {\sc{Assumption 1B: Uniform boundedness of plate displacements.}} There exists a positive constant $R_{max}$ such that for all $N$,
\begin{equation}\label{1Anew}
|\omega^{n - \frac{1}{2}}_{N}| \le R_{max} < R, \qquad \text{ for all } n = 0, 1, ..., N,
\end{equation}
\begin{equation}\label{1B}
\left| (\bd{\eta}^{n}_{N})^{\delta}|_{\Gamma} \right| \le R_{max} < R, \qquad \text{ for all } n = 0, 1, ..., N.
\end{equation}
\item {\sc{Assumption 2B: Uniform invertibility of the {{Lagrangian map}} (Jacobian).}} There exists a positive constant $c_{0}$ such that for all $N$,
\begin{equation}\label{2B}
\det(\bd{I} + \nabla(\bd{\eta}^{n}_{N})^{\delta}) \ge c_{0} > 0, \qquad \text{ for all } n = 0, 1, ..., N. 
\end{equation}
\item {\sc{Assumption 2C: Uniform boundedness of the {{Lagrangian map}} (matrix norm).}} There exists positive constants $c_{1}$ and $c_{2}$ such that for all $N$,
\begin{equation}\label{2C}
|(\bd{I} + \nabla(\bd{\eta}^{n}_{N})^{\delta})^{-1}| \le c_{1}, \qquad |\bd{I} + \nabla(\bd{\eta}^{n}_{N})^{\delta}| \le c_{2}, \qquad \text{ for all } n = 0, 1, ..., N. 
\end{equation}
\end{enumerate}
Then for all $N$:
\begin{itemize}
\item $\boldsymbol{u}_{N}$ is uniformly bounded in $L^{\infty}(0, T; L^{2}(\Omega_{f}))$ and $L^{2}(0, T; H^{1}(\Omega_{f}))$.
\item $\boldsymbol{\eta}_{N}$ is uniformly bounded in $L^{\infty}(0, T; H^{1}(\Omega_{b}))$.
\item $p_{N}$ is uniformly bounded in $L^{\infty}(0, T; L^{2}(\Omega_{b}))$ and $L^{2}(0, T; H^{1}(\Omega_{b}))$.
\item $\omega_{N}$ is uniformly bounded in $L^{\infty}(0, T; H_{0}^{2}(\Gamma))$.
\end{itemize}
In addition, we have the following estimates on the linear interpolations.
\begin{itemize}
\item $\boldsymbol{\overline{\eta}}_{N}$ is uniformly bounded in $W^{1, \infty}(0, T; L^{2}(\Omega_{b}))$.
\item $\boldsymbol{\overline{\omega}}_{N}$ is uniformly bounded in $W^{1, \infty}(0, T; L^{2}(\Gamma))$.
\end{itemize}
\end{lemma}

\begin{remark}\label{invertible}{\sc{A crucial remark about invertibility.}}
At first, it would appear that to show the uniform boundedness results above, we also need to have a fourth assumption, which is Assumption 2A \eqref{2A} from before, that the map $\text{Id} + (\bd{\eta}^{n}_{N})^{\delta}: \Omega_{b} \to \mathbb{R}^{2}$ is injective (and is hence a bijection onto its image), for each $n = 0, 1, ..., N$ and for all $N$. However, this is implied by an injectivity theorem, see Ciarlet \cite{Ciarlet} Theorem 5-5-2. Note also that Assumption 1A \eqref{1A} from before is automatically satisfied once we verify Assumption 1B \eqref{1Anew}, \eqref{1B}. 
In particular, this injectivity theorem is as follows. Since $\det(\bd{I} + \nabla(\bd{\eta}^{n}_{N})^{\delta}) > 0$ by Assumption 2B \eqref{2B}, it suffices to show that $\text{Id} + (\bd{\eta}^{n}_{N})^{\delta} = \bd{\varphi}_{0}$ on $\partial \Omega_{b}$, for some injective mapping $\bd{\varphi}_{0}:\overline{\Omega_{b}} \to \mathbb{R}^{2}$, for example
a standard ALE mapping $\bd{\varphi}_{0}(x, y) = \left(x, y + \left(1 - \frac{y}{R}\right)\omega \right)$ can be used.
This implies the very useful fact that $(\text{Id} + (\bd{\eta}^{n}_{N})^{\delta})(\overline{\Omega_{b}}) = \bd{\varphi}_{0}(\overline{\Omega_{b}})$, which means that the \textit{deformed configuration is fully determined by the behavior on the boundary}. 

%To construct the mapping $\bd{\varphi}_{0}$, we use a standard ALE mapping. Because $(\bd{\eta}^{n}_{N})^{\delta} = \omega \bd{e}_{y}$ on $\Gamma$ for some function $\omega$ with $|\omega| \le R_{max} < R$ (by Assumption 1B \eqref{1Anew}) and $\omega = 0$ on $\partial \Gamma$, with $(\bd{\eta}^{n}_{N})^{\delta}$ satisfying Dirichlet boundary conditions on all other parts of the boundary $\partial \Omega_{b}$, we can define the following injective mapping $\bd{\varphi}_{0}$ satisfying the necessary conditions on $\Omega_{b} = (0, L) \times (0, R)$:
%\begin{equation*}
%\bd{\varphi}_{0}(z, r) = \left(x, y + \left(1 - \frac{y}{R}\right)\omega \right).
%\end{equation*}
%Note also that $\omega$, which is the trace of $(\bd{\eta}^{n}_{N})^{\delta}$ on $\Gamma$, is a continuous function. We observe that this map $\bd{\varphi}_{0}$ is an injective map. Thus, we conclude that $\text{Id} + (\bd{\eta}^{n}_{N})^{\delta}$ is injective for all $n = 0, 1, ..., N$ and for all $N$, if Assumptions 1B and 2B, given by \eqref{1Anew}, \eqref{1B}, \eqref{2B}, hold.
\end{remark}

\begin{proof}
The uniform boundedness of approximate solutions follows from the uniform energy estimates. 
More precisely,  the uniform boundedness of $\bd{u}_{N}$ in $L^{\infty}(0, T; L^{2}(\Omega_{f}))$ follows from Assumption 1B \eqref{1Anew}. The uniform boundedness of $\bd{u}_{N}$ in $L^{2}(0, T; H^{1}(\Omega_{f}))$ follows from Korn's inequality on the fluid domain. 
The uniform boundedness of $\bd{\eta}_{N}$ in $L^{\infty}(0, T; H^{1}(\Omega_{b}))$ follows from combining the uniform energy estimates with Korn's inequality, stated in Proposition \ref{Korn}. To establish the uniform boundedness of $p_{N}$ in $L^{2}(0, T; H^{1}(\Omega_{b}))$, {{we note that the discrete energy estimates in Lemma \ref{discreteenergyestimates} imply the following uniform discrete dissipation bound:}}
\begin{equation*}
{{\sum_{n = 0}^{N - 1}}} \kappa (\Delta t) \int_{\Omega_{b}} \mathcal{J}^{(\eta^{n}_{N})^{\delta}}_{b} |\nabla^{(\eta^{n}_{N})^{\delta}}_{b} p^{n + 1}_{N}|^{2} \le C,
\end{equation*}
for some constant $C$ uniform in $N$, where
$
\mathcal{J}^{(\eta^{n}_{N})^{\delta}}_{b} = \det(\bd{I} + \nabla (\bd{\eta}^{n}_{N})^{\delta}),
$
and
$
\nabla^{(\eta^{n}_{N})^{\delta}}_{b} r = \nabla r \cdot (\bd{I} + \nabla(\bd{\eta}^{n}_{N})^{\delta})^{-1} \ \text{ on } \Omega_{b}.
$
By Assumption 2B \eqref{2B}, we conclude that
\begin{equation*}
(\Delta t) {{\sum_{n = 0}^{N - 1}}} \int_{\Omega_{b}} |\nabla^{(\eta^{n}_{N})^{\delta}}_{b} p^{n + 1}_{N}|^{2} \le C.
\end{equation*}
Since on $\Omega_{b}$, we have that $\nabla p^{n + 1}_{N} = \nabla^{(\eta^{n}_{N})^{\delta}}_{b} p^{n + 1}_{N} \cdot (\bd{I} + \nabla (\bd{\eta}^{n}_{N})^{\delta})$, we use Assumption 2C \eqref{2C}, which implies $|\bd{I} + \nabla (\bd{\eta}^{n}_{N})^{\delta}| \le c_{2}$,  and obtain the estimate
\begin{equation*}
(\Delta t) {{\sum_{n = 0}^{N - 1}}} \int_{\Omega_{b}} |\nabla p^{n + 1}_{N}|^{2} \le |\bd{I} + \nabla (\bd{\eta}^{n}_{N})^{\delta}|^{2} \cdot (\Delta t) {{\sum_{n = 0}^{N - 1}}} \int_{\Omega_{b}} |\nabla^{(\eta^{n}_{N})^{\delta}}_{b} p^{n + 1}_{N}|^{2} \le C,
\end{equation*}
for a constant $C$ independent of $N$. {{Thus, $p_{N}$ is uniformly bounded in $L^{2}(0, T; H^{1}(\Omega_{b}))$, since by the definition of the piecewise constant approximate solution $p_{N}$ in \eqref{approxconstant}, we have that
\begin{equation*}
||p_{N}||^{2}_{L^{2}(0, T; H^{1}(\Omega_{b}))} = (\Delta t) \sum_{n = 0}^{N - 1} \int_{\Omega_{b}} |\nabla p_{N}^{n + 1}|^{2}.
\end{equation*}
}}

\if 1 = 0
Using the cofactor formula for the inverse of a matrix and Assumption 2B, as the determinant appears in the formula for the inverse of a matrix, we conclude that the entries of $(\bd{I} + \nabla (\bd{\eta}^{n}_{N})^{\delta})^{-1}$ are uniformly bounded independently of $n$ and $N$, and hence the matrix norm $|\bd{I} + \nabla (\bd{\eta}^{n}_{N})^{\delta}|^{-1}$ is uniformly bounded independently of $n$ and $N$. This, combined with the fact that $\mathcal{J}^{(\eta^{n}_{N})^{\delta}}$ is uniformly bounded in $n$ and $N$, allows us to conclude that $p_{N}$ is uniformly bounded in $L^{2}(0, T; H^{1}(\Omega_{B}))$. 
\fi 

\end{proof}

The above uniform boundedness result implies the following weak convergence results.

\begin{proposition}\label{prop:weak}
Assume that the three assumptions listed in Lemma \ref{uniform} hold.
Then, there exists a subsequence such that the following weak convergence results hold:
\begin{itemize}
\item $\bd{u}_{N} \rightharpoonup \bd{u}$ weakly* in $L^{\infty}(0, T; L^{2}(\Omega_{f}))$,
\qquad $\bd{u}_{N} \rightharpoonup \bd{u}$ weakly in $L^{2}(0, T; H^{1}(\Omega_{f}))$,
\item $\bd{\eta}_{N} \rightharpoonup \bd{\eta}$ weakly* in $L^{\infty}(0, T; H^{1}(\Omega_{b}))$,
\qquad $\overline{\bd{\eta}}_{N} \rightharpoonup \overline{\bd{\eta}}$ weakly* in $W^{1, \infty}(0, T; L^{2}(\Omega_{b}))$,
\item $p_{N} \rightharpoonup p$ weakly* in $L^{\infty}(0, T; L^{2}(\Omega_{b}))$,
\qquad  $p_{N} \rightharpoonup p$ weakly in $L^{2}(0, T; H^{1}(\Omega_{b}))$, 
\item $\omega_{N} \rightharpoonup \omega$ weakly* in $L^{\infty}(0, T; H_{0}^{2}(\Gamma))$,
\qquad  $\overline{\omega}_{N} \rightharpoonup \overline{\omega}$ weakly* in $W^{1, \infty}(0, T; L^{2}(\Gamma))$.
\end{itemize}
Furthermore, $\bd{\eta} = \bd{\overline{\eta}}$ and $\omega = \overline{\omega}$.
\end{proposition}

To use these results and to be able to construct approximate solutions, it is essential to show that the assumptions from  Lemma \ref{uniform} hold. 
This is given by the following lemma.

\begin{lemma}\label{assumptions}
Suppose that the initial data satisfies $|\omega_{0}| \le R_{0} < R$ for some $R_{0}$, and suppose that $\bd{\eta}_{0}$ has the property that $\text{Id} + (\bd{\eta}_{0})^{\delta}$ is invertible with $\det(\bd{I} + \nabla(\bd{\eta}_{0})^{\delta})\ge c_0 > 0$ on $\Omega_{b}$ for some positive constant $c_{0}$. 
Then,{ \bf{there exists a sufficiently small time $T > 0$}}  such that for all $N$, all three assumptions in Lemma \ref{uniform} hold and the splitting scheme is well defined until time $T$.
\end{lemma}

\begin{proof}
First, notice that the assumptions on the initial data immediately imply 
that the three assumptions from Lemma \ref{uniform} hold for the initial data, i.e., for $n=0$.
In particular, there exist constants $\alpha_{0}$, $\alpha_{1}$, and $\alpha_{2}$ such that
\begin{equation}\label{initial2C}
\det(\bd{I} + \nabla(\bd{\eta}_{0})^{\delta}) \ge \alpha_{0} > 0,
\end{equation}
\begin{equation}\label{initial2B}
|\bd{I} + \nabla(\bd{\eta}_{0})^{\delta}| \ge \alpha_{1} > 0, \qquad |(\bd{I} + \nabla(\bd{\eta}_{0})^{\delta})^{-1}| \ge \alpha_{2} > 0.
\end{equation}
This is because $\det(\bd{I} + \nabla(\bd{\eta}_{0})^{\delta})$, $|\bd{I} + \nabla (\bd{\eta}_{0})^{\delta}|$, and $|(\bd{I} + \nabla(\bd{\eta}_{0})^{\delta})^{-1}|$ are positive continuous functions on the compact set $\overline{\Omega_{b}}$.

Next, we want to define an appropriate time $T>0$ such that the three assumptions hold uniformly for all $N$ and $n\Delta t$ up to time $T$. To do this, we use the energy estimates. Define the initial energy determined by the initial data by $E_{0}$. Then, by the uniform energy estimates, we have that 
{{\begin{equation*}
E^{k + \frac{1}{2}}_{N} \le E_{0}, \qquad E^{k + 1}_{N} \le E_{0}, \qquad \text{ for all } k = 0, 1, ..., N - 1.
\end{equation*}}}
Therefore, after completing both subproblems of the scheme on the time step $[k\Delta t, (k + 1)\Delta t]$, we obtain that 
\begin{equation}\label{energybound1}
||\bd{\dot{\eta}}^{n}_{N}||_{L^{2}(\Omega_{b})} \le C, \qquad \text{ for } n = 0, 1, ..., k + 1,
\end{equation}
\begin{equation}\label{energybound2}
||\omega^{n + \frac{1}{2}}_{N}||_{H_{0}^{2}(\Gamma)} \le C, \qquad \text{ for } n = 0, 1, ..., k,
\end{equation}
\begin{equation}\label{energybound3}
||\zeta^{n + \frac{i}{2}}_{N}||_{L^{2}(\Gamma)} \le C, \qquad \text{ for } 0 \le n + \frac{i}{2} \le k + 1 \qquad \text{ and } i = 0, 1,
\end{equation}
for a constant $C$ depending only on the initial energy $E_{0}$. 

\noindent \textbf{{{Step 1 (Uniform bound on the plate displacements $\omega^{n - \frac{1}{2}}_{N}$)}}.} We first find a condition on $T$ such that Assumption 1B \eqref{1Anew}  is satisfied. Suppose that the linear interpolation $\overline{\omega}_{N}$ is {{defined up to time $(k + 1)\Delta t$, where we recall that the linear interpolation is defined via \eqref{linearinterpolation}}}. Then, by \eqref{energybound2} and \eqref{energybound3} {{and the fact that $\partial_{t} \overline{\omega}_{N} = \zeta_{N}$ from \eqref{approxdt}}}, we have
\begin{equation}\label{Lipschitzomega}
||\overline{\omega}_{N}||_{W^{1, \infty}(0, (k + 1)\Delta t; L^{2}(\Gamma))} \le C,
\end{equation}
\begin{equation}\label{boundedomega}
||\overline{\omega}_{N}||_{L^{\infty}(0, (k + 1)\Delta t; H_{0}^{2}(\Gamma))} \le C,
\end{equation}
where $C$ depends only on $E_{0}$ and is independent of $N$. Thus, following the method in \cite{MuhaCanic13} {{(see in particular equation (73) in \cite{MuhaCanic13})}}, we obtain by an interpolation inequality that for all $t, t + \tau \in [0, (k + 1)\Delta t]$ with $\tau > 0$, 
\begin{equation}\label{interpolation}
||\overline{\omega}_{N}(t + \tau) - \overline{\omega}_{N}(t)||_{H^{1}(\Gamma)} \le C||\overline{\omega}_{N}(t + \tau) - \overline{\omega}_{N}(t)||^{1/2}_{L^{2}(\Gamma)} ||\overline{\omega}_{N}(t + \tau) - \overline{\omega}_{N}(t)||^{1/2}_{H^{2}(\Gamma)}.
\end{equation}
Here, we used a Sobolev interpolation inequality, see for example Theorem 4.17 (pg.~79) of \cite{Adams}. {{By the Lipschitz continuity of $\overline{\omega}_{N}$ taking values in $L^{2}(\Gamma)$ given by \eqref{Lipschitzomega} and by the boundedness of $\overline{\omega}_{N}$ in $H_{0}^{2}(\Gamma)$ given by \eqref{boundedomega},}}
\begin{equation}\label{holderplate}
||\overline{\omega}_{N}(t + \tau) - \overline{\omega}_{N}(t)||_{H^{1}(\Gamma)} \le C \cdot \tau^{1/2} 
\end{equation}
for a constant $C$ depending only on $E_{0}$ (and in particular, not depending on $k$ or $N$). Therefore, setting $t = 0$ and $\tau = (k + 1)\Delta t$ and using the continuous embedding of $H^{1}(\Gamma)$ into $C(\Gamma)$ {(see e.g. \cite[Chapter V]{Adams} for Sobolev embeddings)},
\begin{equation}\label{estT1}
||\omega^{k + 1}_{N} - \omega_{0}||_{C(\Gamma)} \le C \cdot [(k + 1){{\Delta t}}]^{1/2} \le C \cdot T^{1/2},
\end{equation}
where $C$ depends only $E_{0}$. Because $|\omega_{0}| < R$, we can choose $T > 0$ sufficiently small so that 
\begin{equation}\label{Tchoice1}
C \cdot T^{1/2} < R - ||\omega_{0}||_{C(\Gamma)}.
\end{equation}
This will give the first part of Assumption 1B, which is \eqref{1Anew}. 

\medskip

\noindent {{\textbf{Step 2 (Bound on the trace of the Biot displacements $\bd{\eta}^{n}_{N}$ and the Lagrangian map)}}.} Next, we find a condition on $T$ so that the remaining assumptions \eqref{1B}, \eqref{2B}, and \eqref{2C} are satisfied. 
We do this by controlling the behavior of the structure displacement $\bd{\eta}$. First note that 
\begin{equation*}
||\bd{\eta}^{k + 1}_{N} - \bd{\eta}_{0}||_{L^{2}(\Omega_{b})} \le (\Delta t) \sum_{n = 1}^{k + 1} ||\bd{\dot{\eta}}^{n}_{N}||_{L^{2}(\Omega_{b})} \le C (k + 1)(\Delta t) \le CT,
\end{equation*}
for $C$ depending only on $E_{0}${{, where the first inequality follows from the triangle inequality and the definition of $\displaystyle \dot{\bd{\eta}}^{n}_{N} = \frac{\eta^{n}_{N} - \eta^{n - 1}_{N}}{\Delta t}$, and the second inequality follows from \eqref{energybound1}}}. By the odd extension defined in Definition \ref{extension},
\begin{equation*}
||\bd{\eta}^{k + 1}_{N} - \bd{\eta}_{0}||_{L^{2}(\tilde{\Omega}_{b})} \le C\left(||\bd{\eta}^{k + 1}_{N} - \bd{\eta}_{0}||_{L^{2}(\Omega_{b})} + ||\omega^{k + 1}_{N} - \omega_{0}||_{L^{2}(\Gamma)}\right) \le CT,
\end{equation*}
for a constant $C$ depending only on $E_{0}$, where the estimate $||\omega^{k + 1}_{N} - \omega_{0}||_{L^{2}(\Gamma)} \le CT$ follows from the {{Lipschitz estimate}} \eqref{Lipschitzomega}. By regularization, we then have that for a constant depending only on $\delta$ and $E_{0}$,
\begin{equation*}
||(\bd{\eta}_{N}^{k + 1})^{\delta} - (\bd{\eta}_{0})^{\delta}||_{H^{3}(\Omega_{b})} \le C(\delta, E_{0}) \cdot T.
\end{equation*}
By using the trace theorem and the continuous embedding of $H^{2}(\Gamma)$ into $C(\Gamma)$, we thus conclude that
\begin{equation}\label{estT2}
||(\bd{\eta}_{N}^{k + 1})^{\delta}|_{\Gamma} - (\bd{\eta}_{0})^{\delta}|_{\Gamma}||_{C(\Gamma)} \le C(\delta, E_{0}) \cdot T.
\end{equation}
Since $H^{2}(\Omega_{b})$ embeds continuously into $C(\Omega_{b})$, we also have that 
\begin{equation}\label{estT3}
||\nabla (\bd{\eta}_{N}^{k + 1})^{\delta} - \nabla (\bd{\eta}_{0})^{\delta}||_{C(\Omega_{b})} \le C(\delta, E_{0}) \cdot T.
\end{equation}

Note that $\det(\bd{I} + \bd{A})$ is a continuous function of the entries of $\bd{A}$. Also note that the matrix norms $|\bd{I} + \bd{A}|$ and $|(\bd{I} + \bd{A})^{-1}|$ are continuous functions of the matrix $\bd{A}$. Furthermore, we emphasize that the constant $C(\delta, E_{0})$ depends only on $\delta$ and $E_{0}$ and hence is independent of $k$ and $N$. This dependence on $\delta$ is allowable, since for this existence proof, $\delta$ is an arbitrary but fixed regularization parameter.

Thus, there exists $T$ sufficiently small so that by \eqref{estT2} and \eqref{estT3}, the remaining assumptions \eqref{1B}, \eqref{2B}, and \eqref{2C} are satisfied, since these assumptions are all satisfied for the initial displacement $\bd{\eta}_{0}$. Furthermore, we can choose the constants $c_{0}$, $c_{1}$, $c_{2}$, and $R_{max}$ (defined in the statement of those assumptions) independently of $N$ and $n = 0, 1, ..., N$, because of the fact that the constant $C(\delta, E_{0})$ in our estimates does not depend on $k$ (satisfying $(k + 1)\Delta t \le T$) or $N$. 
\end{proof}

\section{Compactness arguments}\label{compactness}

We next want to pass to the limit in the semidiscrete formulation for the approximate solutions, stated in \eqref{semi1} and \eqref{semi2}. Because this is a nonlinear problem with geometric nonlinearities, we must obtain stronger convergence than just weak and weak* convergence in Proposition \ref{prop:weak}, in order to pass to the limit. To do this, we will use compactness arguments of two types: the classical Aubin-Lions compactness theorem for functions defined on fixed domains, and {\emph{generalized Aubin-Lions compactness arguments}} introduced in \cite{AubinLions} for functions defined on moving domains, see also \cite{MuhaCanic13}. We will first deal with compactness arguments for the plate displacement and the Biot domain displacement. Then, we will deal with compactness arguments for the fluid velocity defined on moving domains.

\subsection{Compactness for Biot poroelastic medium displacement}

We show strong convergence of the Biot structure displacements $\overline{\bd{\eta}}_{N}$ by using a standard Aubin-Lions compactness argument. In particular, we have the following strong convergence result for the Biot medium displacement: 

\begin{lemma}\label{AubinLions}
The following compact embedding holds true
$
W^{1, \infty}(0, T; L^{2}(\Omega_{b})) \cap L^{\infty}(0, T; H^{1}(\Omega_{b})) \subset \subset C(0, T; L^{2}(\Omega_{b})),
$
which implies the existence of a subsequence such that 
$$\bd{\overline{\eta}}_{N} \to \bd{\eta} \ {\rm{strongly\  in}} \  C(0, T; L^{2}(\Omega_{b})).$$
\end{lemma}

\begin{proof}
The compact embedding above is a direct consequence of the standard Aubin-Lions compactness lemma in the case of $p = \infty$, which gives a stronger compact embedding into $C(0, T; L^{2}(\Omega_{b}))$ rather than just $L^{\infty}(0, T; L^{2}(\Omega_{b}))$. The fact that we can find a strongly convergent subsequence follows from this compact embedding, once we recall that $\{\overline{\bd{\eta}}_{N}\}_{N = 1}^{\infty}$ are uniformly bounded in the Banach space $W^{1, \infty}(0, T; L^{2}(\Omega_{b})) \cap L^{\infty}(0, T; H^{1}(\Omega_{b}))$ by the uniform energy estimates. 
\end{proof}

\subsection{Compactness for the plate displacement}
The uniform boundedness of the linear interpolation of the plate displacement $\overline{\omega}_{N}$ 
 in $W^{1, \infty}(0, T; L^{2}(\Gamma))$ and $L^{\infty}(0, T; H_{0}^{2}(\Gamma))$ implies {{strong convergence of $\overline{\omega}_{N}$ in $C(0, T; H^{s}(\Gamma))$. Even though the plate displacements are uniformly bounded in $L^{\infty}(0, T; H_{0}^{2}(\Gamma))$ we only get convergence in $C(0, T; H^{s}(\Gamma))$ for $0<s<2$. This is because we will be using Arzela-Ascoli theorem and hence we will lose regularity due to the compact embedding of $H^{2}(\Gamma)$ into $H^{s}(\Gamma)$ for $0 < s < 2$. The precise statement of the compactness results for the approximate plate displacements is as follows:}}

\begin{proposition}\label{platedisplacement}
Given arbitrary $0 < s < 2$, there exists a subsequence such that the following strong convergences hold:
\begin{equation*}
\overline{\omega}_{N} \to \omega, \qquad \text{ in } C(0, T; H^{s}(\Gamma)),
\end{equation*}
\begin{equation*}
\omega_{N} \to \omega, \qquad \text{ in } L^{\infty}(0, T; H^{s}(\Gamma)). 
\end{equation*}
\end{proposition}

\begin{proof}
Using the same argument as in Step 1 of the proof of Lemma \ref{assumptions}, one can show
the following uniform estimate for
the linear interpolations $\overline{\omega}_{N}$ and  $\tau > 0$, $t, t + \tau \in [0, T]$:
{{
\begin{equation}
\begin{array}{rl}\label{holderest8.1}
||\overline{\omega}_{N}(t + \tau) - \overline{\omega}_{N}(t)||_{H^{2\alpha}(\Gamma)} &\le ||\overline{\omega}_{N}(t + \tau) - \overline{\omega}_{N}(t)||^{1 - \alpha}_{L^{2}(\Gamma)} ||\overline{\omega}_{N}(t + \tau) - \overline{\omega}_{N}(t)||^{\alpha}_{H^{2}(\Gamma)} \\
&\le C\tau^{1 - \alpha}, \qquad \text{ for } 0 < \alpha < 1,
\end{array}
\end{equation}}}
where the constant $C$ is independent of $N$, but can depend on the choice of $\alpha$. {{The first inequality in \eqref{holderest8.1} follows from an interpolation estimate for Sobolev spaces (see Theorem 4.17, pg.~79 of \cite{Adams}) and as in the proof of estimate \eqref{interpolation} from Lemma \ref{assumptions}, the second inequality follows from the uniform Lipschitz estimate \eqref{Lipschitzomega} and the uniform boundedness estimate \eqref{boundedomega}.}}
Because {{the constant $C$ in \eqref{holderest8.1}}} is independent of $N$, {{this}} estimate implies that for a given arbitrary $\alpha \in (0, 1)$, the functions $\overline{\omega}_{N}$ are uniformly bounded as functions in $C^{0, 1 - \alpha}(0, T; H^{2\alpha}(\Gamma))$. Hence, the strong convergence of $\overline{\omega}_{N}$ follows directly from the Arzela-Ascoli theorem and the fact that $H^{2\alpha}$ embeds compactly into any $H^{2\alpha - \epsilon}$ for $\epsilon > 0$, once we choose $\alpha \in (0, 1)$ and $\epsilon > 0$ appropriately so that $2\alpha - \epsilon = s$ for a given arbitrary $0 < s < 2$. Hence, we obtain the desired strong convergence, as the equicontinuity condition for the Arzela-Ascoli theorem follows from the above estimate. 

To show a similar strong convergence result for $\omega_{N}$, we must show that
\begin{equation*}
||\omega_{N}(t) - \overline{\omega}_{N}(t)||_{L^{\infty}(0, T; H^{s}(\Gamma))} \to 0,
\end{equation*}
for arbitrary $0 < s < 2$. Once we observe that $\overline{\omega}_{N}(n\Delta t) = \omega_{N}(t)$ for $n\Delta t \le t < (n + 1)\Delta t$, this follows immediately from the above {{H\"{o}lder continuity estimate \eqref{holderest8.1}}}, as
\begin{equation*}
||\omega_{N}(t) - \overline{\omega}_{N}(t)||_{L^{\infty}(0, T; H^{s}(\Gamma))} \le C(\Delta t)^{1 - \frac{s}{2}} \to 0, \qquad \text{ as } N \to \infty.
\end{equation*}
Thus, $\omega_{N}$ and $\overline{\omega}_{N}$ have the same limit in $L^{\infty}(0, T; H^{s}(\Gamma))$ for $0 < s < 2$. 

\end{proof}

Next, we will obtain compactness for the Biot velocity, plate velocity, pore pressure, and fluid velocity. Because the test space \eqref{semitestspace} has the pore pressure and fluid velocity decoupled from the Biot/plate velocity, we can handle the compactness argument for each of these quantities separately.
In particular, we recall the definition of the discrete test space from \eqref{semitestspace}
and note that we can decouple this test space into three smaller test spaces, one for the Biot/plate displacement/velocity, one for the pore pressure, and one for the fluid velocity. In the next section we show compactness results for the Biot velocity and plate velocity, which must be treated together since they are coupled by a kinematic coupling condition at the plate interface $\Gamma$.

\subsection{Compactness for the Biot velocity and plate velocity}

{{Here, we will state and prove a  compactness result for the Biot and plate velocities $(\bd{\xi}_{N}, \zeta_{N})$, by showing the existence of convergent subsequences that converge in $L^{2}(0, T; H^{-s}(\Omega_{b}) \times H^{-s}(\Gamma))$ for $-1/2 < s < 0$. We remark that we must consider negative spatial Sobolev spaces for the Biot/plate velocities for the following two reasons:
\begin{itemize}
\item First, our existence result in Theorem \ref{MainThm1} includes the purely elastic case in which the viscoelasticity coefficients $\mu_{v}, \lambda_{v}$ are allowed to be zero. Hence, we can only expect the Biot velocities in the finite-energy spaces to have spatial regularity of at most $L^{2}(\Omega_{b})$. 
\item {{Second, for the plate velocities $\zeta_{N}$, we must consider negative spatial Sobolev spaces on $\Gamma$ since by the coupling conditions \eqref{mass} and \eqref{bjs}, it is {\bf{not true}} that the plate velocities {{$\zeta_{N}$}} {{are equal to the}} traces of the fluid velocities 
$\bd{u}_{N} \in H^{1}(\Omega_{f})$, which is typically the case in FSI with purely elastic structures and no-slip condition. Therefore, we do not get any higher regularity of the plate velocities than what we get from the finite energy spaces, which implies that
 the plate velocities $\zeta_{N}$ are only at most $L^{2}(\Gamma)$.}}
\end{itemize}
The main compactness result for the Biot/plate velocities is as follows:}}

\begin{theorem}\label{compactvelocities}
For $-1/2 < s < 0$, there exists a subsequence such that
\begin{equation*}
(\bd{\xi}_{N}, \zeta_{N}) \to (\bd{\xi}, \zeta) \  \text{ strongly in } L^{2}(0, T; H^{-s}(\Omega_{b}) \times H^{-s}(\Gamma)).
\end{equation*}
\end{theorem}

\begin{proof}
We will establish this result by using a compactness criterion for piecewise constant functions due to Dreher and J\"{u}ngel \cite{DreherJungel}. To simplify arguments, we define a slightly more regular Biot/plate velocity test space:
\begin{equation}\label{Qv}
\mathcal{Q}_{v} = \{(\bd{\psi}, \varphi) \in (V_{d} \cap H^{2}(\Omega_{b})) \times H_{0}^{2}(\Gamma) : \bd{\psi} = \varphi \bd{e}_{y} \text{ on } \Gamma\}.
\end{equation}
We will use the following chain of embeddings
\begin{equation*}
L^{2}(\Omega_{b}) \times L^{2}(\Gamma) \subset \subset H^{-s}(\Omega_{b}) \times H^{-s}(\Gamma) \subset \mathcal{Q}_{v}',
\end{equation*}
where the first embedding is compact, as {{required for the}} Dreher-J\"{u}ngel compactness criterion \cite{DreherJungel}.

Let $\tau_{\Delta t}$ denote the time shift $\tau_{\Delta t} f(t, \cdot) = f(t - \Delta t, \cdot)$ for a function $f$ defined on $[0, T]$. 
As required by the Dreher-J\"{u}ngel compactness criterion \cite{DreherJungel}, to obtain compactness we must verify that the following inequality is satisfied for 
a uniform constant 
$C$ and for all $\Delta t = T/N$:
\begin{equation}\label{DreherJungelcondition}
\left|\left|\frac{\tau_{\Delta t}(\bd{\xi}_{N}, \zeta_{N}) - (\bd{\xi}_{N}, \zeta_{N})}{\Delta t}\right|\right|_{L^{1}({{\Delta t}}, T; \mathcal{Q}_{v}')} + ||(\bd{\xi}_{N}, \zeta_{N})||_{L^{\infty}(0, T; L^{2}(\Omega_{b}) \times L^{2}(\Gamma))} \le C.
\end{equation}

The second term in this inequality is uniformly bounded by  Lemma \ref{uniform},  which gives exactly the uniform boundednenss of 
$(\bd{\xi}_{N}, \zeta_{N})$ in $L^{\infty}(0, T; L^{2}(\Omega_{b}) \times L^{2}(\Gamma))$.

To deal with the first term in \eqref{DreherJungelcondition} we use the coupled semidiscrete formulation \eqref{semi1}, \eqref{semi2} and
set the test functions $\bd{v}$ and $r$ for the fluid velocity and Biot pore pressure to be zero because  we are considering only the Biot and plate velocities.
 We obtain that for all test functions $(\bd{\psi}, \varphi) \in \mathcal{Q}_{v}$, where $\mathcal{Q}_{v}$ is defined in \eqref{Qv}, the following holds:
\begin{align*}
&\rho_{b} \int_{\Omega_{b}} \left(\frac{\bd{\xi}^{n + 1}_{N} - \bd{\xi}^{n}_{N}}{\Delta t}\right) \cdot \bd{\psi} + \rho_{p} \int_{\Gamma} \left(\frac{\zeta^{n + 1}_{N} - \zeta^{n}_{N}}{\Delta t}\right) \cdot \varphi \\
&= -\int_{\Gamma} \left(\frac{1}{2} \bd{u}^{n + 1}_{N} \cdot \bd{u}^{n}_{N} - p^{n + 1}_{N}\right) (\bd{\psi} \cdot \bd{n}^{\omega^{n}_{N}}) - \int_{\Gamma} \frac{\beta}{\mathcal{J}^{\omega^{n}_{N}}_{\Gamma}} (\zeta^{n + 1}_{N} \bd{e}_{y} - \bd{u}^{n + 1}_{N}) \cdot \bd{\tau}^{\omega^{n}_{N}} (\bd{\psi} \cdot \bd{\tau}^{\omega^{n}_{N}}) \\
&- 2\mu_{e}\int_{\Omega_{b}} \bd{D}(\bd{\eta}^{n + 1}_{N}) : \bd{D}(\bd{\psi}) - \lambda_{e} \int_{\Omega_{b}} (\nabla \cdot \bd{\eta}^{n + 1}_{N}) (\nabla \cdot \bd{\psi}) - 2\mu_{v} \int_{\Omega_{b}} \bd{D}(\bd{\xi}^{n + 1}_{N}) : \bd{D}(\bd{\psi}) \\
&- \lambda_{v} \int_{\Omega_{b}} (\nabla \cdot \bd{\xi}^{n + 1}_{N}) (\nabla \cdot \bd{\psi}) + \alpha \int_{\Omega_{b}} \mathcal{J}^{(\eta^{n}_{N})^{\delta}}_{b} p^{n + 1}_{N} \nabla^{(\eta^{n}_{N})^{\delta}}_{b} \cdot \bd{\psi} - \int_{\Gamma} \Delta \omega^{n + \frac{1}{2}}_{N} \cdot \Delta \varphi.
\end{align*}
The estimate for the first term in \eqref{DreherJungelcondition} will follow if we can estimate the {{right-hand side}} in terms of the ${\mathcal{Q}_{v}'}$ norm.
For this purpose consider an arbitrary $||(\bd{\psi}, \varphi)||_{\mathcal{Q}_{v}} \le 1$, so that $||\bd{\psi}||_{H^{2}(\Omega_{b})} \le 1$ and $||\varphi||_{H_{0}^{2}(\Gamma)} \le 1$. By the uniform estimates in Lemma \ref{uniform} and the regularity of the test functions in \eqref{Qv}, it is clear that the terms on the right hand side are all uniformly bounded by a constant $C$, independent of $||(\bd{\psi}, \varphi)||_{\mathcal{Q}_{v}} \le 1$, 
except possibly the term 
\begin{equation*}
\displaystyle \alpha \int_{\Omega_{b}} \mathcal{J}^{(\eta^{n}_{N})^{\delta}}_{b} p^{n + 1}_{N} \nabla^{(\eta^{n}_{N})^{\delta}}_{b} \cdot \bd{\psi}.
\end{equation*}
To estimate this term we recall the definitions
\begin{equation*}
\mathcal{J}^{(\eta^{n}_{N})^{\delta}}_{b} = \det(\bd{I} + \nabla (\bd{\eta}^{n}_{N})^{\delta}), \qquad \nabla^{(\eta^{n}_{N})^{\delta}} \cdot \bd{\psi} = \text{tr}\left[\nabla \bd{\psi} \cdot (I + \nabla (\bd{\eta}^{n}_{N})^{\delta})^{-1}\right].
\end{equation*}
By assumption 2C \eqref{2C} and the fact that $||\bd{\psi}||_{H^{1}(\Omega_{b})} \le 1$, we have that $||\nabla^{(\eta^{n}_{N})^{\delta}} \cdot \bd{\psi}||_{L^{2}(\Omega_{b})}$ is uniformly bounded, while by the boundedness of $\bd{\eta}^{n}_{N}$ in $H^{1}(\Omega_{b})$, we have that $|\mathcal{J}^{(\eta^{n}_{N})^{\delta}}_{b}| \le C$. Therefore, using the fact that $p_{N}$ is uniformly bounded in $L^{\infty}(0, T; L^{2}(\Omega_{b}))$, we obtain the desired estimate
\begin{equation*}
\left|\alpha \int_{\Omega_{b}} \mathcal{J}^{(\eta^{n}_{N})^{\delta}}_{b} p^{n + 1}_{N} \nabla^{(\eta^{n}_{N})^{\delta}}_{b} \cdot \bd{\psi}\right| \le C.
\end{equation*}
Finally, we conclude that
\begin{equation*}
\left|\left|\frac{(\bd{\xi}^{n + 1}_{N}, \zeta^{n + 1}_{N}) - (\bd{\xi}^{n}_{N}, \zeta^{n}_{N})}{\Delta t}\right|\right|_{\mathcal{Q}_{v}'} \le C, \  \text{ for a constant $C$ that is independent of $n$ and $N$},
\end{equation*}
and since 
\begin{equation*}
\sum_{n = 1}^{N - 1} (\Delta t) \left|\left|\frac{(\bd{\xi}^{n + 1}_{N}, \zeta^{n + 1}_{N}) - (\bd{\xi}^{n}_{N}, \zeta^{n}_{N})}{\Delta t}\right|\right|_{\mathcal{Q}_{v}'} \\
\le (\Delta t) \sum_{n = 1}^{N - 1} C \le CT,
\end{equation*}
we conclude that \eqref{DreherJungelcondition} holds for a uniform constant $C$. This establishes the desired result. 

\end{proof}

\subsection{Compactness for the pore pressure}

\begin{theorem}
There exists a subsequence such that
\begin{equation*}
p_{N} \to p \  \text{strongly in $L^{2}(0, T; L^{2}(\Omega_{b}))$}. 
\end{equation*}
\end{theorem}

\begin{proof}
The proof is based on a similar application of the Dreher-J\"{u}ngel compactness criterion for piecewise constant functions \cite{DreherJungel} as in the previous compactness result. 
We first observe that we have the following chain of embeddings {{$H^{1}(\Omega_{b}) \subset\subset L^{2}(\Omega_{b}) \subset (V_{p} \cap H^{2}(\Omega_{b}))'$}},
and so by the Dreher-J\"{u}ngel compactness criterion \cite{DreherJungel} it suffices to show that the following inequality holds for a constant $C$ independent of $N$:
\begin{equation}\label{DreherJungelpressure}
\left|\left|\frac{\tau_{\Delta t}p_{N} - p_{N}}{\Delta t} \right|\right|_{L^{1}(\Delta t, T; (V_{p} \cap H^{2}(\Omega_{b}))')} + ||p_{N}||_{L^{2}(0, T; H^{1}(\Omega_{b}))} \le C.
\end{equation}

To obtain this estimate, 
we observe that the approximate solutions for the pore pressure satisfy the following weak formulation for all test functions $r \in V_{p}$, where $V_{p}$ is defined by \eqref{Vp}:
\begin{align*}\label{weakpore}
&c_{0} \int_{\Omega_{b}} \left(\frac{p^{n + 1}_{N} - p^{n}_{N}}{\Delta t}\right) \cdot r - \alpha \int_{\Omega_{b}} \mathcal{J}^{(\eta^{n}_{N})^{\delta}}_{b} \dot{\bd{\eta}}^{n + 1}_{N} \cdot \nabla^{(\eta^{n}_{N})^{\delta}}_{b} r - \alpha \int_{\Gamma} (\dot{\bd{\eta}}^{n + 1}_{N} \cdot \bd{n}^{(\omega^{n}_{N})^{\delta}}) r \\
&+ \kappa \int_{\Omega_{b}} \mathcal{J}^{(\eta^{n}_{N})^{\delta}}_{b} \nabla^{(\eta^{n}_{N})^{\delta}}_{b} p^{n + 1}_{N} \cdot \nabla^{(\eta^{n}_{N})^{\delta}}_{b} r - \int_{\Gamma} [(\bd{u}^{n + 1}_{N} - \dot{\bd{\eta}}^{n + 1}_{N}) \cdot \bd{n}^{\omega^{n}_{N}}] r = 0.
\end{align*}
We use more regularity for the test space $V_{p} \cap H^{2}(\Omega_{b})$ to make the following estimates simpler. 
We compute that for any $r \in V_{p} \cap H^{2}(\Omega_{b})$ we have
\begin{align*}
&c_{0} \int_{\Omega_{b}} \left(\frac{p^{n + 1}_{N} - p^{n}_{N}}{\Delta t}\right) \cdot r = \alpha \int_{\Omega_{b}} \mathcal{J}^{(\eta^{n}_{N})^{\delta}}_{b} \bd{\xi}^{n + 1}_{N} \cdot \nabla^{(\eta^{n}_{N})^{\delta}}_{b} r + \alpha \int_{\Gamma} (\zeta^{n + 1}_{N} \bd{e}_{y} \cdot \bd{n}^{(\omega^{n}_{N})^{\delta}}) r \\
&- \kappa \int_{\Omega_{b}} \mathcal{J}^{(\eta^{n}_{N})^{\delta}}_{b} \nabla^{(\eta^{n}_{N})^{\delta}}_{b} p^{n + 1}_{N} \cdot \nabla^{(\eta^{n}_{N})^{\delta}}_{b} r + \int_{\Gamma} [(\bd{u}^{n + 1}_{N} - \zeta^{n + 1}_{N} \bd{e}_{y}) \cdot \bd{n}^{\omega^{n}_{N}}] r.
\end{align*}
We estimate the right hand side for $||r||_{V_{p} \cap H^{2}(\Omega_{b})} \le 1$. Recall that $\mathcal{J}^{(\eta^{n}_{N})^{\delta}}_{b} = \det(\bd{I} + \nabla (\bd{\eta}^{n}_{N})^{\delta})$, 
\begin{equation*}
\nabla^{(\eta^{n}_{N})^{\delta}}_{b} r = \left(\frac{\partial r}{\partial \tilde{x}}, \frac{\partial r}{\partial \tilde{y}}\right) \cdot (\bd{I} + \nabla (\bd{\eta}^{n}_{N})^{\delta})^{-1}, \quad \text{ and } \quad \nabla^{(\eta^{n}_{N})^{\delta}}_{b} p^{n + 1}_{N} = \left(\frac{\partial p^{n + 1}_{N}}{\partial \tilde{x}}, \frac{\partial p^{n + 1}_{N}}{\partial \tilde{y}}\right) \cdot (\bd{I} + \nabla (\bd{\eta}^{n}_{N})^{\delta})^{-1}.
\end{equation*}
We have by Assumption 2C \eqref{2C} that $|(\bd{I} + \nabla(\bd{\eta}^{n}_{N})^{\delta})^{-1}|$ is uniformly bounded, and furthermore, $\mathcal{J}^{(\eta^{n}_{N})^{\delta}}_{b}$ is positive and bounded above. By combining these facts with standard estimates we obtain that 
\begin{equation*}
\left|\left|\frac{p^{n + 1}_{N} - p^{n}_{N}}{\Delta t}\right|\right|_{(V_{p} \cap H^{2}(\Omega_{b}))'} \le C \  \text{ for a constant $C$ that is independent of $n$ and $N$.}
\end{equation*}
Combining this with the fact that $p_{N}$ is uniformly bounded in $L^{2}(0, T; H^{1}(\Omega_{b}))$ gives the desired estimate  \eqref{DreherJungelpressure}.
\end{proof}

\subsection{Compactness for the fluid velocity}

We will obtain convergence of the fluid velocity along a subsequence by using a {\emph{generalized Aubin-Lions compactness theorem}} for functions defined on moving domains, stated as Theorem 3.1 in \cite{AubinLions}. {{To help the reader, we state Theorem 3.1 from \cite{AubinLions} at the end of this manuscript, in the appendix, Section~\ref{appendix2}.}} The reason we must use the {generalized Aubin-Lions compactness theorem} is that the approximate fluid velocities are defined on different time-dependent fluid domains. To prepare for an application of the generalized Aubin-Lions compactness argument we will 
{\emph{map our approximate fluid problem back onto 
the physical domain}} 
\begin{equation*}
\Omega^{n}_{f, N} = \{(x, y) \in \mathbb{R}^{2} : 0 \le x \le L, -R \le y \le \omega^{n}_{N}(x)\},
\end{equation*}
where we redefine the fluid velocity solution and  test spaces as follows:
\begin{equation}\label{VnN}%and \label{QnN}
V^{n + 1}_{N} = \{\bd{u} \in H^{1}(\Omega^{n}_{f, N}) : \nabla \cdot \bd{u} = 0 \text{ on } \Omega^{n}_{f, N}, \bd{u} = 0 \text{ on } \partial \Omega^{n}_{f, N} \setminus \Gamma^{n}_{N}\}, \quad Q^{n}_{N} = V^{n + 1}_{N} \cap H^{3}(\Omega^{n}_{f, N}).
\end{equation}
The approximate fluid velocity  $\bd{u}^{n + 1}_{N} \in V^{n + 1}_{N}$ on the physical domain satisfies the following semidiscrete formulation:
\begin{align}\label{semifluid}
&\int_{\Omega^{n}_{f, N}} \frac{\bd{u}^{n + 1}_{N} - \tilde{\bd{u}}^{n}_{N}}{\Delta t} \cdot \boldsymbol{v} + 2\nu \int_{\Omega^{n}_{f, N}} \boldsymbol{D}(\boldsymbol{u}_{N}^{n + 1}) : \boldsymbol{D}(\boldsymbol{v}) 
\nonumber
\\
&+ \frac{1}{2} \int_{\Omega^{n}_{f, N}} \left[\left(\left(\tilde{\boldsymbol{u}}^{n}_{N} - \zeta^{n + \frac{1}{2}}_{N} \frac{R + y}{R + \omega^{n}_{N}} \boldsymbol{e}_{y}\right) \cdot \nabla \boldsymbol{u}^{n + 1}_{N}\right) \cdot \boldsymbol{v} - \left(\left(\tilde{\boldsymbol{u}}^{n}_{N} - \zeta^{n + \frac{1}{2}}_{N} \frac{R + y}{R + \omega^{n}_{N}} \boldsymbol{e}_{y}\right) \cdot \nabla \boldsymbol{v}\right) \cdot \boldsymbol{u}^{n + 1}_{N}\right] 
\nonumber
\\
&+ \frac{1}{2R} \int_{\Omega^{n}_{f, N}} \frac{R}{R + \omega^{n}_{N}} \zeta^{n + \frac{1}{2}}_{N} \boldsymbol{u}^{n + 1}_{N} \cdot \boldsymbol{v} + \frac{1}{2} \int_{\Gamma^{n}_{N}} (\boldsymbol{u}^{n + 1}_{N} - \boldsymbol{\dot{\eta}}^{n + 1}_{N}) \cdot \boldsymbol{n} (\tilde{\boldsymbol{u}}^{n}_{N} \cdot \boldsymbol{v}) 
\nonumber
\\
&- \int_{\Gamma^{n}_{N}} \left(\frac{1}{2}\boldsymbol{u}^{n + 1}_{N} \cdot \tilde{\boldsymbol{u}}^{n}_{N} - p^{n + 1}_{N}\right) (\boldsymbol{v} \cdot \boldsymbol{n}) - \beta \int_{\Gamma^{n}_{N}} (\boldsymbol{\dot{\eta}}^{n + 1}_{N} - \boldsymbol{u}^{n + 1}_{N}) \cdot \boldsymbol{\tau} (\bd{v} \cdot \boldsymbol{\tau}) = 0,
\ \forall \bd{v} \in Q^{n}_{N}
\end{align}
where 
{{
\begin{equation}\label{tildeudef}
\tilde{\bd{u}}^{n}_{N} = \bd{u}^{n}_{N} \circ \bd{\Phi}^{\omega^{n - 1}_{N}}_{f} \circ (\bd{\Phi}^{\omega^{n}_{N}}_{f})^{-1},
\end{equation}}}
 $\bd{u}^{n}_{N}$ is originally defined on $\Omega^{n - 1}_{f, N}$, and 
the ALE map $\bd{\Phi}^{\omega^{n}_{N}}_{f}: \Omega_{f} \to \Omega^{n}_{f, N}$ is defined by \eqref{phif}.

To be able to compare functions on different physical domains we introduce a maximal domain $\Omega^{M}_{f}$ which contains all the physical domains. 
The existence of such a domain, and the extensions of the velocity functions onto the maximal domain are discussed next.

\subsubsection{Extension to maximal domain}

We consider the following {\emph{maximal fluid domain}} which contains all the physical fluid domains:
%The result in Proposition \ref{extension} allows us to define a maximal domain $\Omega^{M}_{f}$ defined by the function $M(x)$, containing all of the physical approximate fluid domains $\Omega^{n}_{f, N}$:
\begin{equation*}
\Omega^{M}_{f} = \{(x, y) \in \mathbb{R}^{2} : 0 \le x \le L, -R \le y \le M(x)\},
\end{equation*}
where the function $M(x)$ is obtained from the following proposition,  established in Lemma 2.5 in \cite{CanicLectureNotes} and Lemma 4.5 in \cite{AubinLions}
in the context of fluid-structure interaction between an incompressible viscous Newtonian fluid and an elastic Koiter shell:
\begin{proposition}
There exists smooth functions $m(x)$ and $M(x)$ defined on $\Gamma = [0, L]$, {{satisfying $m(0) = m(L) = M(0) = M(L) = 0$}}, such that 
\begin{equation*}
m(x) \le \omega^{n}_{N}(x) \le M(x), \qquad \text{ for all } x \in [0, L], N, \text{ and } n = 0, 1, ..., N.
\end{equation*}
Furthermore, there exist smooth functions $m^{n, l}_{N}(x)$ and $M^{n, l}_{N}(x)$ defined for positive integers $N$, $n = 0, 1, ..., N - 1$ and $l = 0, 1, ..., N - n$, such that
\begin{enumerate}
\item $m^{n, l}_{N}(x) \le \omega^{n + i}_{N}(x) \le M^{n, l}_{N}(x), \qquad$ for all $x \in [0, L]$ and $i = 0, 1, ..., l$.
\item $M^{n, l}_{N}(x) - m^{n, l}_{N}(x) \le C\sqrt{l\Delta t}, \qquad$ for all $x \in [0, L]$.
\item $||M^{n, l}_{N}(x) - m^{n, l}_{N}(x)||_{L^{2}(\Gamma)} \le C(l\Delta t)$,
\end{enumerate}
where $C$ is independent of $n$, $l$, and $N$. Finally, the functions $M^{n, l}_{N}(x)$ and $m^{n, l}_{N}(x)$ for all $n$, $l$, and $N$, are Lipschitz continuous with a Lipschitz constant that is uniformly bounded above by some constant $L > 0$ independent of $n$, $l$, and $N$. 
\end{proposition}

\if 1 = 0

\begin{proof}
The existence of the functions $m(x)$ and $M(x)$ is immediately guaranteed by the inequality \eqref{estT1}, and they can just be taken to be constant functions on $\Gamma$. 

The construction of the functions $m^{n, l}_{N}(x)$ and $M^{n, l}_{N}(x)$ is more involved. Let us define the following functions:
\begin{equation*}
q^{n, l}_{N}(x) = \min_{n \le k \le n + l} \omega^{k}_{N}(x), \qquad Q^{n, l}_{N}(x) = \max_{n \le k \le n + l} \omega^{k}_{N}(x),
\end{equation*}
and
\begin{equation*}
r^{n, l}_{N}(x) = \min_{\substack{n \le k \le n + l, \\ \zeta \in [0, L] \cap [x - l\Delta t, x + l\Delta t]}} \omega^{k}_{N}(\zeta), \qquad R^{n, l}_{N}(x) = \max_{\substack{n \le k \le n + l, \\ \zeta \in [0, L] \cap [z - l\Delta t, z + l\Delta t]}} \omega^{k}_{N}(\zeta).
\end{equation*}

Consider $t, t + \tau \in [0, T]$, with $\tau > 0$. Using the interpolation inequality \eqref{interpolation}, we can prove that
\begin{equation*}
||\overline{\omega}_{N}(t + \tau) - \overline{\omega}_{N}(t)||_{H^{1}(\Gamma)} \le C\tau^{1/2},
\end{equation*}
for a constant $C$ independent of $t$ and $\tau$. Thus, using the continuous embedding of $H^{1}(\Gamma)$ into $C(\Gamma)$, we conclude that
\begin{equation*}
Q^{n, l}_{N}(x) - q^{n, l}_{N}(x) \le C\sqrt{l\Delta t}, \qquad \text{ for all } x \in [0, L],
\end{equation*}
for a constant $C$ that is independent of $n$, $l$, and $N$.  

Next, we want to show a corresponding inequality for the functions $r^{n, l}_{N}(x)$ and $R^{n, l}_{N}(x)$. To do this, we note that by the uniform boundedness of the approximate functions $\omega_{N}$ in $L^{\infty}(0, T; H^{2}(\Gamma))$, we have that the structure displacements $\omega^{n}_{N}$ are all uniformly Lipschitz with a Lipschitz constant uniformly bounded above by $\alpha$. Therefore, we conclude that
\begin{equation}\label{QRineq}
Q^{n, l}_{N}(x) \le R^{n, l}_{N}(x) \le Q^{n, l}_{N}(x) + \alpha (l\Delta t), \qquad q^{n, l}_{N}(x) - \alpha (l \Delta t) \le r^{n, l}_{N}(x) \le q^{n, l}_{N}(x), \qquad \text{ for all } x \in [0, L].
\end{equation}
Therefore, since $\alpha l\Delta t \le C\sqrt{l\Delta t}$ for a constant $C$ independent of $n$, $l$, and $N$ (since $l\Delta t \le T$), we obtain that
\begin{equation*}
R^{n, l}_{N}(x) - r^{n, l}_{N}(x) \le C\sqrt{l\Delta t}, \qquad \text{ for all } x \in [0, L]. 
\end{equation*}

Next, we claim that
\begin{equation*}
||R^{n, l}_{N}(x) - r^{n, l}_{N}(x)||_{L^{2}(\Gamma)} \le C(l\Delta t).
\end{equation*}
To see this, first observe that for all $z \in [0, L]$,
\begin{equation*}
|Q^{n, l}_{N}(x) - q^{n, l}_{N}(x)| \le \sum_{k = n}^{n + l - 1} |\omega^{k + 1}_{N}(x) - \omega^{k}_{N}(x)|.
\end{equation*}
Hence, by the uniform boundedness estimates,
\begin{equation*}
||Q^{n, l}_{N}(x) - q^{n, l}_{N}(x)||_{L^{2}(\Gamma)} \le \sum_{k = n}^{n + l - 1} ||\omega^{k + 1}_{N}(x) - \omega^{k}_{N}(x)||_{L^{2}(\Gamma)} = \Delta t \sum_{k = n}^{n + l - 1} ||v^{k + \frac{1}{2}}_{N}||_{L^{2}(\Gamma)} \le C(l\Delta t).
\end{equation*}
By \eqref{QRineq}, we have that 
\begin{equation*}
|R^{n, l}_{N}(x) - r^{n, l}_{N}(x)| \le |Q^{n, l}_{N}(x) - q^{n, l}_{N}(x)| + 2\alpha (l\Delta t),
\end{equation*}
and hence,
\begin{equation*}
||R^{n, l}_{N}(x) - r^{n, l}_{N}(x)||_{L^{2}(\Gamma)} \le ||Q^{n, l}(x) - q^{n, l}_{N}(x)||_{L^{2}(\Gamma)} + C(l\Delta t).
\end{equation*}
Thus, we conclude that for a constant $C$ independent of $n$, $l$, and $N$,
\begin{equation*}
||R^{n, l}_{N}(x) - r^{n, l}_{N}(x)||_{L^{2}(\Gamma)} \le C(l\Delta t).
\end{equation*}

This would complete the proof, except that $R^{n, l}_{N}(x)$ and $r^{n, l}_{N}(x)$ are not necessarily smooth functions on $\Gamma$. Thus, we will convolve with a smooth compactly supported function, but we must convolve in such a way so that the resulting smooth functions still bound the functions $\omega^{n}_{N}(x)$ appropriately. Let $\varphi$ be a smooth, symmetric, compactly supported function on $\mathbb{R}$ with support in $[-1, 1]$. Then, we define
\begin{equation*}
\varphi_{\epsilon}(x) = \epsilon^{-1} \varphi(\epsilon^{-1}x).
\end{equation*}
To have the convolution defined on all of $[0, L]$, we extend $R^{n, l}_{N}(x)$ and $r^{n, l}_{N}(x)$ to the larger domain $[-L, 2L]$ by using an even reflection at both ends $x = 0$ and $x = L$. We will then define, for $x \in [0, L]$, 
\begin{equation*}
M^{n, l}_{N}(x) = \varphi_{l\Delta t} * R^{n, l}_{N}(x), \qquad m^{n, l}_{N}(x) = \varphi_{l\Delta t} * r^{n, l}_{N}(x).
\end{equation*}
Because $||\phi_{\epsilon}||_{L^{1}(\Gamma)} = 1$ for all $\epsilon$, by Young's convolution inequality, we conclude that
\begin{equation*}
M^{n, l}_{N}(x) - m^{n, l}_{N}(x) \le C\sqrt{l\Delta t}, \qquad \text{ and } \qquad ||M^{n, l}_{N}(x) - m^{n, l}_{N}(x)||_{L^{2}(\Gamma)} \le C(l\Delta t),
\end{equation*}
for a constant $C$ independent of $n$, $l$, and $N$. Furthermore, since the definitions of $R^{n, l}_{N}(x)$ and $r^{n, l}_{N}(x)$ contain maxima/minima of $\omega^{k}_{N}(\zeta)$, $k = n, n + 1, ..., n + l$, over $\zeta \in [0, L] \cap [x - l\Delta t, x + l\Delta t]$, it follows by construction that $M^{n, l}_{N} \ge \omega^{n + i}_{N} \ge m^{n, l}_{N}$ for all $i = 0, 1, ..., l$. 

Finally, we show the uniform Lipschitz property of $M^{n, l}_{N}(x)$, noting that the same argument will work for $m^{n, l}_{N}(x)$. We recall that by the uniform boundedness of $\omega^{n}_{N}$ in $H_{0}^{2}(\Gamma)$, the structure displacements $\omega^{n}_{N}$ are uniformly Lipschitz with Lipschitz constant less than or equal to some $L > 0$. We can then conclude that the functions $Q^{n, l}_{N}$ are uniformly Lipschitz with Lipschitz constant less than or equal to $L$ also. 

We next transfer this uniform Lipschitz estimate to the function $R^{n, l}_{N}(x)$. By definition,
\begin{equation*}
R^{n, l}_{N}(x) = \max_{\zeta \in [0, L] \cap [x - l\Delta t, x + l\Delta t]} Q^{n, l}_{N}(\zeta).
\end{equation*}
To show the Lipschitz property, consider $x_{1}, x_{2} \in [0, L]$. Without loss of generality, suppose that $x_{1} < x_{2}$. There exists $x_{1}^{*}$ in $[0, L] \cap [x_{1} - l\Delta t, x_{1} + l\Delta t]$ such that
\begin{equation*}
R^{n, l}_{N}(x_{1}) = Q^{n, l}_{N}(x_{1}^{*}).
\end{equation*}
Then, there exists a point $x_{2}^{*} \in [0, L] \cap [x_{2} - l\Delta t, x_{2} + l\Delta t]$ such that 
\begin{equation*}
|x_{2}^{*} - x_{1}^{*}| \le |x_{2} - x_{1}|.
\end{equation*}
Thus, 
\begin{equation*}
|Q^{n, l}_{N}(x_{2}^{*}) - Q^{n, l}_{N}(x_{1}^{*})| \le L|x_{2} - x_{1}|,
\end{equation*}
and hence, by the definition of the function $R^{n, l}_{N}(x)$, 
\begin{equation*}
R^{n, l}_{N}(x_{2}) \ge R^{n, l}_{N}(x_{1}) - L|x_{2} - x_{1}|.
\end{equation*}
By interchanging the roles of $x_{1}$ and $x_{2}$, we then obtain
\begin{equation*}
|R^{n, l}_{N}(x_{1}) - R^{n, l}_{N}(x_{2})| \le L|x_{2} - x_{1}|.
\end{equation*}
So $R^{n, l}_{N}$ are all uniformly Lipschitz with Lipschitz constant bounded above by $L$. Finally, the same property holds for $M^{n, l}_{N}(x)$, since the Lipschitz constant bound is not changed by convolution with any $\varphi_{\epsilon}$.

\end{proof}

\fi

Once the maximal fluid domain is defined, we can extend the fluid velocities $\bd{u}^{n}_{N}$ from $\Omega^{n}_{f, N}$ to this common maximal domain $\Omega^{M}_{f}$, using extensions by zero in $\Omega^{M}_{f} \cap (\Omega^{n}_{f, N})^{c}$. 
 Notice that since  $\omega^{n}_{N}(x)$ are all uniformly Lipschitz,
 the {extensions by zero} of the $H^{1}$ functions 
$\bd{u}^{n}_{N}$ defined on Lipschitz domains to $\Omega^{M}_{f}$ are uniformly bounded in $H^{s}(\Omega^{M}_{f})$ for all $s$ such that $0 < s < 1/2$. 
{{Indeed, we have the following lemma, which follows from Theorem 2.7 in \cite{Mikhailov}.}}
\begin{lemma}
The approximate fluid velocities $\{\bd{u}_{N}\}_{N = 1}^{\infty}$ defined on the maximal fluid domain $\Omega^{M}_{f}$ by extension by zero are uniformly bounded in $L^{2}(0, T; H^{s}(\Omega^{M}_{f}))$ for $s \in (0, 1/2)$.
\end{lemma}

\subsubsection{Velocity convergence via a generalized Aubin-Lions compactness argument}\label{VelConv}

We now show strong convergence as $N\to\infty$ along a subsequence of the approximate fluid velocities $\bd{u}_{N}$,  which are now functions in time defined on the fixed maximal domain $\Omega^{M}_{f}$.

\begin{proposition}\label{convergence}
The sequence $\bd{u}_{N}$ is relatively compact in $L^{2}(0, T; L^{2}(\Omega^{M}_{f}))$.
\end{proposition}

\begin{proof}
The proof is based on using the generalized Aubin-Lions compactness theorem for problems on moving domains, {{which is Theorem 3.1 of  \cite{AubinLions}, restated in this manuscript for the reader's convenience as Theorem~\ref{Compactness} in Section \ref{appendix2}.}}
For this purpose we define the Hilbert spaces $V$ and $H$  from the statement of the theorem to be 
\begin{equation*}
H = L^{2}(\Omega^{M}_{f}), \qquad V = H^{s}(\Omega^{M}_{f}), \qquad \text{ for } 0 < s < 1/2,
\end{equation*}
where we note that  indeed $V \subset \subset H$ as required by Theorem 3.1 in \cite{AubinLions}. 
Additionally, the spaces $(V^{n}_{\Delta t},Q^{n}_{\Delta t})$ from the statement of the theorem correspond to our spaces 
$(V^{n}_{N},Q^{n}_{N})$ defined by \eqref{VnN}. Notice that $V^{n}_{N}\times Q^{n}_{N}$
embeds continuously into $V \times V$ as required by the statement of Theorem 3.1 in \cite{AubinLions}, 
where the embedding can be achieved by the extension by zero operator to the maximal domain $\Omega^{M}_{f}$, uniformly in $n$ and $N$. 

To obtain compactness of the sequence $\bd{u}_{N}$ in $L^2(0,T,H)$, by Theorem 3.1 in \cite{AubinLions}, seven properties need to be satisfied
by the sequence $\bd{u}_{N}$ and the spaces $V^{n}_{N}$ and $Q^{n}_{N}$. They are called Properties A1-3, B, and C1-3.

The proof that approximate solutions $\bd{u}_{N}$ satisfy Properties A1-3 and C1-3 is analogous to the corresponding proof
 in \cite{AubinLions} (Section 4.2).
 The main difficulty is to verify Property B, which is a condition on equicontinuity of $\bd{u}_{N}$, stated as follows:
 
 {\bf{Property B, \cite{AubinLions}}}. There exists a constant $C > 0$ independent of $N$ such that
\begin{equation}\label{originalpropB}
\left|\left|P^{n}_{N} \frac{\bd{u}^{n + 1}_{N} - \bd{u}^{n}_{N}}{\Delta t}\right|\right|_{(Q^{n}_{N})'} \le C\left(1 + ||\bd{u}^{n + 1}_{N}||_{V^{n + 1}_{N}}\right), \qquad \text{ for all } n = 0, 1, ..., N - 1,
\end{equation}
{ where $P^{n}_{N}$ denotes the orthogonal projection onto the closed subspace $\overline{Q^{n}_{N}}^{H}$ of the Hilbert space $H$.}

The sequence $\bd{u}_{N}$ constructed in this manuscript, however, {\emph{does not satisfy this property}}. {{Nevertheless}}, $\bd{u}_{N}$ satisfy the following
{\emph{generalized Property B}} which implies the desired equicontinuity under which the generalized Aubin-Lions theorem from \cite{AubinLions} still holds:
 
{\bf{Generalized Property B.}} There exist a constant $C$ independent of $n$ and $N$, an exponent $p$, $1 \le p < 2$, and a sequence of nonnegative 
numbers $\{a^{n}_{N}\}_{n = 0}^{N - 1}$ for each $N$, satisfying $(\Delta t) \sum_{n = 0}^{N - 1} |a^{n}_{N}|^{2} \le C$ uniformly in $N$, such that
\begin{equation}\label{propB}
\left|\left|P^{n}_{N} \frac{\bd{u}^{n + 1}_{N} - \bd{u}^{n}_{N}}{\Delta t}\right|\right|_{(Q^{n}_{N})'} \le C\left(a^{n}_{N} + ||\bd{u}^{n}_{N}||_{V^{n}_{N}} + ||\bd{u}^{n + 1}_{N}||_{V^{n + 1}_{N}}\right)^{p}, \qquad \text{ for all } n = 0, 1, ..., N - 1.
\end{equation}
%where $P^{n}_{N}$ denotes the orthogonal projection onto the closed subspace $\overline{Q^{n}_{N}}^{H}$ of the Hilbert space $H$. 
{{Recall the statement of  \cite[Theorem 3.1]{AubinLions}, which can be found in the Appendix of the present manuscript in Section~\ref{appendix2}, Theorem~\ref{Compactness}. 
}}
{
\begin{theorem}[Generalized Aubin-Lions Compactness Result II]
	Assume that properties A1-3 and C1-3 of \cite[Theorem 3.1]{AubinLions} are satisfied. Furthermore,  assume that Generalized Property B above is satisfied. Then conclusion of \cite[Theorem 3.1]{AubinLions} holds, {{namely, $\{\bu_{N}\}$ is relatively compact in $L^2(0,T;H)$}}.
\end{theorem}
}
\begin{proof}
%Indeed, with this Generalized Property B  the compactness theorem, Theorem 3.1 in \cite{AubinLions} still holds, 
%as we still obtain the essential equicontinuity estimate needed in the proof. 
{We just need to prove that the essential equicontinuity estimate in the proof of \cite[Theorem 3.1.]{AubinLions} still holds under the modified assumption.}
In particular, for the original form of Property B in \eqref{originalpropB}, one has from Lemma 3.1 in \cite{AubinLions} the following equicontinuity estimate for a constant $C > 0$ that is independent of $N$:
\begin{equation*}
||P^{n, l}_{\Delta t} (\bd{u}^{n + l}_{N} - \bd{u}^{n}_{N})||_{(Q^{n, l}_{N})'} \le C\sqrt{l\Delta t}.
\end{equation*}
With the generalized form of Property B that we use above in \eqref{propB}, the same arguments as in the proof of Lemma 3.1 in \cite{AubinLions} will still give rise to the following equicontinuity estimate for a constant $C > 0$ that is independent of $N$:
\begin{equation*}
||P^{n, l}_{\Delta t} (\bd{u}^{n + l}_{N} - \bd{u}^{n}_{N})||_{(Q^{n, l}_{N})'} \le C(l\Delta t)^{1 - \frac{p}{2}},
\end{equation*}
where the generalized Aubin-Lions compactness theorem on moving domains still holds with this new equicontinuity estimate.
This is because $1 \le p < 2 $ and hence, $C(l\Delta t)^{1 - \frac{p}{2}}$ still converges to zero as $\Delta t \to 0$. 
\end{proof}

We can now complete the proof of Proposition~\ref{convergence} by verifying that our sequence $\bd{u}^{n}_{N}$ indeed {{satisfies}} the Generalized Property B.
\vskip 0.1in
\noindent
{\emph{Verification that $\bd{u}^{n}_{N}$ satisfies the Generalized Property B.}}
First, recall that by definition,
\begin{equation}\label{dualspace}
\left|\left|P^{n}_{N} \frac{\bd{u}^{n + 1}_{N} - \bd{u}^{n}_{N}}{\Delta t}\right|\right|_{(Q^{n}_{N})'} = \max_{||\bd{v}||_{Q^{n}_{N}} \le 1} \left|\int_{\Omega^{n}_{f, N}} \frac{\bd{u}^{n + 1}_{N} - \bd{u}^{n}_{N}}{\Delta t} \cdot \bd{v} d\bd{x}\right|.
\end{equation}
To estimate the right hand-side, we use
\begin{equation}\label{dualest}
\left|\int_{\Omega^{n}_{f, N}} \frac{\bd{u}^{n + 1}_{N} - \bd{u}^{n}}{\Delta t} \cdot \bd{v} d\bd{x}\right| \le \left|\int_{\Omega^{n}_{f, N}} \frac{\bd{u}^{n + 1}_{N} - \tilde{\bd{u}}^{n}_{N}}{\Delta t} \cdot \bd{v} d\bd{x}\right| + \left|\int_{\Omega^{n}_{f, N}} \frac{\tilde{\bd{u}}^{n}_{N} - \bd{u}^{n}_{N}}{\Delta t} \cdot \bd{v} d\bd{x}\right|.
\end{equation}
To estimate the first term on the right hand-side we use the semidiscrete formulation for the fluid velocity on the physical domain given by
\eqref{semifluid} to obtain
\begin{align}\label{aux}
&\left|\int_{\Omega^{n}_{f, N}} \frac{\bd{u}^{n + 1}_{N} - \tilde{\bd{u}}^{n}_{N}}{\Delta t} \cdot \bd{v} d\bd{x}\right| \le 2\nu \left|\int_{\Omega^{n}_{f, N}} \boldsymbol{D}(\boldsymbol{u}_{N}^{n + 1}) : \boldsymbol{D}(\boldsymbol{v})\right| 
\nonumber\\
&+ \frac{1}{2} \left|\int_{\Omega^{n}_{f, N}} \left[\left(\left(\tilde{\boldsymbol{u}}^{n}_{N} - \zeta^{n + \frac{1}{2}}_{N} \frac{R + y}{R + \omega^{n}_{N}} \boldsymbol{e}_{y}\right) \cdot \nabla \boldsymbol{u}^{n + 1}_{N}\right) \cdot \boldsymbol{v} - \left(\left(\tilde{\boldsymbol{u}}^{n}_{N} - \zeta^{n + \frac{1}{2}}_{N} \frac{R + y}{R + \omega^{n}_{N}} \boldsymbol{e}_{y} \right) \cdot \nabla \boldsymbol{v}\right) \cdot \boldsymbol{u}^{n + 1}_{N}\right]\right| \nonumber\\
&+ \frac{1}{2R} \left|\int_{\Omega^{n}_{f, N}} \frac{R}{R + \omega^{n}_{N}} \zeta^{n + \frac{1}{2}}_{N} \boldsymbol{u}^{n + 1}_{N} \cdot \boldsymbol{v}\right| + \frac{1}{2} \left|\int_{\Gamma^{n}_{N}} (\boldsymbol{u}^{n + 1}_{N} - \boldsymbol{\dot{\eta}}^{n + 1}_{N}) \cdot \boldsymbol{n} (\tilde{\boldsymbol{u}}^{n}_{N} \cdot \boldsymbol{v})\right| \nonumber\\
&+ \left|\int_{\Gamma^{n}_{N}} \left(\frac{1}{2}\boldsymbol{u}^{n + 1}_{N} \cdot \tilde{\boldsymbol{u}}^{n}_{N} - p^{n + 1}_{N}\right) (\bd{v} \cdot \boldsymbol{n})\right| + \beta \left|\int_{\Gamma^{n}_{N}} (\boldsymbol{\dot{\eta}}^{n + 1}_{N} - \boldsymbol{u}^{n + 1}_{N}) \cdot \boldsymbol{\tau} (\bd{v} \cdot \boldsymbol{\tau}) \right|.
\end{align}
We can bound the terms on the right hand-side uniformly in $n$, $N$, and $||\bd{v}||_{Q^{n}_{N}} \le 1$ as follows. By the boundedness of $\bd{u}^{n + 1}_{N}$ in the uniform energy estimates we immediately have 
\begin{equation*}
2\nu \left|\int_{\Omega^{n}_{f, N}} \bd{D}(\bd{u}^{n + 1}_{N}) : \bd{D}(\bd{v})\right| \le C||\bd{u}^{n + 1}_{N}||_{H^{1}(\Omega^{n}_{f, N})}.
\end{equation*}
The second term on the right hand-side of the above inequality is bounded as follows. 
First notice that because $||\bd{v}||_{Q^{n}_{N}} \le 1$, and by the definition of $Q^{n}_{N}$ in \eqref{VnN}, we have that $\bd{v}$ is bounded in $H^{3}(\Omega^{n}_{f, N})$, and hence, $\bd{v}$ and $\nabla \bd{v}$ are bounded {in $L^{\infty}(\Omega^{n}_{f, N})$}. Furthermore, by the boundedness of the fluid velocity $\bd{u}^{n}_{N}$ on the reference domain due to the uniform energy estimate, and by the uniform boundedness of the Jacobian of the ALE map $\Phi^{\omega^{n}_{N}}_{f}$, we obtain the following bound:
\begin{align*}
&\frac{1}{2} \left|\int_{\Omega^{n}_{f, N}} \left[\left(\left(\tilde{\boldsymbol{u}}^{n}_{N} - \zeta^{n + \frac{1}{2}}_{N} \frac{R + y}{R + \omega^{n}_{N}} \boldsymbol{e}_{y}\right) \cdot \nabla \boldsymbol{u}^{n + 1}_{N}\right) \cdot \boldsymbol{v} - \left(\left(\tilde{\boldsymbol{u}}^{n}_{N} - \zeta^{n + \frac{1}{2}}_{N} \frac{R + y}{R + \omega^{n}_{N}} \boldsymbol{e}_{y} \right) \cdot \nabla \boldsymbol{v}\right) \cdot \boldsymbol{u}^{n + 1}_{N}\right]\right| \\
&\le C \left(||\tilde{\bd{u}}^{n}_{N}||_{L^{2}(\Omega^{n}_{f, N})} + ||\zeta^{n + \frac{1}{2}}||_{L^{2}(\Gamma)}\right) ||\bd{u}^{n + 1}_{N}||_{H^{1}(\Omega^{n}_{f, N})} \cdot ||\bd{v}||_{H^{3}(\Omega^{n}_{f, N})} \le C||\bd{u}^{n + 1}_{N}||_{H^{1}(\Omega^{n}_{f, N})}.
\end{align*}
Similarly, the next term in \eqref{aux} is bounded as follows:
\begin{equation*}
\frac{1}{2R} \left|\int_{\Omega^{n}_{f, N}} \frac{R}{R + \omega^{n}_{N}} \zeta^{n + \frac{1}{2}}_{N} \boldsymbol{u}^{n + 1}_{N} \cdot \boldsymbol{v}\right| \le C ||\zeta^{n + \frac{1}{2}}_{N}||_{L^{2}(\Gamma)} ||\bd{u}^{n + 1}_{N}||_{L^{2}(\Omega^{n}_{f, N})} \cdot ||\bd{v}||_{H^{3}(\Omega^{n}_{f, N})} \le C ||\bd{u}^{n + 1}_{N}||_{L^{2}(\Omega^{n}_{f, N})}.
\end{equation*}
To bound the next term we observe that $||\bd{\dot{\eta}}^{n + 1}_{N}||_{L^{2}(\Gamma)}$ is bounded uniformly and furthermore, the arc length element on $\Gamma^{n}_{N}$ is uniformly bounded pointwise since $\eta^{n}_{N}$ is uniformly bounded in $H_{0}^{2}(\Gamma)$. Therefore, by using the trace inequality on $\Omega_{f}$ we have the following estimate:
\begin{align*}
\frac{1}{2} &\left|\int_{\Gamma^{n}_{N}} (\boldsymbol{u}^{n + 1}_{N} - \boldsymbol{\dot{\eta}}^{n + 1}_{N}) \cdot \boldsymbol{n} (\tilde{\boldsymbol{u}}^{n}_{N} \cdot \boldsymbol{v})\right| \\
&\le C\left(||\bd{u}^{n + 1}_{N}||_{L^{4}(\Gamma)} \cdot ||\bd{u}^{n}_{N}||_{L^{4}(\Gamma)} \cdot ||\bd{v}||_{L^{2}(\Gamma)} + ||\bd{\dot{\eta}}^{n + 1}_{N}||_{L^{2}(\Gamma)} \cdot ||\bd{u}^{n}_{N}||_{L^{4}(\Gamma)} \cdot ||\bd{v}||_{L^{4}(\Gamma)}\right) \\
&\le C\left(||\bd{u}^{n + 1}_{N}||_{H^{1/4}(\Gamma)} \cdot ||\bd{u}^{n}_{N}||_{H^{1/4}(\Gamma)} \cdot ||\bd{v}||_{H^{1}(\Omega_{f})} + ||\bd{\dot{\eta}}^{n + 1}_{N}||_{L^{2}(\Gamma)} \cdot ||\bd{u}^{n}_{N}||_{H^{1/4}(\Gamma)} \cdot ||\bd{v}||_{H^{1/4}(\Gamma)}\right) \\
&\le C\left(||\bd{u}^{n + 1}_{N}||_{H^{3/4}(\Omega_{f})} \cdot ||\bd{u}^{n}_{N}||_{H^{3/4}(\Omega_{f})}  + ||\bd{\dot{\eta}}^{n + 1}_{N}||_{L^{2}(\Gamma)} \cdot ||\bd{u}^{n}_{N}||_{H^{3/4}(\Omega_{f})} \right) \cdot ||\bd{v}||_{H^{{{1}}}(\Omega_{f})} \\
&\le C\left(||\bd{u}^{n + 1}_{N}||_{L^{2}(\Omega_{f})}^{1/4} ||\bd{u}^{n + 1}_{N}||_{H^{1}(\Omega_{f})}^{3/4} ||\bd{u}^{n}_{N}||_{L^{2}(\Omega_{f})}^{1/4} ||\bd{u}^{n}_{N}||_{H^{1}(\Omega_{f})}^{3/4}  + ||\bd{\dot{\eta}}^{n + 1}_{N}||_{L^{2}(\Gamma)} ||\bd{u}^{n}_{N}||_{L^{2}(\Omega_{f})}^{1/4} ||\bd{u}^{n}_{N}||_{H^{1}(\Omega_{f})}^{3/4} \right) \\
&\le C\left(||\bd{u}^{n + 1}_{N}||_{H^{1}(\Omega_{f})}^{3/4} \cdot ||\bd{u}^{n}_{N}||_{H^{1}(\Omega_{f})}^{3/4}  + ||\bd{u}^{n}_{N}||_{H^{1}(\Omega_{f})}^{3/4} \right) \le C\left[1 + \left(||\bd{u}^{n}_{N}||_{V^{n}_{N}} + ||\bd{u}^{n + 1}_{N}||_{V^{n}_{N}}\right)^{3/2}\right].
\end{align*}
{{{{More precisely, in}} the above chain of inequalities, we used the following results to justify the steps:
\begin{itemize}
\item In the first inequality, we used H\"{o}lder's inequality with the exponents $2$, $4$, and $4$, the fact that $\tilde{\bd{u}}^{n}_{N}$ and $\bd{u}^{n}_{N}$ have the same trace along the interface $\Gamma$ by \eqref{tildeudef}, and the uniform boundedness of the Jacobian $\mathcal{J}^{\omega^{n}_{N}}_{\Gamma}$ defined in \eqref{Jf}.
\item In the second inequality, we use Sobolev embedding on the one-dimensional domain $\Gamma$. 
\item In the third inequality, we use the trace theorem on $\Omega_{f}$ to bound the trace along $\Gamma$.
\item In the fourth inequality, we use interpolation, and the fact that since the test function satisfies $||\bd{v}||_{Q^{n}_{N}} \le 1$ where the test space $Q^{n}_{N}$
 is defined in \eqref{VnN}, we have that $||\bd{v}||_{H^{1}(\Omega_{f})}$ is uniformly bounded  on the reference domain  independently of $n$ and $N$ by the uniform bounds on $\omega_{N}$. {{The uniform bounds on $\omega_{N}$ ensure uniform boundedness  of $\mathcal{J}^{\omega}_{f}$ defined in \eqref{Jf}, which is the term that appears when transforming integrals back to the reference domain.}} 
\item In the fifth inequality, we use uniform boundedness of $||\bd{u}^{n}_{N}||_{L^{2}(\Omega_{f})}$ and $||\dot{\bd{\eta}}^{n}_{N}||_{L^{2}(\Gamma)}$ which is implied by the previous uniform energy estimates, see Lemma \ref{uniform}.
\item We use the inequality $ab \le \frac{1}{2}(a^{2} + b^{2})$. 
\end{itemize}
}}

{{\noindent Next, we estimate the second-to-last term in \eqref{aux} in a way similar to the immediately preceding term, using similar justifications as given directly above:}}
\begin{align*}
\Big|\int_{\Gamma^{n}_{N}} &\left(\frac{1}{2} \bd{u}^{n + 1}_{N} \cdot \tilde{\bd{u}}^{n}_{N} - p^{n + 1}_{N}\right) (\bd{v} \cdot \bd{n}) \Big| \\
&\le C\left(||\bd{u}^{n + 1}_{N}||_{L^{4}(\Gamma)} \cdot ||\bd{u}^{n}_{N}||_{L^{4}(\Gamma)} \cdot ||\bd{v}||_{L^{2}(\Gamma)} + ||p^{n + 1}_{N}||_{L^{2}(\Gamma)} \cdot ||\bd{v}||_{L^{2}(\Gamma)}\right) \\
&\le C\left(||\bd{u}^{n + 1}_{N}||_{H^{1/4}(\Gamma)} \cdot ||\bd{u}^{n}_{N}||_{H^{1/4}(\Gamma)} \cdot ||\bd{v}||_{H^{1}(\Omega_{f})} + ||p^{n + 1}_{N}||_{H^{1}(\Omega_{b})} \cdot ||\bd{v}||_{H^{1}(\Omega_{f})}\right) \\
&\le C\left(||\bd{u}^{n + 1}_{N}||_{H^{3/4}(\Omega_{f})} \cdot ||\bd{u}^{n}_{N}||_{H^{3/4}(\Omega_{f})} + ||p^{n + 1}_{N}||_{H^{1}(\Omega_{b})}\right) \\
&\le C\left(||\bd{u}^{n + 1}_{N}||_{L^{2}(\Omega_{f})}^{1/4} ||\bd{u}^{n + 1}_{N}||_{H^{1}(\Omega_{f})}^{3/4} \cdot ||\bd{u}^{n}_{N}||_{L^{2}(\Omega_{f})}^{1/4} ||\bd{u}^{n}_{N}||_{H^{1}(\Omega_{f})}^{3/4} + ||p^{n + 1}_{N}||_{H^{1}(\Omega_{b})}\right) \\
&\le C\left[1 + \left(||p^{n + 1}_{N}||_{H^{1}(\Omega_{b})} + ||\bd{u}^{n}_{N}||_{V^{n}_{N}} + ||\bd{u}^{n + 1}_{N}||_{V^{n + 1}_{N}}\right)^{3/2}\right].
\end{align*}
Finally, we estimate the last term
\begin{align*}
&\beta\left|\int_{\Gamma^{n}_{N}} (\bd{\dot{\eta}}^{n + 1}_{N} - \bd{u}^{n + 1}_{N}) \cdot \bd{\tau} (\bd{v} \cdot \bd{\tau})\right| \le C\left(||\dot{\bd{\eta}}^{n + 1}_{N}||_{L^{2}(\Gamma)} \cdot ||\bd{v}||_{L^{2}(\Gamma)} + ||\bd{u}^{n + 1}_{N}||_{L^{2}(\Gamma)} \cdot ||\bd{v}||_{L^{2}(\Gamma)}\right) \\
&\le C\left(1 + ||\bd{u}^{n + 1}_{N}||_{H^{1}(\Omega_{f})}\right).
\end{align*}
Therefore, we obtain the final estimate of the first term in \eqref{dualest} which implies the existence of a constant $C$ independent of $n$ and $N$, such that
\begin{multline}\label{dualest1}
\max_{||\bd{v}||_{Q^{n}_{N}} \le 1} \left|\int_{\Omega^{n}_{f, N}} \frac{\bd{u}^{n + 1}_{N} - \tilde{\bd{u}}^{n}_{N}}{\Delta t} \cdot \bd{v} d\bd{x}\right| \le C \left(a^{n}_{N} + ||\bd{u}^{n}_{N}||_{V^{n}_{N}} + ||\bd{u}^{n + 1}_{N}||_{V^{n + 1}_{N}}\right)^{3/2}, \\
\text{ for } a^{n}_{N} := 1 + ||p^{n + 1}_{N}||_{H^{1}(\Omega_{b})}, \text{ where } (\Delta t) \sum_{n = 0}^{N - 1} |a^{n}_{N}|^{2} \le 2\left[(\Delta t)N + ||p_{N}||^{2}_{L^{2}(0, T; H^{1}(\Omega_{b}))}\right] \le C.
\end{multline}

To complete the estimate \eqref{dualest}, it remains to show that the second term $\displaystyle \left|\int_{\Omega^{n}_{f, N}} \frac{\tilde{\bd{u}}^{n}_{N} - \bd{u}^{n}_{N}}{\Delta t} \cdot \bd{v} d\bd{x}\right|$ is uniformly bounded. This follows from the same estimates as those 
presented in \cite{AubinLions} which show that there exists a constant $C$ independent of $n$ and $N$, such that 
\begin{equation}\label{dualest2}
\max_{||\bd{v}||_{Q^{n}_{N}} \le 1} \left|\int_{\Omega_{f, N}^{n}} \frac{\tilde{\bd{u}}^{n}_{N} - \bd{u}^{n}_{N}}{\Delta t} \cdot \bd{v} d\bd{x}\right| \le C.
\end{equation}
Combining \eqref{dualest1} and \eqref{dualest2} with \eqref{dualspace} and \eqref{dualest} establishes Generalized Property B
and completes the proof of Proposition \ref{convergence}.
\end{proof}
 
\section{Passing to the limit in the regularized weak formulation}\label{limitpassage}

We have so far established the following strong convergence results:
\begin{align*}
&\bd{\eta}_{N} \to \bd{\eta}, \quad \text{ in } C(0, T; L^{2}(\Omega_{b})),
\\
&\omega_{N} \to \omega, \quad \text{ in } L^{\infty}(0, T; H^{s}(\Gamma)) \text{ for } 0 < s < 2,
\\
&\zeta_{N}^{*} \to \zeta, \quad \text{ in } L^{2}(0, T; H^{-s}(\Gamma)), \text{ for } -1/2 < s < 0, 
\\
&\zeta_{N} \to \zeta, \quad \text{ in } L^{2}(0, T; H^{-s}(\Gamma)), \text{ for } -1/2 < s < 0, 
\\
&\bd{\xi}_{N} \to \bd{\xi}, \quad \text{ in } L^{2}(0, T; H^{-s}(\Omega_{b})), \text{ for } -1/2 < s < 0, 
\\
&\bd{u}_{N} \to \bd{u}, \quad \text{ in } L^{2}(0, T; L^{2}(\Omega^{M}_{f})), \qquad
p_{N} \to p, \quad \text{ in } L^{2}(0, T; L^{2}(\Omega_{b})),
\end{align*}
where $\zeta_{N}^{*}$ and $\zeta_{N}$ converge to the same limit in $L^{2}(0, T; H^{-s}(\Gamma))$ for $-1/2 < s < 0$ due to the numerical dissipation estimates $\sum_{n = 1}^{N} ||\zeta^{n}_{N} - \zeta^{n - \frac{1}{2}}_{N}||_{L^{2}(\Gamma)}^{2} \le C$, which imply that $||\zeta_{N} - \zeta_{N}^{*}||_{L^{2}(0, T; L^{2}(\Gamma))} \to 0$. 

These strong convergence results will be used to pass to the limit in the semidiscrete formulation of the coupled problem \eqref{semi1} and show that the limit satisfies the weak formulation
of the regularized problem. Before we can do this, there are two more convergence results that need to be established.
One is a strong convergence result for the traces for the fluid velocity on the boundary of the fluid domain, and the other is a convergence result for
the test functions, which are defined on approximate moving domains. 

We start with the convergence result for the trace of the fluid velocity $\hat{\bd{u}}_{N}|_{\Gamma}$  along $\Gamma$.

\subsection{Strong convergence of the fluid velocity traces on $\Gamma$}
\begin{proposition}\label{traceconvergence}
The traces  $\hat{\bd{u}}_{N}|_{\Gamma}$ of the approximate fluid velocities on $\Gamma$ converge to the trace of the limiting fluid velocity on $\Gamma$
as $N\to\infty$:
\begin{equation*}
\hat{\bd{u}}_{N}|_{\Gamma} \to \hat{\bd{u}}|_{\Gamma}, \qquad \text{ in } L^{2}(0, T; H^{s - \frac{1}{2}}(\Gamma)), \qquad \text{ for } s \in (0, 1),
\end{equation*}
where $\hat{\bd{u}}_{N} = \bd{u}_{N} \circ \Phi^{\tau_{\Delta t} \omega_{N}}_{f}$ and $\hat{\bd{u}} = \bd{u} \circ \Phi^{\omega}_{f}$. 
\end{proposition}

To prove Proposition \ref{traceconvergence}, we will use the following elementary lemma.

\begin{lemma}\label{Hstrace}
Suppose that the functions $\{f_{n}\}_{n = 1}^{\infty}$ and $f$ are all uniformly bounded in $L^{2}(0, T; H^{1}(\Omega_{f}))$ and $f_{n} \to f$ in $L^{2}(0, T; L^{2}(\Omega_{f}))$. Then, $f_{n} \to f$ in $L^{2}(0, T; H^{s}(\Omega_{f}))$ {{for $s \in (0, 1)$}} and hence $f_{n}|_{\Gamma} \to f|_{\Gamma}$ in $L^{2}(0, T; H^{s - \frac{1}{2}}(\Gamma))$ for {{$s \in (1/2, 1)$}}. 
\end{lemma}

\begin{proof}[Proof of Lemma \ref{Hstrace}]
{{For $s \in (0, 1)$, we compute using Sobolev interpolation that 
\begin{multline*}
||f_{n} - f||_{L^{2}(0, T; H^{s}(\Omega_{f}))}^{2} = \int_{0}^{T} ||(f_{n} - f)(t)||_{H^{s}(\Omega_{f})}^{2} dt \\
\le \int_{0}^{T} ||(f_{n} - f)(t)||_{L^{2}(\Omega_{f})}^{2(1 - s)} \cdot ||(f_{n} - f)(t)||_{H^{1}(\Omega_{f})}^{2s} dt \le ||f_{n} - f||_{L^{2}(0, T; L^{2}(\Omega_{f}))}^{2(1 - s)} \cdot ||f_{n} - f||_{L^{2}(0, T; H^{1}(\Omega_{f}))}^{2s}.
\end{multline*}
The result then follows from the fact that $||f_{n} - f||_{L^{2}(0, T; H^{1}(\Omega_{f}))} \le C$ for a constant $C$ that does not depend on $N$, the assumption that $||f_{n} - f||_{L^{2}(0, T; L^{2}(\Omega_{f}))} \to 0$ as $N \to \infty$, and the trace embedding which gives that $||f_{n}|_{\Gamma} - f|_{\Gamma}||_{L^{2}(0, T; H^{s - \frac{1}{2}}(\Gamma))}^{2} \le ||f_{n} - f||_{L^{2}(0, T; H^{s}(\Omega_{f}))}^{2}$ for $s \in (1/2, 1)$.}} 
\end{proof}

We can use the elementary lemma above to show the desired strong convergence of the fluid velocity traces. 

\begin{proof}[Proof of Proposition \ref{traceconvergence}]
We would like to combine the fact that $\bd{u}_{N} \to \bd{u}$ in $L^{2}(0, T; L^{2}(\Omega^{M}_{f}))$ 
with the fact that $\bd{u}_{N}$ and $\bd{u}$ are all uniformly bounded in $L^{2}(0, T; H^{1}(\Omega_{f}(t)))$ for all $N$,
 to deduce strong convergence of the traces of the fluid velocities using {{Lemma \ref{Hstrace}}}. We do this in the following steps.

\medskip

\noindent \textbf{Step 1.} We show that $\hat{\bd{u}}_{N} \to \hat{\bd{u}}$ on $L^{2}(0, T; L^{2}(\Omega_{f}))$, for $\hat{\bd{u}}_{N}$ and $\hat{\bd{u}}$ 
defined on the reference fluid domain. 

To prove this, we compute $||\hat{\bd{u}}_{N} - \hat{\bd{u}}||_{L^{2}(0, T; L^{2}(\Omega_{f}))}^{2}$ using the functions
 $\bd{u}_{N}$ and $\bd{u}$ which are defined on the maximal domain $\Omega^{M}_{f}$:
\begin{align*}
&||\hat{\bd{u}}_{N} - \hat{\bd{u}}||_{L^{2}(0, T; L^{2}(\Omega_{f}))}^{2} = \int_{0}^{T} \int_{\Omega_{f}} \left|\bd{u}_{N}\left(t, x, y + \left(1 + \frac{y}{R}\right) \tau_{\Delta t} \omega_{N}\right) - \bd{u}\left(t, x, y + \left(1 + \frac{y}{R}\right) \omega\right)\right|^{2} \\
&\le 2(I_{1} + I_{2}), 
\end{align*}
where
\begin{equation*}
I_{1} = \int_{0}^{T} \int_{\Omega_{f}} \left|\bd{u}_{N}\left(t, x, y + \left(1 + \frac{y}{R}\right) \tau_{\Delta t} \omega_{N}\right) - \bd{u}\left(t, x, y + \left(1 + \frac{y}{R}\right) \tau_{\Delta t} \omega_{N}\right)\right|^{2},
\end{equation*}
\begin{equation*}
I_{2} = \int_{0}^{T} \int_{\Omega_{f}} \left|\bd{u}\left(t, x, y + \left(1 + \frac{y}{R}\right) \tau_{\Delta t} \omega_{N}\right) - \bd{u}\left(t, x, y + \left(1 + \frac{y}{R}\right) \omega\right)\right|^{2}.
\end{equation*}
We show that $I_1 \to 0$ as $N\to\infty$ by using the fact that $1 + \frac{\omega^{n}_{N}}R$ is uniformly bounded from above by a positive constant, 
and the fact that $\Omega^{M}_{f}$ contains all of the domains $\Omega^{n}_{f, N}$, so that we can estimate:
{{\begin{align*}
&I_{1} = \sum_{n = 0}^{N - 1} \int_{n\Delta t}^{(n + 1)\Delta t} \int_{\Omega^{n}_{f, N}} \left(1 + \frac{\omega^{n}_{N}}{R}\right) |\bd{u}^{n + 1}_{N} - \bd{u}|^{2} \le C \sum_{n = 0}^{N - 1} \int_{n\Delta t}^{(n + 1)\Delta t} \int_{\Omega^{n}_{f, N}} |\bd{u}^{n + 1}_{N} - \bd{u}|^{2} \\
&\le C ||\bd{u}_{N} - \bd{u}||^{2}_{L^{2}(0, T; L^{2}(\Omega^{M}_{f}))} \to 0.
\end{align*}}}
For $I_{2}$, we break up the integral into two parts:
\begin{equation*}
I_{2} = I_{2, 1} + I_{2, 2},
\end{equation*}
where 
\begin{align*}
&I_{2, 1} = \int_{0}^{T} \int_{0}^{L} \int_{-R}^{\min(0, y^{*}(t, x))}  \left|\bd{u}\left(t, x, y + \left(1 + \frac{y}{R}\right) \tau_{\Delta t} \omega_{N}\right) - \bd{u}\left(t, x, y + \left(1 + \frac{y}{R}\right) \omega\right)\right|^{2},
\\
&I_{2, 2} = \int_{0}^{T} \int_{0}^{L} \int_{\min(0, y^{*}(t, x))}^{0} \left|\bd{u}\left(t, x, y + \left(1 + \frac{y}{R}\right) \tau_{\Delta t}\omega_{N}\right)\right|^{2},
\end{align*}
for $y^{*}(t, x) = \frac{\omega - \tau_{\Delta t} \omega_{N}}{R + \tau_{\Delta t} \omega_{N}}$. We can interpret $y^{*}(t, x)$ as the $y$ value for which $y + \left(1 + \frac{y}{R}\right) \tau_{\Delta t} \omega_{N} = \omega$. 
Now, note that
{{\begin{align*}
I_{2, 1}& \le \int_{0}^{T} \int_{0}^{L} \int_{-R}^{\min(0, y^{*}(t, x))} \left(\int_{y + \left(1 + \frac{y}{R}\right)\tau_{\Delta t}\omega_{N}}^{y + \left(1 + \frac{y}{R}\right)\omega} |\partial_{y} \bd{u}(t, x, y')| dy'\right)^{2} \\
&\le \int_{0}^{T} \int_{0}^{L} \int_{-R}^{\min(0, y^{*}(t, x))} \left(\int_{y + \left(1 + \frac{y}{R}\right)\tau_{\Delta t}\omega_{N}}^{y + \left(1 + \frac{y}{R}\right)\omega} |\partial_{y} \bd{u}(t, x, y')|^{2} dy' \right) \cdot \left(1 + \frac{y}{R}\right) \cdot |\omega - \tau_{\Delta t} \omega_{N}|,
\end{align*}
where we applied Cauchy-Schwarz to the inner $dy'$ integral.}} We note that $\tau_{\Delta t} \omega_{N} \to \omega$ pointwise uniformly on $[0, T] \times \Gamma$ as $N \to \infty$
by {{the convergence $\overline{\omega}_{N} \to \omega$ in $C(0, T; H^{s}(\Gamma))$ for $0 < s < 2$ (from Proposition \ref{platedisplacement}) and 
the estimate \eqref{holderplate}}}.
Combining this with the fact that $||\nabla \bd{u}||_{L^{2}(0, T; L^{2}(\Omega^{\omega}_{f}(t)))}$ is bounded, we have that $I_{2, 1} \to 0$ as $N \to \infty$. 

Next, by {{the fundamental theorem of calculus and Jensen's inequality}},
\begin{align*}
I_{2, 2} &\le \int_{0}^{T} \int_{0}^{L} |\min(0, y^{*}(t, x))| \cdot \max_{w \in [-R, \omega(t, x)]} |\bd{u}(t, x, w)|^{2} \\
&\le {{C}} \int_{0}^{T} \int_{0}^{L} |\min(0, y^{*}(t, x))| \cdot \int_{-R}^{\omega(t, x)} |\partial_{r} \bd{u}(t, x, y')|^{2} dy',
\end{align*}
{{for a constant $C$ that is independent of $N$}}, so we conclude that $I_{2, 2} \to 0$ as $N \to \infty$ by the fact that $|\min(0, y^{*}(t, x))| \to 0$ uniformly on $[0, T] \times \Gamma$, and by the boundedness of $||\nabla \bd{u}||_{L^{2}(0, T; L^{2}(\Omega^{\omega}_{f}(t))}$. Thus, we have that $||\hat{\bd{u}}_{N} - \hat{\bd{u}}||_{L^{2}(0, T; L^{2}(\Omega_{f})))} \to 0$. { Note that this maximum is well defined for almost every $(t,x)$, since Sobolev functions are absolutely continuous for almost every line $\{(t, x, y) : -R \le y \le \omega(t, x)\}$ with fixed $(t, x)$, see \cite[Theorem 2, Section 4.9]{EvansFine}.} 

\medskip

\noindent \textbf{Step 2.} We claim that the functions $\hat{\bd{u}}_{N}$ for positive integers $N$ and $\hat{\bd{u}}$ are all uniformly bounded in $L^{2}(0, T; H^{1}(\Omega_{f}))$. Recall from Lemma \ref{uniform} that the approximate solutions $\hat{\bd{u}}_{N}$ are uniformly bounded in $L^{2}(0, T; H^{1}(\Omega_{f}))$. Since $\hat{\bd{u}}$ is the strong limit of $\hat{\bd{u}}_{N}$ in $L^{2}(0, T; L^{2}(\Omega_{f}))$ and $\hat{\bd{u}}_{N}$ converge weakly in $L^{2}(0, T; H^{1}(\Omega_{f}))$ along a subsequence to a weak limit which hence must also be $\hat{\bd{u}}$, we conclude that $\hat{\bd{u}}$ is also in $L^{2}(0, T; H^{1}(\Omega_{f}))$, which establishes the desired result of this step. 

\medskip

\noindent \textbf{Step 3.} From Step 1, we have that $\hat{\bd{u}}_{N} \to \hat{\bd{u}}$ in $L^{2}(0, T; L^{2}(\Omega_{f}))$ and from Step 2, the functions $\hat{\bd{u}}_{N}$ and $\hat{\bd{u}}$ are bounded in $L^{2}(0, T; H^{1}(\Omega_{f}))$ independently of $N$, so we can conclude the proof of Proposition \ref{traceconvergence} by using Lemma \ref{Hstrace}. 
\end{proof}

%%%%%%%%%%%%%%%%%%%

\subsection{Convergence of the test functions on approximate fluid domains}

The main difficulty in passing to the limit will be the test functions for the fluid velocity. In particular, on the fixed reference domain $\Omega_{f}$ for the fluid, we note that the test functions for the fluid velocity in \eqref{testspace} satisfy $\nabla^{\omega}_{f} \cdot \bd{v} = 0 \text{ on } \Omega_{f}$, where $\omega$ is the solution for the plate displacement. However, the test functions for the fluid velocity in the semidiscrete formulation in the semidiscrete test space $\mathcal{Q}^{n + 1}_{N}$, defined by \eqref{semitestspace}, satisfy $\nabla^{\omega^{n}_{N}} \cdot \bd{v} = 0 \text{ on } \Omega_{f}$. Hence, we need a way of comparing test functions in $\mathcal{Q}^{n + 1}_{N}$ to test functions in the actual test space $\mathcal{V}^{\omega}_{\text{test}}$.

To do this, recall that we have defined the maximal domain $\Omega^{M}_{f}$ that contains all of the numerical fluid domains $\Omega_{f, N}^{n}$. 
We then propose to work with the test functions that are defined on $\Omega^{M}_{f}$,
and are constructed in such a way that the restrictions of those test functions to the domain defined by the plate displacement $\omega$,
and composed with the 
ALE mapping  $\Phi^{\omega}_{f}$ defined in \eqref{phif}, gives a space of test functions $\mathcal{X}^{\omega}_{f}$ that is dense in 
the fluid velocity test space $\mathcal{V}^{\omega}_{f}$.
The space of all such test functions defined on $\Omega^{M}_{f}$ is denoted by $\mathcal{X}$ and it is defined as follows.

{\bf{The test space $\mathcal{X}$:}} The test space $\mathcal{X}$ consists of  functions $\bd{v}\in C_{c}^{1}([0, T); H^{1}(\Omega^{M}_{f}))$ 
satisfying the following properties for each $t \in [0, T)$:
\begin{enumerate}
\item For each $t \in [0, T)$, $\bd{v}(t)$ is a smooth vector-valued function on $\Omega^{M}_{f}$.
\item $\nabla \cdot \bd{v}(t) = 0$ on $\Omega^{M}_{f}$ for all $t \in [0, T)$. 
\item $\bd{v}(t) = 0$ on $\partial \Omega^{M}_{f} \setminus \Gamma_{M}$ for all $t \in [0, T)$, where $\Gamma_{M} = \{(x, M(x)) : 0 \le x \le L\}$ is the top boundary of the maximal fluid domain $\Omega^{M}_{f}$. 
\end{enumerate}

Given $\bd{v} \in \mathcal{X}$, define
\begin{equation}\label{testf}
\tilde{\bd{v}} = \bd{v}|_{\Omega^{\omega}_{f}} \circ \bd{\Phi}^{\omega}_{f} \quad  {\rm and} \quad 
\tilde{\bd{v}}_{N} = \bd{v}|_{\Omega^{\omega_{N}}_{f}} \circ \bd{\Phi}^{\omega_{N}}_{f}.
\end{equation}
The test functions $\tilde{\bd{v}}$ are dense in the fluid velocity test space $\mathcal{V}^{\omega}_{f}$ {{associated with the fixed domain formulation}}, 
and the test functions $\tilde{\bd{v}}_{N}$ restricted to $[n\Delta t, (n + 1)\Delta t)$ are dense in $V^{\omega^{n}_{N}}_{f}$,
where $V^{\omega^{n}_{N}}_{f}$ is the velocity test space for the semidiscretized problem(s) given in \eqref{semitestspace}.
Therefore, for each fixed $N$, we can consider the semidiscrete formulation with the test function $\tilde{\bd{v}}_{N}$, which we emphasize is discontinuous in time, due to the jumps in $\omega_{N}$ at each $n\Delta t$. To pass to the limit as $N\to\infty$
we can use the same approach as in Lemma 7.1 in \cite{MuhaCanic13} and Lemma 2.8 in \cite{CanicLectureNotes}, to obtain
the following strong convergence results of the velocity test functions $\tilde{\bd{v}}_{N}$ and their gradients, which will allow us to pass to the limit in the 
semidiscrete weak formulations:
\begin{proposition}\label{ConvergenceTestFunctions}
Consider $\bd{v} \in \mathcal{X}$, and $\tilde{\bd{v}}$ and $\tilde{\bd{v}}_{N}$ defined in \eqref{testf}. Then
\begin{equation*}
\tilde{\bd{v}}_{N} \to \bd{v}, \qquad \nabla \tilde{\bd{v}}_{N} \to \nabla \bd{v},
\end{equation*}
pointwise, uniformly on $[0, T] \times \Omega_{f}$, as $N \to \infty$.  
\end{proposition}

 \begin{remark} We emphasize that we were able to construct such a test space ${\cal{X}}$ because in the definition of the full test space $\mathcal{V}^{\omega}_{\text{test}}$ in \eqref{testspace}, the only component of the test space whose definition depends on the plate displacement is the fluid velocity, and fortunately, this fluid velocity component of the test space is \textit{decoupled from the other components.}
 This is a feature of fluid-poroelastic structure interaction problems. In the purely elastic case of FSI, the fluid velocity test space is 
 coupled to that of the structure, and the construction of the test functions that converge on the approximate fluid domains in more involving, 
 see e.g., \cite{CDEM,MuhaCanic13,BorSunMultiLayered}. 
 \end{remark}

\subsection{Passing to the limit}

We are now in a position to pass to the limit in the semidiscrete formulation. {{To do this, we consider any test function $(\tilde{\bd{v}}_{N}, \varphi, \bd{\psi}, r)$ for a given $\bd{v} \in \mathcal{X}$ and for each $n = 0, 1, ..., N - 1$, we test the semidiscrete formulation \eqref{semi1} with $(\tilde{\bd{v}}_{N}(t), \varphi(t), \bd{\psi}(t), r(t))$ for each $t \in [n\Delta t, (n + 1)\Delta t)$, integrate the resulting expressions in time from $t = n\Delta t$ to $t = (n + 1)\Delta t$, and then finally sum over $n = 0, 1, ..., N - 1$ to get an integral over the entire time interval $[0, T]$. Then, using the definition of the approximate solutions from Section \ref{approximate}, we thus}} obtain that for all $(\tilde{\bd{v}}_{N}, \varphi, \bd{\psi}, r)$ in the test space with $\bd{v} \in \mathcal{X}$, the following holds:
\begin{align*}
&\int_{0}^{T} \int_{\Omega_{f}} \left(1 + \frac{\tau_{\Delta t} \omega_{N}}{R}\right) \partial_{t}\overline{\bd{u}}_{N} \cdot \tilde{\boldsymbol{v}}_{N} + \frac{1}{2} \int_{0}^{T} \int_{\Omega_{f}} \left(1 + \frac{\tau_{\Delta t}\omega_{N}}{R}\right) \Bigg[\left(\left(\tau_{\Delta t} \boldsymbol{u}_{N} - \zeta_{N} \frac{R + y}{R} \boldsymbol{e}_{y}\right) \cdot \nabla^{\tau_{\Delta t} \omega_{N}}_{f} \bd{u}_{N}\right) \cdot \tilde{\boldsymbol{v}}_{N} \\
&- \left(\left(\tau_{\Delta t} \boldsymbol{u}_{N} - \zeta_{N} \frac{R + y}{R} \boldsymbol{e}_{y}\right) \cdot \nabla^{\tau_{\Delta t} \omega_{N}}_{f} \tilde{\boldsymbol{v}}_{N}\right) \cdot \bd{u}_{N}\Bigg] + \frac{1}{2R} \int_{0}^{T} \int_{\Omega_{f}} \zeta_{N} \boldsymbol{u}_{N} \cdot \tilde{\boldsymbol{v}}_{N} \\
&+ \frac{1}{2} \int_{0}^{T} \int_{\Gamma} (\bd{u}_{N} - \zeta_{N}^{*} \bd{e}_{y}) \cdot \boldsymbol{n}^{\tau_{\Delta t} \omega_{N}} (\tau_{\Delta t} \boldsymbol{u}_{N} \cdot \tilde{\boldsymbol{v}}_{N}) + 2\nu \int_{0}^{T} \int_{\Omega_{f}} \left(1 + \frac{\tau_{\Delta t} \omega_{N}}{R}\right) \boldsymbol{D}^{\tau_{\Delta t} \omega_{N}}_{f}(\boldsymbol{u}_{N}) : \boldsymbol{D}^{\tau_{\Delta t} \omega_{N}}_{f}(\tilde{\boldsymbol{v}}_{N}) \\
&+ \int_{0}^{T} \int_{\Gamma} \left(\frac{1}{2} \bd{u}_{N} \cdot \tau_{\Delta t} \boldsymbol{u}_{N} - p_{N} \right)(\boldsymbol{\psi} - \tilde{\boldsymbol{v}}_{N})\cdot \boldsymbol{n}^{\tau_{\Delta t} \omega_{N}} + \frac{\beta}{\mathcal{J}^{\tau_{\Delta t} \omega_{N}}_{\Gamma}} \int_{0}^{T} \int_{\Gamma} (\zeta_{N}^{*}\bd{e}_{y} - \bd{u}_{N}) \cdot \boldsymbol{\tau}^{\tau_{\Delta t} \omega_{N}} (\boldsymbol{\psi} - \tilde{\boldsymbol{v}}_{N}) \cdot \boldsymbol{\tau}^{\tau_{\Delta t} \omega_{N}} \\
&+ \rho_{b} \int_{0}^{T} \int_{\Omega_{b}} \left(\frac{\bd{\xi}_{N} - \tau_{\Delta t} \bd{\xi}_{N}}{\Delta t}\right) \cdot \boldsymbol{\psi} + \rho_{p} \int_{0}^{T} \int_{\Gamma} \partial_{t}\overline{\zeta}_{N} \cdot \varphi + 2\mu_{e} \int_{0}^{T} \int_{\Omega_{b}} \boldsymbol{D}(\bd{\eta}_{N}) : \boldsymbol{D}(\boldsymbol{\psi}) \\
&+ \lambda_{e} \int_{0}^{T} \int_{\Omega_{b}} (\nabla \cdot \bd{\eta}_{N}) (\nabla \cdot \boldsymbol{\psi}) + 2\mu_{v} \int_{0}^{T} \int_{\Omega_{b}} \bd{D}(\bd{\xi}_{N}) : \bd{D}(\bd{\psi}) + \lambda_{v} \int_{0}^{T} \int_{\Omega_{b}} (\nabla \cdot \bd{\xi}_{N}) (\nabla \cdot \bd{\psi}) \\
&- \alpha \int_{0}^{T} \int_{\Omega_{b}} \mathcal{J}^{(\tau_{\Delta t} \eta_{N})^{\delta}}_{b} p_{N} \nabla^{(\tau_{\Delta t} \eta_{N})^{\delta}}_{b} \cdot \boldsymbol{\psi} + c_{0} \int_{0}^{T} \int_{\Omega_{b}} \partial_{t} \overline{p}_{N} \cdot r - \alpha \int_{0}^{T} \int_{\Omega_{b}} \mathcal{J}^{(\tau_{\Delta t} \eta_{N})^{\delta}}_{b} \bd{\xi}_{N} \cdot \nabla^{(\tau_{\Delta t} \eta_{N})^{\delta}}_{b} r \\
&- \alpha \int_{0}^{T} \int_{\Gamma} (\zeta_{N}^{*}\bd{e}_{y} \cdot \boldsymbol{n}^{(\tau_{\Delta t} \omega_{N})^{\delta}}) r + \kappa \int_{0}^{T} \int_{\Omega_{b}} \mathcal{J}^{(\tau_{\Delta t} \eta_{N})^{\delta}}_{b} \nabla^{(\tau_{\Delta t} \eta_{N})^{\delta}}_{b} p_{N} \cdot \nabla^{(\tau_{\Delta t} \eta_{N})^{\delta}}_{b} r \\
&- \int_{0}^{T} \int_{\Gamma} [(\bd{u}_{N} - \zeta_{N}^{*}\bd{e}_{y}) \cdot \boldsymbol{n}^{\tau_{\Delta t} \omega_{N}}]r + \int_{0}^{T} \int_{\Gamma} \Delta \omega_{N} \cdot \Delta \varphi = 0.
\end{align*}
{{For additional details about the limit passage in the semidiscrete formulation \eqref{semi1}, we refer the reader to the discussion in Section 7.2 of \cite{MuhaCanic13} and Section 2.7.2 of \cite{CanicLectureNotes}.}}

Using the strong convergence results established above, combined with the previously established weak convergence results  in Proposition \ref{prop:weak}, we can pass to the limit in all of the terms in the semidiscrete weak formulation except those involving time derivatives. However, we can handle these by a discrete integration by parts. For example, for the first integral, we can use a discrete integration by parts to obtain:
\begin{align*}
&\int_{0}^{T} \int_{\Omega_{f}} \left(1 + \frac{\tau_{\Delta t} \omega_{N}}{R}\right) \partial_{t} \overline{\bd{u}}_{N} \cdot \tilde{\bd{v}}_{N} \\
&\to -\int_{0}^{T} \int_{\Omega_{f}} \left(1 + \frac{\omega}{R}\right) \bd{u} \cdot \partial_{t} \tilde{\bd{v}} - \frac{1}{R} \int_{0}^{T} \int_{\Omega_{f}} (\partial_{t}\omega)\bd{u} \cdot \tilde{\bd{v}} - \int_{\Omega_{f}} \left(1 + \frac{\omega_{0}}{R}\right) \bd{u}(0) \cdot \tilde{\bd{v}}(0),
\end{align*}
where $\tilde{\bd{v}}_{N} = \bd{v} \circ \Phi^{\tau_{\Delta t} \omega_{N}}_{f}$ and $\tilde{\bd{v}} = \bd{v} \circ \Phi^{\omega}_{f}$ for $\bd{v} \in \mathcal{X}$. See for example pg. 79-81 in \cite{CanicLectureNotes}.

The limiting weak formulation holds for all velocity test functions in the smooth test space, which can be extended to the 
general test space $\mathcal{V}^{\omega}_{\text{test}}$ defined in \eqref{testspace} by using a density argument. 
Therefore, we have shown that the approximate weak solutions converge, up to a subsequence, to a weak solution to the regularized problem, as
stated in Theorem~\ref{MainThm1}.

This completes the main result of this manuscript, {{stated in Theorem~\ref{MainThm1},}} providing existence of a weak solution to the nonlinearly coupled,
regularized fluid-poroviscoelastic structure interaction problem, given in Definition~\ref{regularizedfixed}.

We conclude this section by making the important observation that the weak solution that we have constructed to the regularized FPSI problem satisfies the desired energy estimate. This will be important for showing weak-classical consistency in the next section, and can be shown easily by using the discrete energy estimate for the approximate solutions.

\begin{proposition}\label{energyestimate}{\bf{(Energy estimate for the limiting solution to the regularized problem.)}}
The weak solution $(\bd{u}, \bd{\eta}, p, \omega)$ constructed from the splitting scheme as the limit of approximate solutions satisfies the following energy estimate for almost every $t \in [0, T]$:
\begin{align}\label{energyinequality}
&\frac{1}{2}\int_{\Omega_{f}(t)} |\bd{u}|^{2} + \frac{1}{2}\rho_{b}\int_{\Omega_{b}} |\bd{\xi}|^{2} + \frac{1}{2} c_{0} \int_{\Omega_{b}} |p|^{2} + \mu_{e} \int_{\Omega_{b}} |\bd{D}(\bd{\eta})|^{2} 
\nonumber \\
&+ \frac{1}{2}\lambda_{e} \int_{\Omega_{b}} |\nabla \cdot \bd{\eta}|^{2} + \frac{1}{2} \rho_{p} \int_{\Gamma} |\zeta|^{2} + \frac{1}{2} \int_{\Gamma} |\Delta \omega|^{2} + 2\nu \int_{0}^{t} \int_{\Omega_{f}(s)} |\bd{D}(\bd{u})|^{2} \nonumber \\
&+ 2\mu_{v} \int_{0}^{t} \int_{\Omega_{b}} |\bd{D}(\bd{\xi})|^{2} + \lambda_{v} \int_{0}^{t} \int_{\Omega_{b}} |\nabla \cdot \bd{\xi}|^{2} + \kappa \int_{0}^{t} \int_{\Omega^{\delta}_{b}(s)} |\nabla p|^{2} + \beta \int_{0}^{t} \int_{\Gamma(s)} |(\zeta \bd{e}_{y} - \bd{u}) \cdot \bd{\tau})|^{2} \le E_{0},
\end{align} 
{{where $E_{0}$ is the initial energy of the problem.}}
\end{proposition}

\begin{proof}
The approximate solutions $(\bd{u}_{N}, \bd{\eta}_{N}, p_{N}, \omega_{N})$ satisfy the following energy inequality:
\begin{align*}
&\frac{1}{2}\int_{\Omega_{f, N}(t)} |\bd{u}_{N}|^{2} + \frac{1}{2}\rho_{b}\int_{\Omega_{b}} |\bd{\xi}_{N}|^{2} + \frac{1}{2} c_{0} \int_{\Omega_{b}} |p_{N}|^{2} + \mu_{e} \int_{\Omega_{b}} |\bd{D}(\bd{\eta}_{N})|^{2} \\
&+ \frac{1}{2}\lambda_{e} \int_{\Omega_{b}} |\nabla \cdot \bd{\eta}_{N}|^{2} + \frac{1}{2} \rho_{p} \int_{\Gamma} |\zeta_{N}|^{2} + \frac{1}{2} \int_{\Gamma} |\Delta \omega_{N}|^{2} + 2\nu\int_{0}^{t} \int_{\Omega_{f, N}(s)} |\bd{D}(\bd{u}_{N})|^{2} \\
&+ 2\mu_{v}\int_{0}^{t} \int_{\Omega_{b}} |\bd{D}(\bd{\xi}_{N})|^{2} + \lambda_{v} \int_{0}^{t} \int_{\Omega_{b}} |\nabla \cdot \bd{\xi}_{N}|^{2} + \kappa \int_{0}^{t} \int_{\Omega_{b, N}^{\delta}(s)} |\nabla p_{N}|^{2} + \beta \int_{0}^{t} \int_{\Gamma(s)} |(\zeta_{N}^{*} \bd{e}_{y} - \bd{u}_{N}) \cdot \bd{\tau}|^{2} \le E_{0}.
\end{align*}
{{By using the weak and weak-star convergences of the approximate solutions, stated in Proposition \ref{prop:weak} and {{the weak lower semicontinuity property of the norms}}, we can pass to the limit in the energy inequality, one recovers the energy inequality \eqref{energyinequality}.}}
\end{proof}

\section{{{Weak-classical}} consistency}\label{weakstrong}

We have now shown the existence of weak solutions to the regularized FPSI problem \eqref{deltaweakphysical}. However, it is not clear that the solutions to this regularized problem are physically relevant, since the regularized weak formulation is not equivalent to the original weak formulation without the regularization. However, we will demonstrate the following {{weak-classical}} consistency result: given a spatially and temporally smooth solution $(\bd{u}, \bd{\eta}, p, \omega)$ to the FPSI problem, then the weak solutions to the regularized problem with regularization parameter $\delta$, which we will denote by $(\bd{u}_{\delta}, \bd{\eta}_{\delta}, p_{\delta}, \omega_{\delta})$, converge to the smooth solution as $\delta \to 0$. 

\subsection{Notation}

Since we will have to use spatial convolution of the solution to the regularized problem $(\bd{u}_{\delta}, \bd{\eta}_{\delta}, p_{\delta}, \omega_{\delta})$,
and spatial convolution of the smooth solution $(\bd{u}, \bd{\eta}, p, \omega)$, we introduce the following notation to avoid additional superscripts
involving $\delta$. 
\begin{enumerate}
\item Recall that $(\bd{u}_{\delta}, \bd{\eta}_{\delta}, p_{\delta}, \omega_{\delta})$ denotes the weak solutions to the regularized problem \eqref{deltaweakphysical};
\item We will use $(\bd{u}, \bd{\eta}, p, \omega)$ to denote a {{spatially and temporally smooth solution}} to the original FPSI problem \eqref{Biot1}, \eqref{Biot2}, \eqref{plate}, \eqref{NS1}, \eqref{mass}-\eqref{pressurebalance};
\item We will use the superscript $\delta$ notation
\begin{equation*}
{\bd{\eta}}_{\delta}^\delta = (\bd{\eta}_{\delta})^{\delta} := \delta^{-2} \bd{\eta}_{\delta} * \sigma(x/\delta)
\end{equation*}
to denote the spatial convolution defined by \eqref{etadelta} of the weak solution to the regularized problem with the smooth convolution $\delta$ kernel;
\item Similarly, {{in the same spirit as in 3.,}} we will use 
\begin{equation}\label{tildeeta}
{\bd{\eta}}^\delta = (\bd{\eta})^{\delta} = \delta^{-2} \bd{\eta} * \sigma(x/\delta)
\end{equation}
to denote the spatial convolution of the {{classical}} solution $\bd{\eta}$ with the convolution kernel;
\item
We will use superscript $\delta$ to denote the physical Biot domain under the regularized displacement:
\begin{equation}\label{tildeomega}
{\Omega}_{b,\delta}^\delta(t) = (\bd{I} + {\bd{\eta}}_{\delta}^\delta(t))(\Omega_{b}).
\end{equation}
\end{enumerate}

\noindent
{\bf{Weak formulations reformulated.}} We note that even though the weak formulation \eqref{weakphysical} and the regularized weak formulation \eqref{deltaweakphysical} are stated up until a fixed final time $T$, we can reformulate the weak formulation for almost every time $t \in [0, T]$ by using a cutoff function (see for example the proof of Lemma \ref{weakcontinuity} in the appendix where this is done explicitly). 

Thus,  the {\bf{{{classical (temporally and spatially smooth)}}} solution} $(\bd{u}, \bd{\eta}, p, \omega)$ satisfies the following {\emph{non-regularized  weak formulation}}
 for almost all $t \in [0, T]$, for all test functions $(\bd{v}, \varphi, \bd{\psi}, r) \in \mathcal{V}_{\text{test}}$ with the (moving domain) test space $\mathcal{V}_{\text{test}}$ defined in \eqref{testspacemoving}:
\begin{align}\label{weaknoregularization}
&-\int_{0}^{t} \int_{\Omega_{f}(s)} \boldsymbol{u} \cdot \partial_{t}\boldsymbol{v} + \frac{1}{2} \int_{0}^{t} \int_{\Omega_{f}(s)} [((\boldsymbol{u} \cdot \nabla) \boldsymbol{u}) \cdot \boldsymbol{v} - ((\boldsymbol{u} \cdot \nabla)\boldsymbol{v}) \cdot \boldsymbol{u}] + \frac{1}{2} \int_{0}^{t} \int_{\Gamma_{1}(s)} (\boldsymbol{u} \cdot \boldsymbol{n} - 2\boldsymbol{\xi}_{1} \cdot \boldsymbol{n}) \boldsymbol{u} \cdot \boldsymbol{v} 
\nonumber \\
&+ 2\nu \int_{0}^{t} \int_{\Omega_{f}(s)} \boldsymbol{D}(\boldsymbol{u}) : \boldsymbol{D}(\boldsymbol{v}) + \int_{0}^{t} \int_{\Gamma_{1}(s)} \left(\frac{1}{2}|\boldsymbol{u}|^{2} - p\right)(\psi_{n} - v_{n}) + \beta \int_{0}^{t}  \int_{\Gamma_{1}(s)} (\bd{\xi} - \bd{u}) \cdot \bd{t} (\psi_{t} - v_{t}) 
\nonumber \\
&- \rho_{p} \int_{0}^{t} \int_{\Gamma} \partial_{t} \omega \cdot \partial_{t} \varphi + \int_{0}^{t} \int_{\Gamma} \Delta \omega \cdot \Delta \varphi - \rho_{b} \int_{0}^{t} \int_{\Omega_{b}} \partial_{t} \boldsymbol{\eta}_{1} \cdot \partial_{t}\boldsymbol{\psi} + 2\mu_{e} \int_{0}^{t} \int_{\Omega_{b}} \boldsymbol{D}(\boldsymbol{\eta}_{1}) : \boldsymbol{D}(\boldsymbol{\psi}) 
\nonumber \\
&+ \lambda_{e} \int_{0}^{t} \int_{\Omega_{b}} (\nabla \cdot \boldsymbol{\eta}_{1})(\nabla \cdot \boldsymbol{\psi}) + 2\mu_{v} \int_{0}^{t} \int_{\Omega_{b}} \bd{D}(\partial_{t}\bd{\eta}) : \bd{D}(\bd{\psi}) + \lambda_{v} \int_{0}^{t} \int_{\Omega_{b}} (\nabla \cdot \partial_{t} \bd{\eta})(\nabla \cdot \bd{\psi}) 
\nonumber \\
&- \alpha \int_{0}^{t} \int_{\Omega_{b}(s)} p \nabla \cdot \boldsymbol{\psi} - c_{0} \int_{0}^{t} \int_{\Omega_{b}} p \partial_{t}r - {{\alpha \int_{0}^{t} \int_{\Omega_{b}(s)} \frac{D}{Dt} \boldsymbol{\eta} \cdot \nabla r - \alpha \int_{0}^{t} \int_{\Gamma(s)} (\boldsymbol{\xi} \cdot \boldsymbol{n}) r}} 
\nonumber \\
&+ \kappa \int_{0}^{t} \int_{\Omega_{b}(s)} \nabla p \cdot \nabla r - {{\int_{0}^{t} \int_{\Gamma_{1}(s)} ((\boldsymbol{u} - \boldsymbol{\xi}) \cdot \boldsymbol{n}) r}}
\nonumber \\
&= -\int_{\Omega_{f}(t)} \boldsymbol{u}(t) \cdot \boldsymbol{v}(t) - \rho_{p} \int_{\Gamma} \zeta(t) \cdot \boldsymbol{\psi}(t) - \rho_{b} \int_{\Omega_{b}} \bd{\xi}(t) \cdot \boldsymbol{\psi}(t) - c_{0} \int_{\Omega_{b}} p(t) \cdot r(t) 
\nonumber \\
&+ \int_{\Omega_{f}(0)} \boldsymbol{u}_{0} \cdot \boldsymbol{v}(0) + \rho_{p} \int_{\Gamma} \beta_{0} \cdot \boldsymbol{\psi}(0) + \rho_{b} \int_{\Omega_{b}} \bd{\xi}_{0} \cdot \boldsymbol{\psi}(0) + c_{0} \int_{\Omega_{b}} p_{0} \cdot r(0).
\end{align}

Similarly,  the {\bf{solution to the regularized FPSI problem}} $(\bd{u}_{\delta}, \bd{\eta}_{\delta}, p_{\delta}, \omega_{\delta})$ satisfies the following {\emph{regularized weak formulation}} for every test function $(\boldsymbol{v}, \varphi, \boldsymbol{\psi}, r) \in \mathcal{V}_{\text{test}}$,
and for almost every $t \in [0, T_{\delta}]$ where the final time $T_{\delta}$ potentially depends on $\delta$:
\begin{align}\label{weakdelta}
&-\int_{0}^{t} \int_{\Omega_{f,\delta}(s)} \boldsymbol{u}_{\delta} \cdot \partial_{t}\boldsymbol{v} + \frac{1}{2} \int_{0}^{t} \int_{\Omega_{f,\delta}(s)} [((\boldsymbol{u}_{\delta} \cdot \nabla) \boldsymbol{u}_{\delta}) \cdot \boldsymbol{v} - ((\boldsymbol{u}_{\delta} \cdot \nabla)\boldsymbol{v}) \cdot \boldsymbol{u}_{\delta}] 
\nonumber \\
&+ \frac{1}{2} \int_{0}^{t} \int_{\Gamma_{\delta}(s)} (\boldsymbol{u}_{\delta} \cdot \boldsymbol{n} - 2\boldsymbol{\xi}_{\delta} \cdot \boldsymbol{n}) \boldsymbol{u}_{\delta} \cdot \boldsymbol{v} 
+ 2\nu \int_{0}^{t} \int_{\Omega_{f,\delta}(s)} \boldsymbol{D}(\boldsymbol{u}_{\delta}) : \boldsymbol{D}(\boldsymbol{v})\nonumber \\
&  + \int_{0}^{t} \int_{\Gamma_{\delta}(s)} \left(\frac{1}{2}|\boldsymbol{u}_{\delta}|^{2} - p_{\delta}\right)(\psi_{n} - v_{n})
+ \beta \int_{0}^{t} \int_{\Gamma_{\delta}(s)} (\bd{\xi}_{\delta} - \bd{u}_{\delta}) \cdot \bd{t} (\psi_{t} - v_{t}) 
\nonumber\\
&- \rho_{p} \int_{0}^{t} \int_{\Gamma} \partial_{t} \omega_{\delta} \cdot \partial_{t} \varphi + \int_{0}^{t} \int_{\Gamma} \Delta \omega_{\delta} \cdot \Delta \varphi - \rho_{b} \int_{0}^{t} \int_{\Omega_{b}} \partial_{t} \boldsymbol{\eta}_{\delta} \cdot \partial_{t}\boldsymbol{\psi} + 2\mu_{e} \int_{0}^{t} \int_{\Omega_{b}} \boldsymbol{D}(\boldsymbol{\eta}_{\delta}) : \boldsymbol{D}(\boldsymbol{\psi}) 
\nonumber\\
&+ \lambda_{e} \int_{0}^{t} \int_{\Omega_{b}} (\nabla \cdot \boldsymbol{\eta}_{\delta})(\nabla \cdot \boldsymbol{\psi}) + 2\mu_{v} \int_{0}^{t} \int_{\Omega_{b}} \bd{D}(\partial_{t} \bd{\eta}_{\delta}) : \bd{D}(\bd{\psi}) + \lambda_{v} \int_{0}^{t} \int_{\Omega_{b}} (\nabla \cdot \partial_{t}\bd{\eta}_{\delta}) (\nabla \cdot \bd{\psi}) 
\nonumber\\
&- \alpha \int_{0}^{t} \int_{{\Omega}_{b,\delta}^\delta(s)} p_{\delta} \nabla \cdot \boldsymbol{\psi} - c_{0} \int_{0}^{t} \int_{\Omega_{b}} p_{\delta} \partial_{t}r - \alpha \int_{0}^{t} \int_{{\Omega}_{b,\delta}^\delta(s)} \frac{D^{\delta}}{Dt} \boldsymbol{\eta}_{\delta} \cdot \nabla r - \alpha \int_{0}^{t} \int_{{\Gamma}_{\delta}(s)} (\boldsymbol{\xi}_{\delta} \cdot \boldsymbol{n}) r 
\nonumber\\
&+ \kappa \int_{0}^{t} \int_{{\Omega}_{b,\delta}^\delta(s)} \nabla p_{\delta} \cdot \nabla r - \int_{0}^{t} \int_{\Gamma_{\delta}(s)} ((\boldsymbol{u}_{\delta} - \boldsymbol{\xi}_{\delta}) \cdot \boldsymbol{n}) r 
\nonumber\\
&= -\int_{\Omega_{f,\delta}(t)} \boldsymbol{u}_{\delta}(t) \cdot \boldsymbol{v}(t) - \rho_{p} \int_{\Gamma} \zeta_{\delta}(t) \cdot \varphi(t) - \rho_{b} \int_{\Omega_{b}} \boldsymbol{\xi}_{\delta}(t) \cdot \boldsymbol{\psi}(t) - c_{0} \int_{\Omega_{b}} p_{\delta}(t) \cdot r(t) 
\nonumber\\
&+ \int_{\Omega_{f}(0)} \boldsymbol{u}_{0} \cdot \boldsymbol{v}(0) + \rho_{p} \int_{\Gamma} \beta_{0} \cdot \varphi(0) + \rho_{b} \int_{\Omega_{b}} \boldsymbol{\xi}_{0} \cdot \boldsymbol{\psi}(0) + c_{0} \int_{\Omega_{b}} p_{0} \cdot r(0),
\end{align}
where $\frac{D^{\delta}}{Dt}$ is the material derivative with respect to the regularized Biot displacement. We remark that while our existence proof in the previous sections holds for both a purely elastic and viscoelastic Biot medium, our {{weak-classical consistency}} result will hold in the specific case of a Biot \textit{poroviscoelastic} medium so that the viscoelasticity parameters $\mu_{v}$ and $\lambda_{v}$ are strictly positive, and hence, the plate velocity $\zeta_{\delta} \bd{e}_{y}$ in the weak formulation is equivalently the trace of the Biot medium velocity $\bd{\xi}_{\delta} \in L^{2}(0, T; H^{1}(\Omega_{b}))$ along $\Gamma$. 

\subsection{Statement of the result}
In the remainder of the manuscript, we will prove the {{weak-classical consistency}} result. Before stating the result, we need to introduce some additional notation. 
Namely, to prove the {{weak-classical consistency,}} we will subtract the weak formulations for the two solutions 
$\bd{u}$ and $\bd{u}_{\delta}$ and test formally with the difference of the two solutions $\bd{v} = \bd{u} - \bd{u}_{\delta}$. 
 However, the functions $\bd{u}$ and $\bd{u}_{\delta}$ are defined on different domains, and {{the test functions on the physical domains are required to be divergence-free, so we need to be able to transfer $\bd{u}_{\delta}$ to the physical domain for $\bd{u}$ in a way that preserves the divergence-free condition. To do this, we emphasize that we cannot use the usual ALE mapping $\Phi^{\omega}_{f}: \Omega_{f} \to \Omega_{f}(t)$ defined in \eqref{phif}. This is because given two structure displacements $\omega_{1}$ and $\omega_{2}$ which define two respective fluid domains $\Omega^{\omega_{1}}_{f}$ and $\Omega^{\omega_{2}}_{f}$ and a divergence-free function $\bd{u}_{1}$ on $\Omega^{\omega_{1}}_{f}$, the function $\bd{u} \circ \hat{\bd{\Phi}}^{\omega_{1}}_{f} \circ (\hat{\bd{\Phi}}^{\omega_{2}}_{f})^{-1}$ transformed via the ALE mapping is not necessarily divergence-free on $\Omega^{\omega_{2}}_{f}$. Therefore, we will have to use a different transformation to bring a divergence-free function defined on one fluid domain to a divergence-free function defined on another fluid domain.}}

For this purpose consider the two fluid domains
\begin{equation*}
\Omega_{f}(t) = \{(x, y) \in \mathbb{R}^{2} : 0 \le x \le L, -R \le y \le \omega(t, x)\},
\end{equation*}
\begin{equation*}
\Omega_{f,\delta}(t) = \{(x, y) \in \mathbb{R}^{2} : 0 \le x \le L, -R \le y \le \omega_{\delta}(t, x)\},
\end{equation*}
that are associated to the plate displacements $\omega$ and $\omega_{\delta}$.

We define a map between $\Omega_{f}(t)$ and $\Omega_{f,\delta}(t)$, and a transformation that sends functions on one domain to functions on the other domain as follows. Let  $\psi_{\delta}(t): \Omega_{f,\delta}(t) \to \Omega_{f}(t)$ be the mapping defined by
\begin{equation}\label{psi}
\psi_{\delta}(t, x, y) = (t, x, \gamma_{\delta}(t, x)(R + y) - R), \ {\rm{where}} \ 
\gamma_{\delta}(t, x) = \frac{R + \omega(t, x)}{R + \omega_{\delta}(t, x)}.
\end{equation}
This mapping, unfortunately,  does not preserve the divergence free condition. 
However, if we calculate the gradient of the composite mapped function we get 
\begin{equation}\label{gradidentity}
\nabla (\bd{u} \circ \psi_{\delta}) = [(\nabla \bd{u}) \circ \psi_{\delta}] J_{\delta}
\end{equation} 
where
\begin{equation}\label{J}
J_{\delta}(t, x, y) = \begin{pmatrix} 1 & 0 \\ (R + y) \partial_{x}\gamma_{\delta}(t, x) & \gamma_{\delta}(t, x) \\ \end{pmatrix}.
\end{equation}
Similarly, for the regularized problem we define
\begin{equation}\label{Jtilde}
\tilde{J}_{\delta} = J_{\delta} \circ \psi_{\delta}^{-1} = \begin{pmatrix} 1 & 0 \\ (R + y) \gamma_{\delta}^{-1} \partial_{x} \gamma_{\delta}(t, x) & \gamma_{\delta}(t, x) \\ \end{pmatrix}.
\end{equation}
These Jacobian matrices will now be used to define the transformations that map divergence free functions to divergence free functions.

\begin{definition}
{\bf{Part I:}} Given a divergence-free function $\bd{u}$ on $\Omega_{f}(t)$, the following transformation 
$\widehat{\phantom{u}} : {\bd{u}} \mapsto \widehat{\bd{u}} $ maps ${\bd{u}}$  to a divergence free function 
$\widehat{\bd{u}}$ on $\Omega_{f,\delta}(t)$:
\begin{equation}\label{uhat}
\widehat{\bd{u}} = \gamma_{\delta} J^{-1}_{\delta} \cdot (\bd{u} \circ \psi_{\delta}).
\end{equation}
\vskip 0.1in
\noindent
{\bf{Part II:}} 
Given a divergence-free function $\bd{u}_{\delta}$ on $\Omega_{f,\delta}(t)$ the following transformation 
$\widecheck{\phantom{u}}: \bd{u}_{\delta} \mapsto \widecheck{\bd{u}}_{\delta}$
maps $\bd{u}_{\delta}$ to a divergence free function $\widecheck{\bd{u}}_{\delta}$ on $\Omega_{f}(t)$:
\begin{equation}\label{ucheck}
\widecheck{\bd{u}}_{\delta} = \gamma_{\delta}^{-1} \tilde{J}_{\delta} \cdot (\bd{u}_{\delta} \circ \psi_{\delta}^{-1}).
\end{equation}
\vskip 0.1in
\noindent
\end{definition}

\begin{remark}
Both transformations preserve the trace of functions along $\Gamma$.
\end{remark}

Note that even though the definition of $\widehat{\bd{u}}$ depends on $\delta$, we will not explicitly notate this dependence, as $\delta$ will be clear from the context. 
We now state the {{weak-classical consistency}} result.
\begin{theorem}\label{weakstrongunique}{\bf{({{Weak-classical consistency}})}}
Let $(\bd{\eta}_{0}, \bd{\xi}_{0}, \omega_{0}, \zeta_{0}, p_{0}, \bd{u}_{0})$ be smooth initial data for the nonlinearly coupled FPSI problem
 \eqref{Biot1}, \eqref{Biot2}, \eqref{plate}, \eqref{NS1}, \eqref{mass}-\eqref{pressurebalance}.
Suppose $(\bd{\eta}, \omega, p, \bd{u})$ is a {{classical (temporally and spatially smooth)}} solution to this FPSI problem on the time interval $[0, T]$. 
Let $(\bd{\eta}_{\delta}, \omega_{\delta}, p_{\delta}, \bd{u}_{\delta})$ denote the weak solution to the regularized FPSI problem \eqref{deltaweakphysical}
with regularity parameter $\delta$.

Then the following holds true:
\begin{enumerate}
\item $(\bd{\eta}_{\delta}, \omega_{\delta}, p_{\delta}, \bd{u}_{\delta})$ {{is defined on the time interval $[0, T]$ for all $\delta > 0$, where the final time $T$ is \textit{independent of $\delta$}}};
\item The energy norm of the difference between the two solutions  $E_{\delta}(t)$ converges to zero as $\delta \to 0$, for all $t \in [0, T]$, where
{{
\begin{align}\label{energydiff}
E_{\delta}(t) &:= ||(\widehat{\bd{u}} - \bd{u}_{\delta})(t)||^{2}_{L^{2}(\Omega_{f,\delta}(t))} + \int_{0}^{t} ||\bd{D}(\widehat{\bd{u}} - \bd{u}_{\delta})(s)||^{2}_{L^{2}(\Omega_{f,\delta}(s))} ds 
\nonumber \\
&+ ||(\bd{\xi} - \bd{\xi}_{\delta})(t)||_{L^{2}(\Gamma)}^{2} + ||(\omega - \omega_{\delta})(t)||^{2}_{H^{2}(\Gamma)} + ||(\bd{\xi} - \bd{\xi}_{\delta})(t)||^{2}_{L^{2}(\Omega_{b})} 
\nonumber \\
&+ ||\bd{D}(\bd{\eta} - \bd{\eta}_{\delta})(t)||^{2}_{L^{2}(\Omega_{b})} + ||(\nabla \cdot (\bd{\eta} - \bd{\eta}_{\delta}))(t)||_{L^{2}(\Omega_{b})}^{2} + \int_{0}^{t} ||\bd{D}(\bd{\xi} - \bd{\xi}_{\delta})(s)||^{2}_{L^{2}(\Omega_{b})} ds 
\nonumber \\
&+ \int_{0}^{t} ||\nabla \cdot (\bd{\xi} - \bd{\xi}_{\delta})(s)||_{L^{2}(\Omega_{b})}^{2} ds + ||(p - p_{\delta})(t)||^{2}_{L^{2}(\Omega_{b})} + \int_{0}^{t} ||\nabla(p - p_{\delta})(s)||^{2}_{L^{2}({\Omega}^\delta_{b,\delta}(s))} ds.
\end{align}
}}
\end{enumerate}
\end{theorem}

\noindent {\emph{{Preview of the main steps of the proof of {{weak-classical consistency}.}}}}\label{strategy}
The proof is based on {{Gronwall's}} inequality for $E_{\delta}(t)$. However, there are several obstacles to applying Gronwall's inequality
 due to the fact that we are working on a moving domain problem.  We summarize those main obstacles, and the main ideas behind 
their resolution here. 

The main idea is to estimate the energy difference between $(\bd{u}, \bd{\eta}, p, \omega)$ and $(\bd{u}_{\delta}, \bd{\eta}_{\delta}, p_{\delta}, \omega_{\delta})$, defined in \eqref{energydiff} and obtain an estimate for $E_{\delta}(t)$ in terms of $E_{\delta}(0)$, the integral of $E_{\delta}(s)$ for times $s \in [0, t]$, and other terms that have sufficiently strong convergence in $\delta$ as $\delta \to 0$:
{{\begin{equation}\label{Gronwall0}
E_{\delta}(t) \le C\left(\int_{0}^{t} ||\nabla \bd{\eta} - \nabla {\bd{\eta}}^\delta||^{2}_{L^{2}(\Omega_{b})} ds + \int_{0}^{t} E_{\delta}(s) ds\right)
\end{equation}}}
and then apply Gronwall's inequality to obtain 
{{\begin{equation}\label{Gronwallresult}
E_{\delta}(t) \le C\delta^{3}e^{Ct},
\end{equation}}}
where $C$ is independent of $\delta$, and conclude that $E_{\delta}(t) \to 0$ as $\delta \to 0$. {{We remark that the factor of $\delta^{3}$ appearing in the Gronwall estimate comes from an estimate of the convergence rate of the spatial convolution $\bd{\eta}^{\delta}$ to $\bd{\eta}$ in $H^{1}(\Omega_{b})$, which we establish in the upcoming Lemma \ref{strongconv}.}}

%%%%%%%%%%%%

To do this, we will test the weak formulations for $\bd{u}$ and $\bd{u}_{\delta}$ with appropriate test functions and use the energy inequality {{\eqref{energyinequality} from Proposition \ref{energyestimate}}}.
More precisely, the main steps in the proof are:
\begin{enumerate}
\item Test the {{non-regularized weak formulation  \eqref{weaknoregularization}}}  for the {{classical solution}} $(\bd{u}, \bd{\eta}, p, \omega)$ with the
 ``difference" of $(\bd{u}, \partial_{t} \bd{\eta}, p, \partial_{t} \omega)$ and $(\bd{u}_{\delta}, \partial_{t} \bd{\eta}_{\delta}, p_{\delta}, \partial_{t} \omega_{\delta})$, where the notion of the difference between these two solutions will be made precise in Section \ref{GronwallSection};
\item Test the {{regularized weak formulation \eqref{weakdelta}}}  for $(\bd{u}_{\delta}, \bd{\eta}_{\delta}, p_{\delta}, \omega_{\delta})$ with $(\bd{u}, \partial_{t}\bd{\eta}, p, \partial_{t}\omega)$;
\item Rewrite the {{energy inequality \eqref{energyinequality}}} for $(\bd{u}_{\delta}, \bd{\eta}_{\delta}, p_{\delta}, \omega_{\delta})$ so that it parallels the terms in the weak formulation {{\eqref{weakdelta}}};
\item Combine the equations from Step 1, Step 2, and Step 3. This will give us an expression that we can analyze term by term in order to obtain estimate
{{\eqref{Gronwall0}}} for the energy difference $E_{\delta}(t)$. Details will be presented in Section~\ref{weakstrongestimate};
\item {{Now, we have that the inequality \eqref{Gronwall0} and the resulting Gronwall estimate \eqref{Gronwallresult} are proven locally in time, namely, on the interval $[0,T_\delta]$ along which certain boundedness assumptions on the Lagrangian map hold for the solution of the regularized problem. However, we want the estimate \eqref{Gronwallresult} to hold along the \textit{entire} time interval 
$[0,T]$, along which the {{classical solution}} is defined. Hence, we will use a bootstrap argument on boundedness assumptions of the Lagrangian map in order to propagate the estimate \eqref{Gronwall0} to the entire time interval $[0, T]$, see Section~\ref{bootstrap}.}}
\item Apply Gronwall's inequality to {{\eqref{Gronwall0}}} holding on $[0,T]$ to obtain the following bound for  $E_{\delta}(t)$:
\begin{equation*}
E_{\delta}(t) \le C\delta^{3}e^{Ct},
\end{equation*}
where $C$ is independent of $\delta$, and conclude that $E_{\delta}(t) \to 0$ as $\delta \to 0$.
\end{enumerate}

%Most of the proof, carried out in Section \ref{weakstrongestimate}, will involve estimating, term by term, the various quantities that arise from combining the weak formulations in Steps 1 and 2, and the energy estimate in Step 3. As noted in Step 4 above, this will allow us to obtain an inequality for the energy $E_{\delta}(t)$ that can be used to conclude the proof of Theorem \ref{weakstrongunique} by an application of Gronwall's inequality. 
Before we start with the proof of {{weak-classical}} consistency, we emphasize that there are two main {\bf{mathematical difficulties}} 
that need to be addressed in the proof:
\begin{enumerate}
\item In step 1 above, we want to test  \eqref{weaknoregularization} with the difference of $(\bd{u}, \partial_{t} \bd{\eta}, p, \partial_{t} \omega)$ and $(\bd{u}_{\delta}, \partial_{t} \bd{\eta}_{\delta}, p_{\delta}, \partial_{t}\omega_{\delta})$. This is formal because the test functions in $\mathcal{V}_{\text{test}}$, defined in \eqref{testspacemoving}, must be continuously differentiable in time, and furthermore, for the fluid velocities, the difference between $\bd{u}$ and $\bd{u}_{\delta}$ does not make sense, since these functions are defined on different fluid domains. Thus, we must carefully define which test functions we will use. This is addressed at the beginning of Section~\ref{GronwallSection} below.
\item As mentioned in step 5 above, the regularized weak formulation involves integrals on the physical time-dependent Biot domain ${\Omega}_{b,\delta}^\delta(t)$, which give an extra factor of $\det(\bd{I} + \nabla {\bd{\eta}}_{\delta}^\delta)$ in the integrand from the Jacobian, when the integrals are transferred to the fixed reference Biot domain $\Omega_{b}$. This factor cannot be estimated in the finite energy space where $\bd{\eta}_{\delta}$ is only bounded uniformly in $\delta$ in the function space $L^{\infty}(0, T; H^{1}(\Omega_{b}))$.  {{To obtain pointwise estimates of this term that hold on the time interval $[0,T]$, where $T$ is independent of $\delta$, 
we need to  use a \textit{bootstrap argument} to get from the local pointwise estimates on $[0,T_\delta]$, where $T_\delta$ depends on $\delta$, to the global, uniform estimates on $[0,T]$.}}
 This is addressed in Section \ref{bootstrap} below.
\end{enumerate}

\if 1 = 0
\textbf{Step 1:} \textit{Use bilinearity to expand out each of the terms.} Using bilinearity, we have that there are three quantities to estimate, which we will denote by $I_{1}$, $I_{2}$, and $I_{3}$. We have that 
\begin{equation*}
E(t) = I_{1}(t) + I_{2}(t) - 2I_{3}(t),
\end{equation*}
where 
\begin{multline*}
I_{1}(t) := ||\widehat{\bd{u}}_{1}||_{L^{2}(\Omega_{f,\delta}(t))}^{2} + \int_{0}^{t} ||\bd{D}(\widehat{\bd{u}}_{1})(s)||^{2}_{L^{2}(\Omega_{f,\delta}(s))} ds + ||\bd{\xi}(t)||_{L^{2}(\Gamma)}^{2} \\
+ ||\omega(t)||^{2}_{H^{2}(\Gamma)} + ||\bd{\xi}(t)||_{L^{2}(\Omega_{b})}^{2} + ||\bd{D}(\bd{\eta})||^{2}_{L^{2}(\Omega_{b})} + ||\nabla \cdot \bd{\eta}||_{L^{2}(\Omega_{b})} \\
+ \int_{0}^{t} ||\bd{D}(\bd{\xi})(s)||^{2}_{L^{2}(\Omega_{b})} ds + ||p(t)||_{L^{2}(\Omega_{b})}^{2} + \int_{0}^{t} ||\nabla p(s)||^{2}_{L^{2}(\tilde{\Omega}_{b, 2, \delta}(s))},
\end{multline*}
\begin{multline*}
I_{2}(t) := ||\bd{u}_{\delta}||_{L^{2}(\Omega_{f,\delta}(t))}^{2} + \int_{0}^{t} ||\bd{D}(\bd{u}_{\delta})(s)||^{2}_{L^{2}(\Omega_{f,\delta}(s))} ds + ||\bd{\xi}_{\delta}(t)||_{L^{2}(\Gamma)}^{2} \\
+ ||\omega_{\delta}(t)||^{2}_{H^{2}(\Gamma)} + ||\bd{\xi}_{\delta}(t)||_{L^{2}(\Omega_{b})}^{2} + ||\bd{D}(\bd{\eta}_{\delta})||^{2}_{L^{2}(\Omega_{b})} + ||\nabla \cdot \bd{\eta}_{\delta}||_{L^{2}(\Omega_{b})} \\
+ \int_{0}^{t} ||\bd{D}(\bd{\xi}_{\delta})(s)||^{2}_{L^{2}(\Omega_{b})} ds + ||p_{\delta}(t)||_{L^{2}(\Omega_{b})}^{2} + \int_{0}^{t} ||\nabla p_{\delta}(s)||^{2}_{L^{2}(\tilde{\Omega}_{b, 2, \delta}(s))},
\end{multline*}
\begin{multline*}
I_{3}(t) := (\widehat{\bd{u}}_{1}(t), \bd{u}_{\delta}(t))_{L^{2}(\Omega_{f,\delta}(t))} + \int_{0}^{t} (\bd{D}(\widehat{\bd{u}}_{1}), \bd{D}(\bd{u}_{\delta}))_{L^{2}(\Omega_{f,\delta}(s))} ds \\
+ (\bd{\xi}(t), \bd{\xi}_{\delta}(t))_{L^{2}(\Gamma)} + (\omega(t), \omega_{\delta}(t))_{H^{2}(\Gamma)} \\ 
+ (\bd{\xi}(t), \bd{\xi}_{\delta}(t))_{L^{2}(\Omega_{b})} + (\bd{D}(\bd{\eta}), \bd{D}(\bd{\eta}_{\delta}))_{L^{2}(\Omega_{b})} + (\nabla \cdot \bd{\eta}(t), \nabla \cdot \bd{\eta}_{\delta}(t))_{L^{2}(\Omega_{b})} \\
+ \int_{0}^{t} (\bd{D}(\bd{\xi})(s), \bd{D}(\bd{\xi}_{\delta})(s))_{L^{2}(\Omega_{b})} ds + (p(t), p_{\delta}(t))_{L^{2}(\Omega_{b})} + \int_{0}^{t} (\nabla p(s), \nabla p_{\delta}(s))_{L^{2}(\tilde{\Omega}_{b, 2, \delta}(s))} ds.
\end{multline*}
In particular, $I_{1}(t)$ contains terms that involve only the {{classical}} solution, $I_{2}(t)$ contains terms that involve only the weak solution, and $I_{3}(t)$ contains the cross terms. 
\fi 

\if 1 = 0

The idea is that we want to test in the weak formulations formally with the difference between $(\bd{u}, \bd{\eta}, p, \omega)$ and $(\bd{u}_{\delta}, \bd{\eta}_{\delta}, p_{\delta}, \omega_{\delta})$. However, we need to rigorously define what we mean by the difference of these solutions, since the time-dependent fluid domains associated with $\bd{u}$ and $\bd{u}_{\delta}$ are different. Hence, we must transform the fluid velocities $\bd{u}_{\delta}$ defined on $\Omega^{\omega_{\delta}}$ onto the fluid domain $\Omega^{\omega}$ defined by $\omega$ given by the strong solution, in such a way so that the divergence-free condition on the physical domains is preserved. 

In addition, the test functions must be at least continuously differentiable in time, and hence, before we use the ``difference of the solutions" as a test function, we must regularize in time. We will denote the regularization parameter in time by $\alpha$, and leave the notation so that $\delta$ is exclusively used as the regularization parameter for the spatial regularization of the entire FPSI problem. 

We resolve both of these issues by defining the following functions. Define the following two-by-two matrix:
\begin{equation*}
K(s, t, z, r) = 
\begin{pmatrix} 
\frac{\eta(s, z)}{\eta(t, z)} & 0 \\
-r\partial_{z}\left(\frac{\eta(s, z)}{\eta(t, z)}\right) & 1 \\
\end{pmatrix}
\end{equation*}
Let $j(t)$ denote a standard compactly supported function in one dimension, with $\text{supp}(j) \in (-1, 1)$, $\int_{\mathbb{R}} j dt = 1$ that is even, and decreasing for $j \ge 0$. We then scale $j$ as
\begin{equation*}
j_{\alpha}(t) := \frac{1}{\alpha}j\left(\frac{t}{\alpha}\right).
\end{equation*}
In addition, to convolve in time with parameter $\alpha$, we need to extend our solutions to solutions on the larger time interval $[-\alpha, T + \alpha]$, where $T$ is the final time of the given smooth weak solution. We extend as follows, using the same notation for the original solutions on the interval $[0, T]$ and for the extensions. First, we consider the fluid velocity and the poroelastic pressure.
\begin{equation*}
\bd{u}(t) = \bd{u}_{0}, \qquad \bd{u}_{\delta}(t) = \bd{u}_{0}, \qquad \text{ for } t \le 0,
\end{equation*}
\begin{equation*}
\bd{u}(t) = 0, \qquad \bd{u}_{\delta}(t) = 0, \qquad \text{ for } t \ge T.
\end{equation*}
We do a similar extension by the initial data for $t \le 0$ and by zero for $t \ge T$ for the pressure solution. 

Next, we consider the extension for the poroelastic structure displacement and the plate displacement.
\begin{equation*}
\eta(t) = \eta_{0}, \qquad \eta_{\delta}(0) = \eta_{0}, \qquad \text{ for } t \le 0,
\end{equation*}
\begin{equation*}
\eta(t) = \eta(T), \qquad \eta_{\delta}(T) = \eta_{\delta}(T), \qquad \text{ for } t \ge T.
\end{equation*}
For the plate displacement, we do a similar extension by the initial data and the final data for $t \le 0$ and $t \ge T$ respectively. This is because the poroelastic structure displacement and the plate displacement are both continuous quantities so both the initial and final values of these quantities can be interpreted in a pointwise sense. 

With these functions extended to functions on $t \in \mathbb{R}$, we can convolve in time as follows. For a function $f: (-\infty, \infty) \to X$ where $X$ is a fixed Banach space, we define $f_{\alpha}$ as
\begin{equation*}
f_{\alpha} = \int_{-\infty}^{\infty} f(s) j_{\alpha}(t - s) ds.
\end{equation*}
We note that $f_{\alpha}$ is smooth in time, and takes values in $X$. For functions defined on fixed Banach spaces for each point in time, such as the plate displacement, or the Biot poroelastic displacement, this method of convolution is sufficient.

However, for a function, such as the fluid velocity, which is defined on a moving domain $\Omega_{f}(t)$ which is time-dependent, such a convolution does not work, since the fluid velocity belongs to a different (time-dependent) Banach space for each time $t$. Thus, for the fluid velocity, we define
\begin{equation*}
\bd{u}_{\alpha}(z, r) = \int_{-\infty}^{\infty} K(s, t, z, r) \bd{u}\left(s, z, \frac{\eta(s, z)}{\eta(t, z)} r\right) j_{\alpha}(t - s) ds.
\end{equation*}
We note that this convolution has the essential property that $\bd{u}_{\alpha}$ is still divergence free. Unlike the standard convolution in time, it is not true that $\bd{u}_{\alpha}$ is smooth in time, but one can verify that $\bd{u}_{\alpha}$ is still differentiable in time. 

{{TODO: I have done all of the Gronwall estimates in the next sections without paying attention to the following two important subtleties: (1) the weak solutions $\bd{u}_{\delta}$ are not actually differentiable in time, so we must use this time convolution defined above and let $\alpha \to 0$ to get the estimates more rigorously, (2) we have assumed that the weak solutions $\bd{u}_{\delta}$ formally satisfy the energy inequality. One can probably prove that this is true, though there is probably a way to rewrite the argument to bypass this assumption, but this requires modifying the following arguments slightly.}}

\fi

\subsection{Gronwall's Inequality}\label{GronwallSection}
We show that the following Gronwall's inequality holds for almost all $t \in [0, T_{\delta}]$, where $T_\delta$ depends on $\delta$. Later on we will use a bootstrap argument to show that the {{weak-classical consistency}} holds uniformly, on the entire interval $[0,T]$ on which the {{classical}} solution exists. 

\begin{lemma}\label{GronwallLemma}{\bf{Gronwall's estimate.}} 
{{Suppose that there exist constants $c > 0$ and $C > 0$ which are independent of $\delta$ such that the following estimates hold for almost all $t \in [0, T_{\delta}]$: 
\begin{align}
\label{det}
\det &(\bd{I} + \nabla {\bd{\eta}}_{\delta}^\delta) \ge c > 0,
\\
\label{norm}
0 < c \le&  |\bd{I} + \nabla {\bd{\eta}}_{\delta}^\delta| \le C, \qquad \text{ pointwise in } \overline{\Omega_{b}},
\\
\label{grad}
&|\nabla {\bd{\eta}}_{\delta}^\delta| \le C, \qquad \text{ pointwise in } \overline{\Omega_{b}},
\end{align}
where the final time $T_{\delta}$ potentially depends on $\delta$.}} Furthermore, let $\bd{\eta}$ and ${\bd{\eta}}^\delta$ denote the smooth solution and its regularization, defined on $[0,T]$, and 
$E_\delta$ be the energy norm difference \eqref{energydiff}.
Then the following inequality hold:
\begin{equation}\label{Gronwall}
E_{\delta}(t) \le C\left(\int_{0}^{t} ||\nabla \bd{\eta} - \nabla {\bd{\eta}}^\delta||^{2}_{L^{2}(\Omega_{b})} ds + \int_{0}^{t} E_{\delta}({{s}}) ds\right),
\end{equation}
where $E_{\delta}(t)$ is defined by \eqref{energydiff}.
Furthermore, 
\begin{equation*}
E_{\delta}(t) \le C\delta^{3}e^{Ct}.
\end{equation*}
%as long as the three conditions \eqref{det}, \eqref{norm}, and \eqref{grad} hold, where the constant $C$ is independent of $\delta$. 
\end{lemma}

%By using the bootstrap argument described in Section \ref{bootstrap}, we complete the proof of the weak-classical consistency result described in the statement of Theorem \ref{weakstrongunique}.

To prove Gronwall's inequality, we want to test the non-regularized weak formulation formally with the difference between $(\bd{u}, \partial_{t}\bd{\eta}, p, \partial_{t}\omega)$ and $(\bd{u}_{\delta}, \partial_{t}\bd{\eta}_{\delta}, p_{\delta}, \partial_{t}\omega_{\delta})$. However, there are two reasons why this is not rigorously justified. First, $\partial_{t}\bd{\eta} - \partial_{t}\bd{\eta}_{\delta}$ is not a continuously differentiable function in time as is required for the test functions, and hence, we will use a \textit{convolution in time} and pass to the limit as the convolution parameter goes to zero. Second, the fluid velocities give an additional difficulty, as the fluid velocities are defined on \textit{time-dependent moving domains}. Thus, we must transfer the fluid velocities between different time-dependent domains in order to make sense of the ``difference" between $\bd{u}$ and $\bd{u}_{\delta}$ as a test function. Furthermore, the way in which we do this transformation and the way in which we perform the convolution in time must both respect the divergence-free nature of the fluid velocity on the time-dependent domain. We will address both of these difficulties as follows.

\vskip 0.1in
\noindent
{\bf{Construction of appropriate test functions $(\bd{u}, \partial_{t}\bd{\eta}, p, \partial_{t}\omega) - (\bd{u}_{\delta}, \partial_{t}\bd{\eta}_{\delta}, p_{\delta}, \partial_{t}\omega_{\delta})$:}}

\vskip 0.1in
{\bf{Difficulty 1: Lack of regularity in time.}} We address the first difficulty by defining a convolution in time. This will allow us to regularize $\partial_{t}(\bd{\eta} - \bd{\eta}_{\delta}) = \bd{\xi} - \bd{\xi}_{\delta}$, $p - p_{\delta}$, and $\partial_{t}(\omega - \omega_{\delta}) = \zeta - \zeta_{\delta}$ so that these functions are continuously differentiable in time. Since the classical solution is already continuously differentiable in time, we only need to regularize the weak solutions to the regularized problem. Because these differences are all defined on fixed domains, we can use a standard convolution in time. 

{\bf{Convolution in time.}} Let $j(\cdot): \mathbb{R} \to \mathbb{R}$ be a compactly supported even function with $\text{supp}(j) \subset [-1, 1]$ and $\displaystyle \int_{\mathbb{R}} j = 1$, and we define $j_{\ttt}(t) = \ttt^{-1}j(\ttt^{-1}t)$, where $\ttt > 0$ is the convolution parameter in time.

Consider $\ttt > 0$. Extend $\bd{\xi}_{\delta}$, $p_{\delta}$, and $\zeta_{\delta}$ to the larger interval $[-\ttt, T + \ttt]$ by reflecting across $t = 0$ and $t = T$. 
For example, define:
\begin{align*}
\bd{\xi}_{\delta}(t) &= \bd{\xi}_{\delta}(-t), \ \text{ for } t \in [-\ttt, 0],
\\
\bd{\xi}_{\delta}(t) &= \bd{\xi}_{\delta}(2T - t), \  \text{ for } t \in [T, T + \ttt].
\end{align*}
Convolution in time is then defined by:  
\begin{equation*}
(\bd{\xi}_{\delta})_{\ttt}(t) = \bd{\xi}_{\delta}(t, \cdot) * j_{\ttt} = \int_{\mathbb{R}} \bd{\xi}_{\delta}(s) j_{\ttt}(t - s) ds, \ {\rm{for}} \ t\in[0,T].
\end{equation*}
The convolutions $(p_{\delta})_{\ttt}$ and $(\zeta_{\delta})_{\ttt}$  are defined similarly. With these definitions we can now test with $\bd{\xi} - (\bd{\xi}_{\delta})_{\ttt}$, $p - (p_{\delta})_{\ttt}$, and $\zeta - (\zeta_{\delta})_{\ttt}$. 
\vskip 0.1in

{\bf{Difficulty 2: Velocities are defined on moving domains.}} Because the fluid velocities are defined on moving time-dependent domains, we cannot directly apply a convolution in time. We must first be able to transform fluid velocities from one domain to another, while preserving the divergence-free condition, and then convolve in time. 
The transformation of fluid velocities from one domain to another, while preserving the divergence-free condition, will be performed using the following matrix, {{which was inspired by the transformation introduced in \cite{InoueWakimoto}} (see also \cite{WeakStrongFSI})}:
\begin{equation}\label{Kdelta}
K(s, t, x, y) = 
\begin{pmatrix}
\frac{R + \omega(s, x)}{R + \omega(t, x)} & 0 \\ -(R + y)\partial_{x}\left(\frac{R + \omega(s, x)}{R + \omega(t, x)}\right) & 1 \\
\end{pmatrix}.
\end{equation}
This matrix has the following essential property: {{if $\bd{u}(x, y)$ is a divergence-free function on the domain $\Omega_{f}(s)$ defined by the structure displacement $\omega(s, x)$, then the function
\begin{equation*}
K(s, t, x, y) \bd{u}\left(x, \frac{R + \omega(s, x)}{R + \omega(t, x)} (R + y) - R \right)
\end{equation*}
is a divergence-free vector field on the domain $\Omega_{f}(t)$ defined by the structure displacement $\omega(t, x)$.}}
\vskip 0.1in
{\bf{Combined transformation of fluid velocities and convolution in time:}} We can now use this transformation to convolve in time, as follows. We extend $\bd{u}_{\delta}$ to $[-\ttt, T + \ttt]$ by reflection, as above, and define, for $t \in [0, T]$,
\begin{equation}\label{ualpha}
(\bd{u}_{\delta})_{\ttt}(t) = \int_{\mathbb{R}} K_{2, \delta}(s, t, x, y) \bd{u}_{\delta}\left(s, x, \frac{R + \omega_{\delta}(s, x)}{R + \omega_{\delta}(t, x)}(R + y) - R\right)j_{\ttt}(t - s) ds.
\end{equation}
For a divergence-free function $\bd{v}$, extended as above in time to $[-\ttt, T + \ttt]$, we can define $\bd{v}_{\ttt}$ on $\Omega_{f}(t)$ analogously by
\begin{equation*}
\bd{v}_{\ttt}(t) = \int_{\mathbb{R}} K_{1}(s, t, x, y) \bd{v}\left(s, x, \frac{R + \omega(s, x)}{R + \omega(t, x)}(R + y) - R\right) j_{\ttt}(t - s) ds.
\end{equation*}
Here, $K_{1}(s, t, x, y)$ and $K_{2, \delta}(s, t, x, y)$ are defined as $K(s, t, x, y)$ with the choices of $\omega = \omega$ and $\omega = \omega_{\delta}$ respectively. 
An example of such a function $\bd{v}$ which will be convenient to consider on $\Omega_{f}(t)$ is the function $\widecheck{\bd{u}}_{\delta}$ defined on $\Omega_{f}(t)$, which is the function $\bd{u}_{\delta}$ defined on $\Omega_{f,\delta}(t)$ transferred in a divergence-free manner, as described above, onto the domain $\Omega_{f}(t)$. Specifically,
\begin{equation*}
\widecheck{\bd{u}}_{\delta}(t, x, y) = \begin{pmatrix}
\frac{R + \omega_{\delta}(t, x)}{R + \omega(t, x)} & 0 \\ -(R + y)\partial_{x}\left(\frac{R + \omega_{\delta}(t, x)}{R + \omega(t, x)}\right) & 1 \\
\end{pmatrix} \cdot {{\bd{u}_{\delta}}}\left({{t, x, \frac{R + \omega_{\delta}(t, x)}{R + \omega(t, x)}(R + y) - R}}\right).
\end{equation*}

We present the main properties of $(\bd{u}_{\delta})_{\ttt}$ in the proposition below,  which are a specific case of Lemma 2.6 in \cite{WeakStrongFSI}.

\begin{proposition}\label{alphaconvprop}
Fix an arbitrary $\delta > 0$. 
Given $\bd{u}_{\delta} \in L^{2}(0, T; H^{1}(\Omega_{b}(t))$ and $\omega, \omega_{\delta} \in H_{0}^{2}(\Gamma)$, 
the following properties hold:
\begin{itemize}
\item Divergence-free condition: $\text{div}\left[(\bd{u}_{\delta})_{\ttt}\right] = 0$ and $\text{div}[(\widecheck{\bd{u}}_{\delta})_{\ttt}] = 0$,  $\forall \ttt > 0$ and $\forall t \in [0, T]$;
\\
\item {{Convergence properties:
\begin{align*}
&(\bd{u}_{\delta})_{\ttt} \to \bd{u}_{\delta} \qquad \text{ strongly in } L^{p}(0, T; L^{q}(\Omega_{f,\delta}(t))), \qquad \text{ for all } p \in [1, \infty), q \in [1, 2),
\\
&(\widecheck{\bd{u}}_{\delta})_{\ttt} \to \widecheck{\bd{u}}_{\delta} \qquad \text{ strongly in } L^{p}(0, T; L^{q}(\Omega_{f, 1}(t))), \qquad \text{ for all } p \in [1, \infty), q \in [1, 2),
\\
&(\bd{u}_{\delta})_{\ttt} \rightharpoonup \bd{u}_{\delta} \qquad \text{ weakly in } L^{2}(0, T; W^{1, p}(\Omega_{f,\delta}(t))), \qquad \text{ for all } p \in [1, 2),
\\
&(\widecheck{\bd{u}}_{\delta})_{\ttt} \rightharpoonup \widecheck{\bd{u}}_{\delta} \qquad \text{ weakly in } L^{2}(0, T; W^{1, p}(\Omega_{f, 1}(t))), \qquad \text{ for all } p \in [1, 2).
\end{align*}}}
\end{itemize}
\end{proposition}

\begin{proof}\label{weakstrongestimate}{\bf{(Proof of Gronwall's estimate.)}}

We begin by testing the weak formulation \eqref{weaknoregularization} for the classical solution $(\bd{u}, \bd{\eta}, p, \omega)$ to the original non-regularized problem with 
\begin{equation}\label{test1}
\bd{v} = \bd{u} - (\widecheck{\bd{u}}_{\delta})_{\ttt}, \quad \varphi = \zeta - (\zeta_{\delta})_{\ttt}, \quad \bd{\psi} = \bd{\xi} - (\bd{\xi}_{\delta})_{\ttt}, \quad r = p - (p_{\delta})_{\ttt},
\end{equation}
and then test the regularized weak formulations \eqref{weakdelta} for the weak solutions $(\bd{u}_{\delta}, \bd{\eta}_{\delta}, p_{\delta}, \omega_{\delta})$ with
\begin{equation}\label{test2delta}
\bd{v} = \widehat{\bd{u}}, \quad \varphi = \zeta, \quad \bd{\psi} = \bd{\xi}, \quad r = p.
\end{equation}
Next, we rewrite the energy estimate in Proposition \ref{energyestimate}, which holds for the function $\bd{u}_{\delta}$, 
 in a more convenient form by adding extra terms that will cancel out, in order to have the energy inequality parallel the weak formulation term by term. In particular, we have that for almost every $t \in [0, T_{\delta}]$,
\begin{multline}\label{energydelta}
\frac{1}{2}\int_{\Omega_{f,\delta}(t)} |\bd{u}_{\delta}|^{2} + \frac{1}{2} \int_{0}^{t} \int_{\Omega_{f}(s)} [((\bd{u}_{\delta} \cdot \nabla) \bd{u}_{\delta}) \cdot \bd{u}_{\delta} - ((\bd{u}_{\delta} \cdot \nabla) \bd{u}_{\delta}] + \frac{1}{2} \int_{0}^{t} \int_{\Gamma_{\delta}(s)} (\bd{u}_{\delta} \cdot \bd{n} - 2\bd{\xi}_{\delta} \cdot \bd{n})\bd{u}_{\delta} \cdot \bd{u}_{\delta} \\
+ 2\nu \int_{0}^{t} \int_{\Omega_{f,\delta}(s)} |\bd{D}(\bd{u}_{\delta})|^{2} + \int_{0}^{t} \int_{\Gamma_{\delta}(s)} \left(\frac{1}{2} |\bd{u}_{\delta}|^{2} - p_{\delta}\right) (\bd{\xi}_{\delta} - \bd{u}_{\delta}) \cdot \bd{n} + \beta \int_{0}^{t} \int_{\Gamma_{\delta}(s)} |(\bd{\xi}_{\delta} - \bd{u}_{\delta}) \cdot \bd{t})|^{2} \\
+ \frac{1}{2} \rho_{p} \int_{\Gamma} |\bd{\xi}_{\delta}|^{2} + \frac{1}{2} \int_{\Gamma} |\Delta \omega_{\delta}|^{2} + \frac{1}{2}\rho_{b}\int_{\Omega_{b}} |\bd{\xi}_{\delta}|^{2} + \mu_{e} \int_0^t\int_{\Omega_{b}} |\bd{D}(\bd{\eta}_{\delta})(s)|^{2} \\
+ \frac{1}{2}\lambda_{e}  \int_{\Omega_{b}} |\nabla \cdot \bd{\eta}_{\delta}(s)|^{2} + 2\mu_{v} \int_{0}^{t} \int_{\Omega_{b}} |\bd{D}(\bd{\xi}_{\delta})|^{2} + \lambda_{v} \int_{0}^{t} \int_{\Omega_{b}} |\nabla \cdot \bd{\xi}_{\delta}|^{2} \\
- \aalpha \int_{0}^{t} \int_{{\Omega}^\delta_{b,\delta}(s)} p_{\delta} \nabla \cdot \bd{\xi}_{\delta} + \frac{1}{2} c_{0} \int_0^t\int_{\Omega_{b}} |p_{\delta}(s)|^{2} - \aalpha \int_{0}^{t} \int_{{\Omega}^\delta_{b,\delta}(s)} \frac{{D}^\delta}{Dt} \bd{\eta}_{\delta} \cdot \nabla p_{\delta} - \aalpha \int_{0}^{t} \int_{{\Gamma}^\delta_{\delta}(s)} (\bd{\xi}_{\delta} \cdot \bd{n})p_{\delta} \\
+ \kappa \int_{0}^{t} \int_{{\Omega}^\delta_{b,\delta}(s)} |\nabla p_{\delta}|^{2} - \int_{0}^{t} \int_{\Gamma_{\delta}(s)} ((\bd{u}_{\delta} - \bd{\xi}_{\delta}) \cdot \bd{n}) p_{\delta} \le \frac{1}{2}\int_{\Omega_{f}(0)} |\bd{u}_{0}|^{2} + \frac{1}{2} \rho_{p} \int_{\Gamma} |\bd{\xi}_{0}|^{2} \\
+ \frac{1}{2} \int_{\Gamma} |\Delta \omega_{0}|^{2} + \frac{1}{2}\rho_{b}\int_{\Omega_{b}} |\bd{\xi}_{0}|^{2} + \mu_{e} \int_{\Omega_{b}} |\bd{D}(\bd{\eta}_{0})|^{2} + \frac{1}{2}\lambda_{e} \int_{\Omega_{b}} |\nabla \cdot \bd{\eta}_{0}|^{2} + \frac{1}{2} c_{0} \int_{\Omega_{b}} |p_{0}|^{2}.
\end{multline}
Finally, we combine the weak formulation for $\bd{u}$ tested with \eqref{test1}, subtract the regularized weak formulation for $\bd{u}_{\delta}$ tested with \eqref{test2delta}, and add the energy estimate \eqref{energydelta} for $\bd{u}_{\delta}$ to obtain 
an expression of the form
\begin{equation}\label{sumT}
\sum_{i = 1}^{18} T_{i} \le 0,
\end{equation}
where the terms $T_{i}$ are given below. 
We have to estimate each term, and the combined estimate will give the Gronwall's inequality \eqref{Gronwall}. 
To make this section more concise, we summarize the final estimates here, and
present details of the derivation of these terms and the estimates in Appendix~\ref{appendix2}. 

{{As a notational note, in many of the estimates on the terms $T_{i}$ that follow, we will use Cauchy's inequality with $\epsilon$ often: $\displaystyle |ab| \le \epsilon |a|^{2} + C(\epsilon) |b|^{2}$ where $\epsilon > 0$ is any parameter and $C(\epsilon)$ is a constant that depends on the final choice of $\epsilon > 0$. In the inequalities that appear, $\epsilon > 0$ will hence be a parameter appearing on dissipative terms that will, at the conclusion of the estimates, be chosen small enough so that the dissipative terms from the estimates on $T_{i}$ with $\epsilon$ can be absorbed by the dissipative terms in \eqref{energydiff} to give the final inequality \eqref{Gronwall0}.}}

\vskip 0.1in
\noindent
{\bf{Term T1.}}  Term $T_1$ is defined as follows:
{{
\begin{align}
T_{1} = &-\int_{0}^{t} \int_{\Omega_{f}(s)} \bd{u} \cdot \partial_{t} \left[\bd{u} - (\widecheck{\bd{u}}_{\delta})_{\ttt}\right] - \frac{1}{2} \int_{0}^{t} \int_{\Gamma(s)} (\bd{\xi} \cdot \bd{n}) \bd{u} \cdot [\bd{u} - (\widecheck{\bd{u}}_{\delta})_{\ttt}] + \int_{\Omega_{f}(t)} \bd{u}(t) \cdot [\bd{u} - (\widecheck{\bd{u}}_{\delta})_{\ttt}](t) 
\\
& - \int_{\Omega_{f}(0)} \bd{u}(0) \cdot [\bd{u} - (\widecheck{\bd{u}}_{\delta})_{\ttt}](0) - \int_{0}^{t} \int_{\Omega_{f,\delta}(s)} \bd{u}_{\delta} \cdot \partial_{t} \widehat{\bd{u}} - \frac{1}{2} \int_{0}^{t} \int_{\Gamma_{\delta}(s)} (\bd{\xi}_{\delta} \cdot \bd{n}_{\delta}) \bd{u}_{\delta} \cdot \widehat{\bd{u}} 
\nonumber\\
&+ \int_{\Omega_{f,\delta}(t)} \bd{u}_{\delta}(t) \cdot \widehat{\bd{u}}(t) - \int_{\Omega_{f}(0)} \bd{u}_{\delta}(0) \cdot \widehat{\bd{u}}(0) + \frac{1}{2} \int_{\Omega_{f,\delta}(t)} |\bd{u}_{\delta}(t)|^{2} - \frac{1}{2} \int_{\Omega_{f,\delta}(0)} |\bd{u}_{0}|^{2}.
\label{T1}
\end{align}}}
This term is estimated so that after taking the limit as $\nu \to 0$, the contribution of this term becomes
\begin{equation*}
T_{1} = \frac{1}{2} \int_{\Omega_{f,\delta}(t)} |(\widehat{\bd{u}} - \bd{u}_{\delta})(t)|^{2} + {{R_{1}}},
\end{equation*}
where
{{
\begin{align*}
{{|R_{1}|}} \le &\epsilon \int_{0}^{t} ||\widehat{\bd{u}} - \bd{u}_{\delta}||^{2}_{H^{1}(\Omega_{f,\delta}(s))} \\
&+ C(\epsilon) \left(\int_{0}^{t} ||\omega - \omega_{\delta}||^{2}_{H^{2}(\Gamma)} + \int_{0}^{t} ||\partial_{t}\omega - \partial_{t}\omega_{\delta}||_{L^{2}(\Gamma)}^{2} + \int_{0}^{t} ||\widehat{\bd{u}} - \bd{u}_{\delta}||_{L^{2}(\Omega_{f,\delta}(s))}^{2}\right).
\end{align*}}}

\vskip 0.1in
\noindent
{\bf{Term T2.}} Term $T_2$ is defined as follows:
{{
\begin{align}
T_{2} = &\frac{1}{2} \int_{0}^{t} \int_{\Omega_{f}(s)} ((\bd{u} \cdot \nabla) \bd{u}) \cdot [\bd{u} - (\widecheck{\bd{u}}_{\delta})_{\ttt}] - \frac{1}{2} \int_{0}^{t} \int_{\Omega_{f}(s)} (\bd{u} \cdot \nabla) [\bd{u} - (\widecheck{\bd{u}}_{\delta})_{\ttt}] \cdot \bd{u} 
\nonumber\\
&- \frac{1}{2} \int_{0}^{t} \int_{\Omega_{f,\delta}(s)} ((\bd{u}_{\delta} \cdot \nabla) \bd{u}_{\delta}) \cdot (\widehat{\bd{u}} - \bd{u}_{\delta}) + \frac{1}{2} \int_{0}^{t} \int_{\Omega_{f,\delta}(s)} ((\bd{u}_{\delta} \cdot \nabla) (\widehat{\bd{u}} - \bd{u}_{\delta})) \cdot \bd{u}_{\delta}.
\label{T2}
\end{align}
}}

After taking the limit $\nu\to 0$, term $T_2$  can be estimated as follows:
{{
\begin{equation*}
|T_{2}| \le \epsilon \int_{0}^{t} ||\nabla(\widehat{\bd{u}} - \bd{u}_{\delta})||_{L^{2}(\Omega_{f,\delta}(s))}^{2} + C(\epsilon)\left(\int_{0}^{t} ||\omega - \omega_{\delta}||^{2}_{H^{2}(\Gamma)} + \int_{0}^{t} ||\widehat{\bd{u}} - \bd{u}_{\delta}||_{L^{2}(\Omega_{f,\delta}(s))}^{2}\right).
\end{equation*}
}}

\vskip 0.1in
\noindent
{\bf{Term T3.}} Term $T_3$ is defined as follows:
{{
\begin{align}
T_{3} = &\frac{1}{2} \int_{0}^{t} \int_{\Gamma(s)} (\bd{u} \cdot \bd{n} - \bd{\xi} \cdot \bd{n}) \bd{u} \cdot [\bd{u} - (\widecheck{\bd{u}}_{\delta})_{\ttt}] - \frac{1}{2} \int_{0}^{t} \int_{\Gamma_{\delta}(s)} (\bd{u}_{\delta} \cdot \bd{n}_{\delta} - \bd{\xi}_{\delta} \cdot \bd{n}_{\delta}) \bd{u}_{\delta} \cdot \widehat{\bd{u}} 
\nonumber \\
&+ \frac{1}{2} \int_{0}^{t} \int_{\Gamma(s)} |\bd{u}|^{2} (\bd{\xi} \cdot \bd{n} - \bd{u} \cdot \bd{n}) - \frac{1}{2} \int_{0}^{t} \int_{\Gamma(s)} |\bd{u}|^{2} [(\bd{\xi}_{\delta})_{\ttt} \cdot \bd{n} - (\widecheck{\bd{u}}_{\delta})_{\ttt} \cdot \bd{n}] - \frac{1}{2} \int_{0}^{t} \int_{\Gamma_{\delta}(s)} |\bd{u}_{\delta}|^{2} (\bd{\xi} \cdot \bd{n}_{\delta} - \widehat{\bd{u}} \cdot \bd{n}_{\delta}).
\label{T3}
\end{align}
}}

After taking the limit $\nu\to 0$, term $T3$ can be estimated as follows:
{{\begin{align*}
|T_{3}| \le &\epsilon \int_{0}^{t} ||\widehat{\bd{u}} - \bd{u}_{\delta}||_{H^{1}(\Omega_{f,\delta}(s))}^{2} \\
&+ C(\epsilon) \left(\int_{0}^{t} ||\omega - \omega_{\delta}||^{2}_{H^{2}(\Gamma)} + \int_{0}^{t} ||\bd{\xi} - \bd{\xi}_{\delta}||^{2}_{L^{2}(\Gamma)} + \int_{0}^{t} ||\widehat{\bd{u}} - \bd{u}_{\delta}||_{L^{2}(\Omega_{f,\delta}(s))}^{2}\right).
\end{align*}
}}

\vskip 0.1in
\noindent
{\bf{Term T4.}} Term $T_4$ is defined as follows:
{{
\begin{align}
T_{4} = & 2\nu \int_{0}^{t} \int_{\Omega_{f}(s)} \bd{D}(\bd{u}) : \bd{D}(\bd{u} - (\widecheck{\bd{u}}_{\delta})_{\ttt}) - 2\nu \int_{0}^{t} \int_{\Omega_{f,\delta}(s)} \bd{D}(\bd{u}_{\delta}) : \bd{D}(\widehat{\bd{u}} - \bd{u}_{\delta}).
\label{T4}
\end{align}
}}

After taking the limit $\nu\to 0$, term $T4$ can be estimated as follows:
{{
\begin{equation*}
T_{4} = 2\nu \int_{0}^{t} \int_{\Omega_{f,\delta}(s)} |\bd{D}(\widehat{\bd{u}} - \bd{u}_{\delta})|^{2} + {{R_{4}}},
\end{equation*}
}}
where
{{
\begin{equation*}
{{|R_{4}|}} \le \epsilon \int_{0}^{t} ||\bd{D}(\widehat{\bd{u}} - \bd{u}_{\delta})||^{2}_{L^{2}(\Omega_{f,\delta}(s))} + C(\epsilon) \left(\int_{0}^{t} ||\omega - \omega_{\delta}||^{2}_{H^{2}(\Gamma)} + \int_{0}^{t} ||\widehat{\bd{u}} - \bd{u}_{\delta}||_{L^{2}(\Omega_{f,\delta}(s))}^{2}\right). 
\end{equation*}
}}

\vskip 0.1in
\noindent
{\bf{Term T5}.}  Term $T_5$ is defined as follows:
{{
\begin{multline*}
T_{5} = \beta\int_{0}^{t} \int_{\Gamma(s)} (\bd{\xi} - \bd{u}) \cdot \bd{\tau}(s) [(\bd{\xi} - (\bd{\xi}_{\delta})_{\ttt}) \cdot \bd{\tau}(s)  - (\bd{u} - (\widecheck{\bd{u}}_{\delta})_{\ttt}) \cdot \bd{\tau}(s)] \\
 - \beta\int_{0}^{t} \int_{\Gamma_{\delta}(s)} (\bd{\xi}_{\delta} - \bd{u}_{\delta}) \cdot \bd{\tau}_{\delta}(s) [(\bd{\xi} - \bd{\xi}_{\delta}) \cdot \bd{\tau}_{\delta}(s) - (\widehat{\bd{u}} - \bd{u}_{\delta}) \cdot \bd{\tau}_{\delta}(s)],
\end{multline*}
where $\bd{\tau}(s)$ is the unit tangent vector to $\Gamma(s)$ and $\bd{\tau}_{\delta}(s)$ is the unit tangent vector to $\Gamma_{\delta}(s)$. 
}}

After taking the limit $\nu\to 0$, term $T_5$ can be estimated as follows:
{{
\begin{equation*}
T_{5} = \beta \int_{0}^{t} \int_{\Gamma_{\delta}(s)} |(\bd{\xi} - \bd{\xi}_{\delta}) \cdot \bd{\tau}_{\delta}(s) - (\widehat{\bd{u}} - \bd{u}_{\delta}) \cdot \bd{\tau}_{\delta}(s)|^{2} + R_{5},
\end{equation*}
}}
where
{{
\begin{equation*}
|R_{5}| \le \epsilon \int_{0}^{t} ||\bd{D}(\widehat{\bd{u}} - \bd{u}_{\delta})||_{L^{2}(\Omega_{f,\delta}(s))} + C(\epsilon)\left(\int_{0}^{t} ||\omega - \omega_{\delta}||_{H^{2}(\Gamma)}^{2} + \int_{0}^{t} ||\bd{\xi} - \bd{\xi}_{\delta}||_{L^{2}(\Gamma)}^{2}\right).
\end{equation*}
}}

\vskip 0.1in
\noindent
{\bf{Terms T6-T8.}} Terms $T_6$-$T_8$ are defined as follows:
{{
\begin{align}
T_{6} = &-\rho_{p} \int_{0}^{t} \int_{\Gamma} \zeta \cdot \partial_{t}\left[\zeta - (\zeta_{\delta})_{\ttt}\right] + \rho_{p} \int_{\Gamma} \zeta(s) \cdot  \left[\zeta(t) - (\zeta_{\delta})_{\ttt}(t)\right] - \rho_{p} \int_{\Gamma} \zeta(0) \cdot \left[\zeta(0) - (\zeta_{\delta})_{\ttt}(0)\right] 
\nonumber \\
&+ \rho_{p} \int_{0}^{t} \int_{\Gamma} \zeta_{\delta} \cdot \partial_{t} \zeta - \rho_{p} \int_{\Gamma} \zeta_{\delta}(t) \cdot \zeta(t) + \rho_{p} \int_{\Gamma} \zeta_{\delta}(0) \cdot \zeta(0) + \frac{1}{2}\rho_{p} \int_{\Gamma} |\zeta_{\delta}(t)|^{2} - \frac{1}{2} \rho_{p} \int_{\Gamma} |\zeta_{0}|^{2}.
\label{T6}
\\
T_{7} = &\int_{0}^{t} \int_{\Gamma} \Delta \omega \cdot \Delta \left[\zeta - (\zeta_{\delta})_{\ttt}\right] - \int_{0}^{t} \int_{\Gamma} \Delta \omega_{\delta} \cdot \Delta \zeta + \frac{1}{2} \int_{\Gamma} |\Delta \omega_{\delta}(t)|^{2} - \frac{1}{2} \int_{\Gamma} |\Delta \omega_{0}|^{2}.
\\
T_{8} = &-\rho_{b} \int_{0}^{t} \int_{\Omega_{b}} \partial_{t}\bd{\eta} \cdot \partial_{t}\left[\bd{\xi} - (\bd{\xi}_{\delta})_{\ttt}\right] + \rho_{b} \int_{\Omega_{b}} \bd{\xi}(t) \cdot \left[\bd{\xi}(s) - (\bd{\xi}_{\delta})_{\ttt}(s)\right] - \rho_{b} \int_{\Omega_{b}} \bd{\xi}(0) \cdot \left[\bd{\xi}(0) - (\bd{\xi}_{\delta})_{\ttt}(0)\right] 
\nonumber \\
&+ \rho_{b} \int_{0}^{t} \int_{\Omega_{b}} \partial_{t}\bd{\eta}_{\delta} \cdot \partial_{t} \bd{\xi} - \rho_{b} \int_{\Omega_{b}} \bd{\xi}_{\delta}(t) \cdot \bd{\xi}(t) + \rho_{b} \int_{\Omega_{b}} \bd{\xi}_{\delta}(0) \cdot \bd{\xi}(0) + \frac{1}{2} \rho_{b} \int_{\Omega_{b}} |\bd{\xi}_{\delta}(t)|^{2} - \frac{1}{2} \rho_{b} \int_{\Omega_{b}} |\bd{\xi}_{0}|^{2}.
\end{align}
}}

After taking the limit $\nu\to 0$,  the terms $T_6$-$T_8$ become:

{{
\begin{equation*}
T_{6} = \frac{1}{2} \rho_{p} \int_{\Gamma} |(\zeta - \zeta_{\delta})(t)|^{2}, \ \ \ 
T_{7} = \frac{1}{2} \int_{\Gamma} |\Delta (\omega - \omega_{\delta})(t)|^{2}, \ \ \ T_{8} = \frac{1}{2} \rho_{b} \int_{\Omega_{b}} |(\bd{\xi} - \bd{\xi}_{\delta})(t)|^{2}.
\end{equation*}
}}

\vskip 0.1in
\noindent
{\bf{Terms T9-T12.}} Terms $T_9$-$T_{12}$ are defined as follows:
{{
\begin{align}
T_{9} = &2\mu_{e} \int_{0}^{t} \int_{\Omega_{b}} \bd{D}(\bd{\eta}) : \bd{D}\left[\bd{\xi} - (\bd{\xi}_{\delta})_{\ttt}\right] - 2\mu_{e} \int_{0}^{t} \int_{\Omega_{b}} \bd{D}(\bd{\eta}_{\delta}) : \bd{D}(\bd{\xi}) + \mu_{e} \int_{\Omega_{b}} |\bd{D}(\bd{\eta}_{\delta})(t)|^{2} - \mu_{e} \int_{\Omega_{b}} |\bd{D}(\bd{\eta}_{0})|^{2}.
\nonumber \\
T_{10} = &\lambda_{e} \int_{0}^{t} \int_{\Omega_{b}} (\nabla \cdot \bd{\eta}) \left(\nabla \cdot \left[\bd{\xi} - (\bd{\xi}_{\delta})_{\ttt}\right]\right) - \lambda_{e} \int_{0}^{t} \int_{\Omega_{b}} (\nabla \cdot \bd{\eta}_{\delta}) (\nabla \cdot \bd{\xi}) + \frac{1}{2} \lambda_{e} \int_{\Omega_{b}} |\nabla \cdot \bd{\eta}_{\delta}(t)|^{2} - \frac{1}{2} \lambda_{e} \int_{\Omega_{b}} |\nabla \cdot \bd{\eta}_{0}|^{2}.
\nonumber\\
T_{11} = &2\mu_{v} \int_{0}^{t} \int_{\Omega_{b}} \bd{D}(\bd{\xi}) : \bd{D}\left[\bd{\xi} - (\bd{\xi}_{\delta})_{\ttt}\right] - 2\mu_{v} \int_{0}^{t} \int_{\Omega_{b}} \bd{D}(\bd{\xi}_{\delta}) : \bd{D}(\bd{\xi}) + 2\mu_{v} \int_{0}^{t} \int_{\Omega_{b}} |\bd{D}(\bd{\xi}_{\delta})|^{2}.
\nonumber\\
T_{12} = &\lambda_{v} \int_{0}^{t} \int_{\Omega_{b}} (\nabla \cdot \bd{\xi}) \left(\nabla \cdot \left[\bd{\xi} - (\bd{\xi}_{\delta})_{\ttt}\right]\right) - \lambda_{v} \int_{0}^{t} \int_{\Omega_{b}} (\nabla \cdot \bd{\xi}_{\delta}) (\nabla \cdot \bd{\xi}) + \lambda_{v} \int_{0}^{t} \int_{\Omega_{b}} |\nabla \cdot \bd{\xi}_{\delta}|^{2}.
\label{T12}
\end{align}
}}

Because $\bd{\xi}_{\delta} \in L^{2}(0, T; H^{1}(\Omega_{b}))$ where $\Omega_{b}$ is a fixed domain, we have that $(\bd{\xi}_{\delta})_{\ttt} \to \bd{\xi}_{\delta}$ strongly in $L^{2}(0, T; H^{1}(\Omega_{b}))$. Hence, as $\nu \to 0$, we have that Terms 9-12 converge to the following:
\begin{equation*}
T_{9} = \mu_{e} \int_{\Omega_{b}} |\bd{D}(\bd{\eta} - \bd{\eta}_{\delta})(t)|^{2}, \qquad T_{10} = \frac{1}{2} \lambda_{e} \int_{\Omega_{b}} |\nabla \cdot (\bd{\eta} - \bd{\eta}_{\delta})(t)|^{2},
\end{equation*}
\begin{equation*}
T_{11} = 2\mu_{v} \int_{0}^{t} \int_{\Omega_{b}} |\bd{D}(\bd{\xi} - \bd{\xi}_{\delta})|^{2}, \qquad T_{12} = \lambda_{v} \int_{0}^{t} \int_{\Omega_{b}} |\nabla \cdot (\bd{\xi} - \bd{\xi}_{\delta})|^{2}.
\end{equation*}

\vskip 0.1in
\noindent
{\bf{Term T13.}}  Term $T_{13}$ is defined as follows:
{{
\begin{align*}
T_{13} = &-\aalpha \int_{0}^{t} \int_{\Omega_{b}(s)} p \left(\nabla \cdot \left[\bd{\xi} - (\bd{\xi}_{\delta})_{\ttt}\right]\right) + \aalpha \int_{0}^{t} \int_{{\Omega}^\delta_{b,\delta}(s)} p_{\delta} \left(\nabla \cdot (\bd{\xi} - \bd{\xi}_{\delta})\right).
\end{align*}
}}

After taking the limit $\nu\to 0$, term {{$T_{13}$}} can be estimated as follows:
{{
\begin{align*}
|T_{13}| \le &C(\epsilon)\int_{0}^{t} ||\nabla \bd{\eta} - \nabla {\bd{\eta}}^\delta ||_{L^{2}(\Omega_{b})}^{2} + \epsilon \int_{0}^{t} ||\nabla(\bd{\xi} - \bd{\xi}_{\delta})||_{L^{2}(\Omega_{b})}^{2} \\
&+ C(\epsilon)\left(\int_{0}^{t} ||\nabla \bd{\eta} - \nabla \bd{\eta}_{\delta}||^{2}_{L^{2}(\Omega_{b})} + \int_{0}^{t} ||\omega - \omega_{\delta}||^{2}_{H^{2}(\Gamma)} + \int_{0}^{t} ||p - p_{\delta}||^{2}_{L^{2}(\Omega_{b})} \right).
\end{align*}
}}

\vskip 0.1in
\noindent
{\bf{Term T14.}}  Term $T_{14}$ is defined as follows:
{{
\begin{align}
T_{14} = &-c_{0} \int_{0}^{t} \int_{\Omega_{b}} p \cdot \partial_{t} \left[p - (p_{\delta})_{\ttt}\right] + c_{0} \int_{\Omega_{b}} p(t) \cdot \left[p(t) - (p_{\delta})_{\ttt}(t)\right] - c_{0} \int_{\Omega_{b}} p_{0} \cdot \left[p(0) - (p_{\delta})_{\ttt}(0)\right] 
\nonumber \\
&+ c_{0} \int_{0}^{t} \int_{\Omega_{b}} p_{\delta} \cdot \partial_{t} p - c_{0} \int_{\Omega_{b}} p_{\delta}(t) \cdot p(t) + c_{0} \int_{\Omega_{b}} |p_{0}|^{2} + \frac{1}{2} c_{0} \int_{\Omega_{b}} |p_{\delta}(t)|^{2} - \frac{1}{2} c_{0} \int_{\Omega_{b}} |p_{0}|^{2}.
\label{T14}
\end{align}
}}

This term can be handled in the same way as Terms 6-8. In the limit as $\nu \to 0$, the contribution from this term is
{{
\begin{equation*}
T_{14} = \frac{1}{2} c_{0} \int_{\Omega_{b}} |(p - p_{\delta})(t)|^{2}.
\end{equation*}
}}

\vskip 0.1in
\noindent
{\bf{Term T15.}} Term $T_{15}$ is defined as follows:
{{
\begin{align}
T_{15} = &-\aalpha \int_{0}^{t} \int_{\Omega_{b}(s)} \bd{\xi} \cdot \nabla\left[p - (p_{\delta})_{\ttt}\right] + \aalpha \int_{0}^{t} \int_{{\Omega}^\delta_{b,\delta}(s)} \bd{\xi}_{\delta} \cdot \nabla (p - p_{\delta}).
\label{T15}
\end{align}
}}

After taking the limit $\nu\to 0$, term {{$T_{15}$}} can be estimated as follows:
{{
\begin{align*}
|T_{15}| \le &\epsilon \int_{0}^{t} ||\nabla(p - p_{\delta})||_{L^{2}(\Omega^{\delta}_{b, \delta}(s))}^{2} + C(\epsilon) \int_{0}^{t} ||\nabla \bd{\eta} - \nabla {\bd{\eta}}^\delta ||_{L^{2}(\Omega_{b})}^{2} \\
&+ C(\epsilon) \left(\int_{0}^{t} ||\nabla \bd{\eta} - \nabla \bd{\eta}_{\delta}||_{L^{2}(\Omega_{b})}^{2} + \int_{0}^{t} ||\omega - \omega_{\delta}||^{2}_{H^{2}(\Gamma)} + \int_{0}^{t} ||\partial_{t} \bd{\eta} - \partial_{t} \bd{\eta}_{\delta}||^{2}_{L^{2}(\Omega_{b})}\right).
\end{align*}
}}

\vskip 0.1in
\noindent
{\bf{Term T16.}}  Term $T_{16}$ is defined as follows:
{{
\begin{align}
T_{16} = &-\aalpha \int_{0}^{t} \int_{\Gamma(s)} (\bd{\xi} \cdot \bd{n}) \left[p - (p_{\delta})_{\ttt}\right] + \aalpha \int_{0}^{t} \int_{{\Gamma}^\delta_{\delta}(s)} (\bd{\xi}_{\delta} \cdot \bd{n}) (p - p_{\delta}).
\label{T16}
\end{align}
}}

After passing to the limit as $\nu \to 0$, this term can be estimated as follows:
{{
\begin{multline*}
|T_{16}| \le \epsilon\left(\int_{0}^{t} ||\nabla \bd{\xi} - \nabla \bd{\xi}_{\delta}||^{2}_{L^{2}(\Omega_{b})} + \int_{0}^{t} ||\nabla p - \nabla p_{\delta}||_{L^{2}(\Omega^{\delta}_{b, \delta}(s))}^{2}\right) + C(\epsilon)\left(\int_{0}^{t} ||p - p_{\delta}||^{2}_{L^{2}(\Omega_{b})} \right. \\
\left. + \int_{0}^{t} ||\nabla \bd{\eta} - \nabla \bd{\eta}^{\delta}||_{L^{2}(\Omega_{b})}^{2} + \int_{0}^{t} ||\nabla \bd{\eta} - \nabla \bd{\eta}_{\delta}||_{L^{2}(\Omega_{b})}^{2} + \int_{0}^{t} ||\omega - \omega_{\delta}||_{H^{2}(\Gamma)}^{2} + \int_{0}^{t} ||\bd{\xi} - \bd{\xi}_{\delta}||^{2}_{L^{2}(\Omega_{b})}\right).
\end{multline*}
}}

\vskip 0.1in
\noindent
{\bf{Term T17.}}  Term $T_{17}$ is defined as follows:
{{
\begin{align}
T_{17} = &\kappa \int_{0}^{t} \int_{\Omega_{b}(s)} \nabla p \cdot \nabla \left[p - (p_{\delta})_{\ttt}\right] - \kappa \int_{0}^{t} \int_{{\Omega}^\delta_{b,\delta}(s)} \nabla p_{\delta} \cdot \nabla (p - p_{\delta}).
\label{T17}
\end{align}
}}

This term can be estimated as follows:
{{
\begin{equation*}
T_{17} \le \kappa \int_{0}^{t} \int_{{\Omega}^\delta_{b,\delta}(s)} |\nabla(p - p_{\delta})|^{2} + {{R_{17}}},
\end{equation*}
}}
where the remainder is bounded by 
{{
\begin{align*}
{{|R_{17}|}} \le &\epsilon \int_{0}^{t} ||\nabla(p - p_{\delta})||^{2}_{L^{2}({\Omega}^\delta_{b,\delta}(s))} \\
&+ C(\epsilon) \left(\int_{0}^{t} ||\nabla \bd{\eta} - \nabla {\bd{\eta}}^\delta||_{L^{2}(\Omega_{b})}^{2} + \int_{0}^{t} ||\nabla \bd{\eta} - \nabla \bd{\eta}_{\delta}||_{L^{2}(\Omega_{b})}^{2} + \int_{0}^{t} ||\omega - \omega_{\delta}||_{H^{2}(\Gamma)}^{2} \right).
\end{align*}
}}

\vskip 0.1in
\noindent
{\bf{Term T18.}} Term $T_{18}$ is defined as follows:
{{
\begin{align*}
T_{18} &= \int_{0}^{t} \int_{\Gamma(s)} p (\bd{u} - \bd{\xi}) \cdot \bd{n} - \int_{0}^{t} \int_{\Gamma(s)} p [(\bd{u}_{\delta})_{\ttt} - (\bd{\xi}_{\delta})_{\ttt}] \cdot \bd{n} - \int_{0}^{t} \int_{\Gamma_{\delta}(s)} p_{\delta}(\bd{u} - \bd{\xi}) \cdot \bd{n} 
\nonumber \\
&+ \int_{0}^{t} \int_{\Gamma_{\delta}(s)} p_{\delta}(\bd{u}_{\delta} - \bd{\xi}_{\delta}) \cdot \bd{n} - \int_{0}^{t} \int_{\Gamma(s)} ((\bd{u} - \bd{\xi}) \cdot \bd{n}) [p - (p_{\delta})_{\ttt}] + \int_{0}^{t} \int_{\Gamma_{\delta}(s)} ((\bd{u}_{\delta} - \bd{\xi}_{\delta}) \cdot \bd{n}) (p - p_{\delta}).
\end{align*}
}}

This term can be estimated as follows:
{{
\begin{align*}
|T_{18}| \le &\epsilon \left(\int_{0}^{t} ||\bd{D}(\widehat{\bd{u}} - \bd{u}_{\delta})||^{2}_{L^{2}(\Omega_{f,\delta}(s))} + \int_{0}^{t} ||\nabla (\bd{\xi} - \bd{\xi}_{\delta})||_{L^{2}(\Omega_{b})} + \int_{0}^{t} ||\nabla(p - p_{\delta})||^{2}_{L^{2}({\Omega}^\delta_{b,\delta}(s))}\right) \\
&+ C(\epsilon) \int_{0}^{t} ||\omega - \omega_{\delta}||^{2}_{H^{2}(\Gamma)}.
\end{align*}
}}

{{The combined estimates for the terms $T_1$-$T_{18}$ give the estimate:
\begin{multline*}
E_{\delta}(t) \le \epsilon \left(\int_{0}^{t} ||\widehat{\bd{u}} - \bd{u}_{\delta}||^{2}_{H^{1}(\Omega_{f, \delta}(s))} + \int_{0}^{t} ||\bd{D}(\widehat{\bd{u}} - \bd{u}_{\delta})||^{2}_{L^{2}(\Omega_{f, \delta}(s))} \right. \\
\left. + \int_{0}^{t} ||\nabla (\bd{\xi} - \bd{\xi}_{\delta})||^{2}_{L^{2}(\Omega_{b})} + \int_{0}^{t} ||\nabla (p - p_{\delta})||^{2}_{L^{2}(\Omega_{b, \delta}^{\delta}(s))}\right) \\
+ C(\epsilon) \left(\int_{0}^{t} ||\omega - \omega_{\delta}||_{H^{2}(\Gamma)}^{2} + \int_{0}^{t} ||\bd{\xi} - \bd{\xi}_{\delta}||^{2}_{L^{2}(\Gamma)} + \int_{0}^{t} ||\widehat{\bd{u}} - \bd{u}_{\delta}||^{2}_{L^{2}(\Omega_{f, \delta}(s))} + \int_{0}^{t} ||\nabla \bd{\eta} - \nabla \bd{\eta}^{\delta}||^{2}_{L^{2}(\Omega_{b})} \right. \\
\left. + \int_{0}^{t} ||\nabla (\bd{\eta} - \bd{\eta}_{\delta})||_{L^{2}(\Omega_{b})}^{2} + \int_{0}^{t} ||p - p_{\delta}||_{L^{2}(\Omega_{b})}^{2} + \int_{0}^{t} ||\bd{\xi} - \bd{\xi}_{\delta}||^{2}_{L^{2}(\Omega_{b})}\right),
\end{multline*}
where $\epsilon > 0$ remains to be chosen, and $C(\epsilon)$ is a constant depending only on $\epsilon$, that is independent of $t$ and $\delta$. By using Korn's inequality, Poincar\'{e}'s inequality, and by choosing $\epsilon$ sufficiently small to absorb the terms on the right hand side into the dissipation terms in $E_{\delta}(t)$ defined by \eqref{energydiff}, we then obtain the final inequality \eqref{Gronwall}. This finishes the proof of the Gronwall estimate presented in 
Lemma~\ref{GronwallLemma}.}}
\end{proof}

%%%%%%%%%%%%%%%%%BOOTSTRAP%%%%%%%%%%%%%%%%%%%%%%%%

All that is left to show to complete the proof of {{weak-classical consistency}} stated in Theorem~\ref{weakstrongunique},
is to argue that Gronwall's inequality \eqref{Gronwall} holds 
for all $t \in [0,T]$ where $T$ is independent of $\delta$. This will also imply the first statement in the theorem, which states that
$(\bd{\eta}_{\delta}, \omega_{\delta}, p_{\delta}, \bd{u}_{\delta})$ is uniformly defined on the time interval $[0, T]$ for all $\delta > 0$.
In order to do this we use a bootstrap argument presented in the next subsection.

\subsection{Bootstrap argument}\label{bootstrap}

{{To obtain the desired Gronwall estimate as stated in Lemma \ref{GronwallLemma}, we need the following uniform bounds \eqref{det}, \eqref{norm}, and \eqref{grad}}} on the factor $\det(\bd{I} + \nabla {\bd{\eta}}_{\delta}^\delta)$,
which appears in the regularized weak formulation \eqref{deltaweakphysical} defined on the fixed reference domain $\Omega_{b}$:
{{\begin{align*}
\det(\bd{I} + \nabla {\bd{\eta}}_{\delta}^\delta) &\ge c > 0,
\\
0 < c \le |\bd{I} + \nabla {\bd{\eta}}_{\delta}^\delta| &\le C, \qquad \text{ pointwise in } \overline{\Omega_{b}},
\\
|\nabla {\bd{\eta}}_{\delta}^\delta| &\le C, \qquad \text{ pointwise in } \overline{\Omega_{b}},
\end{align*}}}
which need to hold {\emph{for all $t\in[0,T]$ where $T>0$ is independent of $\delta$}}. 
Notice that we only have uniform boundedness of $\bd{\eta}_{\delta}$ with respect to $\delta$ in $L^{\infty}(0, T; H^{1}(\Omega_{b}))$,
which implies that $\det(\bd{I} + \nabla {\bd{\eta}}_{\delta}^\delta)$ is uniformly bounded with respect to $\delta$ only in $L^{\infty}(0, T; L^{1}(\Omega_{b}))$, which is insufficient for estimating any integrands with this factor.

To get around this difficulty we use the following strategy. Recall that by the way  the weak solution to the regularized problem was constructed using the splitting scheme, we have that there exists a sufficiently small constant $c$ (uniform in $\delta$) such that 
\begin{equation}\label{local}
\det(\bd{I} + \nabla {\bd{\eta}}_{\delta}^\delta) \ge c > 0, 
\end{equation}
for all $t \in [0, T_{\delta}]$ where $T_{\delta} > 0$ {\emph{may depend on $\delta$}}. This estimate holds at least locally, although not locally uniformly,
for each $\delta > 0$. 
{{In fact, similarly, the following three estimates \eqref{det}, \eqref{norm}, and \eqref{grad} from Lemma \ref{GronwallLemma} hold locally,  for $t \in [0,T_\delta]$, where {\emph{$T_\delta$ may depend on $\delta$}},
with positive constants $c$ and $C$ that are independent of $\delta$:}}
{{\begin{align*}
\det &(\bd{I} + \nabla {\bd{\eta}}_{\delta}^\delta) \ge c > 0,
\\
0 < c \le&  |\bd{I} + \nabla {\bd{\eta}}_{\delta}^\delta| \le C, \qquad \text{ pointwise in } \overline{\Omega_{b}},
\\
&|\nabla {\bd{\eta}}_{\delta}^\delta| \le C, \qquad \text{ pointwise in } \overline{\Omega_{b}}.
\end{align*}}}
These estimates imply that for sufficiently small $c>0$, the following inequality {{also holds locally, for all $t \in [0, T_{\delta}]$}}:
\begin{equation}\label{inverse}
0 < C^{-1} \le |(\bd{I} + \nabla {\bd{\eta}}_{\delta}^\delta)^{-1}| \le c^{-1}.
\end{equation}

Let  $[0,T]$ denote the time interval on which the {{classical}} solution $\bd{\eta}$ exists. 
Then, we can choose {{$c>0$}} and $C>0$ so that the inequalities \eqref{local}-\eqref{inverse} also hold for the {{classical}} solution 
{\emph{for all $t\in [0,T]$}}. 

{{We will now show that the time interval on which 
 estimates \eqref{local}-\eqref{inverse} hold for the weak solution of the regularized problem ${\bd{\eta}}_{\delta}^\delta$ can, in fact, be extended to the entire 
 interval $[0,T]$, namely, that estimates \eqref{local}-\eqref{inverse} hold globally, uniformly in $\delta$, where $T$ is independent of $\delta$. We will do this by using a classical bootstrap argument the steps of which we present below. This bootstrap argument will propagate the desired estimates which hold a priori locally on $[0, T_{\delta}]$ (for a time $T_{\delta}$ possibly depending on $\delta$) to an entire uniform time interval $[0, T]$.}}
{{The global uniform estimates will follow}} if we can show that $\nabla \bd{\eta}$ and $\nabla {\bd{\eta}}_{\delta}^\delta$ are pointwise uniformly ``close'', i.e., 
\begin{equation}\label{est3}
|(\nabla \bd{\eta} - \nabla {\bd{\eta}}_{\delta}^\delta)(t, x)| \to 0 \ {\rm{pointwise\  uniformly\  in}} \ [0, T] \times \Omega_{b}\  {\rm{as}} \ \delta \to 0. 
\end{equation}

To obtain this estimate we start with the main proof of 
Gronwall's inequality {\emph{under the assumptions that \eqref{det}, \eqref{norm}, and \eqref{grad} are locally valid 
for $t \in [0,T_\delta]$}}:
\begin{equation}\label{GronAux}
E_{\delta}(t) \le C_{1}\int_{0}^{t} ||({\bd{\eta}}^\delta - \bd{\eta})(s)||^{2}_{H^{1}(\Omega_{b})} ds + C_{2}\int_{0}^{t} E_{\delta}(s) ds,
\end{equation}
where the constants $C_{1}$ and $C_{2}$ are independent of $\delta$.
%
%We will use Gronwall's inequality to get an estimate that will hold as long as the assumptions \eqref{det}, \eqref{norm}, and \eqref{grad} are valid. The resulting estimate we obtain will formally be of the form
%\begin{equation*}
%E_{\delta}(t) \le C_{1}\int_{0}^{t} ||({\bd{\eta}}^\delta - \bd{\eta})(s)||^{2}_{H^{1}(\Omega_{b})} ds + C_{2}\int_{0}^{t} E_{\delta}(s) ds,
%\end{equation*}
%where the constants $C_{1}$ and $C_{2}$ are independent of $\delta$ and $E_{\delta}(t)$ is the energy difference between the classical solution and the weak solution to the regularized problem with regularity parameter $\delta$, defined by \eqref{energydiff}.
%
Then, by Lemma~\ref{strongconv} below, we obtain that the first term on the right hand-side above can be estimated as follows:
\begin{equation}\label{convolutionbootstrap}
||{\bd{\eta}}^\delta - \bd{\eta} ||_{H^{1}(\Omega_{b})} \le C\delta^{3/2}, \qquad \text{ for all } t \in [0, T], 
\end{equation}
since the classical solution $\bd{\eta}$ is spatially smooth, and ${\bd{\eta}}^\delta$ is the convolution of $\bd{\eta}$ with the smooth $\delta$ kernel, defined in 
\eqref{tildeeta}. 
With this essential observation, the Gronwall estimate based on \eqref{GronAux} gives
\begin{equation*}
E_{\delta}(t) \le C_{1}\left(\int_{0}^{t} ||({\bd{\eta}}^\delta - \bd{\eta})(s)||^{2}_{H^{1}(\Omega_{b})} ds\right) e^{C_{2}t} \le C_{1}\left(\int_{0}^{T} ||({\bd{\eta}}^\delta - \bd{\eta})(s)||^{2}_{H^{1}(\Omega_{b})} ds\right) e^{C_{2}t} \le C\delta^{3}e^{C_{2}t}.
\end{equation*} 
By the definition of $E_{\delta}(t)$ and an application of Poincare's and Korn's inequalities on $\Omega_{b}$, see Proposition \ref{Korn}, 
this implies that the following terms in the definition of $E_{\delta}(t)$
$$||(\bd{\eta} - \bd{\eta}_{\delta})(t)||_{H^{1}(\Omega_{b})}\le C \delta^{3/2}, \ {\rm{and}}  \ ||(\omega - \omega_{\delta})(t)||_{H^{2}(\Gamma)} \le C \delta^{3/2} \to 0 \ {\rm{as}} \ \delta \to 0$$
converge to zero as $\delta \to 0$ at a rate of $\delta^{3/2}$, as long as the assumptions \eqref{det}, \eqref{norm}, and \eqref{grad} hold. 
Therefore, by H\"{o}lder's inequality, for sufficiently small $\delta>0$, we can prove that the following estimate holds:
\begin{align}\label{closedbootstrap}
|(\nabla {\bd{\eta}}^\delta - \nabla {\bd{\eta}}_{\delta}^\delta)(t, x)|& = \left|\int_{\tilde{\Omega}_{b}} (\nabla \bd{\eta} - \nabla \bd{\eta}_{\delta})(t, y) \sigma_{\delta}(x - y) dy\right| \le C \delta^{3/2} \cdot \delta^{-1} \to 0, 
 \\
& \text{pointwise uniformly in $[0, T] \times \Omega_{b}$ as $\delta \to 0$},
\nonumber
\end{align}
where $C$ is independent of $\delta$.
%We want to show that, in fact, 
%\begin{equation}\label{est3}
%|(\nabla \bd{\eta} - \nabla {\bd{\eta}}_{\delta}^\delta)(t, x)| \to 0 \ {\rm{pointwise\  uniformly\  in}} \ [0, T] \times \Omega_{b}\  {\rm{as}} \ \delta \to 0. 
%\end{equation}
More precisely, notice that the convolution integral in \eqref{closedbootstrap} is defined on the domain ${\tilde{\Omega}_{b}}$, which is triple the size of
the domain $\Omega_b$.
{{Furthermore, we recall that the convolution and the larger domain $\tilde{\Omega}_{b}$ are defined using odd extensions as in Definition \ref{extension}.}}
Thus, by the definition of the odd extensions of $\bd{\eta}$ and $\bd{\eta}_{\delta}$ to the larger domain $\tilde{\Omega}_{b}$, we get
\begin{equation*}
||\bd{\eta} - \bd{\eta}_{\delta}||_{H^{1}(\tilde{\Omega}_{b})} \le C\left(||\bd{\eta} - \bd{\eta}_{\delta}||_{H^{1}(\Omega_{b})} + ||\omega - \omega_{\delta}||_{H^{1}(\Gamma)}\right).
\end{equation*}
In addition, since we have extended the functions $\bd{\eta}$ and $\bd{\eta}_{\delta}$ to the larger domain $\tilde{\Omega}_{b}$, the estimate \eqref{closedbootstrap} holds for all $\delta$ such that $\{(x, y) \in \mathbb{R}^{2}: \text{dist}((x, y), \Omega_{b}) \le \delta\} \subset \tilde{\Omega}_{b}$. 
Thus, $|(\nabla {\bd{\eta}}^\delta - \nabla {\bd{\eta}}_{\delta}^\delta)(t, x)| \to 0$ pointwise uniformly in $[0,T]$ as $\delta \to 0$.

To obtain \eqref{est3} it suffices to show that $|(\nabla \bd{\eta} - \nabla {\bd{\eta}}^\delta)(t, x)| \to 0$ pointwise uniformly in $[0,T]$ as $\delta \to 0$.
This follows from Lemma~\ref{strongconv}. Namely, Lemma~\ref{strongconv} implies
\begin{equation}\label{est2}
|(\nabla \bd{\eta} - \nabla {\bd{\eta}}^\delta)(t, x)| \le C\delta \to 0 \  \text{pointwise uniformly in $[0, T] \times \Omega_{b}$ as $\delta \to 0$}.
\end{equation}
So combining \eqref{est2} with \eqref{closedbootstrap}, we get \eqref{est3}.

{{To conclude the bootstrap argument, we combine the fact that $\text{det}(\bd{I} + \nabla \bd{\eta}) \ge c > 0$ on a time interval $[0, T]$ with the estimate \eqref{est3} (which states that $\nabla \bd{\eta}^{\delta}_{\delta}$ and $\nabla \bd{\eta}$ are uniformly close on the full time interval $[0, T]$ as $\delta \to 0$), to deduce that $\text{det}(\bd{I} + \nabla \bd{\eta}^{\delta}_{\delta}) \ge c > 0$ also uniformly on $[0, T]$, for all sufficiently small $\delta$.}} Similarly, for all sufficiently small $\delta$, the assumptions \eqref{norm} and \eqref{grad} will also hold up to the final time $T > 0$, as we can {{also propagate the estimates \eqref{norm} and \eqref{grad} similarly by combining estimate \eqref{est3} with the fact that these estimates hold for the classical solution $\bd{\eta}$ on some time interval $[0, T]$}}. This closes the bootstrap argument, and so we obtain that the estimate \eqref{det}, and similarly the estimates \eqref{norm} and \eqref{grad}, hold uniformly up to the final time $T > 0$ uniformly in $\delta$.

\if 1 = 0
\begin{remark}(Remark about extensions to three dimensions)
This bootstrap argument is especially useful, because it would extend to the physically relevant case of three dimensions. If one defines a corresponding odd extension (extending Definition \ref{extension} to three spatial dimensions), one can carry out a similar procedure, and use the same bootstrap argument described above. Even in three spatial dimensions, one can still close the bootstrap, as the $L^{2}$ norm of the convolution kernel $\sigma^{\delta}(x - y)$ is on the order of $\delta^{-3/2}$. Hence, we would obtain the following corresponding estimate \eqref{closedbootstrap} in three spatial dimensions
\begin{equation*}
|(\nabla {\bd{\eta}}^\delta - \nabla {\bd{\eta}}_{\delta}^\delta)(t, x)| \sim \delta^{2} \cdot \delta^{-3/2} \to 0,
\end{equation*}
pointwise uniformly as $\delta \to 0$ on the spacetime domain $[0, T] \times \Omega_{b}$, which would still allow us to close the bootstrap argument to establish the weak-classical consistency result.
\end{remark}
\fi

We end this section by proving the following lemma, which establishes convergence of the spatial convolution of the classical solution $\bd{\eta}$ in $H^{1}(\Omega_{b})$, which is needed for the argument described above, {{in the estimates \eqref{convolutionbootstrap} and \eqref{est2}.}}
\begin{lemma}\label{strongconv}
Let $\bd{\eta} \in L^{\infty}(0, T; V_{d})$ be an arbitrary but fixed smooth function in time and space on $[0, T] \times \overline{\Omega_{b}}$, where $V_{d}$ is defined in \eqref{Vd}. Then, there exists a constant $C$ independent of $\delta > 0$, depending only on $\bd{\eta}$, such that 
\begin{equation*}
\max_{t \in [0, T]} ||{\bd{\eta}}^\delta - \bd{\eta}||_{H^{1}(\Omega_{b})} \le C\delta^{3/2},\ {\rm{and}}\ 
|\nabla {\bd{\eta}}^\delta - \nabla \bd{\eta}| \le C\delta\ \quad \forall x \in \overline{\Omega}_{b}\ {\rm{ and}}\  \forall t \in [0, T]. 
\end{equation*}
\end{lemma}

\begin{remark}
More generally, if $f$ is a smooth function on $\mathbb{R}^{2}$ with sufficient decay at infinity, such as a Schwartz function, {{then the argument below}} shows that the function $\tilde{f}$ defined by
\begin{equation*}
\tilde{f} = f * \sigma_{\delta} \qquad \text{ on } \mathbb{R}^{2}
\end{equation*}
would satisfy $||\tilde{f} - f||_{H^{1}(\Omega_{b})} \le C\delta^{2}$ for a constant $C$. However, because we are working on a bounded domain $\Omega_{b}$, we must use an odd extension to define the spatial convolution of $\bd{\eta}$. Since the odd extension of $\bd{\eta}$ to the larger domain $\tilde{\Omega}_{b}$ is not necessarily smooth on $\tilde{\Omega}_{b}$ even if $\bd{\eta}$ is a smooth function on $\overline{\Omega_{b}}$, we incur a loss in our estimate due to potential irregularities of the odd extension due to the behavior of the initial function $\bd{\eta}$ near the boundary $\partial \Omega_{b}$, which gives rise to the convergence rate $\delta^{3/2}$ instead of the optimal rate of convergence $\delta^{2}$. 
\end{remark}

\begin{proof}
Separate the domain $\Omega_{b} = (0, L) \times (0, R)$ into two parts:
\begin{equation*}
\Omega_{b, 1} = (\delta, L - \delta) \times (\delta, R - \delta), \qquad \Omega_{b, 2} = \Omega_{b} \setminus \Omega_{b, 1}.
\end{equation*}
For $\bd{x} \in \Omega_{b, 1}$, we note that because the convolution kernel $\sigma_{\delta}$ is radially symmetric,
\begin{equation*}
({\bd{\eta}}^\delta - \bd{\eta})(\bd{x}) = \int_{\Omega_{b}} \left(\frac{1}{2} \bd{\eta}(\bd{x} + \bd{x}') - \bd{\eta}(\bd{x}) + \frac{1}{2} \bd{\eta}(\bd{x} - \bd{x}')\right) \sigma_{\delta}(\bd{x}') d\bd{x}',
\end{equation*}
\begin{equation*}
(\nabla {\bd{\eta}}^\delta - \nabla \bd{\eta})(\bd{x}) = \int_{\Omega_{b}} \left(\frac{1}{2} \nabla \bd{\eta}(\bd{x} + \bd{x}') - \nabla \bd{\eta}(\bd{x}) + \frac{1}{2} \nabla \bd{\eta}(\bd{x} - \bd{x}') \right) \sigma_{\delta}(\bd{x}') d\bd{x}'.
\end{equation*}
For $\bd{x} \in \Omega_{b, 1}$, these points are at least $\delta$ away from the boundary. Therefore, we have the following estimate for the discretized second derivative:
\begin{equation*}
\left|\frac{1}{2} \bd{\eta}(\bd{x} + \bd{x}') - \bd{\eta}(x) + \frac{1}{2} \bd{\eta}(\bd{x} - \bd{x}')\right| \le C\delta^{2} \quad \text{ for } |\bd{x}'| \le \delta,
\end{equation*}
and similarly for $\nabla \bd{\eta}$, by using the fact that $\bd{\eta}$ is spatially smooth in $\overline{\Omega_{b}}$. 
Therefore,
\begin{equation}\label{omegab1}
|({\bd{\eta}}^\delta - \bd{\eta})(\bd{x})| \le C\delta^{2}, \quad |(\nabla {\bd{\eta}}^\delta - \nabla \bd{\eta})(\bd{x})| \le C\delta^{2}, \qquad \text{ for } \bd{x} \in \Omega_{b, 1},
\end{equation}
for a constant $C$ depending only on $\bd{\eta}$. 

For $\bd{x} \in \Omega_{b, 2}$ we cannot use the same estimate, since after extending $\bd{\eta}$ to the larger domain $\tilde{\Omega}_{b}$, the extended function on $\tilde{\Omega}_{b}$ does not necessarily have a continuous second derivative, as a result of the properties of odd extension, and in fact, there may be discontinuities of the second derivative along the boundary $\partial \Omega_{b}$. However, $\nabla \bd{\eta}$ on the larger domain $\tilde{\Omega}_{b}$ is still \textit{Lipschitz continuous}. Thus, we instead use the equations:
\begin{equation*}
({\bd{\eta}}^\delta - \bd{\eta})(\bd{x}) = \int_{\Omega_{b}} (\bd{\eta}(\bd{x} + \bd{x}') - \bd{\eta}(\bd{x})) \sigma_{\delta}(\bd{x}') d\bd{x}',
\end{equation*}
\begin{equation*}
(\nabla {\bd{\eta}}^\delta - \nabla \bd{\eta})(\bd{x}) = \int_{\Omega_{b}} (\nabla \bd{\eta}(\bd{x} + \bd{x}') - \nabla \bd{\eta}(\bd{x})) \sigma_{\delta}(\bd{x}') d\bd{x}'.
\end{equation*}
Since $\bd{x} \in \Omega_{b, 2}$, even if $|\bd{x}'| \le \delta$, we may have that $\bd{x} + \bd{x}'$ is outside of $\Omega_{b}$. However, due to the Lipschitz continuity of $\nabla \bd{\eta}$ on the larger domain $\tilde{\Omega}_{b}$, we still have the estimates
\begin{equation*}
|\bd{\eta}(\bd{x} + \bd{x}') - \bd{\eta}(\bd{x})| \le C\delta, \quad |\nabla \bd{\eta}(\bd{x} + \bd{x}') - \nabla \bd{\eta}(\bd{x})| \le C\delta, \qquad \text{ for } \bd{x} \in \Omega_{b, 2}, |\bd{x'}| \le \delta,
\end{equation*}
which give
\begin{equation}\label{omegab2}
|({\bd{\eta}}^\delta - \bd{\eta})(\bd{x})| \le C\delta, \quad |(\nabla {\bd{\eta}}^\delta - \nabla \bd{\eta})(\bd{x})| \le C\delta, \qquad \text{ for } \bd{x} \in \Omega_{b, 2}.
\end{equation}
The area of $\Omega_{b, 2}$ is bounded by $(2R + 2L)\delta$, so by \eqref{omegab1} and \eqref{omegab2}, we have $||{\bd{\eta}}^\delta - \bd{\eta}||_{H^{1}(\Omega_{b})} \le C\delta^{3/2}$ for a spatially smooth function $\bd{\eta}$ on $\overline{\Omega_{b}}$, where $C$ depends only on the norms of up to the second spatial derivative of $\bd{\eta}$ on $\overline{\Omega_{b}}$. The generalization of this result to a function $\bd{\eta}$ that also depends on time and is spatially smooth in both space and time follows analogously.

\end{proof}

{{This completes the proof of the weak-classical consistency results. 
This  proof effectively shows that the weak solutions that we have constructed to the regularized FPSI problem converge (in the energy norm on a uniform time interval) as the regularization parameter goes to zero to a classical solution of the original (non-regularized) FPSI problem when such a classical solution to the original FPSI problem exists.}}

\section{Conclusions}
In this manuscript we proved the existence of a weak solution to a fluid-structure interaction problem between the flow of an incompressible, viscous fluid and a multi-layered poroelastic/poroviscoelastic structure consisting of the Biot equations of poro(visco)elasticity and a thin, reticular interface with mass and elastic energy, which is transparent to fluid flow. The fluid and multilayered structure are nonlinearly coupled, giving rise to significant difficulties in the existence proof, associated with the geometric nonlinearity of the coupled problem. The existence proof is constructive, and it consists of two major steps. In the fist step we proved the existence of a weak solution to a regularized problem in the class of finite energy solutions. In the second step we showed that the solution of this regularized problem converges to a {{classical}} solution to the original, nonregularized probroblem as the regularization parameter tends to zero, as long as the original problem possesses a {{classical solution}}. 
While the proof of the existence of a weak solution to the regularized problem only requires that the Biot structure is poroelastic, additional regularity of the Biot poroelatic medium is required to prove the {{weak-classical}} consistency-the Biot structure is assumed to be poroviscoelastic. 
This {{weak-classical}} consistency result also shows that the solution we constructed is unique in the sense of {{weak-classical}} uniqueness. 

{{We make a few comments about extensions of these results on fluid-poroelastic structure interaction to the case of three spatial dimensions, as the model problem discussed in this manuscript involves two spatial dimensions. For the existence proof, the constructive existence proof outlined for the two-dimensional FPSI problem carries out to the case of fluid-poroelastic structure interaction between a fluid modeled by the Navier-Stokes in three spatial dimensions and a three-dimensional Biot poroviscoelastic medium, separated by a two-dimensional reticular plate. In the course of such an analysis to a three-dimensional problem, one would encounter several new difficulties, which we briefly discuss here. First, the odd extension used to define the convolution in Definition \ref{extension} would have to be modified, but a similar odd extension could be used in three dimensions too. More importantly, the plate displacement in the finite energy space is in $H^{2}(\Gamma)$, which for a two-dimensional plate interface separating a 3D fluid and 3D Biot medium, would produce an interface that is only $\alpha$-H\"{o}lder continuous for $0 < \alpha < 1/2$. Thus, we would be working with time-dependent fluid domains $\Omega_{f}(t)$, which are not uniformly Lipschitz, which is a geometric requirement for many classical results such as the trace theorem. However, such problems have already been addressed in the fluid-structure interaction literature, for example in \cite{BorNote3D}, and techniques exist for the analysis of FSI systems in three spatial dimensions, see for example \cite{BorSun3d} and \cite{BorSunNonLinearKoiter}. Hence, the proof of constructive existence of a weak solution to a regularized 3D FPSI problem is expected to carry through similarly without significant challenges to give an analogue of Theorem \ref{MainThm1} for an analogous 3D FPSI model. 

However, we emphasize that it is still an open question to show weak-classical consistency for the 3D FPSI problem. Although one can still show existence of weak solutions to the regularized problem, using the current analysis, an analogue of Theorem \ref{weakstrongunique} is at the moment unattainable. The issue is the rate of convergence in Lemma \ref{strongconv} of the convolution of the odd extension $\bd{\eta}^{\delta}$ to the original displacement $\bd{\eta}$ when $\bd{\eta}$ is a smooth function in time and space, which is on the order of $\delta^{3/2}$. In two-dimensions, the usual convolution kernel $\sigma_{\delta}$ has an $L^{2}$ norm on the order of $\delta^{-1}$, so we get a crucial convergence to zero in the estimate \eqref{closedbootstrap} as $\delta \to 0$ stating that the gradients of $\bd{\eta}_{\delta}$ and $\bd{\eta}^{\delta}_{\delta}$ converge pointwise uniformly in the limit as $\delta \to 0$, which is the estimate that allows us to close our bootstrap argument. In three dimensions, we would lose this convergence to zero since in three dimensions, $\sigma_{\delta}$ has an $L^{2}$ norm on the order of $\delta^{-3/2}$ while the convolution  estimate would still give a convergence rate of $\delta^{3/2}$ for the $L^{2}$ norm of $(\nabla \bd{\eta} - \nabla \bd{\eta}_{\delta})$. Hence, to establish a corresponding weak-classical consistency result for the three dimensional problem, one would either have to improve the convergence rate of the convolution $\bd{\eta}^{\delta}$ to the actual function $\bd{\eta}$ in Lemma \ref{strongconv}, or find an alternative regularization/extension procedure in place of that in Definition \ref{extension} that exhibits a better convergence rate than $\delta^{3/2}$ as in Lemma \ref{strongconv}.}}

{{A final interesting extension of this work is to consider the singular limit as the thin interface thickness converges to zero, and analyze the resulting FPSI problem between the Navier-Stokes equations for an incompressible, viscous fluid and the Biot equations,
nonlinearly coupled over the moving interface, without a reticular plate separating the two.}} Preliminary results indicate that this will be possible under certain assumptions{{, including viscoelasticity of the Biot medium. In this case, we believe that one could obtain an analogous result for existence of a weak solution, first to a regularized problem, either through an adaptation of the methods presented in this manuscript for the case with a reticular plate, or as a singular limit of weak solutions to the regularized problem with a plate as the plate thickness goes to zero. Then, under the assumption of the existence of a classical solution to the FPSI problem with direct Biot-fluid contact (and no plate), one could pursue a similar weak-classical consistency result showing convergence of weak solutions (to the regularized problem) to the classical solution as the regularization parameter goes to zero. Consequently, we would have a weak-classical consistency result of the same type as in the current paper for the FPSI model without a reticular plate.}}
%\section{The Gronwall estimate}\label{Gronwall}

%\input{Gronwall_Section11.tex}

\section*{Acknowledgements}

{{We would like to thank the anonymous referees for their careful reading of our manuscript and for providing insightful comments that improved the quality of this paper. Partial support for this research was provided by the NSF MSPRF fellowship DMS-2303177 (Jeffrey Kuan), the NSF grants DMS-2247000 and DMS-2011319 (Sun\v{c}ica \v{C}ani\'{c} and Jeffrey Kuan), and the Croatian Science Foundation under project number IP-2022-10-2962 (Boris Muha).}}

\appendix
\section{Appendix}\label{appendix}

\subsection{Weak continuity of solutions to the regularized {{FPSI}} problem}\label{appendix1}

In this appendix, we show a result related to weak continuity of solutions to the regularized FPSI problem,
namely, we will show that as $\nu \to 0$:
\begin{equation*}
\int_{\Omega_{f, \delta}(0)} \widehat{\bd{u}}(0) \cdot (\bd{u}_{\delta})_{\nu}(0) \to \int_{\Omega_{f, \delta}(0)} |\bd{u}_{0}|^{2}, \quad \text{ and } \quad \int_{\Omega_{f,\delta}(t)} \widehat{\bd{u}}(t) \cdot (\bd{u}_{\delta})_{\nu}(t) \to \int_{\Omega_{f,\delta}(t)} \widehat{\bd{u}}(t) \cdot \bd{u}_{\delta}(t),
\end{equation*}
for almost all points $0 < t \le T$. 

This result will be used in {{Section \ref{appendix2}}} to estimate the first term $T_{1}$ in \eqref{sumT} in the Gronwall's estimate. 
We will show weak continuity through the following series of lemmas.

\begin{lemma}\label{alphaconv}
Let $\omega \in L^{\infty}(0, T; H_{0}^{2}(\Gamma)) \cap W^{1, \infty}(0, T; L^{2}(\Gamma))$ with
\begin{equation*}
\min_{t \in [0, T], x \in [0, L]} R + \omega(t, x) > 0,
\end{equation*}
define the moving fluid domain $\Omega_{f}^{\omega}(t)$. Then, given $\bd{u} \in L^{2}(0, T; H^{1}(\Omega_{f}^{\omega}(t))) \cap L^{\infty}(0, T; L^{2}(\Omega_{f}^{\omega}(t)))$ where $\Omega_{f}^{\omega}(t) = \{(x, y) \in \mathbb{R}^{2} : 0 \le x \le L, -R \le y \le \omega(t, x)\}$, we have that
\begin{equation*}
||\bd{u}_{\nu}(t, x, y) - \bd{u}(t, x, y)||_{L^{2}(\Omega_{f}^{\omega}(t))} \to 0 \qquad \text{ as } \nu \to 0,
\end{equation*}
for almost all $t \in [0, T]$.
\end{lemma}

\begin{proof}
Recall that in the case of real-valued functions, one shows convergence of the convolution to the function itself almost everywhere by using the Lebesgue differentiation theorem \cite{Evans}. To apply the theorem in this context, we need to apply it to a function taking values in a \textit{fixed} Banach space rather than a time-dependent Banach space.

As a result, we consider the following function,
\begin{equation*}
\bd{v}(t, x, y) = K(t, 0, x, y) \bd{u}\left(t, x, \frac{R + \omega(t, x)}{R + \omega(0, x)}(R + y) - R\right),
\end{equation*}
where we have pulled the fluid velocity back to the fixed initial domain $\Omega^{\omega}_{f}(0)$. We recall the definition of $K(s, t, z, r)$ from \eqref{Kdelta} and its inverse:
\begin{equation*}
K(s, t, x, y) = 
\begin{pmatrix}
\frac{R + \omega(s, x)}{R + \omega(t, x)} & 0 \\ -(R + y)\partial_{x}\left(\frac{R + \omega(s, x)}{R + \omega(t, x)}\right) & 1 \\
\end{pmatrix}, \quad K^{-1}(s, t, x, y) =
\begin{pmatrix}
\frac{R + \omega(t, x)}{R + \omega(s, x)} & 0 \\
(R + y)\frac{R + \omega(t, x)}{R + \omega(s, x)} \partial_{x}\left(\frac{R + \omega(s, x)}{R + \omega(t, x)}\right) & 1 \\
\end{pmatrix}.
\end{equation*}
By the uniform boundedness of $R + \omega(t, x)$ and $|\partial_{x}\omega(t, x)|$, and $\min_{t \in [0, T], x \in [0, L]} R + \omega(t, x) > 0$, it is immediate to see that $\bd{v}(t, z, r)$ is in $L^{\infty}(0, T; L^{2}(\Omega^{\omega}_{f}(0)))$, where we emphasize that $L^{2}(\Omega^{\omega}_{f}(0))$ is a fixed function space that no longer depends on time. 

By Lebesgue's differentiation theorem, almost every $t \in [0, T]$ is a Lebesgue point satisfying
\begin{equation}\label{Lebesgue}
\lim_{\nu \to 0} \frac{1}{2\nu} \int_{t - \nu}^{t + \nu} ||\bd{v}(t, \cdot) - \bd{v}(s, \cdot)||_{L^{2}(\Omega^{\omega}_{f}(0))} ds \to 0.
\end{equation}

Recall that by definition \eqref{ualpha},
\begin{equation*}
\bd{u}_{\nu}(t, x, y) = \int_{\mathbb{R}} K(s, t, x, y) \bd{u}\left(s, x, \frac{R + \omega(s, x)}{R + \omega(t, x)}(R + y) - R\right) j_{\nu}(t - s) ds.
\end{equation*}
Thus, we compute 
\begin{align*}
\bd{u}_{\nu}(t, x, y) - \bd{u}(t, x, y) = &\int_{\mathbb{R}} \left(K(s, t, x, y) \bd{u}\left(s, x, \frac{R + \omega(s, x)}{R + \omega(t, x)}(R + y) - R\right) - \bd{u}(t, x, y)\right) \cdot j_{\nu}(t - s) ds \\
:= & I_{1} + I_{2},
\end{align*}
where 
\begin{align*}
I_{1} = \int_{\mathbb{R}} &K^{-1}\left(t, 0, x, \frac{R + \omega(0, x)}{R + \omega(t, x)}(R + y) - R\right) \cdot \\
&\left(\bd{v}\left(s, x, \frac{R + \omega(0, x)}{R + \omega(t, x)}(R + y) - R\right) - \bd{v}\left(t, x, \frac{R + \omega(0, x)}{R + \omega(t, x)}(R + y) - R\right)\right) j_{\nu}(t - s) ds,
\end{align*}
\begin{align*}
I_{2} = \int_{\mathbb{R}} &\left(K(s, t, x, y) K^{-1}\left(s, 0, x, \frac{R + \omega(0, x)}{R + \omega(t, x)}(R + y) - R\right) - K^{-1}\left(t, 0, x, \frac{R + \omega(0, x)}{R + \omega(t, x)}(R + y) - R\right)\right) \cdot \\
&\bd{v}\left(s, x, \frac{R + \omega(0, x)}{R + \omega(t, x)}(R + y) - R\right) j_{\nu}(t - s) {{ds}}.
\end{align*}

We estimate each of these terms as follows. For $I_{1}$, we compute that
\begin{equation*}
K^{-1}\left(t, 0, x, \frac{R + \omega(0, x)}{R + \omega(t, x)}(R + y) - R\right) = 
\begin{pmatrix} \frac{R + \omega(0, x)}{R + \omega(t, x)} & 0 \\ (R + y) \left(\frac{R + \omega(0, x)}{R + \omega(t, x)}\right)^{2} \partial_{x}\left(\frac{R + \omega(t, x)}{R + \omega(0, x)}\right) & 1 \\ \end{pmatrix},
\end{equation*}
which we note is uniformly bounded on $[0, T]$. Hence, using the fact that $|j_{\nu}(t - s)| \le \frac{1}{\nu}$, we get
\begin{align*}
\small
||I_{1}||_{L^{2}(\Omega^{\omega}_{f}(t))} 
&\le C \cdot \frac{1}{\nu} \int_{t - \nu}^{t + \nu} \left|\left| \bd{v}\left(s, x, \frac{R + \omega(0, x)}{R + \omega(t, x)}(R + y) - R\right) - \bd{v}\left(t, x, \frac{R + \omega(0, x)}{R + \omega(t, x)}(R + y) - R\right) \right| \right|_{L^{2}(\Omega^{\omega}_{f}(t))} ds \\
&\le C \cdot \frac{1}{\nu} \int_{t - \nu}^{t + \nu} \left(\frac{R + \omega(t, x)}{R + \omega(0, x)}\right)^{1/2} ||\bd{v}(s, x, y) - \bd{v}(t, x, y)||_{L^{2}(\Omega^{\omega}_{f}(0))} ds \to 0,
\end{align*}
as $\nu \to 0$ if $t$ is a Lebesgue point, by \eqref{Lebesgue} and the uniform boundedness of $\frac{R + \omega(t, x)}{R + \omega(0, x)}$ on $[0, T]$. 

To estimate $I_{2}$, we can use the continuity in time of $\omega$ and $\partial_{x}\omega$ to calculate that 
\begin{equation*}
\left|K(s, t, x, y) K^{-1}\left(s, 0, x, \frac{R + \omega(0, x)}{R + \omega(t, x)}(R + y) - R\right) - K^{-1}\left(t, 0, x, \frac{R + \omega(0, x)}{R + \omega(t, x)}(R + y) - R\right)\right| \to 0,
\end{equation*}
uniformly in $(x, y)$ as $s \to t$. Now, we estimate
\begin{align*}
\small
||I_{2}||_{L^{2}(\Omega^{\omega}_{f}(t))} 
&\le \int_{\mathbb{R}} \max_{x, y \in \Omega^{\omega}_{f}(t)} \left|K(s, t, x, y) K^{-1}\left(s, 0, x, \frac{R + \omega(0, x)}{R + \omega(t, x)}(R + y) - R\right) 
\right.
\\
&\left. \quad -  K^{-1}\left(t, 0, x, \frac{R + \omega(0, x)}{R + \omega(t, x)}(R + y) - R\right) \right| \\
&\quad \cdot \left|\left|\bd{v}\left(s, x, \frac{R + \omega(0, x)}{R + \omega(t, x)}(R + y) - R\right)\right|\right|_{L^{2}(\Omega^{\omega}_{f}(t))} \cdot j_{\nu}(t - s) ds \\
&\le \int_{\mathbb{R}} \max_{x, y \in \Omega^{\omega}_{f}(t)} \left|K(s, t, x, y) K^{-1}\left(s, 0, x, \frac{R + \omega(0, x)}{R + \omega(t, x)}(R + y) - R\right) 
\right.
\\
& \quad \left. - K^{-1}\left(t, 0, x, \frac{R + \omega(0, x)}{R + \omega(t, x)}(R + y) - R\right)\right| \\
&\quad \cdot \left(\frac{R + \omega(t, x)}{R + \omega(0, x)}\right)^{1/2} \cdot \left|\left|\bd{v}\left(s, x, y\right)\right|\right|_{L^{2}(\Omega^{\omega}_{f}(0))} \cdot j_{\nu}(t - s) ds \\
&\le C\int_{\mathbb{R}} \max_{x, y \in \Omega^{\omega}_{f}(t)} \left|K(s, t, x, y) K^{-1}\left(s, 0, x, \frac{R + \omega(0, x)}{R + \omega(t, x)}(R + y) - R\right) 
\right.
\\
& \left. \quad - K^{-1}\left(t, 0, x, \frac{R + \omega(0, x)}{R + \omega(t, x)}(R + y) - R\right)\right| \cdot j_{\nu}(t - s) ds,
\end{align*}
where we used the fact that $\bd{v} \in L^{\infty}(0, T; L^{2}(\Omega^{\omega}_{f}(0)))$. Thus, we conclude that $||I_{2}||_{L^{2}(\Omega^{\omega}_{f}(t))} \to 0$ as $\nu \to 0$. This completes the proof. 
\end{proof}

We also have a weak continuity lemma, which states that the value of $\bd{u}_{\delta}$ tested against any function in the fluid function space has a continuity property as $t \to 0$. 

\begin{lemma}\label{weakcontinuity}
Consider an arbitrary $\bd{q} \in C^{1}(0, T; V_{f, \delta}(t))$ and the weak solution $\bd{u}_{\delta}$ to the regularized problem for arbitrary $\delta$, where $V_{f,\delta}(t)$ is defined by the displacement $\omega_{\delta}$ and \eqref{Vft}. There exists a measure zero subset $S$ of $[0, T]$ (depending on $\delta$) such that
\begin{equation*}
\lim_{t \to 0, t \in [0, T] \cap S^{c}} \int_{\Omega_{f,\delta}(t)} \bd{u}_{\delta}(t) \cdot \bd{q}(t) = \int_{\Omega_{f, \delta}(0)} \bd{u}_{0} \cdot \bd{q}(0). 
\end{equation*}
\end{lemma}

\begin{proof}
Consider the following function for each $\tau \in [0, T]$ and $\alpha > 0$, given by
\begin{equation}\label{Jtaualpha}
J_{\tau, \nu}(t) = 1 - \int_{0}^{t} j_{\nu}(s - \tau) ds,
\end{equation}
and note that $J_{\tau, \nu}'(t) = -j_{\nu}(t - \tau)$. We want to test the regularized weak formulation for $\bd{u}_{\delta}$ with the test function $J_{\tau, \nu}(t)\bd{q}$ for certain admissible choices of $\tau$. To see which $\tau$ we want to choose, we define the function
\begin{equation*}
\bd{w}(t, x, y) = \frac{R + \omega_{\delta}(t)}{R + \omega_{\delta}(0)} \cdot \bd{u}_{\delta}\left(t, x, \frac{R + \omega_{\delta}(t)}{R + \omega_{\delta}(0)}(R + y) - R\right) \cdot \bd{q}\left(t, x, \frac{R + \omega_{\delta}(t)}{R + \omega_{\delta}(0)}(R + y) - R\right).
\end{equation*}
We claim that $\bd{w} \in L^{\infty}(0, T; L^{1}(\Omega_{f, \delta}(0)))$. To see this, we compute by a change of variables that 
\begin{equation*}
||\bd{w}(t, x, y)||_{L^{1}(\Omega_{f, \delta}(0))} = \int_{\Omega_{f,\delta}(t)} |\bd{u}_{\delta}(t, x, y) \cdot \bd{q}(t, x, y)|,
\end{equation*}
and we then use the fact that $\bd{u}_{\delta}, \bd{q} \in L^{\infty}(0, T; L^{2}(\Omega_{f, \delta}(t)))$. 

Hence, by the Lebesgue differentiation theorem, there exists a measurable subset $S \subset [0, T]$ of measure zero such that every point in $[0, T] \cap S^{c}$ is a Lebesgue point of $\bd{w}$, in the sense that
\begin{equation}\label{Lebesguew}
\lim_{\nu \to 0} \frac{1}{2\nu} \int_{\tau - \nu}^{\tau + \nu} ||\bd{w}(\tau, \cdot) - \bd{w}(s, \cdot)||_{L^{1}(\Omega_{f, \delta}(0))} ds \to 0.
\end{equation}
for every $\tau \in [0, T] \cap S^{c}$. These are the $\tau$ for which we will consider the test function $J_{\tau, \nu}(t)\bd{q}$. For the test functions for the Biot medium and the plate, we will take these test functions to be zero. Hence, in the regularized weak formulation \eqref{weakdelta}, we will test with $(\bd{v}, \varphi, \bd{\psi}, r) = (J_{\tau, \nu}(t)\bd{q}, 0, 0, 0)$.

Hence, we obtain the following equality:
\begin{align*}
&-\int_{0}^{T} \int_{\Omega_{f,\delta}(t)} \bd{u}_{\delta} \cdot \partial_{t}(J_{\tau, \nu}(t)\bd{q}) + \frac{1}{2} \int_{0}^{T} \int_{\Omega_{f,\delta}(t)} [((\bd{u}_{\delta} \cdot \nabla)\bd{u}_{\delta}) \cdot (J_{\tau, \nu}(t)\bd{q}) - ((\bd{u}_{\delta} \cdot \nabla)(J_{\tau, \nu}(t)\bd{q})) \cdot \bd{u}_{\delta}] \\
&+ \frac{1}{2} \int_{0}^{T} \int_{\Gamma_{\delta}(t)} (\bd{u}_{\delta} \cdot \bd{n} - 2\bd{\xi}_{\delta} \cdot \bd{n}) \bd{u}_{\delta} \cdot (J_{\tau, \nu}(t)\bd{q}) 
+ 2\nu \int_{0}^{T} \int_{\Omega_{f,\delta}(t)} \bd{D}(\bd{u}_{\delta}) : \bd{D}(J_{\tau, \nu}(t)\bd{q})
\\
& - \int_{0}^{T} \int_{\Gamma_{\delta}(t)} \left(\frac{1}{2}|\bd{u}_{\delta}|^{2} - p_{\delta}\right) J_{\tau, \nu}(t) q_{n} 
- \beta \int_{0}^{T} \int_{\Gamma_{\delta}(t)} [(\xi_{\delta})_{\tau} - (u_{\delta})_{\tau}] \cdot J_{\tau, \nu}(t) q_{\tau} = \int_{\Omega_{f, \delta}(0)} \bd{u}_{0} \cdot J_{\tau, \nu}(0) \bd{q}(0). 
\end{align*}
Consider $\tau \in (0, T) \cap S^{c}$. We want to pass to the limit as $\nu \to 0$, and then pass to the limit as $\tau \to 0$, in order to obtain the desired result. 

First, we pass to the limit as $\nu \to 0$. We handle the convergences as follows.

\noindent \textbf{First term:} We will show that because $\tau$ is a Lebesgue point of $\bd{w}$, 
\begin{equation*}
-\int_{0}^{T} \int_{\Omega_{f,\delta}(t)} \bd{u}_{\delta} \cdot \partial_{t}(J_{\tau, \nu}(t)\bd{q}) \to \int_{\Omega_{f, \delta}(\tau)} \bd{u}_{\delta}(\tau) \bd{q}(\tau) - \int_{0}^{t} \int_{\Omega_{f,\delta}(t)} \bd{u}_{\delta} \cdot \partial_{t}\bd{q}, \qquad \text{ as } \nu \to 0.
\end{equation*}
We compute that
\begin{equation*}
-\int_{0}^{T} \int_{\Omega_{f,\delta}(t)} \bd{u}_{\delta} \cdot \partial_{t}(J_{\tau, \nu}(t)\bd{q}) = \int_{0}^{T} \int_{\Omega_{f,\delta}(t)} \bd{u}_{\delta} \cdot j_{\nu}(t - \tau) \bd{q} - \int_{0}^{T} \int_{\Omega_{f,\delta}(t)} \bd{u}_{\delta} \cdot J_{\tau, \nu}(t) \partial_{t}\bd{q}.
\end{equation*}
It is easy to see that
\begin{equation*}
\int_{0}^{T} \int_{\Omega_{f,\delta}(t)} \bd{u}_{\delta} J_{\tau, \nu}(t) \partial_{t}\bd{q} \to \int_{0}^{t} \int_{\Omega_{f,\delta}(t)} \bd{u}_{\delta} \partial_{t} \bd{q}.
\end{equation*}
So it remains to show that
\begin{equation*}
\int_{0}^{T} \int_{\Omega_{f,\delta}(t)} \bd{u}_{\delta} \cdot j_{\nu}(t - \tau) \bd{q} \to \int_{\Omega_{f, \delta}(\tau)} \bd{u}_{\delta}(\tau) \bd{q}(\tau), \qquad \text{ as } \nu \to 0. 
\end{equation*}
By a change of variables, we compute that
\begin{multline*}
\int_{0}^{T} \int_{\Omega_{f,\delta}(t)} \bd{u}_{\delta} \cdot j_{\nu}(t - \tau) \bd{q} \\
= \int_{0}^{T} \int_{\Omega_{f, \delta}(\tau)} \frac{R + \omega_{\delta}(t)}{R + \omega_{\delta}(\tau)} \cdot \bd{u}_{\delta}\left(t, x, \frac{R + \omega_{\delta}(t)}{R + \omega_{\delta}(\tau)}(R + y) - R\right) \cdot j_{\nu}(t - \tau) \bd{q}\left(t, x, \frac{R + \omega_{\delta}(t)}{R + \omega_{\delta}(\tau)}(R + y) - R\right) \\
= \int_{0}^{T} \int_{\Omega_{f, \delta}(\tau)} \frac{R + \omega_{\delta}(0)}{R + \omega_{\delta}(\tau)} \bd{w}\left(t, x, \frac{R + \omega_{\delta}(0)}{R + \omega_{\delta}(\tau)}(R + y) - R\right) \cdot j_{\nu}(t - \tau) = \int_{0}^{T} \int_{\Omega_{f, \delta}(0)} \bd{w}(t, x, y) \cdot j_{\nu}(t - \tau).
\end{multline*}
By \eqref{Lebesguew}, we have that
\begin{equation*}
\int_{0}^{T} \int_{\Omega_{f, \delta}(0)} \bd{w}(t, x, y) \cdot j_{\nu}(t - \tau) \to \int_{\Omega_{f, \delta}(0)} \bd{w}(\tau, x, y) = \int_{\Omega_{f, \delta}(\tau)} \bd{u}_{\delta}(\tau) \cdot \bd{q}(\tau),
\end{equation*}
which establishes the desired convergence. 

\noindent \textbf{Final term:} It is immediate to see that for all sufficiently small $\nu > 0$,
\begin{equation*}
\int_{\Omega_{f, \delta}(0)} \bd{u}_{0} \cdot J_{\tau, \nu}(0) \bd{q}(0) = \int_{\Omega_{f, \delta}(0)} \bd{u}_{0} \cdot \bd{q}(0).
\end{equation*}
We can now easily take $\nu \to 0$ in the remaining terms to obtain that for any $\tau \in (0, T) \cap S^{c}$, 
\begin{multline*}
\int_{\Omega_{f, \delta}(\tau)} \bd{u}_{\delta}(\tau) \cdot \bd{q}(\tau) - \int_{0}^{t} \int_{\Omega_{f,\delta}(t)} \bd{u}_{\delta} \cdot \partial_{t} \bd{q} + \frac{1}{2} \int_{0}^{t} \int_{\Omega_{f,\delta}(t)} [((\bd{u}_{\delta} \cdot \nabla) \bd{u}_{\delta}) \cdot \bd{q} - ((\bd{u}_{\delta} \cdot \nabla) \bd{q}) \cdot \bd{u}_{\delta}] \\
+ \frac{1}{2} \int_{0}^{t} \int_{\Gamma_{\delta}(t)} (\bd{u}_{\delta} \cdot \bd{n} - 2\bd{\xi}_{\delta} \cdot \bd{n}) \bd{u}_{\delta} \cdot \bd{q} + 2\nu \int_{0}^{t} \int_{\Omega_{f,\delta}(t)} \bd{D}(\bd{u}_{\delta}) : \bd{D}(\bd{q}) \\
- \int_{0}^{t} \int_{\Gamma_{\delta}(t)} \left(\frac{1}{2} |\bd{u}_{\delta}|^{2} - p_{\delta}\right) q_{n} - \beta \int_{0}^{t} \int_{\Gamma_{\delta}(t)} [(\xi_{\delta})_{\tau} - (u_{\delta})_{\tau}] \cdot q_{\tau} = \int_{\Omega_{f, \delta}(0)} \bd{u}_{0} \cdot \bd{q}(0). 
\end{multline*}
Passing to the limit as $\tau \to 0$ with $\tau \in (0, T) \cap S^{c}$ gives the desired result. 

\end{proof}

\begin{lemma}\label{qtilde}
Let $\bd{u}_{0}$ be divergence free and smooth on $\overline{\Omega_{f}(0)}$. Define
\begin{equation}\label{tildeq}
\tilde{\bd{q}}(t, x, y) = K_{\delta}(0, t, x, y) \bd{u}_{0}\left(x, \frac{R + \omega_{\delta}(0, x)}{R + \omega_{\delta}(t, x)}(R + y) - R\right),
\end{equation}
where $K_{\delta}$ is given by \eqref{Kdelta}. Then, there exists a sequence of functions $\tilde{\bd{q}}_{m} \in C^{1}_{c}(0, T; V_{f, \delta}(t))$, with $V_{f, \delta}(t)$ determined by the plate displacement $\omega_{\delta}$ via the definition \eqref{Vft}, such that
\begin{equation*}
\max_{0 \le t \le T} ||\tilde{\bd{q}} - \tilde{\bd{q}}_{m}||_{L^{2}(\Omega_{f, \delta}(t))} \to 0, \qquad \text{ as } m \to \infty.
\end{equation*}
\end{lemma}

\begin{proof}
There exists a rectangular two-dimensional maximal domain $\Omega_{M}$ of the form $[0, L] \times [-R, R_{max}]$ for some positive constant $R_{max}$ that contains all of the domains $\Omega_{f, \delta}(t)$ for $t \in [0, T]$. We will extend $\tilde{\bd{q}}$ to the maximal spacetime domain $[0, T] \times \Omega_{M}$ by extending vertically in the radial direction by the trace of $\tilde{\bd{q}}$ along $\Gamma_{\delta}(t)$. In particular, we define
\begin{multline}\label{traceext}
\small
\tilde{\bd{q}}(t, x, y) = K(0, t, x, \omega_{\delta}(t, x)) \bd{u}_{0}\left(x, \omega_{\delta}(0, x)\right), 
\qquad \text{ for } (t, x, y) \in ([0, T] \times \Omega_{M}) - ([0, T] \times \Omega_{f, \delta}(t)).
\end{multline}
Note that this extension preserves the \textit{divergence free} property.

We have the following two claims about the extended function, considered as a function on the \textit{fixed} maximal domain $\Omega_{M}$. First, we claim that $\tilde{\bd{q} }\in L^{\infty}(0, T; H^{1}(\Omega_{M}))$. Second, we claim that $\tilde{\bd{q}} \in C(0, T; L^{2}(\Omega_{M}))$. To see that $\tilde{\bd{q}} \in L^{\infty}(0, T; H^{1}(\Omega_{M}))$, we note that $\omega_{\delta}$ and $\partial_{x} \omega_{\delta}$ are bounded uniformly pointwise, and furthermore $\bd{u}_{0}$ and its first spatial derivatives are bounded by assumption. In addition, $\partial_{x}^{2} \omega_{\delta} \in L^{\infty}(0, T; L^{2}(\Gamma))$, which allows us to conclude that $\tilde{\bd{q}} \in L^{\infty}(0, T; H^{1}(\Omega_{M}))$. 

Next, we want to verify that $\tilde{\bd{q}} \in C(0, T; L^{2}(\Omega_{M}))$. Consider any $t \in [0, T]$ and consider any $s \in [0, T]$ with $s \ne t$. We define the following regions:
\begin{align*}
A(s, t) = \Omega^{M}_{f} \cap (\Omega_{f, \delta}(s) \cup \Omega_{f, \delta}(t))^{c},
\
&B(s, t) = [\Omega_{f, \delta}(s) \cap (\Omega_{f, \delta}(t))^{c}] \cup [(\Omega_{f, \delta}(s))^{c} \cap \Omega_{f, \delta}(t)],
\\
&C(s, t) = \Omega_{f, \delta}(s) \cap \Omega_{f, \delta}(t).
\end{align*}
Consider $\epsilon > 0$. We want to find $h > 0$ such that 
\begin{equation}\label{L2continuity}
||\tilde{\bd{q}}(t, \cdot) - \tilde{\bd{q}}(s, \cdot)||^{2}_{L^{2}(\Omega_{M})} \le \epsilon, \qquad \text{ for all } s \in (t - h, t + h) \cap [0, T]. 
\end{equation}
We compute that
\begin{align}\label{IABC}
||\tilde{\bd{q}}(t, \cdot) - \tilde{\bd{q}}(s, \cdot)||_{L^{2}(\Omega_{M})}^{2} 
&= \int_{A(s, t)} |\tilde{\bd{q}}(t, x, y) - \tilde{\bd{q}}(s, x, y)|^{2} + \int_{B(s, t)} |\tilde{\bd{q}}(t, x, y) - \tilde{\bd{q}}(s, x, y)|^{2} 
\nonumber 
\\
&+ \int_{C(s, t)} |\tilde{\bd{q}}(t, x, y) - \tilde{\bd{q}}(s, x, y)|^{2} 
= I_{A} + I_{B} + I_{C}.  
\end{align}
We estimate each of the terms $I_{A}$, $I_{B}$, and $I_{C}$ separately. 

For $I_{A}$, we recall that we are extending by the trace as in \eqref{traceext} on $A(s, t)$, so we have that
\begin{equation*}
I_{A} = \int_{A(s, t)} |K_{\delta}(0, t, x, \omega_{\delta}(t, x)) - K_{\delta}(0, s, x, \omega_{\delta}(s, x))|^{2} \cdot |\bd{u}_{0}(x, \omega_{\delta}(0, x))|^{2}.
\end{equation*}
We have that $|\bd{u}_{0}(x, \omega_{\delta}(0, x))| \le M_{1}$ for some constant $M_{1}$ by the fact that $\bd{u}_{0}$ is continuous on $\overline{\Omega_{f}(0)}$. By continuity, we can choose $h > 0$ sufficiently small so that
\begin{equation*}
|K_{\delta}(0, t, x, \omega_{\delta}(t, x)) - K_{\delta}(0, s, x, \omega_{\delta}(s, x))|^{2} < \frac{\epsilon}{3M_{1}^{2}(R + R_{max})L}, \qquad \text{ for all } s \in (t - h, t + h) \cap [0, T]. 
\end{equation*}
Thus, for all $s \in (t - h, t + h) \cap [0, T]$, 
\begin{equation*}
I_{A} \le |A(s, t)| \cdot \frac{\epsilon}{3 (R + R_{max})L} \le \frac{\epsilon}{3}.
\end{equation*}

For $I_{B}$, we will use the fact that $\omega_{\delta}$ does not change much in time over small time intervals, by continuity. We note that there exists a uniform constant $M_{2}$ such that $|\tilde{\bd{q}}| \le M_{2}$ on $[0, T] \times \Omega_{M}$.  Hence, 
\begin{equation*}
I_{B} = \int_{B(s, t)} |\tilde{\bd{q}}(t, z, r) - \tilde{\bd{q}}(s, z, r)|^{2} \le |B(s, t)| \cdot 4M_{2}^{2} = 4M_{2}^{2} \int_{0}^{L} |\omega_{\delta}(t, x) - \omega_{\delta}(s, x)| dx.
\end{equation*}
Because $\omega_{\delta} \in L^{\infty}(0, T; H^{2}_{0}(\Gamma)) \cap W^{1, \infty}(0, T; L^{2}(\Gamma))$, there exists $h > 0$ sufficiently small such that
\begin{equation*}
|\omega_{\delta}(t, x) - \omega_{\delta}(s, x)| \le \frac{\epsilon}{12M_{2}^{2}L}, \qquad \text{ for all } x \in [0, L] \text{ and } s \in (t - h, t + h) \cap [0, T].
\end{equation*}
This allows us to conclude that $I_{B} \le \frac{\epsilon}{3}$, for all $s \in (t - h, t + h) \cap [0, T]$. 

For $I_{C}$, we refer to the definition of $\tilde{\bd{q}}$ in \eqref{tildeq} and note that $K_{\delta}(0, t, x, y)$ is continuous in time uniformly in $(x, y) \in [0, L] \times [-R, R_{max}]$, $\bd{u}_{0}$ is uniformly continuous as a function on $\overline{\Omega_{f}(0)}$, and $\omega_{\delta}(t, x)$ is continuous in time uniformly in $x \in [0, L]$. Hence, there exists $h > 0$ sufficiently small such that
\begin{equation*}
|\tilde{\bd{q}}(t, x, y) - \tilde{\bd{q}}(s, x, y)|^{2} \le \frac{\epsilon}{3(R + R_{max})L}, \qquad \text{ for all } (x, y) \in C(s, t) \text{ and } s \in (t - h, t + h) \cap [0, T],
\end{equation*}
which gives the desired result that $I_{C} \le \frac{\epsilon}{3}$ for all $s \in (t - h, t + h) \cap [0, T]$. Thus, by using \eqref{IABC}, we have established \eqref{L2continuity}.

Since $\tilde{\bd{q}} \in L^{\infty}(0, T; H^{1}(\Omega_{M})) \cap C(0, T; L^{2}(\Omega_{M}))$, we can extend $\tilde{\bd{q}}$ to a continuous function on all of $\mathbb{R}$ as follows. We can find an increasing sequence $T_{m}$ with $T_{m} \to T$ as $m \to \infty$, such that $\tilde{\bd{q}}(T_{m}) \in H^{1}(\Omega_{M})$ for all $m$. Define an extension $\hat{\bd{q}}_{m}$ for each $m$ to all of $\mathbb{R}$ by $\hat{\bd{q}}_{m} = \tilde{\bd{q}}$ if $t \in [0, T_{m}]$, 
\begin{equation*}
\hat{\bd{q}}_{m} = \tilde{\bd{q}}(0), \qquad \text{ if } t < 0, \qquad \hat{\bd{q}}_{m} = \tilde{\bd{q}}(T_{m}), \qquad \text{ if } t > T_{m}.
\end{equation*}
Define
\begin{equation*}
\tilde{\bd{q}}_{m} = \hat{\bd{q}}_{m} * j_{1/m},
\end{equation*}
where the convolution is a convolution in time with $j_{\nu}$ for $\alpha = 1/m$. Because $\hat{\bd{q}}_{m} \in L^{\infty}(0, T; H^{1}(\Omega_{M})) \cap C(0, T; L^{2}(\Omega_{M}))$ with $\hat{\bd{q}}_{m}$ being divergence free for every $t \in [0, T]$, we have that $\tilde{\bd{q}}_{m}$ restricted to $\bigcup_{t \in [0, T]} \{t\} \times \Omega_{f, \delta}(t)$ gives a function in $C^{1}([0, T); V_{f, \delta}(t))$, where $V_{f, \delta}(t)$ is the space defined in \eqref{Vft} with the plate displacement $\omega_{\delta}$. The fact that
\begin{equation*}
\max_{0 \le t \le T} ||\tilde{\bd{q}} - \tilde{\bd{q}}_{m}||_{L^{2}(\Omega_{f, \delta}(t))} \to 0, \qquad \text{ as } m \to \infty,
\end{equation*}
follows from the uniform continuity of $\tilde{\bd{q}}$ on $[0, T]$ as a function taking values in $L^{2}(\Omega_{M})$, convergence properties of convolutions, and the fact that $\tilde{\bd{q}} \in C(0, T; L^{2}(\Omega_{M}))$ which gives the convergence
\begin{equation*}
\max_{t \in [T_{m}, T]} ||\tilde{\bd{q}}(T) - \tilde{\bd{q}}(t)||_{L^{2}(\Omega_{M})} \to 0, \qquad \text{ as } m \to \infty.
\end{equation*}
\end{proof}

\begin{lemma}\label{qtildecontinuity}
For the function $\tilde{\bd{q}}$ defined in \eqref{tildeq}, there exists a measure zero subset $S$ of $[0, T]$ such that 
\begin{equation*}
\lim_{t \to 0, t \in [0, T] \cap S^{c}} \int_{\Omega_{f,\delta}(t)} \bd{u}_{\delta}(t) \cdot \tilde{\bd{q}}(t) = \int_{\Omega_{f, \delta}(0)} \bd{u}_{0} \cdot \tilde{\bd{q}}(0). 
\end{equation*}

\begin{proof}
Note that because $\partial_{t}\tilde{\bd{q}}$ is not necessarily in $H^{1}(\Omega_{f, \delta}(t))$, $\tilde{\bd{q}}$ is not a valid test function. Thus, we use the sequence $\tilde{\bd{q}}_{m} \in C^{1}(0, T; \mathcal{V}_{f, \delta}(t))$ from Lemma \ref{qtilde}, which satisfies
\begin{equation*}
\max_{0 \le t \le T} ||\tilde{\bd{q}} - \tilde{\bd{q}}_{m}||_{L^{2}(\Omega_{f, \delta}(t))} \to 0, \qquad \text{ as } m \to \infty.
\end{equation*}

We can then apply Lemma \ref{weakcontinuity} to each of the test functions $\tilde{\bd{q}}_{m}$, to deduce that there exists a measure zero subset $S_{m}$ of $[0, T]$ such that
\begin{equation*}
\lim_{t \to 0, t \in [0, T] \cap S_{m}^{c}} \int_{\Omega_{f,\delta}(t)} \bd{u}_{\delta}(t) \cdot \tilde{\bd{q}}_{m}(t) = \int_{\Omega_{f, \delta}(0)} \bd{u}_{0} \cdot \tilde{\bd{q}}_{m}(0). 
\end{equation*}
In addition, by uniform boundedness, $\bd{u}_{\delta} \in L^{\infty}(0, T; L^{2}(\Omega_{f, \delta}(t)))$, and hence, there exists a measure zero subset $S_{0}$ of $[0, T]$, and a positive constant $C$ such that $||\bd{u}_{0}||_{L^{2}(\Omega_{f, \delta}(0))} \le C$, and 
\begin{equation}\label{S0}
||\bd{u}_{\delta}(t)||_{L^{2}(\Omega_{f, \delta}(t))} \le C, \qquad \text{ for all } t \in S_{0}^{c}.
\end{equation}
Define $S = S_{0} \cup \bigcup_{m \ge 1} S_{m}$, which is also a measure zero subset of $[0, T]$. Then, for each $m$, 
\begin{equation}\label{Sm}
\lim_{t \to 0, t \in [0, T] \cap S^{c}} \int_{\Omega_{f,\delta}(t)} \bd{u}_{\delta}(t) \cdot \tilde{\bd{q}}_{m}(t) = \int_{\Omega_{f, \delta}(0)} \bd{u}_{0} \cdot \tilde{\bd{q}}_{m}(0). 
\end{equation}
By passing to the limit in $m$, we claim that in addition,
\begin{equation*}
\lim_{t \to 0, t \in [0, T] \cap S^{c}} \int_{\Omega_{f,\delta}(t)} \bd{u}_{\delta}(t) \cdot \tilde{\bd{q}}(t) = \int_{\Omega_{f, \delta}(0)} \bd{u}_{0} \cdot \tilde{\bd{q}}(0). 
\end{equation*}

To see this, consider $\epsilon > 0$. We claim that there exists $h > 0$ sufficiently small such that for all $t \in (0, h) \cap S^{c}$, 
\begin{equation*}
\left|\int_{\Omega_{f,\delta}(t)} \bd{u}_{\delta}(t) \cdot \tilde{\bd{q}}(t) - \int_{\Omega_{f, \delta}(0)} \bd{u}_{0} \cdot \tilde{\bd{q}}(0)\right| < \epsilon.
\end{equation*}
We can choose $M$ sufficiently large such that $\displaystyle \max_{0 \le t \le T} ||\tilde{\bd{q}} - \tilde{\bd{q}}_{M}||_{L^{2}(\Omega_{f, \delta}(t))} < \frac{\epsilon}{3C}$, where $C$ is defined by \eqref{S0}. Therefore, for all $t \in [0, T] \cap S^{c}$,
\begin{equation*}
\left|\int_{\Omega_{f,\delta}(t)} \bd{u}_{\delta}(t) \cdot \tilde{\bd{q}}(t) - \int_{\Omega_{f,\delta}(t)} \bd{u}_{\delta}(t) \cdot \tilde{\bd{q}}_{M}(t)\right| < \frac{\epsilon}{3}.
\end{equation*}
In addition,
\begin{equation*}
\left|\int_{\Omega_{f, \delta}(0)} \bd{u}_{0} \cdot \tilde{\bd{q}}(0) - \int_{\Omega_{f, \delta}(0)} \bd{u}_{0} \cdot \tilde{\bd{q}}_{M}(0)\right| < \frac{\epsilon}{3}.
\end{equation*}
By applying \eqref{Sm} with $m = M$, we can choose $h > 0$ sufficiently small such that for all $t \in (0, h) \cap S^{c}$,
\begin{equation*}
\left|\int_{\Omega_{f,\delta}(t)} \bd{u}_{\delta}(t) \cdot \tilde{\bd{q}}_{M}(t) - \int_{\Omega_{f, \delta}(0)} \bd{u}_{0} \cdot \tilde{\bd{q}}_{M}(0)\right| < \frac{\epsilon}{3}.
\end{equation*}
Thus, by applying the triangle inequality, we have that for all $t \in (0, h) \cap S^{c}$,
\begin{equation*}
\left|\int_{\Omega_{f,\delta}(t)} \bd{u}_{\delta}(t) \cdot \tilde{\bd{q}}(t) - \int_{\Omega_{f, \delta}(0)} \bd{u}_{0} \cdot \tilde{\bd{q}}(0)\right| < \epsilon,
\end{equation*}
which establishes the desired result. 

\end{proof}

\end{lemma}

We can now prove the final result of this appendix. We recall the definition of $\widehat{\bd{u}}$ from \eqref{uhat}.

\begin{lemma}\label{alphaconvolution}
In the limit as $\nu \to 0$ we have the following convergence results:
\begin{equation*}
\int_{\Omega_{f, \delta}(0)} \widehat{\bd{u}}(0) \cdot (\bd{u}_{\delta})_{\nu}(0) \to \int_{\Omega_{f, \delta}(0)} |\bd{u}_{0}|^{2}, \quad \text{ and } \quad \int_{\Omega_{f,\delta}(t)} \widehat{\bd{u}}(t) \cdot (\bd{u}_{\delta})_{\nu}(t) \to \int_{\Omega_{f,\delta}(t)} \widehat{\bd{u}}(t) \cdot \bd{u}_{\delta}(t),
\end{equation*}
for almost all points $t \in (0, T]$. 
\end{lemma}

\begin{proof}
The second convergence for almost all points $t \in (0, T]$ follows directly from Lemma \ref{alphaconv} and the fact that $\widehat{\bd{u}} \in L^{\infty}(0, T; L^{2}(\Omega_{f, \delta}(t)))$. 

So we just need to verify the convergence at $t = 0$. To do this, we note that $\widehat{\bd{u}}(0) = \bd{u}_{0}$. Hence,
\begin{align*}
\int_{\Omega_{f, \delta}(0)} &\widehat{\bd{u}}(0) \cdot (\bd{u}_{\delta})_{\nu}(0) \\
&= \int_{\Omega^{\omega_{0}}} \left(\int_{\mathbb{R}} K_{\delta}(s, 0, x, y) \bd{u}_{\delta}\left(s, x, \frac{R + \omega_{\delta}(s, x)}{R + \omega_{\delta}(0, x)}(R + y) - R\right) j_{\delta}(t - s) ds\right) \bd{u}_{0}(x, y) dx dy \\
&= \int_{\mathbb{R}} \left(\int_{\Omega^{\omega_{0}}} K_{\delta}(s, 0, x, y) \bd{u}_{\delta} \left(s, x, \frac{R + \omega_{\delta}(s, x)}{R + \omega_{\delta}(0, x)}(R + y) - R\right) \cdot \bd{u}_{0}(x, y) dx dy \right) j_{\nu}(t - s) ds \\
&=  \int_{\mathbb{R}} \left(\int_{\Omega_{f, \delta}(s)} \bd{u}_{\delta}(s, x, y) \cdot \frac{R + \omega_{\delta}(0, x)}{R + \omega_{\delta}(s, x)} K_{\delta}^{t}\left(s, 0, x, \frac{R + \omega_{\delta}(0, x)}{R + \omega_{\delta}(s, x)}(R + y) - R\right) \right. \\
&\quad \left. \cdot \bd{u}_{0}\left(x, \frac{R + \omega_{\delta}(0, x)}{R + \omega_{\delta}(s, x)}(R + y) - R\right) dx dy\right) j_{\nu}(t - s) ds.
\end{align*}

We compute 
\begin{align*}
&\frac{R + \omega_{\delta}(0, x)}{R + \omega_{\delta}(s, x)} \cdot K_{\delta}^{t}\left(s, 0, x, \frac{R + \omega_{\delta}(0, x)}{R + \omega_{\delta}(s, x)}(R + y) - R\right) = \begin{pmatrix}
1 & (R + y)\nabla\left(\frac{R + \omega_{\delta}(0, x)}{R + \omega_{\delta}(s, x)}\right) \\
0 & \frac{R + \omega_{\delta}(0, x)}{R + \omega_{\delta}(s, x)} \\
\end{pmatrix} 
\\
&= \begin{pmatrix}
\frac{R + \omega_{\delta}(0, x)}{R + \omega_{\delta}(s, x)} & 0 \\
-(R + y)\nabla\left(\frac{R + \omega_{\delta}(0, x)}{R + \omega_{\delta}(s, x)}\right) & 1 \\
\end{pmatrix} 
+
\begin{pmatrix}
1 - \frac{R + \omega_{\delta}(0, x)}{R + \omega_{\delta}(s, x)} & (R + y)\nabla\left(\frac{R + \omega_{\delta}(0, x)}{R + \omega_{\delta}(s, x)}\right) \\
(R + y)\nabla\left(\frac{R + \omega_{\delta}(0, x)}{R + \omega_{\delta}(s, x)}\right) & \frac{R + \omega_{\delta}(0, x)}{R + \omega_{\delta}(s, x)} - 1 \\
\end{pmatrix} \\
&:= K_{\delta}(0, s, x, y) + R_{\delta}(0, s, x, y).
\end{align*}

Hence,
\begin{align*}
&\int_{\Omega_{f, \delta}(0)} \widehat{\bd{u}}(0) \cdot (\bd{u}_{\delta})_{\nu}(0) \\
&= \int_{\mathbb{R}} \left(\int_{\Omega_{f, \delta}(s)} \bd{u}_{\delta} (s, x, y) \cdot K_{\delta}(0, s, x, y) \bd{u}_{0}\left(x, \frac{R + \omega_{\delta}(0, x)}{R + \omega_{\delta}(s, x)}(R + y) - R\right) dx dy \right) j_{\nu}(t - s) ds \\
&+ \int_{\mathbb{R}} \left(\int_{\Omega_{f, \delta}(s)} \bd{u}_{\delta} (s, x, y) \cdot R_{\delta}(0, s, x, y) \bd{u}_{0}\left(x, \frac{R + \omega_{\delta}(0, x)}{R + \omega_{\delta}(s, x)}(R + y) - R\right) dx dy \right) j_{\nu}(t - s) ds = I_{K, \delta} + I_{R, \delta}. 
\end{align*}
Note that
\begin{equation*}
I_{K, \delta} = \int_{\mathbb{R}} \left(\int_{\Omega_{f, \delta}(s)} \bd{u}_{\delta}(s, x, y) \cdot \tilde{\bd{q}}(s, x, y) dx dy \right) j_{\nu}(t - s) ds
\end{equation*}
where $\tilde{\bd{q}}$ is defined by \eqref{tildeq}. Since $\bd{u}_{\delta}(s) = \bd{u}_{\delta}(-s)$ so that $\omega_{\delta}(s) = \omega_{\delta}(-s)$ for $s \le 0$ 
(see the extension procedure), we conclude by Lemma \ref{qtildecontinuity} that 
\begin{equation*}
I_{K, \delta} \to \int_{\Omega_{f, \delta}(0)} \bd{u}_{0} \cdot \tilde{\bd{q}}(0) = \int_{\Omega_{f, \delta}(0)} |\bd{u}_{0}|^{2}, \qquad \text{ as } \nu \to 0. 
\end{equation*}

So it suffices to show that $I_{R, \delta} \to 0$ as $\nu \to 0$. This follows from the fact that $|R_{\delta}| \to 0$ uniformly as $s \to 0$. In particular,
\begin{equation*}
\int_{\Omega_{f, \delta}(s)} \left|\bd{u}_{\delta}(s, x, y) \cdot \bd{u}_{0}\left(x, \frac{R + \omega_{\delta}(0, x)}{R + \omega_{\delta}(s, x)}(R + y) - R\right)\right| dx dy \le C, \qquad \text{ for almost all } s \in [0, T],
\end{equation*}
by the boundedness of $\bd{u}_{\delta} \in L^{\infty}(0, T; L^{2}(\Omega_{f, \delta}(t))$ and the fact that $\bd{u}_{0}$ is uniformly bounded. In addition, by the continuity properties of $\omega_{\delta}$ in time, we have that 
\begin{equation*}
\max_{(x, y) \in \overline{\Omega_{f, \delta}(s)}} |R_{\delta}(0, s, x, y)| \to 0, \qquad \text{ as } s \to 0,
\end{equation*}
which implies that $I_{R, \delta} \to 0$ as $\nu \to 0$. This completes the proof. 

We will use this result in the next section to estimate the first term $T_1$, see \eqref{sumT} in the Gronwall's estimate.
\subsection{Gronwall's terms estimates}\label{appendix2}
In this appendix  we provide details of the derivation of the terms appearing in \eqref{sumT} and the calculations
providing the desired estimates of the terms in \eqref{sumT} used to prove Gronwall's estimate in Section~\ref{GronwallSection}. 

\vskip 0.1in
\noindent
{\bf{Term T1.}} To derive term $T_1$, defined in \eqref{T1}, we first 
multiply the weak formulation \eqref{weaknoregularization} for $\bd{u}$ with the test function $\bd{v} = \bd{u} - (\widecheck{\bd{u}}_{\delta})_{\nu}$ 
to obtain the terms:
{{
\begin{align*}
T_{1, 1} &= -\int_{0}^{t} \int_{\Omega_{f}(s)} \bd{u} \cdot \partial_{t} \left[\bd{u} - (\widecheck{\bd{u}}_{\delta})_{\nu}\right] - \frac{1}{2} \int_{0}^{t} \int_{\Gamma(s)} (\bd{\xi} \cdot \bd{n}) \bd{u} \cdot [\bd{u} - (\widecheck{\bd{u}}_{\delta})_{\nu}] \\
&+ \int_{\Omega_{f}(t)} \bd{u}(t) \cdot [\bd{u} - (\widecheck{\bd{u}}_{\delta})_{\nu}](t) - \int_{\Omega_{f}(0)} \bd{u}(0) \cdot [\bd{u} - (\widecheck{\bd{u}}_{\delta})_{\nu}](0),
\end{align*}
}}
where $\Omega_{f}(0)$ is the fluid domain corresponding to the initial structure displacement $\omega_{0}$. We note that $\bd{u}$ is smooth in time and $(\widecheck{\bd{u}}_{\delta})_{\nu}$  is differentiable in time as a result of the time convolution. Thus, by the Reynold's transport theorem, 
{{
\begin{equation*}
T_{1, 1} = \int_{0}^{t} \int_{\Omega_{f}(s)} \partial_{t}\bd{u} \cdot [\bd{u} - (\widecheck{\bd{u}}_{\delta})_{\nu}] + \frac{1}{2} \int_{0}^{t} \int_{\Gamma(s)} (\bd{\xi} \cdot \bd{n}) \bd{u} \cdot [\bd{u} - (\widecheck{\bd{u}}_{\delta})_{\nu}].
\end{equation*}
}}
Because $\bd{u}$ is smooth and by the weak convergence properties of $(\widecheck{\bd{u}}_{\delta})_{\nu}$ in Proposition \ref{alphaconvprop}, 
{{
\begin{equation*}
T_{1, 1} = \int_{0}^{t} \int_{\Omega_{f}(s)} \partial_{t}\bd{u} \cdot [\bd{u} - \widecheck{\bd{u}}_{\delta}] + \frac{1}{2} \int_{0}^{t} \int_{\Gamma(s)} (\bd{\xi} \cdot \bd{n}) \bd{u} \cdot [\bd{u} - \widecheck{\bd{u}}_{\delta}] + K_{1, 1, \nu},
\end{equation*}
}}
where $K_{1, 1, \nu} \to 0$ as $\nu \to 0$. Using estimates as found in \cite{WeakStrongFSI}, we can transfer the first integral from {{$\Omega_{1}(s)$ to $\Omega_{f,\delta}(s)$}} at the cost of an additional term, so that 
{{
\begin{equation*}
T_{1, 1} = \int_{0}^{t} \int_{\Omega_{f,\delta}(s)} \partial_{t} \widehat{\bd{u}} \cdot (\widehat{\bd{u}} - \bd{u}_{\delta}) + \frac{1}{2} \int_{0}^{t} \int_{\Gamma(s)} (\bd{\xi} \cdot \bd{n}) \bd{u} \cdot (\bd{u} - \widecheck{\bd{u}}_{\delta}) + {{\tilde{R}_{1}}} + K_{1, 1, \nu},
\end{equation*}
}}
where
{{
\begin{multline*}
{{|\tilde{R}_{1}|}} \le \epsilon \int_{0}^{t} ||\widehat{\bd{u}} - \bd{u}_{\delta}||^{2}_{H^{1}(\Omega_{f,\delta}(s))} \\
+ C(\epsilon) \left(\int_{0}^{t} ||\omega - \omega_{\delta}||^{2}_{H^{2}(\Gamma)} + \int_{0}^{t} ||\partial_{t}\omega - \partial_{t}\omega_{\delta}||_{L^{2}(\Gamma)}^{2} + \int_{0}^{t} ||\widehat{\bd{u}} - \bd{u}_{\delta}||_{L^{2}(\Omega_{f,\delta}(s))}^{2}\right).
\end{multline*}
}}
Thus, by using Proposition \ref{alphaconvprop} again,
{{
\begin{equation}\label{T1-1}
T_{1, 1} = \int_{0}^{t} \int_{\Omega_{f,\delta}(s)} \partial_{t} \widehat{\bd{u}} \cdot (\widehat{\bd{u}} - (\bd{u}_{\delta})_{\nu}) + \frac{1}{2} \int_{0}^{t} \int_{\Gamma(s)} (\bd{\xi} \cdot \bd{n}) \bd{u} \cdot (\bd{u} - (\widecheck{\bd{u}}_{\delta})_{\nu}) + {{\tilde{R}_{1} + K_{1, 1, \nu}}},
\end{equation}
}}
where {{$K_{1, 1, \nu} \to 0$}} as $\nu \to 0$. 

Next, we test the regularized weak formulation for $\bd{u}_{\delta}$ with $\widehat{\bd{u}}$ and obtain the following terms:
{{
\begin{align*}
T_{1, 2} = -\int_{0}^{t} \int_{\Omega_{f,\delta}(s)} \bd{u}_{\delta} \cdot \partial_{t} \widehat{\bd{u}} - \frac{1}{2} \int_{0}^{t} \int_{\Gamma_{\delta}(s)} (\bd{\xi}_{\delta} \cdot \bd{n}_{\delta}) \bd{u}_{\delta} \cdot \widehat{\bd{u}} 
+ \int_{\Omega_{f,\delta}(t)} \bd{u}_{\delta}(t) \cdot \widehat{\bd{u}}(t) - \int_{\Omega_{f}(0)} \bd{u}_{\delta}(0) \cdot \widehat{\bd{u}}(0). 
\end{align*}
}}
We want to integrate by parts in time, but $\bd{u}_{\delta}$ is not necessarily smooth in time. Thus, we replace $\bd{u}_{\delta}$ by its time regularization $(\bd{u}_{\delta})_{\nu}$ at the cost of a term $K_{1, 2, \nu}$ which goes to zero as $\nu \to 0$ by Proposition \ref{alphaconvprop}. Combining this with the Reynold's transport theorem, we get:
{{
\begin{equation}\label{T1-2}
T_{1, 2} = \int_{0}^{t} \int_{\Omega_{f,\delta}(s)} \partial_{t}\left[(\bd{u}_{\delta})_{\nu}\right] \cdot \widehat{\bd{u}} + \frac{1}{2} \int_{0}^{t} \int_{\Gamma_{\delta}(s)} (\bd{\xi}_{\delta} \cdot \bd{n}_{\delta}) (\bd{u}_{\delta})_{\nu} \cdot \widehat{\bd{u}} + K_{1, 2, \nu},
\end{equation}
}}
where $K_{1, 2, \nu} \to 0$ as $\nu \to 0$. 

Now, from the energy inequality, we obtain the terms
\begin{equation}\label{T1-3}
T_{1, 3} = \frac{1}{2} \int_{\Omega_{f,\delta}(t)} |\bd{u}_{\delta}(t)|^{2} - \frac{1}{2} \int_{\Omega_{f,\delta}(0)} |\bd{u}_{\delta}(0)|^{2}.
\end{equation}

Using the Reynold's transport theorem, the total contribution $T_{1} = T_{1, 1} - T_{1, 2} + T_{1, 3}$ is
{{
\begin{align*}
T_{1} &= \frac{1}{2}\int_{\Omega_{f,\delta}(t)} |\widehat{\bd{u}}(t)|^{2} - \frac{1}{2}\int_{\Omega_{f,\delta}(0)} |\widehat{\bd{u}}(0)|^{2} - \int_{\Omega_{f,\delta}(t)} (\widehat{\bd{u}} \cdot (\bd{u}_{\delta})_{\nu})(t) + \int_{\Omega_{f,\delta}(0)} (\widehat{\bd{u}} \cdot (\bd{u}_{\delta})_{\nu})(0) \\
&+ \frac{1}{2} \int_{\Omega_{f,\delta}(t)} |\bd{u}_{\delta}(t)|^{2} - \frac{1}{2} \int_{\Omega_{f,\delta}(0)} |\bd{u}_{\delta}(0)|^{2} - \frac{1}{2} \int_{0}^{t} \int_{\Gamma_{\delta}(s)} (\bd{\xi}_{\delta} \cdot \bd{n}_{\delta}) \widehat{\bd{u}} \cdot (\widehat{\bd{u}} - (\bd{u}_{\delta})_{\nu}) \\
&+ \frac{1}{2} \int_{0}^{t} \int_{\Gamma(s)} (\bd{\xi} \cdot \bd{n}) \bd{u} \cdot (\bd{u} - (\widecheck{\bd{u}}_{\delta})_{\nu}) + {{\tilde{R}_{1} + K_{1, 1, \nu}}} + K_{1, 2, \nu}. 
\end{align*}
}}
By Proposition \ref{alphaconvprop}, $(\bd{u}_{\delta})_{\nu}$ and $(\widecheck{\bd{u}}_{\delta})_{\nu}$ converge weakly to $\bd{u}_{\delta}$ and $\widecheck{\bd{u}}_{\delta}$ respectively, weakly in {{$L^{2}(0, T, W^{1, p}(\Omega_{f,\delta}(t)))$ and $L^{2}(0, T, W^{1, p}(\Omega_{f, 1}(t)))$}} for all $p \in [1, 2)$. 
Furthermore,  by Lemma \ref{alphaconvolution} proved in the appendix above, we have that
\begin{equation}\label{appendixconv}
\int_{\Omega_{f,\delta}(0)} (\widehat{\bd{u}} \cdot (\bd{u}_{\delta})_{\nu})(0) \to \int_{\Omega_{f,\delta}(0)} (\widehat{\bd{u}} \cdot \bd{u}_{\delta})(0), \quad \int_{\Omega_{f,\delta}(t)} (\widehat{\bd{u}} \cdot (\bd{u}_{\delta})_{\nu})(t) \to \int_{\Omega_{f,\delta}(t)} (\widehat{\bd{u}} \cdot \bd{u}_{\delta})(t).
\end{equation}
Thus, taking the limit as $\nu \to 0$, the contribution of this term is now
{{
\begin{align*}
T_{1} &= \frac{1}{2} \int_{\Omega_{f,\delta}(s)} |(\widehat{\bd{u}} - \bd{u}_{\delta})(t)|^{2} - \frac{1}{2} \int_{\Omega_{f,\delta}(0)} |(\widehat{\bd{u}} - \bd{u}_{\delta})(0)|^{2} \\
&- \frac{1}{2} \int_{0}^{t} \int_{\Gamma_{\delta}(s)} (\bd{\xi}_{\delta} \cdot \bd{n}_{\delta}) \widehat{\bd{u}} \cdot (\widehat{\bd{u}} - \bd{u}_{\delta}) + \frac{1}{2} \int_{0}^{t} \int_{\Gamma(s)} (\bd{\xi} \cdot \bd{n}) \bd{u} \cdot (\bd{u} - \widecheck{\bd{u}}_{\delta}) + {{\tilde{R}_{1}}}.
\end{align*}
}}
Since $\widehat{\bd{u}}(0) = \bd{u}_{\delta}(0) = \bd{u}_{0}$, we obtain after some standard estimates that
\begin{equation*}
T_{1} = \frac{1}{2} \int_{\Omega_{f,\delta}(t)} |(\widehat{\bd{u}} - \bd{u}_{\delta})(t)|^{2} + {{R_{1}}},
\end{equation*}
where
{{
\begin{align*}
{{|R_{1}|}} &\le \epsilon \int_{0}^{t} ||\widehat{\bd{u}} - \bd{u}_{\delta}||^{2}_{H^{1}(\Omega_{f,\delta}(s))} \\
&+ C(\epsilon) \left(\int_{0}^{t} ||\omega - \omega_{\delta}||^{2}_{H^{2}(\Gamma)} + \int_{0}^{t} ||\partial_{t}\omega - \partial_{t}\omega_{\delta}||_{L^{2}(\Gamma)}^{2} + \int_{0}^{t} ||\widehat{\bd{u}} - \bd{u}_{\delta}||_{L^{2}(\Omega_{f,\delta}(s))}^{2}\right).
\end{align*}
}}
This completes the calculations associated with term $T_1$.

\vskip 0.1in
\noindent
{\bf{Term T2.}} To estimate term $T_2$, defined in \eqref{T2} above, we notice that 
since $(\widecheck{\bd{u}}_{\delta})_{\nu}$ converges weakly to $\widecheck{\bd{u}}_{\delta}$ in {{$L^{2}(0, T; W^{1, p}(\Omega_{f, \delta}(t)))$}} for $p \in [1, 2)$ by Proposition \ref{alphaconvprop}, and because $\bd{u}$ is smooth, as $\nu \to 0$, we have that $T_{2}$ converges to
{{
\begin{align*}
T_{2} &:= \frac{1}{2} \int_{0}^{t} \int_{\Omega_{f}(s)} ((\bd{u} \cdot \nabla) \bd{u}) \cdot (\bd{u} - \widecheck{\bd{u}}_{\delta}) - \frac{1}{2} \int_{0}^{t} \int_{\Omega_{f}(s)} ((\bd{u} \cdot \nabla) (\bd{u} - \widecheck{\bd{u}}_{\delta})) \cdot \bd{u} \\
&- \frac{1}{2} \int_{0}^{t} \int_{\Omega_{f,\delta}(s)} ((\bd{u}_{\delta} \cdot \nabla) \bd{u}_{\delta}) \cdot (\widehat{\bd{u}} - \bd{u}_{\delta}) + \frac{1}{2} \int_{0}^{t} \int_{\Omega_{f,\delta}(s)} ((\bd{u}_{\delta} \cdot \nabla) (\widehat{\bd{u}} - \bd{u}_{\delta})) \cdot \bd{u}_{\delta}.
\end{align*}
}}
We note that the quantity
{{$\displaystyle 
\frac{1}{2} \int_{0}^{t} \int_{\Omega_{f,\delta}(s)} ((\bd{u}_{\delta} \cdot \nabla) \bd{u}_{\delta}) \cdot \bd{u}_{\delta},
$}}
is well-defined because {{$\bd{u}_{\delta} \in L^{\infty}(0, T; L^{2}(\Omega_{f,\delta}(t))) \cap L^{2}(0, T; H^{1}(\Omega_{f,\delta}(t)))$}}, which by interpolation is in {{$L^{4}(0, T; H^{1/2}(\Omega_{f,\delta}(t)))$, and hence by Sobolev inequalities embeds into $L^{4}(0, T; L^{4}(\Omega_{f,\delta}(t)))$}}. 

We want to transfer the integrals
{{
\begin{equation}\label{term2integrals}
\int_{0}^{t} \int_{\Omega_{f}(s)} ((\bd{u} \cdot \nabla) \bd{u}) \cdot (\bd{u} - \widecheck{\bd{u}}_{\delta}), \qquad \int_{0}^{t} \int_{\Omega_{f}(s)} ((\bd{u} \cdot \nabla) (\bd{u} - \widecheck{\bd{u}}_{\delta})) \cdot \bd{u},
\end{equation}
}}
to integrals on {{$\Omega_{f,\delta}(s)$}} by using the map {{$\psi_{\delta}(s, \cdot): \Omega_{f,\delta}(s) \to \Omega_{f}(s)$}} defined by \eqref{psi}. We use 
\begin{equation*}
\widehat{\bd{u}} = \gamma_{\delta} J_{\delta}^{-1} \cdot (\bd{u} \circ \psi_{\delta}), \qquad \widehat{\bd{u}} - \bd{u}_{\delta} = \gamma_{\delta} J_{\delta}^{-1} \cdot ((\bd{u} - \widecheck{\bd{u}}_{\delta}) \circ \psi_{\delta}),
\end{equation*}
where we recall the definitions of the appropriate terms from  \eqref{psi}, \eqref{J}, \eqref{uhat}, and \eqref{ucheck}.

Following arguments found in \cite{WeakStrongFSI}, we obtain the following estimates. We have, using \eqref{gradidentity}, that
{{
\begin{align}\label{term2transfer1}
&\int_{0}^{t} \int_{\Omega_{f}(s)} ((\bd{u} \cdot \nabla) \bd{u}) \cdot (\bd{u} - \widecheck{\bd{u}}_{\delta}) 
= \int_{0}^{t} \int_{\Omega_{f,\delta}(s)} \gamma_{\delta} [(\nabla (\bd{u} \circ \psi_{\delta})) J_{\delta}^{-1} (\bd{u} \circ \psi_{\delta})] \cdot (\bd{u} - \widecheck{\bd{u}}_{\delta}) \circ \psi_{\delta} 
\nonumber \\
&= \int_{0}^{t} \int_{\Omega_{f,\delta}(s)} [(\nabla (\bd{u} \circ \psi_{\delta})) \widehat{\bd{u}}] \cdot [\gamma_{\delta}^{-1}J_{\delta}(\widehat{\bd{u}} - \bd{u}_{\delta})] 
\nonumber
\\
&= \int_{0}^{t} \int_{\Omega_{f,\delta}(s)} [(\nabla(\bd{u} \circ \psi_{\delta})) \widehat{\bd{u}}] \cdot (\widehat{\bd{u}} - \bd{u}_{\delta}) - \int_{0}^{t} \int_{\Omega_{f,\delta}(s)} [(\nabla(\bd{u} \circ \psi_{\delta})) \widehat{\bd{u}}] \cdot [(I - \gamma_{\delta}^{-1}J_{\delta})(\widehat{\bd{u}} - \bd{u}_{\delta})] 
\nonumber \\
&= \int_{0}^{t} \int_{\Omega_{f,\delta}(s)} ((\nabla \widehat{\bd{u}})\widehat{\bd{u}}) \cdot (\widehat{\bd{u}} - \bd{u}_{\delta}) + \int_{0}^{t} \int_{\Omega_{f,\delta}(s)} (\nabla ((I - \gamma_{\delta} J_{\delta}^{-1})(\bd{u} \circ \psi_{\delta})) \widehat{\bd{u}}) \cdot (\widehat{\bd{u}} - \bd{u}_{\delta}) 
\nonumber \\
&- \int_{0}^{t} \int_{\Omega_{f,\delta}(s)} [(\nabla(\bd{u} \circ \psi_{\delta})) \widehat{\bd{u}}] \cdot [(I - \gamma_{\delta}^{-1}J_{\delta})(\widehat{\bd{u}} - \bd{u}_{\delta})] 
\nonumber \\
&= \int_{0}^{t} \int_{\Omega_{f,\delta}(s)} [(\widehat{\bd{u}} \cdot \nabla) \widehat{\bd{u}}] \cdot (\widehat{\bd{u}} - \bd{u}_{\delta}) + {{R_{2, 1}}},
\end{align}
}}
where
{{
\begin{align*}
{{R_{2, 1}}} = \int_{0}^{t} \int_{\Omega_{f,\delta}(s)} (\nabla ((I - \gamma_{\delta} J_{\delta}^{-1})(\bd{u} \circ \psi_{\delta})) \widehat{\bd{u}}) \cdot (\widehat{\bd{u}} - \bd{u}_{\delta}) 
- \int_{0}^{t} \int_{\Omega_{f,\delta}(s)} [(\nabla(\bd{u} \circ \psi_{\delta})) \widehat{\bd{u}}] \cdot [(I - \gamma_{\delta}^{-1}J_{\delta})(\widehat{\bd{u}} - \bd{u}_{\delta})].
\end{align*}
}}
In the following estimates, we will repeatedly use the following inequalities, which hold for a constant $C$ that is independent of $\delta$:
\begin{equation*}
|\gamma_{\delta}^{-1}J_{\delta} - I| \le C(|\gamma_{\delta}^{-1} - 1| + |\nabla \gamma_{\delta}|) \le C||\omega - \omega_{\delta}||_{H^{2}(\Gamma)},
\end{equation*}
\begin{equation*}
|\gamma_{\delta} J_{\delta}^{-1} - I| \le C(|\gamma_{\delta} - 1| + |\nabla \gamma_{\delta}|) \le C||\omega - \omega_{\delta}||_{H^{2}(\Gamma)},
\end{equation*}
\begin{equation}\label{nablagammaJ}
|\nabla (\gamma_{\delta} J_{\delta}^{-1})| \le C(|\partial_{x}\gamma_{\delta}| + |\partial_{xx}\gamma_{\delta}|) \le C(||\omega - \omega_{\delta}||_{H^{2}(\Gamma)} + |\partial_{xx}(\omega - \omega_{\delta})|),
\end{equation}
so that 
\begin{equation}\label{gammaJinv}
||\nabla(\gamma_{\delta} J_{\delta}^{-1})||_{L^{2}(\Omega_{f,\delta}(t))} \le C||\omega - \omega_{\delta}||_{H^{2}(\Gamma)}.
\end{equation}
To obtain \eqref{nablagammaJ}, we estimate $|\partial_{xx}\gamma_{\delta}|$ by using the fact that $\omega$ is smooth so that $|\partial_{xx}\omega| \le C$ and a direct computation of $\partial_{xx} \gamma_{\delta}$.
Using these estimates, the Leibniz rule, and the smoothness of $\bd{u}$,  we get
{{
\begin{align*}
&\left|\int_{0}^{t} \int_{\Omega_{f,\delta}(s)} (\nabla ((I - \gamma_{\delta} J_{\delta}^{-1})(\bd{u} \circ \psi_{\delta})) \widehat{\bd{u}}) \cdot (\widehat{\bd{u}} - \bd{u}_{\delta}) \right| \\
&\le C \int_{0}^{t} ||\omega - \omega_{\delta}||_{H^{2}(\Gamma)} ||\widehat{\bd{u}} - \bd{u}_{\delta}||_{L^{2}(\Omega_{f,\delta}(s))} \le C\left(\int_{0}^{t} ||\omega - \omega_{\delta}||^{2}_{H^{2}(\Gamma)} + \int_{0}^{t} ||\widehat{\bd{u}} - \bd{u}_{\delta}||^{2}_{L^{2}(\Omega_{f,\delta}(s))}\right).
\end{align*}
}}
By using \eqref{gradidentity}, and  the fact that $|J_{\delta}| \le C$ is uniformly bounded, due to the fact that $|J_{\delta}| \le C(1 + ||\omega - \omega_{\delta}||_{H^{2}(\Gamma)}) \le C$ is uniformly bounded, we obtain a similar estimate:
{{
\begin{equation*}
\int_{0}^{t} \int_{\Omega_{f,\delta}(s)} [(\nabla(\bd{u} \circ \psi_{\delta})) \widehat{\bd{u}}] \cdot [(I - \gamma_{\delta}^{-1}J_{\delta})(\widehat{\bd{u}} - \bd{u}_{\delta})] \le C\left(\int_{0}^{t} ||\omega - \omega_{\delta}||^{2}_{H^{2}(\Gamma)} + \int_{0}^{t} ||\widehat{\bd{u}} - \bd{u}_{\delta}||^{2}_{L^{2}(\Omega_{f,\delta}(s))}\right).
\end{equation*}
}}
Thus, we obtain
{{
\begin{equation}\label{T2R1}
{{|R_{1}|}} \le C\left(\int_{0}^{t} ||\omega - \omega_{\delta}||^{2}_{H^{2}(\Gamma)} + \int_{0}^{t} ||\widehat{\bd{u}} - \bd{u}_{\delta}||^{2}_{L^{2}(\Omega_{f,\delta}(s))}\right).
\end{equation}
}}

We now focus on the second integral in \eqref{term2integrals}. By using \eqref{gradidentity} we obtain
{{
\begin{align}\label{term2transfer2}
&\int_{0}^{t} \int_{\Omega_{f}(s)} ((\bd{u} \cdot \nabla)(\bd{u} - \widecheck{\bd{u}}_{\delta})) \cdot \bd{u} = \int_{0}^{t} \int_{\Omega_{f,\delta}(s)} \gamma_{\delta} [(\nabla((\bd{u} - \widecheck{\bd{u}}_{\delta}) \circ \psi_{\delta})) J_{\delta}^{-1} (\bd{u} \circ \psi_{\delta})] \cdot (\bd{u} \circ \psi_{\delta}) 
\nonumber \\
&= \int_{0}^{t} \int_{\Omega_{f,\delta}(s)} [(\nabla((\bd{u} - \widecheck{\bd{u}}_{\delta}) \circ \psi_{\delta})) \widehat{\bd{u}}] \cdot (\gamma_{\delta}^{-1}J_{\delta}\widehat{\bd{u}}) 
\nonumber\\
&=  \int_{0}^{t} \int_{\Omega_{f,\delta}(s)} [(\nabla((\bd{u} - \widecheck{\bd{u}}_{\delta}) \circ \psi_{\delta})) \widehat{\bd{u}}] \cdot \widehat{\bd{u}} - \int_{0}^{t} \int_{\Omega_{f,\delta}(s)} [(\nabla((\bd{u} - \widecheck{\bd{u}}_{\delta}) \circ \psi_{\delta})) \widehat{\bd{u}}] \cdot [(I - \gamma_{\delta}^{-1}J_{\delta})\widehat{\bd{u}}] \nonumber\\
&= \int_{0}^{t} \int_{\Omega_{f,\delta}(s)} (\nabla (\widehat{\bd{u}} - \bd{u}_{\delta}) \widehat{\bd{u}}) \cdot \widehat{\bd{u}} + \int_{0}^{t} \int_{\Omega_{f,\delta}(s)} (\nabla [(I - \gamma_{\delta} J_{\delta}^{-1})((\bd{u} - \widecheck{\bd{u}}_{\delta}) \circ \psi_{\delta})] \widehat{\bd{u}}) \cdot \widehat{\bd{u}} 
\nonumber\\
&- \int_{0}^{t} \int_{\Omega_{f,\delta}(s)} [(\nabla((\bd{u} - \widecheck{\bd{u}}_{\delta}) \circ \psi_{\delta})) \widehat{\bd{u}}] \cdot [(I - \gamma_{\delta}^{-1}J_{\delta})\widehat{\bd{u}}] \nonumber \\
&= \int_{0}^{t} \int_{\Omega_{f,\delta}(s)} ((\widehat{\bd{u}} \cdot \nabla) (\widehat{\bd{u}} - \bd{u}_{\delta})) \cdot \widehat{\bd{u}} + {{R_{2, 2}}},
\end{align}
}}
where
{{
\begin{equation*}
R_{2, 2} := \int_{0}^{t} \int_{\Omega_{f,\delta}(s)} (\nabla [(I - \gamma_{\delta} J_{\delta}^{-1})((\bd{u} - \widecheck{\bd{u}}_{\delta}) \circ \psi_{\delta})] \widehat{\bd{u}}) \cdot \widehat{\bd{u}} - \int_{0}^{t} \int_{\Omega_{f,\delta}(s)} [(\nabla((\bd{u} - \widecheck{\bd{u}}_{\delta}) \circ \psi_{\delta})) \widehat{\bd{u}}] \cdot [(I - \gamma_{\delta}^{-1}J_{\delta})\widehat{\bd{u}}].
\end{equation*}
}}
To estimate {{$R_{2, 2}$}}, we will use the following inequalities:
\begin{align*}
|(\bd{u} - \widecheck{\bd{u}}_{\delta}) \circ \psi| &= |\gamma_{\delta}^{-1}J_{\delta} \cdot (\widehat{\bd{u}} - \bd{u}_{\delta})| 
\le C|\widehat{\bd{u}} - \bd{u}_{\delta}|,
\nonumber \\
|\nabla((\bd{u} - \widecheck{\bd{u}}_{\delta}) \circ \psi_{\delta})| 
&= |\nabla(\gamma_{\delta}^{-1}J_{\delta} \cdot (\widehat{\bd{u}} - \bd{u}_{\delta}))| \le |\nabla(\gamma_{\delta}^{-1}J_{\delta})| \cdot |\widehat{\bd{u}} - \bd{u}_{\delta}| + |\gamma_{\delta}^{-1}J_{\delta}| \cdot |\nabla(\widehat{\bd{u}} - \bd{u}_{\delta})| 
\nonumber \\
&\le C(|\nabla(\gamma_{\delta}^{-1}J_{\delta})| \cdot |\widehat{\bd{u}} - \bd{u}_{\delta}| + |\nabla(\widehat{\bd{u}} - \bd{u}_{\delta})|).
\end{align*}
From the fact that $\max\left(|I - \gamma_{\delta}^{-1}J_{\delta}|, |I - \gamma_{\delta} J_{\delta}^{-1}|\right) \le C\min\left(1, ||\omega - \omega_{\delta}||_{H^{2}(\Gamma)}\right)$, 
we obtain:
{{
\begin{align}\label{T2R2}
{{|R_{2, 2}|}} &\le C\left(\int_{0}^{t} \int_{\Omega_{f,\delta}(s)} |\nabla(\gamma_{\delta} J_{\delta}^{-1})| \cdot |(\bd{u} - \widecheck{\bd{u}}_{\delta}) \circ \psi_{\delta}| + \int_{0}^{t} \int_{\Omega_{f,\delta}(s)} |I - \gamma_{\delta} J_{\delta}^{-1}| \cdot |\nabla((\bd{u} - \widecheck{\bd{u}}_{\delta}) \circ \psi_{\delta})| \right.
\nonumber \\
&\left. + \int_{0}^{t} \int_{\Omega_{f,\delta}(s)} |I - \gamma_{\delta}^{-1}J_{\delta}| \cdot |\nabla((\bd{u} - \widecheck{\bd{u}}_{\delta}) \circ \psi_{\delta})|\right) 
\nonumber\\
&\le C\left(\int_{0}^{t} \int_{\Omega_{f,\delta}(s)} \left(|\nabla(\gamma_{\delta}J_{\delta}^{-1})| + |\nabla(\gamma_{\delta}^{-1}J_{\delta})|\right) \cdot |\widehat{\bd{u}} - \bd{u}_{\delta}| + \int_{0}^{t} \int_{\Omega_{f,\delta}(s)} ||\omega - \omega_{\delta}||_{H^{2}(\Gamma)} \cdot |\nabla(\widehat{\bd{u}} - \bd{u}_{\delta})|\right) 
\nonumber\\
&\le \epsilon \int_{0}^{t} ||\nabla(\widehat{\bd{u}} - \bd{u}_{\delta})||_{L^{2}(\Omega_{f,\delta}(s))}^{2} + C(\epsilon)\left(\int_{0}^{t} ||\omega - \omega_{\delta}||^{2}_{H^{2}(\Gamma)} + \int_{0}^{t} ||\widehat{\bd{u}} - \bd{u}_{\delta}||_{L^{2}(\Omega_{f,\delta}(s))}^{2}\right).
\end{align}
}}
In the last line, we use the following estimates, derived similarly as for \eqref{gammaJinv},
\begin{align*}
|\nabla (\gamma_{\delta}^{-1}J_{\delta})| \le C(|\partial_{x}(\gamma_{\delta}^{-1})| + |\partial_{x}\gamma_{\delta}| + |\partial_{xx}\gamma_{\delta}|) &\le C(||\omega - \omega_{\delta}||_{H^{2}(\Gamma)} + |\partial_{xx}(\omega - \omega_{\delta})|),
\\
||\nabla(\gamma_{\delta}^{-1} J_{\delta})||_{L^{2}(\Omega_{f,\delta}(t))} &\le C||\omega - \omega_{\delta}||_{H^{2}(\Gamma)}.
\end{align*}

Therefore, for the expression in \eqref{T2}, after transferring the integrals \eqref{term2transfer1} and \eqref{term2transfer2} and estimating {{$R_{2, 1}$ \eqref{T2R1} and $R_{2, 2}$ \eqref{T2R2}}}, the remaining terms are:
{{
\begin{align*}
&\frac{1}{2} \int_{0}^{t} \int_{\Omega_{f,\delta}(s)} [(\widehat{\bd{u}} \cdot \nabla) \widehat{\bd{u}}] \cdot (\widehat{\bd{u}} - \bd{u}_{\delta}) - [(\widehat{\bd{u}} \cdot \nabla)(\widehat{\bd{u}} - \bd{u}_{\delta})] \cdot \widehat{\bd{u}} \\
&- \frac{1}{2} \int_{0}^{t} \int_{\Omega_{f,\delta}(s))} [(\bd{u}_{\delta} \cdot \nabla) \bd{u}_{\delta}] \cdot (\widehat{\bd{u}} - \bd{u}_{\delta}) - [(\bd{u}_{\delta} \cdot \nabla)(\widehat{\bd{u}} - \bd{u}_{\delta})] \cdot \bd{u}_{\delta} \\
&= \frac{1}{2} \int_{0}^{t} \int_{\Omega_{f,\delta}(s)} [((\widehat{\bd{u}} - \bd{u}_{\delta}) \cdot \nabla) \bd{u}_{\delta}] \cdot \widehat{\bd{u}} - \frac{1}{2} \int_{0}^{t} \int_{\Omega_{f,\delta}(s)} [((\widehat{\bd{u}} - \bd{u}_{\delta}) \cdot \nabla) \widehat{\bd{u}}] \cdot \bd{u}_{\delta}.
\end{align*}
}}
In absolute values, the right hand-side can be bounded as follows:
{{
\begin{equation*}
\le \epsilon \int_{0}^{t} ||\nabla(\widehat{\bd{u}} - \bd{u}_{\delta})||_{L^{2}(\Omega_{f,\delta}(s))}^{2} + C(\epsilon) \int_{0}^{t} ||\widehat{\bd{u}} - \bd{u}_{\delta}||_{L^{2}(\Omega_{f,\delta}(s))}^{2}.
\end{equation*}
}}
Combining this estimate with \eqref{T2R1} and \eqref{T2R2} we obtain
{{
\begin{equation*}
|T_{2}| \le \epsilon \int_{0}^{t} ||\nabla(\widehat{\bd{u}} - \bd{u}_{\delta})||_{L^{2}(\Omega_{f,\delta}(s))}^{2} + C(\epsilon)\left(\int_{0}^{t} ||\omega - \omega_{\delta}||^{2}_{H^{2}(\Gamma)} + \int_{0}^{t} ||\widehat{\bd{u}} - \bd{u}_{\delta}||_{L^{2}(\Omega_{f,\delta}(s))}^{2}\right).
\end{equation*}
}}

\vskip 0.1in
\noindent
{\bf{Term T3.}} To estimate term $T_3$ defined in \eqref{T3}, we start by noting that 
because $\bd{u}$ and $\bd{\xi}$ are smooth, we can pass to the limit as $\nu \to 0$ using Proposition \ref{alphaconvprop} and the fact that $(\bd{\xi}_{\delta})_{\nu} \to \bd{\xi}_{\delta}$ strongly in $L^{2}(0, T; H^{1}(\Omega_{b}))$, so that we can ultimately just test with $\bd{v} = \bd{u} - \widecheck{\bd{u}}_{\delta}$ and $\bd{\psi} = \bd{\xi} - \bd{\xi}_{\delta}$. In the regularized weak formulation for $\bd{u}_{\delta}$, we test with $\bd{u}$ and $\bd{\xi}$. Note that both test functions $\bd{u} - \widecheck{\bd{u}}_{\delta}$ and $\widehat{\bd{u}} - \bd{u}_{\delta}$ have the same trace along $\Gamma(t)$ and $\Gamma_{\delta}(t)$ respectively, which we will formally denote by $\bd{u} - \bd{u}_{\delta}$ along the reference configuration of the interface $\Gamma$. Combining the resulting expressions, we have the following contribution of $T_{3}$ in the limit as $\nu \to 0$:
{{
\begin{align*}
T_{3} &= \frac{1}{2} \int_{0}^{t} \int_{\Gamma(s)} (\bd{u} \cdot \bd{n} - \bd{\xi} \cdot \bd{n}) \bd{u} \cdot (\bd{u} - \widecheck{\bd{u}}_{\delta}) - \frac{1}{2} \int_{0}^{t} \int_{\Gamma_{\delta}(s)} (\bd{u}_{\delta} \cdot \bd{n}_{\delta} - \bd{\xi}_{\delta} \cdot \bd{n}_{\delta}) \bd{u}_{\delta} \cdot \widehat{\bd{u}} \\
&+ \frac{1}{2} \int_{0}^{t} \int_{\Gamma(s)} |\bd{u}|^{2} (\bd{\xi} \cdot \bd{n} - \bd{u} \cdot \bd{n}) - \frac{1}{2} \int_{0}^{t} \int_{\Gamma(s)} |\bd{u}|^{2} (\bd{\xi}_{\delta} \cdot \bd{n} - \widecheck{\bd{u}}_{\delta} \cdot \bd{n}) \\
&- \frac{1}{2} \int_{0}^{t} \int_{\Gamma_{\delta}(s)} |\bd{u}_{\delta}|^{2} (\bd{\xi} \cdot \bd{n}_{\delta} - \widehat{\bd{u}} \cdot \bd{n}_{\delta}) = \frac{1}{2} \int_{0}^{t} \int_{\Gamma(s)} (\bd{\xi} \cdot \bd{n} - \bd{u} \cdot \bd{n}) \bd{u} \cdot \widecheck{\bd{u}}_{\delta} \\
&- \frac{1}{2} \int_{0}^{t} \int_{\Gamma(s)} (\bd{\xi}_{\delta} \cdot \bd{n} - \widecheck{\bd{u}}_{\delta} \cdot \bd{n}) |\bd{u}|^{2} - \frac{1}{2} \int_{0}^{t} \int_{\Gamma_{\delta}(s)} (\bd{\xi} \cdot \bd{n}_{\delta} - \widehat{\bd{u}} \cdot \bd{n}_{\delta}) |\bd{u}_{\delta}|^{2} \\
&+ \frac{1}{2} \int_{0}^{t} \int_{\Gamma_{\delta}(s)} (\bd{\xi}_{\delta} \cdot \bd{n}_{\delta} - \bd{u}_{\delta} \cdot \bd{n}_{\delta}) \bd{u}_{\delta} \cdot \widehat{\bd{u}} = {{R_{3, 1} + R_{3, 2}}},
\end{align*}
}}
where
{{
\begin{equation*}
{{R_{3, 1}}} = \frac{1}{2} \int_{0}^{t} \int_{\Gamma} [(\bd{\xi} - \bd{u}) \cdot \bd{e}_{y}] \bd{u}_{\delta} \cdot (\bd{u} - \bd{u}_{\delta}) - \frac{1}{2} \int_{0}^{t} \int_{\Gamma} [(\bd{\xi}_{\delta} - \bd{u}_{\delta}) \cdot \bd{e}_{y}] \bd{u} \cdot (\bd{u} - \bd{u}_{\delta}),
\end{equation*}
\begin{multline*}
{{R_{3, 2}}} = \frac{1}{2} \int_{0}^{t} \int_{\Gamma} \partial_{x}\omega (\bd{u} \cdot \bd{e}_{x}) \bd{u} \cdot \bd{u}_{\delta} - \frac{1}{2} \int_{0}^{t} \int_{\Gamma} \partial_{x}\omega (\bd{u}_{\delta} \cdot \bd{e}_{x}) |\bd{u}|^{2} - \frac{1}{2} \int_{0}^{t} \int_{\Gamma} \partial_{x}\omega_{\delta} (\bd{u} \cdot \bd{e}_{x}) |\bd{u}_{\delta}|^{2} \\
+ \frac{1}{2} \int_{0}^{t} \int_{\Gamma} \partial_{x}\omega_{\delta} (\bd{u}_{\delta} \cdot \bd{e}_{x}) \bd{u} \cdot \bd{u}_{\delta}.
\end{multline*}
}}

We estimate {{$R_{3, 1}$}} as follows: decompose {{$R_{3, 1}$ as $R_{3, 1} = R_{3, 1, 1} + R_{3, 1, 2}$}}, where
{{
\begin{equation*}
{{R_{3, 1, 1}}} = -\frac{1}{2} \int_{0}^{t} \int_{\Gamma} (\bd{\xi} \cdot \bd{e}_{y}) (\bd{u} - \bd{u}_{\delta}) \cdot (\bd{u} - \bd{u}_{\delta}) + \frac{1}{2} \int_{0}^{t} \int_{\Gamma} [(\bd{\xi} - \bd{\xi}_{\delta}) \cdot \bd{e}_{y}] \bd{u} \cdot (\bd{u} - \bd{u}_{\delta}),
\end{equation*}
\begin{equation*}
{{R_{3, 1, 2}}} = \frac{1}{2} \int_{0}^{t} \int_{\Gamma} (\bd{u} \cdot \bd{e}_{y}) (\bd{u} - \bd{u}_{\delta}) \cdot (\bd{u} - \bd{u}_{\delta}) - \frac{1}{2} \int_{0}^{t} \int_{\Gamma} [(\bd{u} - \bd{u}_{\delta}) \cdot \bd{e}_{y}] \bd{u} \cdot (\bd{u} - \bd{u}_{\delta}).
\end{equation*}
}}
By interpolation, 
{{
\begin{multline*}
{{|R_{3, 1, 1}|}} \le C\left(\int_{0}^{t} ||\widehat{\bd{u}} - \bd{u}_{\delta}||^{1/2}_{L^{2}(\Omega_{f,\delta}(s))} ||\widehat{\bd{u}} - \bd{u}_{\delta}||^{3/2}_{H^{1}(\Omega_{f,\delta}(s))} + \int_{0}^{t} ||\bd{\xi} - \bd{\xi}_{\delta}||_{L^{2}(\Gamma)} ||\widehat{\bd{u}} - \bd{u}_{\delta}||_{H^{1}(\Omega_{f,\delta}(s))}\right) \\
\le \epsilon \int_{0}^{t} ||\widehat{\bd{u}} - \bd{u}_{\delta}||^{2}_{H^{1}(\Omega_{f,\delta}(s))} + C(\epsilon) \left(\int_{0}^{t} ||\widehat{\bd{u}} - \bd{u}_{\delta}||_{L^{2}(\Omega_{f,\delta}(s))}^{2} + \int_{0}^{t} ||\bd{\xi} - \bd{\xi}_{\delta}||_{L^{2}(\Gamma)}^{2}\right).
\end{multline*}
}}
By using the same interpolation inequality, we obtain the following estimate for {{$R_{3, 1, 2}$}}.
{{
\begin{equation*}
{{|R_{3, 1, 2}|}} \le \epsilon \int_{0}^{t} ||\widehat{\bd{u}} - \bd{u}_{\delta}||^{2}_{H^{1}(\Omega_{f,\delta}(s))} + C(\epsilon) \int_{0}^{t} ||\widehat{\bd{u}} - \bd{u}_{\delta}||_{L^{2}(\Omega_{f,\delta}(s))}^{2}.
\end{equation*}
}}

We estimate {{$R_{3, 2}$}} by first rewriting {{$R_{3, 2}$}} as follows:
\begin{align*}
{{R_{3, 2}}} &= -\frac{1}{2}\int_{0}^{t} \int_{\Gamma} (\partial_{x}\omega - \partial_{x}\omega_{\delta}) (\bd{u})_{x} \bd{u} \cdot (\bd{u} - \bd{u}_{\delta}) - \frac{1}{2} \int_{0}^{t} \int_{\Gamma} \partial_{x}\omega_{\delta} (\bd{u})_{x} (\bd{u} - \bd{u}_{\delta}) \cdot (\bd{u} - \bd{u}_{\delta}) \\
&+ \frac{1}{2} \int_{0}^{t} \int_{\Gamma} (\partial_{x}\omega - \partial_{x}\omega_{\delta}) (\bd{u} - \bd{u}_{\delta})_{x} |\bd{u}|^{2} + \frac{1}{2} \int_{0}^{t} \int_{\Gamma} \partial_{x}\omega_{\delta} (\bd{u} - \bd{u}_{\delta})_{x} \bd{u} \cdot (\bd{u} - \bd{u}_{\delta}).
\end{align*}
By interpolation, by the boundedness of $|\partial_{x} \omega|$ and $|\partial_{x}\omega_{\delta}|$, and by the smoothness of $\bd{u}$, we get:
{{
\begin{equation*}
{{|R_{3, 2}|}} \le \epsilon \int_{0}^{t} ||\widehat{\bd{u}} - \bd{u}_{\delta}||_{H^{1}(\Omega_{f,\delta}(s))}^{2} + C(\epsilon) \left(\int_{0}^{t} ||\omega - \omega_{\delta}||^{2}_{H^{2}(\Gamma)} + \int_{0}^{t} ||\widehat{\bd{u}} - \bd{u}_{\delta}||_{L^{2}(\Omega_{f,\delta}(s))}^{2}\right).
\end{equation*}
}}

Hence, by combining the two estimates we get the final estimate for $T_3$:
{{
\begin{align*}
|T_{3}| \le \epsilon \int_{0}^{t} ||\widehat{\bd{u}} - \bd{u}_{\delta}||_{H^{1}(\Omega_{f,\delta}(s))}^{2} 
+ C(\epsilon) \left(\int_{0}^{t} ||\omega - \omega_{\delta}||^{2}_{H^{2}(\Gamma)} + \int_{0}^{t} ||\bd{\xi} - \bd{\xi}_{\delta}||^{2}_{L^{2}(\Gamma)} + \int_{0}^{t} ||\widehat{\bd{u}} - \bd{u}_{\delta}||_{L^{2}(\Omega_{f,\delta}(s))}^{2}\right).
\end{align*}
}}

\vskip 0.1in
\noindent
{\bf{Term T4.}} To estimate term $T_4$, defined in \eqref{T4}, we again 
use Proposition \ref{alphaconvprop} to pass to the limit as $\nu \to 0$ so that the contribution from $T_{4}$ is
{{
\begin{equation}\label{T4equality}
T_{4} := 2\nu \int_{0}^{t} \int_{\Omega_{f}(s)} \bd{D}(\bd{u}) : \bd{D}(\bd{u} - \widecheck{\bd{u}}_{\delta}) - 2\nu \int_{0}^{t} \int_{\Omega_{f,\delta}(s)} \bd{D}(\bd{u}_{\delta}) : \bd{D}(\widehat{\bd{u}} - \bd{u}_{\delta}).
\end{equation}
}}
We want to transfer the integral on $\Omega_{1}(t)$ to $\Omega_{f,\delta}(t)$. Recalling \eqref{gradidentity}, we have that
{{
\begin{equation*}
\int_{0}^{t} \int_{\Omega_{f}(s)} \bd{D}(\bd{u}) : \bd{D}(\bd{u} - \widecheck{\bd{u}}_{\delta}) = \int_{0}^{t} \int_{\Omega_{f,\delta}(s)} \gamma_{\delta} [\nabla(\bd{u} \circ \psi_{\delta}) J_{\delta}^{-1}]^{sym} : [\nabla((\bd{u} - \widecheck{\bd{u}}_{\delta}) \circ \psi_{\delta}) J_{\delta}^{-1}]^{sym},
\end{equation*}
}}
where the superscript {{`sym'}} notation denotes a symmetrization. Following the procedure in \cite{WeakStrongFSI}, we break up the integral as
{{
\begin{multline}\label{T4equality2}
\int_{0}^{t} \int_{\Omega_{f}(s)} \bd{D}(\bd{u}) : \bd{D}(\bd{u} - \widecheck{\bd{u}}_{\delta}) = \int_{0}^{t} \int_{\Omega_{f,\delta}(s)} \bd{D}(\widehat{\bd{u}}) : \bd{D}(\widehat{\bd{u}} - \bd{u}_{\delta}) + {{R_{4, 1} + R_{4, 2} + R_{4, 3} + R_{4, 4}}},
\end{multline}
}}
where
{{
\begin{align*}
{{R_{4, 1}}} &= \int_{0}^{t} \int_{\Omega_{f,\delta}(s)} (\nabla (\bd{u} \circ \psi_{\delta}) J_{\delta}^{-1})^{sym} : [\nabla (\widehat{\bd{u}} - \bd{u}_{\delta}) (J_{\delta}^{-1} - I) + (J_{\delta} - I) \nabla (\widehat{\bd{u}} - \bd{u}_{\delta}) J_{\delta}^{-1}]^{sym},
\\
{{R_{4, 2}}} &= \int_{0}^{t} \int_{\Omega_{f,\delta}(s)} [(I - \gamma_{\delta} J_{\delta}^{-1})\nabla (\bd{u} \circ \psi_{\delta}) + \nabla (\bd{u} \circ \psi_{\delta}) (J_{\delta}^{-1} - I)]^{sym} : \bd{D}(\widehat{\bd{u}} - \bd{u}_{\delta}),
\\
{{R_{4, 3}}} &= \int_{0}^{t} \int_{\Omega_{f,\delta}(s)} (\nabla (\bd{u} \circ \psi_{\delta}) J_{\delta}^{-1})^{sym} : (\gamma_{\delta} \nabla(\gamma_{\delta}^{-1}J_{\delta})(\widehat{\bd{u}} - \bd{u}_{\delta})J_{\delta}^{-1})^{sym},
\\
{{R_{4, 4}}} &= -\int_{0}^{t} \int_{\Omega_{f,\delta}(s)} [(\nabla(\gamma_{\delta} J_{\delta}^{-1})) \bd{u} \circ \psi_{\delta}]^{sym} : \bd{D}(\widehat{\bd{u}} - \bd{u}_{\delta}).
\end{align*}
}}
To verify this equality, one can use the Leibniz rule, the definition $\widehat{\bd{u}} = \gamma_{\delta} J_{\delta}^{-1} \cdot (\bd{u} \circ \psi_{\delta})$, and the identity $\widehat{\bd{u}} - \bd{u}_{\delta} = \gamma_{\delta} J_{\delta}^{-1} \cdot ((\bd{u} - \widecheck{\bd{u}}_{\delta}) \circ \psi_{\delta})$.

We now estimate the terms {{$R_{4, 1}$-$R_{4, 4}$}}. For this purpose we will use the following inequalities:
\begin{align*}
|J_{\delta}^{-1}| &\le C(1 + |\partial_{x} \gamma_{\delta}|), \quad 
|J_{\delta}^{-1} - I| \le C(|\gamma_{\delta}^{-1} - 1| + |\partial_{x} \gamma_{\delta}|), 
\\
 |J_{\delta} - I| &\le C(|\gamma_{\delta} - 1| + |\partial_{x} \gamma_{\delta}|), \quad   |\gamma_{\delta} J_{\delta}^{-1} - I| \le C(|\gamma_{\delta} - 1| + |\partial_{x} \gamma_{\delta}|).
\end{align*}
and, recalling the definition of $\gamma_{\delta}$ in \eqref{psi}, we have the following inequalities: 
\begin{equation*}
|\gamma_{\delta} - 1| \le C||\omega - \omega_{\delta}||_{H^{2}(\Gamma)}, \qquad |\gamma_{\delta}^{-1} - 1| \le C||\omega - \omega_{\delta}||_{H^{2}(\Gamma)},
\end{equation*}
\begin{equation*}
|\partial_{x} \gamma_{\delta}| \le C||\omega - \omega_{\delta}||_{H^{2}(\Gamma)}, \qquad |\partial_{x} (\gamma_{\delta}^{-1})| \le C||\omega - \omega_{\delta}||_{H^{2}(\Gamma)}.
\end{equation*}
Because $|J_{\delta}^{-1}| \le C(1 + |\partial_{x} \gamma_{\delta}|) \le C$ since $|\partial_{x} \gamma_{\delta}|$ is bounded, and because $\bd{u}$ is smooth,
{{
\begin{align*}
{{|R_{4, 1}|}} &\le C\int_{0}^{t} ||\nabla (\widehat{\bd{u}} - \bd{u}_{\delta})||_{L^{2}(\Omega_{f,\delta}(s))} (||\gamma_{\delta}^{-1} - 1||_{L^{2}(\Omega_{f,\delta}(s))} + ||\gamma_{\delta} - 1||_{L^{2}(\Omega_{f,\delta}(s))} + ||\partial_{x} \gamma_{\delta}||_{L^{2}(\Omega_{f,\delta}(s))}) \\
&\le \epsilon \int_{0}^{t} ||\nabla (\widehat{\bd{u}} - \bd{u}_{\delta})||_{L^{2}(\Omega_{f,\delta}(s))}^{2} + C(\epsilon) \int_{0}^{t} ||\omega - \omega_{\delta}||^{2}_{H^{2}(\Gamma)}.
\end{align*}
}}
We also have that
{{
\begin{equation*}
{{|R_{4, 2}|}} \le \epsilon \int_{0}^{t} ||\bd{D} (\widehat{\bd{u}} - \bd{u}_{\delta})||_{L^{2}(\Omega_{f,\delta}(s))}^{2} + C(\epsilon) \int_{0}^{t} ||\omega - \omega_{\delta}||^{2}_{H^{2}(\Gamma)}.
\end{equation*}
}}
For {{$R_{4, 3}$ and $R_{4, 4}$}}, we compute that
\begin{equation*}
\nabla(\gamma_{\delta}^{-1}J_{\delta}) = \nabla \begin{pmatrix} \gamma_{\delta}^{-1} & 0 \\ (R + y)\gamma_{\delta}^{-1}\partial_{x}\gamma_{\delta} & 1 \\ \end{pmatrix}, \qquad \nabla (\gamma_{\delta} J_{\delta}^{-1}) = \nabla \begin{pmatrix} \gamma_{\delta} & 0 \\ -(R + y)\partial_{x}\gamma_{\delta} & 1 \\ \end{pmatrix}.
\end{equation*}
Therefore,
\begin{equation*}
|\nabla (\gamma_{\delta}^{-1}J_{\delta})| \le C(|\partial_{x}(\gamma_{\delta}^{-1})| + |\partial_{x}\gamma_{\delta}| + |\partial_{xx}\gamma_{\delta}|), \qquad |\nabla(\gamma_{\delta} J_{\delta}^{-1})| \le C(|\partial_{x}\gamma_{\delta}| + |\partial_{xx}\gamma_{\delta}|),
\end{equation*}
where we can estimate
\begin{equation*}
|\partial_{xx}\gamma_{\delta}| \le C(||\omega - \omega_{\delta}||_{H^{2}(\Gamma)} |\partial_{xx}\omega| + |\partial_{xx}(\omega - \omega_{\delta})| + ||\omega - \omega_{\delta}||_{H^{2}(\Gamma)}).
\end{equation*}
So since $||\partial_{xx}\omega||_{L^{2}(\Omega_{f,\delta}(t))} \le C$ since $\omega$ is uniformly bounded in $H^{2}(\Gamma)$, we have that
{{
\begin{align*}
{{|R_{4, 3}|}} &\le C \int_{0}^{t} ||\nabla(\gamma_{\delta}^{-1}J_{\delta})||_{L^{2}(\Omega_{f,\delta}(s))} ||\widehat{\bd{u}} - \bd{u}_{\delta}||_{L^{2}(\Omega_{f,\delta}(s))} \le C\int_{0}^{t} ||\omega - \omega_{\delta}||_{H^{2}(\Gamma)} ||\widehat{\bd{u}} - \bd{u}_{\delta}||_{L^{2}(\Omega_{f,\delta}(s))} \\
&\le C\left(\int_{0}^{t} ||\omega - \omega_{\delta}||_{H^{2}(\Gamma)}^{2} + \int_{0}^{t} ||\widehat{\bd{u}} - \bd{u}_{\delta}||_{L^{2}(\Omega_{f,\delta}(s))}^{2}\right).
\end{align*}
}}
Similarly, using $||\nabla(\gamma_{\delta} J_{\delta}^{-1})||_{L^{2}(\Omega_{f,\delta}(t))} \le C||\omega - \omega_{\delta}||_{H^{2}(\Gamma)}$, we have the following estimate for $R_{4}$:
{{ 
\begin{align*}
{{|R_{4, 4}|}} &\le C \int_{0}^{t} ||\nabla(\gamma_{\delta} J_{\delta}^{-1})||_{L^{2}(\Omega_{f,\delta}(s))} ||\bd{D}(\widehat{\bd{u}} - \bd{u}_{\delta})||_{L^{2}(\Omega_{f,\delta}(s))} \\
&\le \epsilon \int_{0}^{t} ||\bd{D}(\widehat{\bd{u}} - \bd{u}_{\delta})||_{L^{2}(\Omega_{f,\delta}(s))}^{2} + C(\epsilon) \int_{0}^{t} ||\omega - \omega_{\delta}||^{2}_{H^{2}(\Gamma)}. 
\end{align*}
}}
We now have the final estimate of $T_4$, obtained after using  \eqref{T4equality} and \eqref{T4equality2} as follows:
{{
\begin{equation*}
T_{4} = 2\nu \int_{0}^{t} \int_{\Omega_{f,\delta}(s)} |\bd{D}(\widehat{\bd{u}} - \bd{u}_{\delta})|^{2} + {{R_{4}}},
\end{equation*}
}}
where
{{
\begin{equation*}
{{|R_{4}|}} \le \epsilon \int_{0}^{t} ||\bd{D}(\widehat{\bd{u}} - \bd{u}_{\delta})||^{2}_{L^{2}(\Omega_{f,\delta}(s))} + C(\epsilon) \left(\int_{0}^{t} ||\omega - \omega_{\delta}||^{2}_{H^{2}(\Gamma)} + \int_{0}^{t} ||\widehat{\bd{u}} - \bd{u}_{\delta}||_{L^{2}(\Omega_{f,\delta}(s))}^{2}\right). 
\end{equation*}
}}

\vskip 0.1in
\noindent
{\bf{Term T5.}} Similarly as before, 
after passing to the limit as $\nu \to 0$ in term $T_5$, defined by \eqref{T5}, the contribution of this term is
{{
\begin{multline}\label{T5}
T_{5} = \beta\int_{0}^{t} \int_{\Gamma(s)} (\bd{\xi} - \bd{u}) \cdot \bd{\tau}(s)[(\bd{\xi} - \bd{\xi}_{\delta}) \cdot \bd{\tau}(s) - (\bd{u} - \widecheck{\bd{u}}_{\delta}) \cdot \bd{\tau}(s)] \\
- \beta\int_{0}^{t} \int_{\Gamma_{\delta}(s)} (\bd{\xi}_{\delta} - \bd{u}_{\delta}) \cdot \bd{\tau}_{\delta}(s) [(\bd{\xi} - \bd{\xi}_{\delta}) \cdot \bd{\tau}_{\delta}(s) - (\widehat{\bd{u}} - \bd{u}_{\delta}) \cdot \bd{\tau}_{\delta}(s)].
\end{multline} 
}}
We note that when we test the weak formulation for $\bd{u}$ with $\bd{v} = \bd{u} - (\widecheck{\bd{u}}_{\delta})_{\nu}$ and $\bd{\psi} = \bd{\xi} - (\bd{\xi}_{\delta})_{\nu}$, we can pass to the limit as $\nu \to 0$ to obtain the first term in $T_{5}$ above, by using similar arguments involving Proposition \ref{alphaconvprop}, as for the previously considered terms. 
This term can now be rewritten as follows:
{{
\begin{equation*}
T_{5} = \beta \int_{0}^{t} \int_{\Gamma_{\delta}(s)} |(\bd{\xi} - \bd{\xi}_{\delta}) \cdot \bd{\tau}_{\delta}(s) - (\widehat{\bd{u}} - \bd{u}_{\delta}) \cdot \bd{\tau}_{\delta}(s)|^{2} + R_{5},
\end{equation*}
}}
where
{{
\begin{multline*}
R_{5} = \beta\int_{0}^{t} \int_{\Gamma(s)} (\bd{\xi} - \bd{u}) \cdot \bd{\tau}(s)[(\bd{\xi} - \bd{\xi}_{\delta}) \cdot \bd{\tau}(s) - (\bd{u} - \widecheck{\bd{u}}_{\delta}) \cdot \bd{\tau}(s)] \\
- \beta\int_{0}^{t} \int_{\Gamma_{\delta}(s)} (\bd{\xi} - \widehat{\bd{u}}) \cdot \bd{\tau}_{\delta}(s)[(\bd{\xi} - \bd{\xi}_{\delta}) \cdot \bd{\tau}_{\delta}(s) - (\widehat{\bd{u}} - \bd{u}_{\delta}) \cdot \bd{\tau}_{\delta}(s)].
\end{multline*}
}}
Denote the arc length elements of {{$\Gamma(t)$}} and $\Gamma_{\delta}(t)$ respectively by $\mathcal{J}^{\omega}_{\Gamma} = \sqrt{1 + |\partial_{x}\omega|^{2}}$ and $\mathcal{J}^{\omega_{\delta}}_{\Gamma} = \sqrt{1 + |\partial_{x}\omega_{\delta}|^{2}}$, 
and {{we recall that we denote the tangent vectors to $\Gamma(t)$ and $\Gamma_{\delta}(t)$ respectively by $\bd{\tau}(t) = \frac{1}{\mathcal{J}^{\omega}_{\Gamma}}(1, \partial_{x}\omega)$ and $\bd{\tau}_{\delta}(t) = \frac{1}{\mathcal{J}^{\omega_{\delta}}_{\Gamma}}(1, \partial_{x}\omega_{\delta})$.}}
We can now rewrite $R_{5}$ by writing everything in terms of the $x$ and $y$ components. For this purpose, recall that $\bd{\xi}$ and $\bd{\xi}_{\delta}$ along the interface displace in only the $y$ direction. We formally express the common trace of $\bd{u} - \widecheck{\bd{u}}_{\delta}$ and $\widehat{\bd{u}} - \bd{u}_{\delta}$ along the reference configuration of the interface $\Gamma$ by $\bd{u} - \bd{u}_{\delta}$. Thus, 
{{
\begin{align*}
R_{5} &= \beta\int_{0}^{t} \int_{\Gamma} (\bd{\xi} - \bd{u}) \cdot (1, \partial_{x}\omega) [(\bd{\xi} - \bd{\xi}_{\delta}) - (\bd{u} - \bd{u}_{\delta})] \cdot \bd{\tau}(s) \\
&- \beta\int_{0}^{t} \int_{\Gamma} (\bd{\xi} - \bd{u}) \cdot (1, \partial_{x}\omega_{\delta}) [(\bd{\xi} - \bd{\xi}_{\delta}) - (\bd{u} - \bd{u}_{\delta})] \cdot \bd{\tau}_{\delta}(s).
\end{align*}
}}
In the previous step, we used the fact that when transferred back to the reference configuration $\Omega_{f}$, $\widehat{\bd{u}} - \bd{u}_{\delta}$ and $\bd{u} - \widecheck{\bd{u}}_{\delta}$ have the same trace along $\Gamma$. Thus, $R_{5} = R_{5, 1} + R_{5, 2}$, where
{{
\begin{align*}
R_{5, 1} &= \beta \int_{0}^{t} \int_{\Gamma} (\bd{\xi} - \bd{u}) \cdot \bd{e}_{y} (\partial_{x}\omega - \partial_{x}\omega_{\delta}) [(\bd{\xi} - \bd{\xi}_{\delta}) - (\bd{u} - \bd{u}_{\delta})] \cdot \bd{\tau}(s),
\\
R_{5, 2} &= \beta \int_{0}^{t} \int_{\Gamma} (\bd{\xi} - \bd{u}) \cdot (1, \partial_{x}\omega_{\delta}) [(\bd{\xi} - \bd{\xi}_{\delta}) - (\bd{u} - \bd{u}_{\delta})] \cdot (\bd{\tau}(s) - \bd{\tau}_{\delta}(s)).
\end{align*}
}}
We can use the fact that $|\partial_{x}\omega|$ and $|\partial_{x}\omega_{\delta}|$ are uniformly bounded to obtain the following estimates:
{{
\begin{equation*}
|R_{5, 1}| \le \epsilon \int_{0}^{t} ||\bd{D}(\widehat{\bd{u}} - \bd{u}_{\delta})||_{L^{2}(\Omega_{f,\delta}(s))} + C(\epsilon)\left(\int_{0}^{t} ||\omega - \omega_{\delta}||_{H^{2}(\Gamma)}^{2} + \int_{0}^{t} ||\bd{\xi} - \bd{\xi}_{\delta}||_{L^{2}(\Gamma)}^{2}\right),
\end{equation*}
}}
where we used the trace inequality, Poincare's inequality, and Korn's inequality for the fluid. For the second term {{$R_{5, 2}$}}, we use the estimate {{$|\bd{\tau}(s) - \bd{\tau}_{\delta}(s)| \le C|\partial_{x}\omega - \partial_{x}\omega_{\delta}|$}} to obtain
{{
\begin{equation*}
|R_{5, 2}| \le \epsilon \int_{0}^{t} ||\bd{D}(\widehat{\bd{u}} - \bd{u}_{\delta})||_{L^{2}(\Omega_{f,\delta}(s))} + C(\epsilon)\left(\int_{0}^{t} ||\omega - \omega_{\delta}||_{H^{2}(\Gamma)}^{2} + \int_{0}^{t} ||\bd{\xi} - \bd{\xi}_{\delta}||_{L^{2}(\Gamma)}^{2}\right).
\end{equation*}
}}
Hence,
{{
\begin{equation*}
T_{5} = \beta \int_{0}^{t} \int_{\Gamma_{\delta}(s)} |(\bd{\xi} - \bd{\xi}_{\delta}) \cdot \bd{\tau}_{\delta}(s) - (\widehat{\bd{u}} - \bd{u}_{\delta}) \cdot \bd{\tau}_{\delta}(s)|^{2} + R_{5},
\end{equation*}
}}
where
{{
\begin{equation*}
|R_{5}| \le \epsilon \int_{0}^{t} ||\bd{D}(\widehat{\bd{u}} - \bd{u}_{\delta})||_{L^{2}(\Omega_{f,\delta}(s))} + C(\epsilon)\left(\int_{0}^{t} ||\omega - \omega_{\delta}||_{H^{2}(\Gamma)}^{2} + \int_{0}^{t} ||\bd{\xi} - \bd{\xi}_{\delta}||_{L^{2}(\Gamma)}^{2}\right).
\end{equation*}
}}

\vskip 0.1in
\noindent
{\bf{Terms T6-T8.}}
We will present estimates only for term $T_6$, defined in \eqref{T6}, as the same procedure will hold for $T_7$ and $T_8$. 
Since $\zeta$ and $\zeta_{\delta}$ are weakly continuous in $L^{2}(\Gamma)$, by the weak formulation, we get:
\begin{align*}
\int_{0}^{t} \int_{\Gamma} \zeta \cdot \partial_{t}\left[(\zeta_{\delta})_{\nu}\right] + \int_{0}^{t} \int_{\Gamma} \zeta_{\delta} \cdot \partial_{t}\zeta 
&= \int_{0}^{t} \int_{\Gamma} \zeta \cdot \partial_{t}\left[(\zeta_{\delta})_{\nu}\right] + \int_{0}^{t} \int_{\Gamma} \zeta_{\delta} \cdot \partial_{t}\left[(\zeta)_{\nu}\right] - \int_{0}^{t} \int_{\Gamma} \zeta_{\delta} \cdot \partial_{t}\left[(\zeta)_{\nu} - \zeta\right] \\
&\to \int_{\Gamma} \zeta(t) \cdot \zeta_{\delta}(t) - \int_{\Gamma} |\zeta_{0}|^{2}.
\end{align*}
This follows from Lemma 2.5 in \cite{WeakStrongFSI}, which implies:
\begin{equation*}
\int_{0}^{t} \int_{\Gamma} \zeta \cdot \partial_{t}\left[(\zeta_{\delta})_{\nu}\right] + \int_{0}^{t} \int_{\Gamma} \zeta_{\delta} \cdot \partial_{t}\left[(\zeta)_{\nu}\right]
\to \int_{\Gamma} \zeta(t) \cdot \zeta_{\delta}(t) - \int_{\Gamma} |\zeta_{0}|^{2}, \qquad \text{ as } \nu \to 0,
\end{equation*}
and from the fact that  $\zeta$ is smooth in space and time, which implies
\begin{equation*}
\int_{0}^{t} \int_{\Gamma} \zeta_{\delta} \cdot \partial_{t}\left[(\zeta)_{\nu} - \zeta\right] \to 0, \qquad \text{ as } \nu \to 0.
\end{equation*}
Furthermore, because $\zeta(0) = \zeta_{\delta}(0) = \zeta_{0}$ weak continuity of $\zeta_{\delta}$ at $t = 0$ implies
 that $\displaystyle \int_{\Gamma} \zeta(0) \cdot [\zeta(0) - (\zeta_{\delta})_{\nu}(0)] \to 0$ as $\nu \to 0$. 
{{Similarly, $\displaystyle \int_{\Gamma} \zeta(s) \cdot [\zeta(s) - (\zeta_{\delta})_{\nu}(s)] \to 0$ as $\nu \to 0$ for almost every $s \in [0, T]$.}}
 Hence, as $\nu \to 0$, the contribution from $T_{6}$ is
 {{
\begin{equation*}
T_{6} = \frac{1}{2} \rho_{p} \int_{\Gamma} |(\zeta - \zeta_{\delta})(t)|^{2}.
\end{equation*}
}}
Similarly, the contributions from $T_7$ and $T_8$ as $\nu \to 0$ are
{{
\begin{equation*}
T_{7} = \frac{1}{2} \int_{\Gamma} |\Delta (\omega - \omega_{\delta})(t)|^{2}, \qquad T_{8} = \frac{1}{2} \rho_{b} \int_{\Omega_{b}} |(\bd{\xi} - \bd{\xi}_{\delta})(t)|^{2}.
\end{equation*}
}}

\vskip 0.1in
\noindent
{\bf{Terms T9-T12.}} Since these calculations are straight forward, a discussion about the limiting expressions as $\nu \to 0$ for terms $T_9$-$T_{12}$ was presented earlier, just under 
\eqref{T12}. 
%Because $\bd{\xi}_{\delta} \in L^{2}(0, T; H^{1}(\Omega_{b}))$ where $\Omega_{b}$ is a fixed domain, we have that $(\bd{\xi}_{\delta})_{\nu} \to \bd{\xi}_{2, \nu}$ strongly in $L^{2}(0, T; H^{1}(\Omega_{b}))$. Hence, we have that Terms 9-12 converge to the following as $\nu \to 0$.
%\begin{equation*}
%T_{9} = \mu_{e} \int_{\Omega_{b}} |\bd{D}(\bd{\eta} - \bd{\eta}_{\delta})(\tau)|^{2}, \qquad T_{10} = \frac{1}{2} \lambda_{e} \int_{\Omega_{b}} |\nabla \cdot (\bd{\eta} - \bd{\eta}_{2})(\tau)|^{2},
%\end{equation*}
%\begin{equation*}
%T_{11} = 2\mu_{v} \int_{0}^{t} \int_{\Omega_{b}} |\bd{D}(\bd{\xi} - \bd{\xi}_{\delta})|^{2}, \qquad T_{12} = \lambda_{v} \int_{0}^{t} \int_{\Omega_{b}} |\nabla \cdot (\bd{\xi} - \bd{\xi}_{\delta})|^{2}.
%\end{equation*}

\vskip 0.1in
\noindent
{\bf{Term T13.}} Similarly as before, by taking the limit as $\nu \to 0$, we have that 
{{
\begin{equation*}
T_{13} = -\aalpha \int_{0}^{t} \int_{\Omega_{b}(s)} p \nabla \cdot (\bd{\xi} - \bd{\xi}_{\delta}) + \aalpha \int_{0}^{t} \int_{{\Omega}^\delta_{b, \delta}(s)} p_{\delta} \nabla \cdot (\bd{\xi} - \bd{\xi}_{\delta}).
\end{equation*}
}}
To estimate this term we use \eqref{nablaeta} and the matrix identity $\bd{B}^{-1} = \frac{1}{\text{det}(\bd{B})} \bd{B}^{C}$ to obtain 
\begin{multline*}
|T_{13}| = \aalpha \left|\int_{0}^{t} \int_{\Omega_{b}} \mathcal{J}^{\eta}_{b} p \nabla^{\eta}_{b} \cdot (\bd{\xi} - \bd{\xi}_{\delta}) - \int_{0}^{t} \int_{\Omega_{b}} \mathcal{J}^{{\eta}^\delta_{\delta}}_{b} p_{\delta} \nabla^{{\eta}^\delta_{\delta}}_{b} \cdot (\bd{\xi} - \bd{\xi}_{\delta}) \right| \\
= \aalpha \left|\int_{0}^{t} \int_{\Omega_{b}} p \cdot \text{tr}\left(\nabla (\bd{\xi} - \bd{\xi}_{\delta}) \cdot (\bd{I} + \nabla \bd{\eta})^{C}\right) - \int_{0}^{t} \int_{\Omega_{b}} p_{\delta} \cdot \text{tr}\left(\nabla (\bd{\xi} - \bd{\xi}_{\delta}) \cdot (\bd{I} + \nabla {\bd{\eta}}^\delta_{\delta})^{C}\right) \right| \le {{R_{13, 1} + R_{13, 2}}},
\end{multline*}
where the superscript $``C"$ denotes the cofactor matrix. The integrals {{$R_{13, 1}$ and $R_{13, 2}$}} are defined as follows: 
\begin{align*}
{{R_{13, 1}}} &= \aalpha \left|\int_{0}^{t} \int_{\Omega_{b}} p \cdot \text{tr}\left(\nabla (\bd{\xi} - \bd{\xi}_{\delta}) \cdot (\nabla (\bd{\eta} - {\bd{\eta}}^\delta_{\delta}))^{C}\right) \right|,
\\
{{R_{13, 2}}} &= \aalpha \left|\int_{0}^{t} \int_{\Omega_{b}} (p - p_{\delta}) \cdot \text{tr}\left(\nabla (\bd{\xi} - \bd{\xi}_{\delta}) \cdot (\bd{I} + \nabla {\bd{\eta}}^\delta_{\delta})^{C}\right)\right|.
\end{align*}
In the previous calculations, we observe that the cofactor matrix operation is linear when the matrices are two by two. Using the fact that $p$ is smooth, the assumption \eqref{grad}, and the fact that 
\begin{equation}\label{convolutionineq}
||\nabla {\bd{\eta}}^\delta - \nabla {\bd{\eta}}^\delta_{\delta}||_{L^{2}(\Omega_{b})} \le {{C||\nabla \bd{\eta} - \nabla \bd{\eta}_{\delta}||_{L^{2}(\tilde{\Omega}_{b})}}} \le C\left(||\nabla \bd{\eta} - \nabla \bd{\eta}_{\delta}||_{L^{2}(\Omega_{b})} + ||\omega - \omega_{\delta}||_{H^{2}(\Gamma)}\right)
\end{equation}
for a constant $C$ independent of $\delta$, by Young's convolution inequality and the definition of odd extension to the larger domain $\tilde{\Omega}_{b}$ in Definition \ref{extension}, we obtain the estimates on {{$R_{13, 1}$ and $R_{13, 2}$}}:
\begin{align*}
{{R_{13, 1}}} &\le \epsilon \int_{0}^{t} ||\nabla(\bd{\xi} - \bd{\xi}_{\delta})||_{L^{2}(\Omega_{b})}^{2} + C(\epsilon)\left(\int_{0}^{t} ||\nabla \bd{\eta} - \nabla {\bd{\eta}}^\delta_{\delta}||_{L^{2}(\Omega_{b})}^{2}\right) \\
&\le C(\epsilon)\int_{0}^{t} ||\nabla \bd{\eta} - \nabla {\bd{\eta}}^\delta||_{L^{2}(\Omega_{b})}^{2} + \epsilon \int_{0}^{t} ||\nabla(\bd{\xi} - \bd{\xi}_{\delta})||_{L^{2}(\Omega_{b})}^{2} \\
&+ C(\epsilon)\left(\int_{0}^{t} ||\nabla \bd{\eta} - \nabla \bd{\eta}_{\delta}||^{2}_{L^{2}(\Omega_{b})} + \int_{0}^{t} ||\omega - \omega_{\delta}||^{2}_{H^{2}(\Gamma)}\right),
\\
{{R_{13, 2}}} &\le \epsilon \int_{0}^{t} ||\nabla(\bd{\xi} - \bd{\xi}_{\delta})||_{L^{2}(\Omega_{b})}^{2} + C(\epsilon) \left(\int_{0}^{t} ||p - p_{\delta}||^{2}_{L^{2}(\Omega_{b})}\right).
\end{align*}
Therefore, the final estimate for $T_{13}$ is as follows:
\begin{align*}
|T_{13}| &\le C(\epsilon)\int_{0}^{t} ||\nabla \bd{\eta} - \nabla {\bd{\eta}}^\delta||_{L^{2}(\Omega_{b})}^{2} + \epsilon \int_{0}^{t} ||\nabla(\bd{\xi} - \bd{\xi}_{\delta})||_{L^{2}(\Omega_{b})}^{2} \\
&+ C(\epsilon)\left(\int_{0}^{t} ||\nabla \bd{\eta} - \nabla \bd{\eta}_{\delta}||^{2}_{L^{2}(\Omega_{b})} + \int_{0}^{t} ||\omega - \omega_{\delta}||^{2}_{H^{2}(\Gamma)} + \int_{0}^{t} ||p - p_{\delta}||^{2}_{L^{2}(\Omega_{b})} \right).
\end{align*}

\vskip 0.1in
\noindent
{\bf{Term T14.}} This term  can be handled in the same way as terms $T_6$-$T_8$. 

\vskip 0.1in
\noindent
{\bf{Term T15.}} We pass to the limit  as $\nu \to 0$ in \eqref{T15}  to obtain:
{{
\begin{equation*}
T_{15} = -\aalpha \int_{0}^{t} \int_{\Omega_{b}(s)} \frac{D}{Dt} \bd{\eta} \cdot \nabla (p - p_{\delta}) + \alpha \int_{0}^{t} \int_{{\Omega}^\delta_{b, \delta}(s)} \frac{{D}^\delta}{Dt} \bd{\eta}_{\delta} \cdot \nabla (p - p_{\delta}).  
\end{equation*}
}}
To estimate this term we pull back to the reference domain and  use \eqref{nablaeta} and the cofactor formula for the matrix inverse to obtain:
\begin{align*}
|T_{15}| &= \aalpha \left|\int_{0}^{t} \int_{\Omega_{b}} \mathcal{J}^{\eta}_{b} \partial_{t} \bd{\eta} \cdot \nabla^{\eta}_{b} (p - p_{\delta}) - \int_{0}^{t} \int_{\Omega_{b}} \mathcal{J}^{{\eta}^\delta_{\delta}}_{b} \partial_{t} \bd{\eta}_{\delta} \cdot \nabla^{{\eta}^\delta_{\delta}}_{b} (p - p_{\delta})\right| \\
&= \aalpha \left|\int_{0}^{t} \int_{\Omega_{b}} \partial_{t} \bd{\eta} \cdot \left[\nabla(p - p_{\delta}) \cdot (\bd{I} + \nabla \bd{\eta})^{C}\right] - \int_{0}^{t} \int_{\Omega_{b}} \partial_{t}\bd{\eta}_{\delta} \cdot \left[\nabla (p - p_{\delta}) \cdot (\bd{I} + \nabla {\bd{\eta}}^\delta_{\delta})^{C}\right] \right|
\le {{R_{15, 1} + R_{15, 2}}},
\end{align*}
where
\begin{align*}
{{R_{15, 1}}} &= \aalpha \left|\int_{0}^{t} \int_{\Omega_{b}} \partial_{t}\bd{\eta} \cdot \left[\nabla(p - p_{\delta}) \cdot (\nabla \bd{\eta} - \nabla {\bd{\eta}}^\delta_{\delta})^{C}\right]\right|,
\\
{{R_{15, 2}}} &= \aalpha \left|\int_{0}^{t} \int_{\Omega_{b}} (\partial_{t} \bd{\eta} - \partial_{t} \bd{\eta}_{\delta}) \cdot \left[\nabla (p - p_{\delta}) \cdot (\bd{I} + \nabla {\bd{\eta}}^\delta_{\delta})^{C}\right] \right|.
\end{align*}

To estimate {{$R_{15, 1}$,}} we use \eqref{det}, \eqref{norm}, and the convolution inequality \eqref{convolutionineq} to obtain:
{{
\begin{align*}
{{R_{15, 1}}} &\le \epsilon \int_{0}^{t} ||\nabla(p - p_{\delta})||_{L^{2}({\Omega}^\delta_{b, \delta}(s))}^{2} + C(\epsilon) \int_{0}^{t} ||\nabla \bd{\eta} - \nabla {\bd{\eta}}^\delta||_{L^{2}(\Omega_{b})}^{2} \\
&+ C(\epsilon) \left(\int_{0}^{t} ||\nabla \bd{\eta} - \nabla \bd{\eta}_{\delta}||_{L^{2}(\Omega_{b})}^{2} + \int_{0}^{t} ||\omega - \omega_{\delta}||^{2}_{H^{2}(\Gamma)}\right).
\end{align*}
}}
Here, we also used the following estimate on the norm of the gradient of the pressure on the reference domain and on the moving domain,
which is obtained by using \eqref{det}, \eqref{norm}, and \eqref{nablaeta}:
\begin{align}\label{gradprescompare}
||\nabla(p - p_{\delta})(t)||^{2}_{L^{2}(\Omega_{b})} &= \int_{\Omega_{b}} |\nabla(p - p_{\delta})|^{2} = \int_{\Omega_{b}} \mathcal{J}^{{\eta}^\delta_{\delta}}_{b} |\nabla^{{\eta}^\delta_{\delta}}_{b} (p - p_{\delta}) \cdot (\bd{I} + \nabla {\bd{\eta}}^\delta_{\delta})|^{2} \cdot (\mathcal{J}^{{\eta}^\delta_{\delta}}_{b})^{-1} 
\nonumber \\
&\le C \int_{\Omega_{b}} \mathcal{J}^{{\eta}^\delta_{\delta}}_{b} |\nabla^{{\eta}^\delta_{\delta}}_{b} (p - p_{\delta})|^{2} = C||\nabla(p - p_{\delta})(t)||_{L^{2}({\Omega}^\delta_{b, \delta}(t))}^{2},
\end{align}
where constant $C$ is independent of $\delta$ and  $t \in [0, T_{\delta}]$. 

The estimate of {{$R_{15, 2}$}} is straight forward:
{{
\begin{equation*}
{{R_{15, 2}}} \le \epsilon \int_{0}^{t} ||\nabla(p - p_{\delta})||_{L^{2}({\Omega}^\delta_{b, \delta}(s))}^{2} + C(\epsilon) \int_{0}^{t} ||\partial_{t} \bd{\eta} - \partial_{t} \bd{\eta}_{\delta}||^{2}_{L^{2}(\Omega_{b})}.
\end{equation*}
}}

From here we get the final estimate of $T_{15}$:
{{
\begin{align*}
|T_{15}| &\le \epsilon \int_{0}^{t} ||\nabla(p - p_{\delta})||_{L^{2}({\Omega}^\delta_{b, \delta}(s))}^{2} + C(\epsilon) \int_{0}^{t} ||\nabla \bd{\eta} - \nabla {\bd{\eta}}^\delta||_{L^{2}(\Omega_{b})}^{2} \\
&+ C(\epsilon) \left(\int_{0}^{t} ||\nabla \bd{\eta} - \nabla \bd{\eta}_{\delta}||_{L^{2}(\Omega_{b})}^{2} + \int_{0}^{t} ||\omega - \omega_{\delta}||^{2}_{H^{2}(\Gamma)} + \int_{0}^{t} ||\partial_{t} \bd{\eta} - \partial_{t} \bd{\eta}_{\delta}||^{2}_{L^{2}(\Omega_{b})}\right).
\end{align*}
}}

\vskip 0.1in
\noindent
{\bf{Term T16.}} To estimate $T_{16}$ defined in \eqref{T16} we start by passing to the limit as $\nu \to 0$ to obtain
{{
\begin{equation*}
T_{16} = -\aalpha \int_{0}^{t} \int_{\Gamma(s)} (\bd{\xi} \cdot \bd{n}) (p - p_{\delta}) + \aalpha \int_{0}^{t} \int_{{\Gamma}^\delta_{\delta}(s)} (\bd{\xi}_{\delta} \cdot {\bd{n}}^\delta_{\delta}) (p - p_{\delta}),
\end{equation*}
}}
where ${\bd{n}}^\delta_{\delta}$ is the upward pointing normal vector to ${\Gamma}^\delta_{\delta}(t)$. We integrate by parts to obtain that {{$|T_{16}| \le R_{16, 1} + R_{16, 2}$}}, where
{{
\begin{equation*}
{{R_{16, 1}}} := \aalpha \left|\int_{0}^{t} \int_{\Omega_{b}(s)} (\nabla \cdot \bd{\xi}) (p - p_{\delta}) - \int_{0}^{t} \int_{{\Omega}^\delta_{b, \delta}(s)} (\nabla \cdot \bd{\xi}_{\delta})(p - p_{\delta})\right|,
\end{equation*}
\begin{equation*}
{{R_{16, 2}}} := \aalpha \left|\int_{0}^{t} \int_{\Omega_{b}(s)} \bd{\xi} \cdot \nabla(p - p_{\delta}) - \int_{0}^{t} \int_{{\Omega}^\delta_{b, \delta}(s)} \bd{\xi}_{\delta} \cdot \nabla(p - p_{\delta})\right|.
\end{equation*}
}}
By using \eqref{nablaeta} and the bootstrap assumption \eqref{grad}, we have that 
\begin{align*}
{{R_{16, 1}}} &= \aalpha \left|\int_{0}^{t} \int_{\Omega_{b}} \mathcal{J}^{\eta}_{b} (\text{tr}(\nabla^{\eta}_{b} \bd{\xi})) (p - p_{\delta}) - \int_{0}^{t} \int_{\Omega_{b}} \mathcal{J}^{{\eta}^\delta_{\delta}}_{b} (\text{tr}(\nabla^{{\eta}^\delta_{\delta}}_{b} \bd{\xi}_{\delta})) (p - p_{\delta}) \right| \\
&= \aalpha \left|\int_{0}^{t} \int_{\Omega_{b}} \text{tr}(\nabla \bd{\xi} \cdot (\bd{I} + \nabla \bd{\eta})^{C}) (p - p_{\delta}) - \int_{0}^{t} \int_{\Omega_{b}} \text{tr}(\nabla \bd{\xi}_{\delta} \cdot (\bd{I} + \nabla {\bd{\eta}}^\delta_{\delta})^{C}) (p - p_{\delta})\right| \\
&\le \aalpha \left|\int_{0}^{t} \int_{\Omega_{b}} \text{tr}(\nabla \bd{\xi} \cdot (\nabla \bd{\eta} - \nabla {\bd{\eta}}^\delta_{\delta})^{C}) (p - p_{\delta})\right| 
+ \aalpha \left|\int_{0}^{t} \int_{\Omega_{b}} \text{tr}(\nabla (\bd{\xi} - \bd{\xi}_{\delta}) \cdot (\bd{I} + \nabla {\bd{\eta}}^\delta_{\delta})^{C}) (p - p_{\delta})\right| \\
&\le C \int_{0}^{t} ||\nabla \bd{\eta} - \nabla {\bd{\eta}}^\delta_{\delta}||_{L^{2}(\Omega_{b})} \cdot ||p - p_{\delta}||_{L^{2}(\Omega_{b})} + C\int_{0}^{t} ||\nabla \bd{\xi} - \nabla \bd{\xi}_{\delta}||_{L^{2}(\Omega_{b})} \cdot ||p - p_{\delta}||_{L^{2}(\Omega_{b})}.
\end{align*}
For {{$R_{16, 2}$,}} we compute 
\begin{align*}
{{R_{16, 2}}} &= \alpha \left|\int_{0}^{t} \int_{\Omega_{b}} \bd{\xi} \cdot \left[\nabla(p - p_{\delta}) \cdot (\bd{I} + \nabla \bd{\eta})^{C}\right] - \int_{0}^{t} \int_{\Omega_{b}} \bd{\xi}_{\delta} \cdot \left[\nabla(p - p_{\delta}) \cdot (\bd{I} + \nabla {\bd{\eta}}^\delta_{\delta})^{C}\right] \right| \\
&\le \alpha \left|\int_{0}^{t} \int_{\Omega_{b}} \bd{\xi} \cdot \left[\nabla(p - p_{\delta}) \cdot (\nabla \bd{\eta} - \nabla {\bd{\eta}}^\delta_{\delta})^{C}\right]\right| + \alpha \left|\int_{0}^{t} \int_{\Omega_{b}} (\bd{\xi} - \bd{\xi}_{\delta}) \cdot \left[\nabla(p - p_{\delta}) \cdot (\bd{I} + \nabla {\bd{\eta}}^\delta_{\delta})^{C}\right] \right| \\
&\le C \int_{0}^{t} ||\nabla p - \nabla p_{\delta}||_{L^{2}(\Omega_{b})} \cdot ||\nabla \bd{\eta} - \nabla {\bd{\eta}}^\delta_{\delta}||_{L^{2}(\Omega_{b})} + C \int_{0}^{t} ||\bd{\xi} - \bd{\xi}_{\delta}||_{L^{2}(\Omega_{b})} \cdot ||\nabla p - \nabla p_{\delta}||_{L^{2}(\Omega_{b})}.
\end{align*}
By the convolution inequality \eqref{convolutionineq} and the previous estimate on the gradient of the pressure \eqref{gradprescompare}, we conclude that 
{{
\begin{multline*}
|T_{16}| \le \epsilon\left(\int_{0}^{t} ||\nabla \bd{\xi} - \nabla \bd{\xi}_{\delta}||^{2}_{L^{2}(\Omega_{b})} + \int_{0}^{t} ||\nabla p - \nabla p_{\delta}||_{L^{2}({\Omega}^\delta_{b, \delta}(s))}^{2}\right) + C(\epsilon)\left(\int_{0}^{t} ||p - p_{\delta}||^{2}_{L^{2}(\Omega_{b})} \right. \\
\left. + \int_{0}^{t} ||\nabla \bd{\eta} - \nabla {\bd{\eta}}^\delta||_{L^{2}(\Omega_{b})}^{2} + \int_{0}^{t} ||\nabla \bd{\eta} - \nabla \bd{\eta}_{\delta}||_{L^{2}(\Omega_{b})}^{2} + \int_{0}^{t} ||\omega - \omega_{\delta}||_{H^{2}(\Gamma)}^{2} + \int_{0}^{t} ||\bd{\xi} - \bd{\xi}_{\delta}||^{2}_{L^{2}(\Omega_{b})}\right).
\end{multline*}
}}

\vskip 0.1in
\noindent
{\bf{Term T17.}} To estimate term $T_{17}$ defined in \eqref{T17} we use \eqref{nablaeta} to compute
\begin{align*}
T_{17} &= \kappa \int_{0}^{t} \int_{\Omega_{b}} \mathcal{J}^{\eta}_{b} \nabla^{\eta}_{b} p \cdot \nabla^{\eta}_{b} (p - p_{\delta}) - \kappa \int_{0}^{t} \int_{\Omega_{b}} \mathcal{J}^{{\eta}^\delta_{\delta}}_{b} \nabla^{{\eta}^\delta_{\delta}}_{b} p_{\delta} \cdot \nabla^{{\eta}^\delta_{\delta}}_{b} (p - p_{\delta}) \\
&= \kappa \int_{0}^{t} \int_{\Omega_{b}} \mathcal{J}^{{\eta}^\delta_{\delta}}_{b} \nabla^{{\eta}^\delta_{\delta}}_{b} (p - p_{\delta}) \cdot \nabla^{{\eta}^\delta_{\delta}}_{b}(p - p_{\delta}) + I_{1} + I_{2} = \kappa \int_{0}^{t} \int_{{\Omega}^\delta_{b, \delta}(t)} |\nabla(p - p_{\delta})|^{2} + {{R_{17, 1} + R_{17, 2}}},
\end{align*}
where
\begin{align*}
{{R_{17, 1}}} &= \kappa \int_{0}^{t} \int_{\Omega_{b}} \mathcal{J}^{\eta}_{b} \nabla^{\eta}_{b} p \cdot \nabla^{\eta}_{b}(p - p_{\delta}) - \kappa \int_{0}^{t} \int_{\Omega_{b}} \mathcal{J}^{{\eta}^\delta_{\delta}}_{b} \nabla^{\eta}_{b} p \cdot \nabla^{{\eta}^\delta_{\delta}}_{b} (p - p_{\delta}),
\\
{{R_{17, 2}}} &= \kappa \int_{0}^{t} \int_{\Omega_{b}} \mathcal{J}^{{\eta}^\delta_{\delta}}_{b} (\nabla^{\eta}_{b} p - \nabla^{{\eta}^\delta_{\delta}}_{b} p) \cdot \nabla^{{\eta}^\delta_{\delta}}_{b}(p - p_{\delta}).
\end{align*}

To estimate {{$R_{17, 1}$,}} we use \eqref{nablaeta} to obtain
\begin{equation*}
{{R_{17, 1}}} = \kappa \int_{0}^{t} \int_{\Omega_{b}} \nabla^{\eta}_{b} p \cdot \Big(\nabla(p - p_{\delta}) \cdot \left[(\bd{I} + \nabla \bd{\eta})^{C} - (\bd{I} + \nabla {\bd{\eta}}^\delta_{\delta})^{C}\right]\Big). 
\end{equation*}
Because $\bd{\eta}$ is smooth, $|\nabla^{\eta}_{b} p| \le C$ uniformly in space and time. Therefore,
\begin{equation*}
{{|R_{17, 1}|}} \le C \int_{0}^{t} ||\nabla (p - p_{\delta})||_{L^{2}(\Omega_{b})} \cdot ||(\nabla \bd{\eta} - \nabla {\bd{\eta}}^\delta_{\delta})^{C}||_{L^{2}(\Omega_{b})}.
\end{equation*}
Using the estimate in \eqref{gradprescompare}, we obtain the desired estimate that
{{
\begin{equation*}
{{|R_{17, 1}|}} \le \epsilon \int_{0}^{t} ||\nabla(p - p_{\delta})||_{L^{2}({\Omega}^\delta_{b, \delta}(s))}^{2} + C(\epsilon) \int_{0}^{t} ||\nabla \bd{\eta} - \nabla {\bd{\eta}}^\delta_{\delta}||^{2}_{L^{2}(\Omega_{b})}.
\end{equation*}
}}

To estimate {{$R_{17, 2}$}}, we use the bootstrap assumption \eqref{grad} that there exists a constant $C$ (independent of $\delta$) such that $|\nabla {\bd{\eta}}^\delta_{\delta}| \le C$ pointwise for $t \in [0, T_{\delta}]$. Therefore, $|(\bd{I} + \nabla \bd{\eta}_{\delta})^{C}|$ is  pointwise uniformly bounded in space and time on the time interval $[0, T_{\delta}]$. Thus, by \eqref{nablaeta},
\begin{equation*}
{{R_{17, 2}}} = \kappa \int_{0}^{t} \int_{\Omega_{b}} (\nabla^{\eta}_{b} p - \nabla^{{\eta}^\delta_{\delta}}_{b} p) \cdot \left[\nabla(p - p_{\delta}) \cdot (\bd{I} + \nabla {\bd{\eta}}^\delta_{\delta})^{C}\right]
\end{equation*}
and hence
\begin{equation*}
{{|R_{17, 2}|}} \le C \int_{0}^{t} ||\nabla^{\eta}_{b} p - \nabla^{{\eta}^\delta_{\delta}}_{b} p||_{L^{2}(\Omega_{b})} \cdot ||\nabla(p - p_{\delta})||_{L^{2}(\Omega_{b})}.
\end{equation*}
We estimate the first pressure term by using  \eqref{nablaeta} to obtain
\begin{align*}
||\nabla^{\eta}_{b} p - \nabla^{{\eta}^\delta_{\delta}}_{b} p||_{L^{2}(\Omega_{b})}^{2} &= \int_{\Omega_{b}} \left|\nabla p \cdot \left[(\bd{I} + \nabla \bd{\eta})^{-1} - (\bd{I} + \nabla {\bd{\eta}}^\delta_{\delta})^{-1}\right]\right|^{2} \\
&= \int_{\Omega_{b}} \left|\nabla p \cdot (\bd{I} + \nabla {\bd{\eta}}^\delta_{\delta})^{-1} [(\bd{I} + \nabla {\bd{\eta}}^\delta_{\delta})(\bd{I} + \nabla \bd{\eta})^{-1} - \bd{I}]\right|^{2} \\
&= \int_{\Omega_{b}} \left|\nabla p \cdot (\bd{I} + \nabla {\bd{\eta}}^\delta_{\delta})^{-1} [(\bd{I} + \nabla {\bd{\eta}}^\delta_{\delta}) - (\bd{I} + \nabla \bd{\eta})] (\bd{I} + \nabla \bd{\eta})^{-1}\right|^{2} \\
&= \int_{\Omega_{b}} \left|\nabla p \cdot (\bd{I} + \nabla {\bd{\eta}}^\delta_{\delta})^{-1} (\nabla {\bd{\eta}}^\delta_{\delta} - \nabla \bd{\eta}) (\bd{I} + \nabla \bd{\eta})^{-1}\right|^{2}.
\end{align*}
Using the fact that $p$ is smooth and the bootstrap assumption \eqref{norm}, we have that
\begin{equation*}
||\nabla^{\eta}_{b} p - \nabla^{{\eta}^\delta_{\delta}}_{b} p||_{L^{2}(\Omega_{b})}^{2} \le C||\nabla {\bd{\eta}}^\delta_{\delta} - \nabla \bd{\eta}||_{L^{2}(\Omega_{b})}^{2}.
\end{equation*}
Therefore, combining this with  \eqref{gradprescompare} we obtain
{{
\begin{equation*}
{{R_{17, 2}}} \le \epsilon \int_{0}^{t} ||\nabla (p - p_{\delta})||^{2}_{L^{2}({\Omega}^\delta_{b, \delta}(s))} + C(\epsilon) \int_{0}^{t} ||\nabla \bd{\eta} - \nabla {\bd{\eta}}^\delta_{\delta}||^{2}_{L^{2}(\Omega_{b})}.
\end{equation*}
}}

The final estimate of $T_{17}$ now follows after the application of the convolution inequality \eqref{convolutionineq}:
{{
\begin{equation*}
T_{17} \le \kappa \int_{0}^{t} \int_{{\Omega}^\delta_{b, \delta}(s)} |\nabla(p - p_{\delta})|^{2} + {{R_{17}}},
\end{equation*}
}}
where the remainder is bounded by 
{{
\begin{align*}
{{|R_{17}|}} &\le \epsilon \int_{0}^{t} ||\nabla(p - p_{\delta})||^{2}_{L^{2}({\Omega}^\delta_{b, \delta}(s))} \\
&+ C(\epsilon) \left(\int_{0}^{t} ||\nabla \bd{\eta} - \nabla {\bd{\eta}}^\delta||_{L^{2}(\Omega_{b})}^{2} + \int_{0}^{t} ||\nabla \bd{\eta} - \nabla \bd{\eta}_{\delta}||_{L^{2}(\Omega_{b})}^{2} + \int_{0}^{t} ||\omega - \omega_{\delta}||_{H^{2}(\Gamma)}^{2} \right).
\end{align*}
}}

\vskip 0.1in
\noindent
{\bf{Term 18.}} Here want to estimate
{{
\begin{align*}
T_{18} &= \int_{0}^{t} \int_{\Gamma(s)} p (\bd{u} - \bd{\xi}) \cdot \bd{n} - \int_{0}^{t} \int_{\Gamma(s)} p (\bd{u}_{\delta} - \bd{\xi}_{\delta}) \cdot \bd{n} 
- \int_{0}^{t} \int_{\Gamma_{\delta}(s)} p_{\delta}(\bd{u} - \bd{\xi}) \cdot \bd{n}_{\delta} + \int_{0}^{t} \int_{\Gamma_{\delta}(s)} p_{\delta}(\bd{u}_{\delta} - \bd{\xi}_{\delta}) \cdot \bd{n}_{\delta} \\
&- \int_{0}^{t} \int_{\Gamma(s)} ((\bd{u} - \bd{\xi}) \cdot \bd{n}) (p - p_{\delta}) + \int_{0}^{t} \int_{\Gamma_{\delta}(s)} ((\bd{u}_{\delta} - \bd{\xi}_{\delta}) \cdot \bd{n}_{\delta}) (p - p_{\delta}) \\
&= -\int_{0}^{t} \int_{\Gamma(s)} p (\bd{u}_{\delta} - \bd{\xi}_{\delta}) \cdot \bd{n} - \int_{0}^{t} \int_{\Gamma_{\delta}(s)} p_{\delta}(\bd{u} - \bd{\xi}) \cdot \bd{n}_{\delta} 
+ \int_{0}^{t} \int_{\Gamma(s)} ((\bd{u} - \bd{\xi}) \cdot \bd{n}) p_{\delta} + \int_{0}^{t} \int_{\Gamma_{\delta}(s)} ((\bd{u}_{\delta} - \bd{\xi}_{\delta}) \cdot \bd{n}_{\delta}) p.
\end{align*}
}}
By mapping all of the integrals back to the reference domain $\Gamma$, we obtain
{{
\begin{align*}
T_{18} &= - \int_{0}^{t} \int_{\Gamma} p(\bd{u}_{\delta} - \bd{\xi}_{\delta}) \cdot (-\partial_{x} \omega, 1) - \int_{0}^{t} \int_{\Gamma} p_{\delta}(\bd{u} - \bd{\xi}) \cdot (-\partial_{x} \omega_{\delta}, 1) \\
&+ \int_{0}^{t} \int_{\Gamma} p_{\delta} (\bd{u} - \bd{\xi}) \cdot (-\partial_{x} \omega, 1) + \int_{0}^{t} \int_{\Gamma} p (\bd{u}_{\delta} - \bd{\xi}_{\delta}) \cdot (-\partial_{x}\omega_{\delta}, 1) \\
&= \int_{0}^{t} \int_{\Gamma} p (\bd{u}_{\delta} - \bd{\xi}_{\delta}) \cdot \bd{e}_{x} (\partial_{x}\omega - \partial_{x}\omega_{\delta}) - \int_{0}^{t} \int_{\Gamma} p_{\delta}(\bd{u} - \bd{\xi}) \cdot \bd{e}_{x} (\partial_{x} \omega - \partial_{x}\omega_{\delta}) \\
&= -\int_{0}^{t} \int_{\Gamma} p[(\bd{u} - \bd{\xi}) \cdot \bd{e}_{x} - (\bd{u}_{\delta} - \bd{\xi}_{\delta}) \cdot \bd{e}_{x}] (\partial_{x} \omega - \partial_{x} \omega_{\delta}) + \int_{0}^{t} \int_{\Gamma} (p - p_{\delta}) (\bd{u} - \bd{\xi}) \cdot \bd{e}_{x} (\partial_{x}\omega - \partial_{x}\omega_{\delta}).
\end{align*}
}}
The absolute value is bounded as follows:
{{
\begin{multline*}
\left|\int_{0}^{t} \int_{\Gamma} p[(\bd{u} - \bd{\xi}) \cdot \bd{e}_{x} - (\bd{u}_{\delta} - \bd{\xi}_{\delta}) \cdot \bd{e}_{x}] (\partial_{x} \omega - \partial_{x} \omega_{\delta})\right| + \left|\int_{0}^{t} \int_{\Gamma} (p - p_{\delta}) (\bd{u} - \bd{\xi}) \cdot \bd{e}_{x} (\partial_{x}\omega - \partial_{x}\omega_{\delta})\right| \\
\le C\left(\int_{0}^{t} ||(\bd{u} - \bd{\xi}) \cdot \bd{e}_{x} - (\bd{u}_{\delta} - \bd{\xi}_{\delta}) \cdot \bd{e}_{x}||_{L^{2}(\Gamma)} ||\partial_{x}\omega - \partial_{x}\omega_{\delta}||_{L^{2}(\Gamma)} + \int_{0}^{t} ||p - p_{\delta}||_{L^{2}(\Gamma)} ||\partial_{x}\omega - \partial_{x}\omega_{\delta}||_{L^{2}(\Gamma)}\right).
\end{multline*}
}}
After the application of the trace theorem, Poincare's inequality, and Korn's inequality we obtain the final estimate:
{{
\begin{align*}
|T_{18}| &\le \epsilon \left(\int_{0}^{t} ||\bd{D}(\widehat{\bd{u}} - \bd{u}_{\delta})||^{2}_{L^{2}(\Omega_{f,\delta}(s))} + \int_{0}^{t} ||\nabla (\bd{\xi} - \bd{\xi}_{\delta})||^{2}_{L^{2}(\Omega_{b})} + \int_{0}^{t} ||\nabla(p - p_{\delta})||^{2}_{L^{2}({\Omega}^\delta_{b, \delta}(s))}\right) \\
&+ C(\epsilon) \int_{0}^{t} ||\omega - \omega_{\delta}||^{2}_{H^{2}(\Gamma)}.
\end{align*}
}}

\end{proof}

{{\subsection{Generalized Aubin-Lions Compactness Theorem \cite{AubinLions}}\label{appendix2}
To help the reader follow the results from Section~\ref{VelConv} we state here the Generalized Aubin-Lions Compactness Theorem, i.e., Theorem 3.1 of \cite{AubinLions}.

\begin{theorem}\label{Compactness} {\rm{\bf{(The Generalized Aubin-Lions Compactness Theorem)}}}
Let $V$ and $H$ be Hilbert spaces such that  $V\subset\subset H$.
Suppose that 
$\{\bu_\dt\} \subset L^2(0,T;H)$ is a sequence such that
$
\bu_\dt (t,\cdot)=\bu_\dt^n(\cdot) \ {\rm on}\; ((n-1)\dt,n\dt],\;n=1,\dots,N,
$
with $N\Delta t = T$.
Let $V_\dt^n$ and $Q_\dt^n$ be Hilbert spaces such that $(V_\dt^n,Q_\dt^n)\hookrightarrow V\times V$,  
where the embeddings are uniformly continuous w.r.t. $\dt$ and $n$,
and $V_\dt^n \subset \subset \overline{Q_\dt^n}^H \hookrightarrow (Q_\dt^n)'$.
Let $\bu_\dt^n \in V_\dt^n$, $n = 1,\dots,N$.
%Furthermore, let $Q_\dt^n \hookrightarrow V$ be such that $V_\dt^n = \overline{Q_\dt^n}^V$, and let $H^n_\dt=\overline{Q^n_\dt}^H$.
If the following is true: 
\begin{itemize}
\item[(A)] There exists a universal constant $C>0$ such that for every $\dt$ 
\begin{enumerate}
\item[A1.] 
$
\sum_{n=1}^N\|\bu^n_\dt\|_{V^n_\dt}^2\dt\leq C,
$
\item[A2.] 
$
\|\bu_\dt\|_{L^\infty(0,T;H)}\leq C,
$
\item[A3.]
$
\|\tau_\dt\bu_\dt-\bu_\dt\|^2_{L^2(\dt,T;H)}\leq C\dt.
$

\end{enumerate}

\item[(B)] 
There exists a universal constant $C>0$ such that
$$
\|P^{n}_\dt\frac{\bu^{n+1}_\dt-\bu^{n}_\dt}{\dt}\|_{(Q^n_\dt)'}\leq C(\|\bu^{n+1}_\dt\|_{V^{n+1}_\dt}+1),\; n=0,\dots,N-1,
$$
where $P^n_\dt$ is the orthogonal projector onto  $\overline{Q^n_\dt}^H$.

\item[(C)]\label{FuncSp} The function spaces $Q_\dt^n$ and $V_\dt^n$ depend smoothly on time in the following sense:
\begin{enumerate}
\item[C1.] 
For every $\dt > 0$, and for every $l\in \{1,\dots,N\}$ and $n\in\{1,\dots,N-l\}$,  
there exists a space $Q^{n,l}_\dt\subset V$ and the operators 
${{J}}^i_{\Delta t,l,n}:Q^{n,l}_\dt\to Q^{n+i}_\dt, i=0,1,\dots,l,$ such that
$\|J^i_{\Delta t,l,n}\bq\|_{Q^{n+i}_\dt}\leq C\|\bq\|_{Q^{n,l}_\dt}, \ \forall \bq\in Q^{n,l}_\dt$, and
%and a universal constant $C>0$ such that for every $\bq\in Q^{n,l}_\dt$  the following holds:
\begin{equation}\label{Ji1}
\Big ((J^{j+1}_{\Delta t,l,n}\bq-J^j_{\Delta t,l,n}\bq),\bu^{n+j+1}_\dt\Big )_H\leq C \dt\|\bq\|_{Q^{n,l}_\dt}\|\bu^{n+j+1}_\dt\|_{V^{n+j+1}_\dt},\quad j\in \{0,\dots,l-1\},
\end{equation}
\begin{equation}\label{Ji2}
\|J^i_{\Delta t,l,n}\bq-\bq\|_H\leq C \sqrt{l\dt}\|\bq\|_{Q^{n,l}_\dt},\quad i\in \{0,\dots,l\},
\end{equation}
where $C > 0$ is independent of $\dt,n$ and $l$.

\item[C2.] 
Let $V^{n,l}_\dt = \overline{Q^{n,l}_\dt}^V$. 
There exist the functions $I^i_{\dt,l,n}:V^{n+i}_\dt\to V^{n,l}_\dt, \  i=0,1,\dots,l,$
and a universal constant $C>0$, such that for every $\bv\in V^{n+i}_\dt$ 
\begin{equation}\label{C21}
\|I^i_{\dt,l,n} \bv\|_{V^{n,l}_\dt}\leq C\|\bv\|_{V^{n+i}_\dt}, \quad i\in \{0,\dots,l\},
\end{equation}
\begin{equation}\label{C22}
\|I^i_{\dt,l,n} \bv-\bv\|_H\leq  g(l\dt)\|\bv\|_{V^{n+i}_\dt},\quad i\in \{0,\dots,l\},
\end{equation}
where $g:\R_+\to\R_+$ is a universal, monotonically increasing function such that $g(h)\to 0$ as $h\to 0$.

\item[C3.] {\rm{Uniform Ehrling property:}}
%Let  $H^{n,l}_\dt = \overline{Q_\dt^{n,l}}^H$. TODO: zasto ovo treba biti tu? Nema ga nigdje u iskazu.
%Then,  $\bv\in V^{n,l}_\dt$ the following uniform Ehrling property holds:
%
For every $\delta>0$ there exists a constant $C(\delta)$ independent of $n,l$ and $\dt$, such that
\begin{equation}\label{Ehrling}
\|\bv\|_{H}\leq \delta\|\bv\|_{V^{n,l}_\dt}+C(\delta)\|\bv\|_{(Q^{n,l}_\dt)'},\quad \bv\in V^{n,l}_\dt.
\end{equation}
\end{enumerate}
\end{itemize}
then $\{\bu_{\Delta t}\}$ is relatively compact in $L^2(0,T;H)$.
\end{theorem}

}}
\bibliographystyle{plain}
%\bibliography{}
\bibliography{FPSIBibliography}

\begin{thebibliography}{10}

\bibitem{Adams}
R.~A. Adams.
\newblock {\em Sobolev spaces}, volume~65 of {\em Pure and Applied
  Mathematics}.
\newblock Academic Press, New York-London, 1975.

\bibitem{AEN19}
I.~Ambartsumyan, V.~J. Ervin, T.~Nguyen, and I.~Yotov.
\newblock A nonlinear {S}tokes-{B}iot model for the interaction of a
  non-{N}ewtonian fluid with poroelastic media.
\newblock {\em ESAIM Math. Model. Numer. Anal.}, 53(6):1915--1955, 2019.

\bibitem{QuainiQuarteroniPoroelastic}
S.~Badia, A.~Quaini, and A.~Quarteroni.
\newblock Coupling {B}iot and {N}avier-{S}tokes equations for modelling
  fluid-poroelastic media interaction.
\newblock {\em Journal of Computational Physics}, 228(21):7986--8014, 2009.

\bibitem{barGruLasTuff2}
V.~Barbu, Z.~Gruji\'{c}, I.~Lasiecka, and A.~Tuffaha.
\newblock Existence of the energy-level weak solutions for a nonlinear
  fluid-structure interaction model.
\newblock In {\em Fluids and waves}, volume 440 of {\em Contemp. Math.}, pages
  55--82. Amer. Math. Soc., Providence, RI, 2007.

\bibitem{BarGruLasTuff}
V.~Barbu, Z.~Gruji\'{c}, I.~Lasiecka, and A.~Tuffaha.
\newblock Smoothness of weak solutions to a nonlinear fluid-structure
  interaction model.
\newblock {\em Indiana Univ. Math. J.}, 57(3):1173--1207, 2008.

\bibitem{Biotwell1}
H.~Barucq, M.~Madaune-Tort, and P.~Saint-Macary.
\newblock Theoretical aspects of wave propagation for {B}iot's consolidation
  problem.
\newblock {\em Monografías del Seminario Matemático García de Galdeano},
  31:449--458, 2004.

\bibitem{Biotwell2}
H.~Barucq, M.~Madaune-Tort, and P.~Saint-Macary.
\newblock On nonlinear {B}iot's consolidation models.
\newblock {\em Nonlinear Anal.}, 63:e985--e995, 2005.

\bibitem{BdV1}
H.~Beir\~{a}o~da Veiga.
\newblock On the existence of strong solutions to a coupled fluid-structure
  evolution problem.
\newblock {\em J. Math. Fluid Mech.}, 6(1):21--52, 2004.

\bibitem{Biot1}
M.~A. Biot.
\newblock General theory of three-dimensional consolidation.
\newblock {\em J. Appl. Phys.}, 12(2):155--164, 1941.

\bibitem{Biot2}
M.~A. Biot.
\newblock Theory of elasticity and consolidation for a porous anisotropic
  solid.
\newblock {\em J. Appl. Phys.}, 26(2):182--185, 1955.

\bibitem{Biotwell7}
L.~Bociu, G.~Guidoboni, R.~Sacco, and J.~T. Webster.
\newblock Analysis of nonlinear poro-elastic and poro-visco-elastic models.
\newblock {\em Arch. Ration. Mech. Anal.}, 222(3):1445--1519, 2016.

\bibitem{BiotWell22}
L.~Bociu, B.~Muha, and J.~T. Webster.
\newblock Weak solutions in nonlinear poroelasticity with incompressible
  constituents.
\newblock {\em Nonlinear Anal. Real World Appl.}, 67:Paper No. 103563, 22,
  2022.

\bibitem{BiotWell23}
L.~Bociu, B.~Muha, and J.~T. Webster.
\newblock Mathematical effects of linear visco-elasticity in quasi-static
  {B}iot models.
\newblock {\em J. Math. Anal. Appl.}, 527(2):Paper No. 127462, 2023.

\bibitem{BCMW21}
L.~Bociu, S.~\v{C}ani\'{c}, B.~Muha, and J.~T. Webster.
\newblock Multilayered poroelasticity interacting with {S}tokes flow.
\newblock {\em SIAM J. Math. Anal.}, 53(6):6243--6279, 2021.

\bibitem{Biotwell8}
L.~Bociu and J.~T. Webster.
\newblock Nonlinear quasi-static poroelasticity.
\newblock {\em J. Differential Equations}, 296:242--278, 2021.

\bibitem{Brenner}
S.~C. Brenner and L.~R. Scott.
\newblock {\em The Mathematical Theory of Finite Element Methods}, volume~15 of
  {\em Texts in Applied Mathematics}.
\newblock Springer Science+Business Media, LLC, New York, third edition, 2008.

\bibitem{BukacZuninoYotovPoroelastic}
M.~Bukac, P.~Zunino, and I.~Yotov.
\newblock Explicit partitioning strategies for the interaction between a fluid
  and a multilayered poroelastic structure: an operator-splitting approach.
\newblock {\em Journal of Computational Physics}, 228(21):7986--8014, 2013.

\bibitem{YifanDES}
S.~Canic, Y.~Wang, and M.~Buka\v{c}.
\newblock A next-generation mathematical model for drug eluting stents.
\newblock {\em {SIAM} J. Appl. Math.}, 81(4):1503--1529, 2021.

\bibitem{CGH14}
P.~Causin, G.~Guidoboni, A.~Harris, D.~Prada, R.~Sacco, and S.~Terragni.
\newblock A poroelastic model for the perfusion of the lamina cribrosa in the
  optic nerve head.
\newblock {\em Math. Biosci.}, 257:33--41, 2014.

\bibitem{Ces17}
A.~Cesmelioglu.
\newblock Analysis of the coupled {N}avier-{S}tokes/{B}iot problem.
\newblock {\em J. Math. Anal. Appl.}, 456(2):970--991, 2017.

\bibitem{CDEM}
A.~Chambolle, B.~Desjardins, M.~J. Esteban, and C.~Grandmont.
\newblock Existence of weak solutions for the unsteady interaction of a viscous
  fluid with an elastic plate.
\newblock {\em J. Math. Fluid Mech.}, 7(3):368--404, 2005.

\bibitem{WeakStrongRigid1}
N.~V. Chemetov, \v{S}. Ne\v{c}asov\'{a}, and B.~Muha.
\newblock Weak-strong uniqueness for fluid-rigid body interaction problem with
  slip boundary condition.
\newblock {\em J. Math. Phys.}, 60(1):011505, 13, 2019.

\bibitem{ChengShkollerCoutand}
C.~H.~A. Cheng, D.~Coutand, and S.~Shkoller.
\newblock Navier-{S}tokes equations interacting with a nonlinear elastic
  biofluid shell.
\newblock {\em SIAM J. Math. Anal.}, 39(3):742--800, 2007.

\bibitem{ChenShkoller}
C.~H.~A. Cheng and S.~Shkoller.
\newblock The interaction of the 3{D} {N}avier-{S}tokes equations with a moving
  nonlinear {K}oiter elastic shell.
\newblock {\em SIAM J. Math. Anal.}, 42(3):1094--1155, 2010.

\bibitem{Ciarlet}
P.~G. Ciarlet.
\newblock {\em Mathematical Elasticity Volume I: Three-Dimensional Elasticity},
  volume~20 of {\em Studies in Mathematics and Its Applications}.
\newblock Elsevier Science Publishers B.V., Amsterdam, 1988.

\bibitem{CioranescuBook}
D.~Cioranescu and J.~Saint~Jean Paulin.
\newblock {\em Homogenization of reticulated structures}, volume 136 of {\em
  Applied Mathematical Sciences}.
\newblock Springer-Verlag, New York, 1999.

\bibitem{DynamicPressure}
C.~Conca, F.~Murat, and O.~Pironneau.
\newblock The {S}tokes and {N}avier-{S}tokes equations with boundary conditions
  involving the pressure.
\newblock {\em Japan. J. Math.}, 20(2):279--318, 1994.

\bibitem{CSS1}
D.~Coutand and S.~Shkoller.
\newblock Motion of an elastic solid inside an incompressible viscous fluid.
\newblock {\em Arch. Ration. Mech. Anal.}, 176(1):25--102, 2005.

\bibitem{CSS2}
D.~Coutand and S.~Shkoller.
\newblock The interaction between quasilinear elastodynamics and the
  {N}avier-{S}tokes equations.
\newblock {\em Arch. Ration. Mech. Anal.}, 179(3):303--352, 2006.

\bibitem{NSDarcy}
M.~Discacciati and A.~Quarteroni.
\newblock Navier-{S}tokes/{D}arcy coupling: modeling, analysis, and numerical
  approximation.
\newblock {\em Rev. Mat. Complut.}, 22(2):315--426, 2009.

\bibitem{DreherJungel}
M.~Dreher and A.~J\"{u}ngel.
\newblock Compact families of piecewise constant functions in ${L}^{p}(0, {T};
  {B})$.
\newblock {\em Nonlinear Anal.}, 75(6):3072--3077, 2012.

\bibitem{Evans}
L.~C. Evans.
\newblock {\em Partial differential equations}, volume~19 of {\em Graduate
  Studies in Mathematics}.
\newblock American Mathematical Society, Providence, second edition, 2010.

\bibitem{EvansFine}
L.~C. Evans and R.~F. Gariepy.
\newblock {\em Measure theory and fine properties of functions}.
\newblock Studies in Advanced Mathematics. CRC Press, Boca Raton, FL, 1992.

\bibitem{GWG15}
V.~Girault, M.~F. Wheeler, B.~Ganis, and M.~E. Mear.
\newblock A lubrication fracture model in a poro-elastic medium.
\newblock {\em Math. Models Methods Appl. Sci.}, 25(4):587--645, 2015.

\bibitem{glowinski2003finite}
R.~Glowinski.
\newblock {\em Finite element methods for incompressible viscous flow, in:
  {P}.{G}.{C}iarlet, {J}.-{L}.{L}ions ({E}ds), {H}andbook of numerical
  analysis}, volume~9.
\newblock North-{H}olland, {A}msterdam, 2003.

\bibitem{CG}
C.~Grandmont.
\newblock Existence of weak solutions for the unsteady interaction of a viscous
  fluid with an elastic plate.
\newblock {\em SIAM J. Math. Anal.}, 40(2):716--737, 2008.

\bibitem{Grandmont16}
C.~Grandmont and M.~Hillairet.
\newblock Existence of global strong solutions to a beam-fluid interaction
  system.
\newblock {\em Arch. Ration. Mech. Anal.}, 220(3):1283--1333, 2016.

\bibitem{FSIforBIO_Lukacova}
C.~Grandmont, M.~Luk\'{a}\v{c}ov\'{a}-Medvid'ov\'{a}, and \v{S}.
  Ne\v{c}asov\'{a}.
\newblock Mathematical and numerical analysis of some {F}{S}{I} problems.
\newblock In T.~Bodn\'{a}r, G.~P. Galdi, and \v{S}. Ne\v{c}asov\'{a}, editors,
  {\em Fluid-structure interaction and biomedical applications}, Advances in
  Mathematical Fluid Mechanics, pages 1--77. Birkh\"{a}user, 2014.

\bibitem{IgnatovaKukavica}
M.~Ignatova, I.~Kukavica, I.~Lasiecka, and A.~Tuffaha.
\newblock On well-posedness for a free boundary fluid-structure model.
\newblock {\em J. Math. Phys.}, 53(11):115624, 13, 2012.

\bibitem{ignatova2014well}
M.~Ignatova, I.~Kukavica, I.~Lasiecka, and A.~Tuffaha.
\newblock On well-posedness and small data global existence for an interface
  damped free boundary fluid-structure model.
\newblock {\em Nonlinearity}, 27(3):467--499, 2014.

\bibitem{InoueWakimoto}
A.~Inoue and M.~Wakimoto.
\newblock On existence of solutions of the {N}avier-{S}tokes equation in a time
  dependent domain.
\newblock {\em J. Fac. Sci. Univ. Tokyo Sect. IA Math.}, 24(2):303--319, 1977.

\bibitem{JM96}
W.~J\"{a}ger and A.~Mikeli\'{c}.
\newblock On the boundary conditions at the contact interface between a porous
  medium and a free fluid.
\newblock {\em Ann. Scuola Norm. Sup. Pisa Cl. Sci. (4)}, 23(3):403--465, 1996.

\bibitem{JM00}
W.~J\"{a}ger and A.~Mikeli\'{c}.
\newblock On the interface boundary condition of {B}eavers, {J}oseph, and
  {S}affman.
\newblock {\em SIAM J. Appl. Math.}, 60(4):1111--1127, 2000.

\bibitem{NonlinearFPSI_CRM23}
J.~Kuan, S.~\v{C}ani\'{c}, and B.~Muha.
\newblock Existence of a weak solution to a regularized moving boundary
  fluid-structure interaction problem with poroelastic media.
\newblock {\em Comptes Rendus M\'{e}canique}, 351(S1):1--30, 2023.

\bibitem{Kuk}
I.~Kukavica and A.~Tuffaha.
\newblock Solutions to a fluid-structure interaction free boundary problem.
\newblock {\em DCDS-A}, 32(4):1355--1389, 2012.

\bibitem{KukavicaTuffahaZiane}
I.~Kukavica, A.~Tuffaha, and M.~Ziane.
\newblock Strong solutions for a fluid structure interaction system.
\newblock {\em Adv. Differential Equations}, 15(3-4):231--254, 2010.

\bibitem{LengererRuzicka}
D.~Lengeler and M.~R{\r u}{\v z}i{\v c}ka.
\newblock Weak solutions for an incompressible {N}ewtonian fluid interacting
  with a {K}oiter type shell.
\newblock {\em Arch. Ration. Mech. Anal.}, 211(1):205--255, 2014.

\bibitem{Lequeurre}
J.~Lequeurre.
\newblock Existence of strong solutions to a fluid-structure system.
\newblock {\em SIAM J. Math. Anal.}, 43(1):389--410, 2011.

\bibitem{LDQ11}
M.~Lesinigo, C.~D'Angelo, and A.~Quarteroni.
\newblock A multiscale {D}arcy-{B}rinkman model for fluid flow in fractured
  porous media.
\newblock {\em Numer. Math.}, 117(4):717--752, 2011.

\bibitem{Mikhailov}
S.~E. Mikhailov.
\newblock Traces, extensions and co-normal derivatives for elliptic systems on
  {L}ipschitz domains.
\newblock {\em J. Math. Anal. Appl.}, 378:324--342, 2011.

\bibitem{BorNote3D}
B.~Muha.
\newblock A note on the {T}race {T}heorem for domains which are locally
  subgraph of a h\"{o}lder continuous function.
\newblock {\em Netw. Heterog. Media}, 9(1):191--196, 2014.

\bibitem{MuhaCanic13}
B.~Muha and S.~{\v C}ani\'c.
\newblock Existence of a weak solution to a nonlinear fluid-structure
  interaction problem modeling the flow of an incompressible, viscous fluid in
  a cylinder with deformable walls.
\newblock {\em Arch. Ration. Mech. Anal.}, 207(3):919--968, 2013.

\bibitem{BorSun3d}
B.~Muha and S.~{\v C}ani\'c.
\newblock A nonlinear, 3{D} fluid-structure interaction problem driven by the
  time-dependent dynamic pressure data: a constructive existence proof.
\newblock {\em Commun. Inf. Syst.}, 13(3):357--397, 2013.

\bibitem{BorSunMultiLayered}
B.~Muha and S.~{\v C}ani\'c.
\newblock Existence of a solution to a fluid-multi-layered-structure
  interaction problem.
\newblock {\em J. Differential Equations}, 256(2):658--706, 2014.

\bibitem{BorSunNonLinearKoiter}
B.~Muha and S.~{\v C}ani\'c.
\newblock Fluid-structure interaction between an incompressible, viscous 3{D}
  fluid and an elastic shell with nonlinear {K}oiter membrane energy.
\newblock {\em Interfaces Free Bound.}, 17(4):465--495, 2015.

\bibitem{BorSunSlip}
B.~Muha and S.~{\v C}ani\'c.
\newblock Existence of a weak solution to a fluid-elastic structure interaction
  problem with the {N}avier slip boundary condition.
\newblock {\em J. Differential Equations}, 260(12):8550--8589, 2016.

\bibitem{AubinLions}
B.~Muha and S.~{\v C}ani\'c.
\newblock A generalization of the {A}ubin-{L}ions-{S}imon compactness lemma for
  problems on moving domains.
\newblock {\em J. Differential Equations}, 266(12):8370--8418, 2019.

\bibitem{Multilayered}
B.~Muha and S.~\v{C}ani\'{c}.
\newblock Existence of a solution to a fluid-multi-layered-structure
  interaction problem.
\newblock {\em J. Differential Equations}, 256(2):658--706, 2014.

\bibitem{WeakStrongRigid2}
B.~Muha, \v{S}. Ne\v{c}asov\'{a}, and A.~Rado\v{s}evi\'{c}.
\newblock A uniqueness result for 3{D} incompressible fluid-rigid body
  interaction problem.
\newblock {\em J. Math. Fluid Mech.}, 23(1):Paper No. 1, 39, 2021.

\bibitem{Biotwell3}
S.~Owczarek.
\newblock A {G}alerkin method for {B}iot consolidation model.
\newblock {\em Math. Mech. Solids}, 15(1):42--56, 2010.

\bibitem{Raymond}
J.-P. Raymond and M.~Vanninathan.
\newblock A fluid-structure model coupling the {N}avier-{S}tokes equations and
  the {L}am\'{e} system.
\newblock {\em J. Math. Pures Appl. (9)}, 102(3):546--596, 2014.

\bibitem{AndrewPartitionedSchemeFPSI}
A.~Scharf, S.~\v{C}ani\'{c}, and Y.~Wang.
\newblock A partitioned scheme for fluid-structure interaction with
  multilayered poroelastic media.
\newblock {\em In draft form.}, 2024.

\bibitem{WeakStrongFSI}
S.~Schwarzacher and M.~Sroczinski.
\newblock Weak-strong uniqueness for an elastic plate interacting with the
  {N}avier-{S}tokes equation.
\newblock {\em S{I}{A}{M} J. Math. Anal.}, 54(4):4104--4138, 2022.

\bibitem{NonlinearFPSI1}
A.~Seboldt, O.~Oyekole, J.~Tamba\v{c}a, and M.~Buka\v{c}.
\newblock Numerical modeling of the fluid-porohyperelastic structure
  interaction.
\newblock {\em SIAM J. Sci. Comput.}, 43(4):A2923--A2948, 2021.

\bibitem{Biotwell4}
R.~E. Showalter.
\newblock Diffusion in poro-elastic media.
\newblock {\em J. Math. Anal. Appl.}, 251(1):310--340, 2000.

\bibitem{Sho05}
R.~E. Showalter.
\newblock Poroelastic filtration coupled to {S}tokes flow.
\newblock In {\em Control theory of partial differential equations}, volume 242
  of {\em Lect. Notes Pure Appl. Math.}, pages 229--241. Chapman \& Hall/CRC,
  Boca Raton, FL, 2005.

\bibitem{Biotwell5}
R.~E. Showalter and N.~Su.
\newblock Partially saturated flow in a poroelastic medium.
\newblock {\em Discrete Contin. Dyn. Syst. Ser. B}, 1(4):403--420, 2001.

\bibitem{CanicLectureNotes}
S.~\v{C}ani\'{c}.
\newblock Fluid-structure interaction with incompressible fluids.
\newblock In L.~C. Berselli and M.~Ru\v zi\v cka, editors, {\em Progress in
  Mathematical Fluid Dynamics}, volume 2272 of {\em Lecture Notes in
  Mathematics}, pages 15--87. Springer, 2020.

\bibitem{Biotwell6}
A.~\v{Z}en\'{i}\v{s}ek.
\newblock The existence and uniqueness theorem in {B}iot's consolidation
  theory.
\newblock {\em Aplikace Matematiky}, 29(3):194--211, 1984.

\bibitem{FluidsCanic}
Y.~Wang, S.~\v{C}ani\'{c}, M.~Buka\v{c}, C.~Blaha, and S.~Roy.
\newblock Mathematical and computational modeling of a poroelastic cell
  scaffold in a bioartificial pancreas.
\newblock {\em Fluids}, 7(7):222, 2022.

\bibitem{YRC14}
J.~Young, B.~Rivi\`ere, Jr. C.~S.~Cox, and K.~Uray.
\newblock A mathematical model of intestinal oedema formation.
\newblock {\em Math. Med. Biol.}, 31(1):1--15, 2014.

\bibitem{NonlinearFPSI2}
R.~Zakerzadeh and P.~Zunino.
\newblock A computational framework for fluid-porous structure interaction with
  large structural deformation.
\newblock {\em Meccanica}, 54:101--121, 2019.

\end{thebibliography}

\end{document}